\documentclass[11pt, reqno]{amsart}
\usepackage{amssymb}
\usepackage{times}
\usepackage{amsmath,amsthm}
\usepackage{amsfonts}
\usepackage{leftidx}
\usepackage{mathrsfs}
\usepackage{enumerate}
\usepackage{abstract}
\usepackage{stmaryrd}
\usepackage{ulem}

\usepackage{float}
\usepackage{hyperref}
\def\R{\mathbb{R}}

\def\d{|\nabla|}
\def\n{\nabla}

\def\p{\partial}

\def\vo{\vspace{1\baselineskip}}

\def\h{\frac{1}{2}}
\def\ta{\thickapprox}
\def\be{\begin{equation}}
\def\ee{\end{equation}}

\newtheorem{theorem}{Theorem}[subsection]
\newtheorem{thm}{Theorem}[section]
\newtheorem{lemma}[theorem]{Lemma}
\newtheorem{proposition}[theorem]{Proposition}
\theoremstyle{definition}
\newtheorem{definition}[theorem]{Definition}

\theoremstyle{remark}
\newtheorem{remark}{Remark}[section]
\numberwithin{equation}{section}
\begin{document}
 \title[3D Finite Depth Capillary waves ]{Global regularity for the 3D finite depth capillary water waves  }
\author{Xuecheng Wang}
\address{Mathematics Department, Princeton University, Princeton, New Jersey,08544, USA}
\email{xuecheng@math.princeton.edu}
\thanks{}
\maketitle
\begin{abstract}

In this paper, we prove global regularity, scattering, and the non-existence of small traveling waves for the  $3D$ finite depth capillary waves system for small initial data. The non-existence of small traveling waves  shows a fundamental difference between  the capillary waves ($\sigma=1, g=0$) and  the gravity-capillary   waves ($\sigma=1,$ $ 0< g< 3$) in the finite depth setting. As, for the later case,    there exists arbitrary small $L^2$ traveling waves.

 Different from the water waves system in the infinite depth setting, the quadratic terms of the same system in the finite depth setting are worse due to the absence of null structure  inside the Dirichlet-Neumann operator.  In the finite depth setting,  the  capillary waves system  has the worst quadratic terms among  the water waves systems with  all possible values of gravity effect constant and surface tension coefficient. It loses favorable cancellations not only in the High $\times$ Low type interaction but also in the High $\times $ High type interaction.

In the worst scenario, the best decay rate  of the nonlinear solution  that one could expect is   $ (1+t) ^ {-1/2} $, because the $3D$ finite depth capillary waves system lacks null structures and there exists   $Q(u, \bar{u})$ type quadratic term, which causes a very large time resonance set and the definite growth of the associated profile.  As a result, the problematic terms  are not only  the quadratic terms  but also the cubic terms.

 Our results and proofs have the following  features and innovations.

\begin{enumerate}
\item[$\bullet$] We are able to control the nonlinear effect of the $Q(u, \bar{u})$ type quadratic term, which  is very delicate. Even for the $2D$ quadratic Schr\"odinger equation, which is a simplified model of the    $3D$ finite depth capillary waves system,  there is no small data global regularity result  when there is no null structure inside the $Q(u, \bar{u})$ type quadratic term.  It is not completely solved even   for the $3D$ quadratic Schr\"odinger equation. On the contrary, when the nonlinearity of quadratic Schr\"odinger equation is $u\bar{u}$, there  exist  finite time blow up solutions for suitably small $L^2$ initial data, see \cite{ikeda}.

 \item[$\bullet$] We introduce a novel method to control the weighted norms, which help us get around a delicate difficulty around the large time resonance set and  further enable us to prove the dispersion estimate for the nonlinear solution. We control two weighted norms at the same time: the lower order weighted norm and the high order weighted norm.  We allow the high order weighted norm to grow appropriately, which helps us to prove that the low order weighted norm doesn't grow over time. Although this idea is not new in the nonlinear wave equations, which have many vector fields commute with the associated linear operator, it usually doesn't work    well for  general dispersive equations, especially when there is no scaling  vector field  available. More weights mean more burdens, which make it more difficult to close the argument.

\item[$\bullet$]   To prove    global regularity for the capillary waves system, we  identify a good variable to prove the dispersion  estimate  and fully exploit  the hidden structures inside the capillary waves  system, which include   the hidden symmetries in both quadratic terms and cubic terms, the conservation law of momentum,   and  the hidden structures of the symbol of  the quadratic terms.

\end{enumerate}

\end{abstract}
\setcounter{tocdepth}{1}
\tableofcontents

\section{Introduction}\label{introduction}

\subsection{The set-up of problem and previous results}
We study the 
 evolution of a constant density inviscid fluid, e.g., water, inside a time dependent domain $\Omega(t)\subset \mathbb{R}^3$, which has a fixed flat bottom $\Sigma$ and a moving interface $\Gamma(t)$. Above the water region $\Omega(t)$ is vacuum. In other words, we only consider the one-fluid problem. We assume that the fluid is irrotational and incompressible. We only consider the effect of surface tension. The effect of gravity is neglected. The problem under consideration is also known as the capillary water waves system.

After normalizing the depth of $\Omega(t)$ to be ``$1$'', we can represents $\Omega(t), \Sigma,$ and $\Gamma(t)$ in the Eulerian coordinates  as follows, 
 \[
\Omega(t):=\{(x,y): x \in \mathbb{R}^2, -1\leq y \leq h(t, x)\}, \]
\[\Gamma(t):= \{(x, h(t, x): x\in \mathbb{R}^2\}, \quad \Sigma:= \{(x,-1): x\in \mathbb{R}^2\},
 \]
where $h(t,x)$ represents the height of interface, which will be a small perturbation of zero.
 
Let $u(t)$ and $p$ denotes the velocity and the pressure of the fluid respectively. Then the evolution of fluid can be described by the free boundary Euler equation  as follows,
\be
\p_t u + u\cdot \nabla u =-\nabla p, \quad \nabla \cdot u =0, \quad \nabla \times u =0, \quad \textup{in}\,\,\, \Omega(t).
\ee
The free surface $\Gamma(t)$ moves with the normal component of the velocity according to the kinematic boundary condition as follows, 
\[
\p_t + u\cdot\nabla \, \textup{ is tangent  to $\cup_{t} \Gamma(t)$}. 
\]
The pressure $p$ satisfies the  Young-Laplace equation as follows,
\[
p= \sigma H(h),\quad \textup{on}\,\, \Gamma(t).
\]
where $\sigma$ denotes the surface tension coefficient, which will be normalized to be ``$1$'', and $H (h) $ represents the mean curvature of the interface, which is given as follows,
\[
H(h)= \nabla\cdot \Big(\frac{\nabla h}{\sqrt{1+|\nabla h|^2}}\Big).
\]
At last, we have the following Neumann type boundary condition on the bottom $\Sigma$, 
\[
u\cdot \vec{\mathbf{n}}=0, \quad  \textup{on}\,\, \Sigma,
\]
which means that the fluid cannot go through the bottom as it is fixed.

Since the velocity field is irrotational, we can represent it in terms of velocity potential $\phi$. Let $\psi$ be the restriction of velocity potential on the boundary $\Gamma(t)$, more precisely, $\psi(t,x):=\phi(t,x,\,h(t,x))$. From divergence free condition and boundary conditions, we can derive the following Laplace equation with two boundary conditions: Neumann type on the bottom and Dirichlet type on the interface,
\begin{equation}\label{harmoniceqn}
(\Delta_x + \p_y^2)\phi=0, \quad \frac{\p \phi}{\p \vec{\mathbf{n}}}\big|_{\Sigma}=0, \quad \phi\big|_{\Gamma(t)} = \psi.
\end{equation}
Hence, we can reduce (for example, see \cite{zakharov}) the motion of fluid to the  evolution of the height ``$h(t, x)$'' and the restricted velocity potential  `` $\psi(t, x)$ '' as follows,
\begin{equation}\label{waterwaves}
\left\{\begin{array}{l}
\p_t h= G(h)\psi,\\
\p_t \psi =    H(h) - \frac{1}{2} |\nabla \psi|^2 + \displaystyle{\frac{(G( h)\psi + \n \,h\cdot\n \psi)^2}{2(1+ |\n\,h|^2)}},
\end{array}\right.
\end{equation}
where $G( h)\psi= \sqrt{1+|\n\,h|^2}\mathcal{N}( h)\psi$  and $\mathcal{N}( h)\psi$ is the Dirichlet-Neumann operator.
 The capillary waves system (\ref{waterwaves}) has the following conserved energy and momentum, see e.g., \cite{benolver},
\be\label{conservedenergy}
\mathcal{H}(  h(t),\psi(t)) := \Big[ \int_{\R^2} \frac{1}{2}\psi(t) G( h(t))\psi(t)   + \frac{  |\nabla h(t)|^2}{1+\sqrt{1+|\nabla h(t)|^2}} d x\Big]= \mathcal{H}( h(0), \psi(0)),
\ee
\be\label{conservedmom}
\int_{\R^2} h(t, x) d x =\int_{\R^2} h(0, x) d x  .
\ee
From  \cite{wang2}[Lemma 3.4], we know that
\be\label{e10}
(\textup{Flat bottom setting}):\quad  \Lambda_{\leq 2}[G(h)\psi]=\d\tanh\d \psi -\nabla\cdot( h \nabla \psi) - \d \tanh\d ( h \d \tanh\d \psi),
\ee
\be\label{e11}
(\textup{Flat bottom setting}):\quad\Lambda_{\leq 2}[\p_t \psi]=\Delta h -\h |\nabla \psi|^2 + \h (\d\tanh \d \psi)^{2},
\ee
where $\Lambda_{\leq 2}[\mathcal{N}]$ denotes  the linear terms and quadratic terms of the nonlinearity $\mathcal{N}$.

Therefore, in the small   solution regime, the following approximation holds,
\[
2\mathcal{H}(\,h(t),\psi(t))\approx   \| \nabla h  \|_{L^2}^2 +  \| \d^{1/2} \sqrt{\tanh(\d)} \psi\|_{L^2}^2
\]
\be\label{approximationconservation}
\approx \| \nabla h  \|_{L^2}^2 +\| \d P_{\leq 1}[\psi]\|_{L^2}^2 +\| \d^{1/2} P_{\geq 1}[\psi]\|_{L^2}^2 .
\ee
 
There is an extensive literature on the study of the   water waves system. Without being exhaustive, we  only discuss some of the history and previous works.

\noindent$\bullet$ \quad Previous results on the  local existence of the water waves system. \quad 

Early works of  Nalimov \cite{nalimov} and Yosihara \cite{yosihara} considered the local well-posedness of the  small perturbation of flat interface such that the Rayleigh-Taylor sign condition holds. It was first discovered by   Wu \cite{wu1, wu2} that  the Rayleigh-Taylor   sign condition holds without the smallness assumptions in the infinite depth setting. She showed the local existence  for arbitrary size of initial data in Sobolev spaces. After the breakthrough of Wu's work, there are many important works on improving the understanding of local well-posedness of the full water waves system and the free boundary Euler equations. Christodoulou-Lindblad \cite{christodoulou} and Lindblad \cite{lindblad} considered the gravity waves with vorticity. Beyer-Gunter \cite{beyer} considered the effect of surface tension. Lannes\cite{lannes} considered the finite depth setting. See also Shatah-Zeng\cite{shatah1}, and Coutand-Shkoller \cite{coutand1}. 
 It turns out that local well-posedness also holds even if the interface  have a unbounded curvature and the bottom is very rough even without regularity assumption (a finite separation condition is required), see the works of Alazard-Burq-Zuily \cite{alazard1,alazard2} for more detailed description.

\noindent$\bullet$ \quad Previous results on the   long time behavior of the water waves system. \quad 

The long time behavior of the water waves system is more difficult and challenging.  Even for  a small perturbation of static solution and flat interface, we only have few results so far. Note that, it is possible to develop a so-called ``splash-singularity'' for a large perturbation, see \cite{fefferman}. 

To obtain  the local existence, it is very important to avoid losing derivatives due to the quasilinear nature of the water waves systems,  which corresponds to the high frequency part of the solution . However, to study the long time behavior and the dispersion of solution over time, the  low  frequency part plays an essential role. It is very interesting to see that the water waves systems in different settings  have  very different behavior at the low frequency part. The methods developed in one setting may be not applicable to another setting.

We first discuss previous results in  the infinite depth setting.  The first long-time result for the water waves system is due to the work of Wu \cite{wu3}, where she proved the almost global existence for the gravity waves in $2D$.   Subsequently, Germain-Masmoudi-Shatah \cite{germain2} and Wu \cite{wu4} proved the global existence for the gravity waves system in $3D$, which is the first global regularity result for the water waves system. Global existence of the capillary waves in $3D$ was also obtained by Germain-Masmoudi-Shatah \cite{germain3}. For the $2D$ gravity waves  system, it is highly nontrivial to bypass the almost global existence.  As first pointed out by  Ionescu-Pusateri \cite{IP1} and independently by Alazard-Delort \cite{alazard}, one has to modify the profile appropriately first to prove global regularity. The solution possesses the modified scattering property instead of the usual scattering. 
 Later, a different interesting proof of the almost global existence was obtained in the holomorphic coordinates by Hunter-Ifrim-Tataru \cite{hunter}, then Ifrim-Tataru \cite{Ifrim1} improved this result and gave another interesting   proof of the global existence. The author \cite{wang1} considered the infinite energy solution of the gravity waves in $2D$, which removed the momentum assumption in  previous small data results. 
 Global existence of capillary waves in $2D$ was also obtained. See  Ionescu-Pusateri   \cite{IP3, IP4} and  Ifrim-Tataru  \cite{Ifrim1}. Very recently, Deng-Ionescu-Pausader-Pusateri \cite{deng1} proved small data  global regularity for the $3D$ gravity-capillary waves ($\sigma, g > 0$), which completes the picture of small data global existence for the $3D$ water waves system in  the infinite depth setting.   The long time behavior of the $2D$ gravity-capillary waves in the infinite depth setting remains open.

Now, we restrict ourself to the finite depth case.The behavior of water waves in the finite depth setting  is more delicate due to the presence of traveling waves, the more complicated structure at the low frequency part, and less favorable quadratic terms.

 For the $3D$ infinite depth water waves system,   there is no small (in $L^2$ sense) traveling waves in all settings from all previous global existence results, see \cite{wu4,germain2,deng1}. However, we do have small traveling waves for the $3D$ finite depth gravity-capillary waves  as long as $\sigma/g > 1/3$.   From the recent work of the author \cite{wang4}, we know that there is no small traveling wave  for the 3D gravity waves system,  i.e., $\sigma/g=0$. So far, it is still not clear whether there exist  small traveling waves  for the $3D$ finite depth gravity-capillary waves when $0< \sigma/g\leq 1/3$.  However, if one can combine the ideas used in \cite{deng1} and \cite{wang4} successfully, then it is reasonable to expect that there is no small traveling waves when $0< \sigma/g < 1/3$.

On the long time behavior side. Only results on the gravity waves have been obtained.
The large time existence was   obtained by  Alvarez-Samaniego-Lannes \cite{lannes} for the $3D$ gravity waves. Recently, the author \cite{wang2,wang4} showed  that global regularity  holds for  the $3D$ gravity waves system for suitably small initial data. For the $2D$ case,   Harrop-Griffiths-Ifrim-Tataru \cite{Ifrim2}  showed the cubic life span ($1/\epsilon^2$) of gravity waves for   small initial data of size $\epsilon$.

 When the surface tension is effective, to the best knowledge of the author,  there is no long time existence result,  which exceeds  the scale of $1/\epsilon$ life span, in either $2D$ or $3D$. 

\subsection{Main difficulties for the capillary waves system in the flat bottom settings}
To help readers understand the main issues and the main difficulties of the capillary waves in the finite depth setting, we compare the capillary waves system in the infinite depth setting and the flat bottom setting.

The question of global regularity for the capillary waves sysetm in the infinite depth setting is already highly nontrivial. Thanks to recent  works of Germain-Masmoudi-Shatah \cite{germain2}, Ionescu-Pusateri\cite{IP3, IP4}, and Ifrim-Tataru   \cite{Ifrim1}, now we understand the  long time behavior of small data solution of the infinite depth capillary waves in both $2D$ and $3D$  very well. 

Unfortunately, as we will explain shortly, we can't   use those understandings of the long time behavior as those understandings in the infinite depth setting are not related at all to the finite detph setting. The effect of the finite depth bottom shows up in the long run, which totally changes the long time behavior.

  One of the main reasons that it is even possible to prove global regularity for the $2D$ capillary waves in the infinite depth setting is that there exist  favorable cancellations in the infinite depth setting, which act like null structures. Note that
\be\label{e1}
(\textup{Infinite depth setting}):\quad \Lambda_{\leq 2}[\p_t h ]= \Lambda_{\leq 2}[G(h)\psi]=\d \psi -\nabla\cdot( h \nabla\psi) - \d(h \d \psi),
\ee
\be\label{e2}
(\textup{Infinite depth setting}):\quad\Lambda_{\leq 2}[\p_t \psi]=\Delta h -\h |\nabla \psi|^2 + \h (\d\psi)^{2}.
\ee

Intuitively speaking, for the infinite depth setting, we are dealing with the following type of quasilinear   dispersive equation,
\[
\textup{Infinite depth}:\quad (\p_t + i |\nabla|^{3/2}) u = \d^{1/2}\Lambda_{2}[G(h)\psi] + i \Lambda_2[\p_t \psi] + \mathcal{R},\quad u = \d^{1/2} h + i \psi.
\]

From (\ref{e1}) and (\ref{e2}), it is easy to check that the size  of symbol of quadratic terms is   ``$0$'' in both   $1\times 1 $ (sizes of input frequencies)  $\rightarrow 0$ (size  of output frequency) type interaction and $1\times 0 \rightarrow 1$ type interaction.  More precisely, we can gain \emph{at least the smallness of  $|\xi|\min\{|\eta|, |\xi-\eta|\}^{1/2}$} from the symbol of quadratic terms, where $\xi-\eta$ and $\eta$ are frequencies of  two  inputs inside the quadratic terms.

 Intuitively speaking, the smallness of  symbol stabilizes the nonlinear effect. As we will see in later discussion, the  smallness of output frequency stabilizes the growth of the Fourier transform of the profile around a small neighborhood of zero frequency, which again leads to the expectation that the growth of Fourier transform of the profile in the medium frequency part is also stabilized in the $1\times 0\rightarrow 1$ type interaction.

However, we lose all those favorable cancellations  for the capillary waves system (\ref{waterwaves}) in the finite depth setting.  From  (\ref{e10})  and (\ref{e11}), it is easy to check that the size  of symbol of quadratic terms is  ``$1$'' in both  $1\times 1 \rightarrow 0$ type interaction and  $1\times 0 \rightarrow 1$ type  interaction. More precisely,  in the $1\times 0 \rightarrow 1$ type interaction, the size of symbol of the quadratic terms of  ``$\p_t h $'' is $1$, meanwhile, in the $1\times 1\rightarrow 0$ type interaction,    the size of symbol of the quadratic terms of   ``$\p_t\psi$'' is $1$.

Because of lacking null structures, we expect much stronger nonlinear effect for the finite depth capillary waves, which makes the question of global regularity in the finite depth setting much harder than the infinite depth setting.

Recall (\ref{waterwaves}), (\ref{e10}), and (\ref{e11}). To understand the long time behavior of solution of (\ref{waterwaves}), we have to understand the long time behavior of the following toy model, which only keeps the quadratic terms of the system (\ref{waterwaves}),
\be\label{e18}
(\textup{Toy Model}):  (\p_t + i \d^{\frac{3}{2}}\sqrt{\tanh(\d)}) u =  \sqrt{\frac{\d}{\tanh \d}} \Lambda_{2}[G(h)\psi] + i \Lambda_2[\p_t \psi],\,\, u:= \sqrt{\frac{\d}{\tanh \d}} h + i \psi.	
\ee

We can even simplify the toy model  further by only considering the low frequency part of (\ref{e18}) and replacing $   \d^{3/2}\sqrt{\tanh(\d)} $ and $\sqrt{\d/\tanh(\d)}$by $\d^2 $ and ``$1$'' respectively. More precisely, we consider the long time behavior of the  toy model of (\ref{e18}) as follows, 
\be\label{e70}
(\textup{Toy model of (\ref{e18}) }):\quad  (\p_t - i \Delta)v= Q_1(v,\bar{v}) + Q_2(v,  {v})+ Q_3(\bar{v}, \bar{v}),\quad  v:\R_t \times \R^2_x \longrightarrow \mathbb{C}, 
\ee
where the symbols $q_i(\xi-\eta, \eta)$ of the quadratic terms $Q_i(\cdot, \cdot)$, $i\in\{1,2,3\}$, satisfy the following estimate, 
\be\label{e20}
\| \mathcal{F}^{-1}[ q_i(\xi-\eta,\eta)\psi_k(\xi)\psi_{k_1}(\xi-\eta)\psi_{k_2}(\eta)]\|_{L^1} \lesssim \min\{2^{2\max\{k_1,k_2\}},1\}, \quad i \in \{1,2,3\},
\ee
which captures the facts that there are at least two derivatives inside (\ref{e18}) and the size of symbol is ``$1$'' in both $1\times 1\rightarrow 0$ type interaction and the $1\times 0\rightarrow 1$ type interaction. 

It turns out that the toy model (\ref{e70}) ( $2D$ quadratic Schr\"odinger equation) is already a  very delicate problem due to the presence of $  v  \bar{v} $ type nonlinearity. \emph{Even  the quadratic Schr\"odinger equation in $3D$  is not completely solved}.

 If without the $ v  \bar{v} $ type quadratic term, then the $1\times 1 \rightarrow 0$ type interaction is actually not very bad.  Note that the phases are all of size $1$  in the $1\times 1 \rightarrow 0$ type interaction if there is no $ v  \bar{v} $ type quadratic term.     The high oscillation of phase in time will also stabilize the growth of the  profile  in a neighborhood of zero frequency even without the smallness comes from the symbol. We refer readers to  the works of Germain-Masmoudi-Shatah \cite{germain4, germain5} for more detailed discussion.

Unfortunately,   we do have  $v \bar{v}$ type quadratic term in     the capillary waves system (\ref{waterwaves}). To see the nonlinear effect of $v \bar{v}$ type quadratic term  in the long run, we consider the toy model (\ref{e70}) for simplicity. 
We define the profile of solution $v(t)$ as $g(t):= e^{-it \Delta} v(t)$ and study the growth of profile over time, which gives us a sense of what the dispersion of solution will be. 

 From the Duhamel's formula, we have
\[
\widehat{g}(t, \xi) = \widehat{g}(0, \xi) + \int_0^t \int_{\R^2} e^{ i2s \xi \cdot \eta }  q_1(\xi-\eta, \eta) \widehat{g}(s, \xi-\eta)  \widehat{\bar{g}}(s, \eta) 
 + e^{ i2s \eta \cdot (\xi-\eta)} q_2(\xi-\eta, \eta) \widehat{ {g}}(s, \xi-\eta) \]
 \be\label{e30}
 \times \widehat{g}(s, \eta)  + e^{i s(|\xi|^2 +|\xi-\eta|^2 +|\eta|^2)}q_3(\xi-\eta, \eta) \widehat{\bar{g}}(s, \xi-\eta)\widehat{\bar{g}}(s, \eta)
 d \eta d s.
 \ee

We start from the first iteration by replacing ``$g(s)$'' on the right side of (\ref{e30})
with the initial data $g(0)$, whose frequency is localized around ``$1$''. As a result, intuitively speaking, in the worst case, we have
\be\label{e35}
   \left\{ \begin{array}{ll}
  |\widehat{g}(t, \xi)| \sim t, & \textup{when}\,\, |\xi| \lesssim  {1}/{t}\\
   1/|\xi| \lesssim |\widehat{g}(t, \xi)|  \lesssim (1+t)^\delta/|\xi|  & \textup{when}\,\, |\xi|\ll 1.
   \end{array}\right.
\ee

Because of the growth of the profile around the zero frequency, generally speaking, the $L^\infty_x$ decay rate of solution is not sharp and is only $(1+t)^{-1/2}$ when $|\xi|\approx (1+t)^{-1/2}$.  Of course, it is just a intuition of what the worst scenario can happen based on the picard iteration.  

Due to the growth of profile at the low frequency part, it will cause again the growth of profile at the  medium frequency part (i.e., frequency of size around ``$1$'') because of the $1\times 0\rightarrow 1$ type interaction. Therefore, generally speaking, one might expect certain instability happens. Recently, Ikeda and Inui \cite{ikeda} showed that there exists a class of small $L^2$ initial data such that the solution of the quadratic Schr\"odinger equation of $|u|^2$ type blows up in a polynomial time in \emph{both $2D$ and $3D$}. 

From (\ref{e35}), it is also reasonable to expect that the solution should   behave better if symbol of quadratic terms contributes certain power of the smallness ``$\xi$''. For the $3D$ quadratic Schr\"odinger equation, recently, the author \cite{wang3} showed that there exists small data global solution as long as the symbol of quadratic terms contribute the smallness of $|\xi|^{\epsilon}$ for any small ``$\epsilon$'', $0< \epsilon\ll 1$. An analogue of this result  should also hold  for the $2D$ quadratic Schr\"odinger equation.  Again, we don't have the  luxury of smallness inside the capillary waves system (\ref{waterwaves}), some other ideas are needed,

\subsection{Main result}

 Despite the nonlinear effect comes  from the quadratic term is very strong, after fully exploiting the structures inside the capillary water waves (\ref{waterwaves}) and using a novel weighted norms method for a good substitution variable, in this paper, we  show that the solution of capillary waves system (\ref{waterwaves}) globally exists and scatters to a linear solution for   small initial data.  

More precisely,  our main theorem is stated as follows, 
 \begin{thm}
Let $N_0=2000, \delta\in (0, 10^{-9}]$, $\tilde{\delta}:=400\delta$, and $\alpha =1/10 $. Assume that the initial data $(h_0, \psi_0)$ satisfies the following smallness condition,
\[
\| (h_0, \psi_0)\|_{H^{N_0 +1/2}} + \sum_{\Gamma \in \{L, \Omega\}}\| (\Gamma  h_0, \Gamma  \psi_0)\|_{H^{10+1/2}} + \sum_{\Gamma^1, \Gamma^2 \in \{L, \Omega\}}\| ( \Gamma^1 \Gamma^2 h_0, \Gamma^1 \Gamma^2 \psi_0)\|_{H^{1/2}} \lesssim \epsilon_0,
\]
where $\epsilon_0$ is a sufficiently small constant, $\Omega:= x^{\perp}\cdot \nabla_x$ and $L:=x \cdot \nabla_x +2$. Then there exists a unique global solution for the capillary water waves system \textup{(\ref{waterwaves})} with initial data $(h_0, \psi_0)$, and   the solution scatters to a linear solution associated with \textup{(\ref{waterwaves})}. Moreover, the following estimate holds,
\be\label{desiredfinalestimate}
  \sup_{t\in [0,T]} (1+t)^{-\delta}\| (\tilde{\Lambda}h, \psi)(t)\|_{H^{N_0}}   + (1+t) \big[ \sum_{k\in \mathbb{Z}}  2^{(1+\alpha)k + 6 k_{+} }\| P_{k}[(h, \psi)(t)]\|_{L^\infty}\big] \lesssim \epsilon_0.
\ee 
where $ \tilde{\Lambda}:=  \d^{1/2} (\tanh \d)^{-1/2}$ 
\end{thm}
\begin{remark}
From (\ref{desiredfinalestimate}), we know that the solution decays over time, which implies that there is no small traveling waves associated with the capillary water waves system (\ref{waterwaves}), i.e., $\sigma/g = \infty$. This result sharply contrasts to the existence of small traveling waves for the  gravity-capillary waves system with $\sigma/g > 1/3.$ 
\end{remark}

\subsection{Main   ideas of proof}
  The idea of proving global regularity for a dispersive equation is classical, which is   iterating  the local existence result by controlling both the energy and the dispersion of solution over time.

The whole argument depends on the expectation of the decay rate of solution, which is very delicate. 
 From (\ref{e35}), it is reasonable to expect that the solution does not decay sharply. Also it is reasonable to expect that the decay rate of solution may be different  if  the  solution is localized around different size of frequencies.

  Generally speaking, one can put as many derivatives as he/she wants  in front of the solution, which certainly helpful in the $1\times 1 \rightarrow 0$ type interaction. However, the more derivatives one puts the less information it reveals about the solution itself. The gain from the weigh of regularity in the $1\times 1\rightarrow 0$ type interaction becomes the corresponding  loss in the $1\times 0 \rightarrow 1$ type interaction. 

 Motivated from (\ref{e35}), we expect that $1+\alpha$ ($0< \alpha < 1$) derivatives of solution decay sharply. The reason why we put ``$1+$'' derivatives instead of $1$ derivative in front of solution is to avoid a summability issue which causes a logarithmic loss in time, which prevents the decay rate to be sharp. In other words, within the exponential time interval of existence, the worst decay rate is $(1+t)^{-1/2}$. However, after the exponential time interval of existence, the worst decay rate is roughly $(1+t)^{-1/2+\delta}$, as we expect a ``$\log(t)$'' growth of the weighted norm of the profile.

 Based on the expectation of the decay rate  of the nonlinear  solution, we expect that the energy of solution only grows appropriately. We remark that the expectation of the sharp decay rate of the derivatives of solution is also very essential to the energy estimate.  If the decay rate of up to two derivatives of solution were not sharp, then   it is not so clear how one can control the energy near the time resonance set, which is very large and complicated.   Especially given the existence of  blow up solution for small initial data result for the quadratic Schr\"odinger equation, which has the same essential structure of very large time resonance set.
  
  More precisely, the time resonance sets of the quadratic terms are defined as follows,
 \[
\mathcal{T}_{\mu, \nu}:=\{(\xi, \eta):\quad \Lambda(|\xi|) - \mu \Lambda(|\xi-\eta|)-\nu \Lambda(| \eta|)=0 \}, \quad \mu, \nu\in\{+,-\}.
 \]
 Note that, the following approximation roughly holds,
 \be\label{timeresonance}
 \mathcal{T}_{+, -} \cap \{(\xi,\eta):  |\xi|, |\eta|\ll 1\} \approx  \{(\xi,\eta):   |\xi|, |\eta|\ll 1,  \Lambda(|\xi|) -   \Lambda(|\xi-\eta|) +  \Lambda(| \eta|)\approx 2 \xi \cdot \eta \approx 0\}.
 \ee
Note that it is possible that $|\xi\cdot \eta|\ll 1/t$ no matter what the sizes of $|\xi|$ and $|\eta|$ are. From (\ref{timeresonance}), we know that the time resonance set is almost everywhere and very complicated.

\subsubsection{Energy estimate: the control of high frequency part of solution} We first point out that the difference of the high frequency part between the infinite depth setting and the flat bottom setting is very little.  Thanks to the works of Alazard-M\'etivier \cite{alazardguy} and Alazard-Burq-Zuilly\cite{alazard1, alazard2},  by using the method of paralinearization and symmetrization, we can find a pair of good unknown variables, such that the equation satisfied by the good unknown variables have symmetries inside, which help us avoiding losing derivatives in the energy estimate. 

However, since the decay rate of solution in the worst scenario is only $(1+t)^{-1/2+\delta}$, a rough $L^2-L^\infty$ type energy estimate is not sufficient to prove that the energy only grows appropriately.  Recall that we expect that the decay rate of $1+\alpha$ derivatives of solution is sharp.  Therefore, we need to pay special attention to the low frequency part of the input that is putted in $L^\infty$ type space.  To this end,  an important step is to understand the structure of low frequency part of the Dirichlet-Neumann operator, which has been studied in details in \cite{wang2}.

We first state our desired energy estimate and then explain the main intuitions behind.  We expect the following new type of energy estimate holds, 
\be\label{e40}
\big|\frac{d}{d t}E(t) \big|\lesssim  E(t)\big(\| (h(t), \psi(t)) \|_{W^{6,1+\alpha}} + \|(h(t),\psi(t))\|_{W^{6,1}}\|(h(t),\psi(t))\|_{W^{6,0}}\big),
\ee
where the function space of type $W^{\gamma,b}$ is defined as follows, 
\be\label{Linftyspace}
\| f\|_{W^{\gamma, b}}:=\sum_{k\in\mathbb{Z}} (2^{\gamma k } + 2^{b k}) \| P_{k } f\|_{L^\infty}, \quad    b< \gamma.
\ee

Since we expect that the decay rate of the $1+\alpha$ derivatives of solution is sharp, the desired new type of energy estimate (\ref{e40}) is sufficient to prove that the energy only grows appropriately.

To derive the new type energy estimate  (\ref{e40}), besides the quadratic terms, we also have to pay special attentions to the low frequency part of the cubic terms. 

Recall (\ref{e10}), (\ref{e11}), and  (\ref{e18}). We know that there are at least two derivatives in total inside the quadratic terms. Note the facts that we don't lose derivatives   after utilizing symmetries during the energy estimates and  the total number of derivatives doesn't decrease in this process. As a result, intuitively speaking, there are only two possible scenarios as follows, (i) including   the  High $\times$ High type interaction,   there are at least two derivatives associated with the input with relatively smaller frequency; (ii) the regularity of the quadratic terms can be lower to  $L^2$. Therefore, we can put the input with larger frequency in $L^\infty$ and put the other input in $L^2$. In whichever scenario, the input which is putted in $L^\infty$  type space  always has two derivatives in front, which explains the first estimate in (\ref{e40}).  A very similar intuition also holds  for cubic and higher order terms, which leads to the second part of (\ref{e40}).

 \subsubsection{The dispersion estimate: sharp decay rate of the $1+\alpha$ derivatives of solution}
 To carry out the analysis of decay estimate, we first identify a good substitution variable, which has the same decay rate as  the original solution. Instead of proving the dispersion estimate for the original variable, our goal is reduced to prove the dispersion estimate for the good substitution variable. To this end, we introduce a novel method of estimating the weighted norms, which helps us to get around the difficulties around the large time resonance set.

The main reason that we study the good substitution variable instead of the original variable is  that the equation satisfied by the good substitution variable has favorable structures, which make it possible to prove that the high order weighted norm only grows appropriately. 

We divide the rest of this subsection into three parts. (i) In the first part, we explain how to find the good substitution variable and why we do so. (ii) In the second part, we explain some main ideas in the estimate of the lower order weighted norm. The goal is to prove that it   doesn't grow over time, which further implies that the decay rate of $1+\alpha$ derivatives of solution is sharp. A crucial ingredient is that the high order weighted norm only grows appropriately. (iii) In the third part,  we explain main ideas behind the estimate of high order weighted norm, which is the most important part of the whole argument.

\noindent $\bullet$\quad \textit{A good substitution variable.\quad} To illustrate  main ideas and for simplicity, we consider the toy model (\ref{e18}).  Define the profile of $u(t)$ as $f(t):=e^{it \Lambda} u(t)$, we have
\[
\widehat{f}(t, \xi) = \widehat{f}(0, \xi)  + \sum_{\mu, \nu\in\{+,-\}} \int_0^t \int_{\R^2} e^{is\Phi^{\mu,\nu}(\xi, \eta)}q_{\mu, \nu}(\xi-\eta, \eta) \widehat{f^{\mu}}(s,\xi-\eta ) \widehat{f^{\nu}}(s, \eta) d \eta d s,  
\]
where $f^{+}:=f=:P_{+}[f]$, $f^{-}:=\bar{f}=:P_{-}[f]$, and  $q_{\mu, \nu}(\xi-\eta, \eta)$ is the symbol of  $u^{\mu} u^{\nu}$ type quadratic term. The   phase $\Phi^{\mu, \nu}(\xi-\eta, \eta)$ is defined as follows,
\[
\Phi^{\mu,\nu}(\xi, \eta)= \Lambda(|\xi|) -\mu\Lambda(|\xi-\eta|) -\nu\Lambda(|\eta|), \quad \Lambda(|\xi|):=|\xi|^{3/2} \sqrt{\tanh(|\xi|)}.
\]
Note that 
\[
\nabla_\eta \Phi^{+, +}(\xi,\eta) = -\Lambda'(|\eta-\xi|)\frac{\eta-\xi}{|\eta-\xi|} - \Lambda'(|\eta| )\frac{\eta }{|\eta |} \Longrightarrow \nabla_\eta \Phi^{+, +}(\xi,\xi/2)=0.
\]
Therefore, we can't do integration by parts in $\eta$ around a small neighborhood of $(\xi, \xi/2)$ (space resonance set) .  Fortunately,  $(\xi, \xi/2)$ doesn't belong to the time resonance set. From the explicit formula, it is easy to check the validity of the following estimate,
\[
\Phi^{+, +}(\xi,\xi/2)= \Lambda(|\xi|) -2\Lambda(|\xi|/2) \sim \Lambda(|\xi|).
\]

Very similarly, it is easy to check that the following estimate holds when $|\eta|\ll |\xi|$ and $\mu=-$ or $|\xi|\ll |\eta|$, $\mu\nu=+$,
\[
|\Phi^{\mu, \nu}(\xi, \eta)| \sim \max\{\Lambda(|\xi|), \Lambda(|\eta|)\}.
\]
We call  those cases as the high oscillation in time cases, in which the associated phase is relatively big.

Therefore , we can use a normal form transformation to remove the high oscillation in time cases as follows,
\be\label{toynormalform}
v:= u + \sum_{\mu, \nu\in \{+,-\}} A_{\mu, \nu}(u^{\mu}, u^{\nu}),  \quad a_{\mu, \nu}(\xi-\eta, \eta) = \sum_{k\in \mathbb{Z}} \frac{ i q_{\mu, \nu}(\xi-\eta, \eta)}{\Phi^{\mu, \nu}(\xi, \eta)} 
\ee
\[
 \times \big( \psi_{\leq k-10}(\eta-\xi/2) \psi_{k}(\xi)+ \frac{1-\mu}{2} \psi_k(\xi)\psi_{\leq k-10}(\eta) + \frac{1+\mu\nu}{2}\psi_k(\xi) \psi_{\geq k+10}(\eta) \big),
\]
where  $a_{\mu, \nu}(\xi-\eta, \eta)$ is the symbol of quadratic terms $A_{\mu, \nu}(\cdot, \cdot)$.   Note that there are at least two derivatives inside the symbol, which covers the loss of dividing the phase. As a result, the normal form transformation is not singular. 

There are two reasons  that we   remove the high oscillation cases. The first reason  is that we want ``$\nabla_\xi \Phi^{\mu, \nu}(\xi,\eta)$'' to be small when $|\eta|\ll |\xi|$, therefore the error term  when ``$\nabla_\xi $'' hits the phase is not so bad. The second reason is that we want ``$\nabla_\eta \Phi^{\mu, \nu}(\xi,\eta)$'' to has a lower bound such that we can always do integration by parts in $\eta$. Those two properties  are very essential in later weighted norm estimate.

It is easy to check that the structure of quadratic terms inside the equation satisfied by $v$ is much better.  Now, the strategy is to prove the decay rate of  $1+\alpha$ derivatives of `` $v$'' is sharp.

So far, the discussion is restricted to the toy model (\ref{e18}). For the capillary waves system (\ref{waterwaves}), we use similar ideas not only for   quadratic terms, but also for   cubic terms and quartic terms, see (\ref{normalformatransfor}). Please refer to  subsection \ref{goodvariable} for more details.

\vo

\noindent $\bullet$\quad \textit{The low  order weighed norm.\quad}  We first define the low order weighted norm $Z_1$-norm and the high order weighted norm $Z_2$-norm as follows, 
\be\label{loworderweightnrom}
 \| g\|_{Z_1}:= \sum_{k\in \mathbb{Z}} \sum_{j\geq -k_{-}} \| g\|_{B_{k,j}}, \quad  \| g\|_{B_{k,j}}:=   (2^{(1+\alpha)k  } +  2^{ 10 k_{+}} )2^j \| \varphi^k_j(x) P_k g 	(x)\|_{L^2}, 
\ee
\be\label{highorderweightnorm}
 \| g\|_{Z_2}:=\sum_{\Gamma^1, \Gamma^2\in \{L, \Omega\}} \| \Gamma^1\Gamma^2 g \|_{L^2} + \| \Gamma^1 g \|_{L^2} ,
\ee
where $\varphi_j^k(x)$ is defined in (\ref{spatiallocalization}), which is first introduced in the work of Ionescu-Pausader \cite{ benoit}. An advantage of using this type space is that it not only localizes the frequency but also   localizes the spatial concentration.  The atomic space of this type has been successfully used in many dispersive PDEs, see \cite{deng1,deng2,guo2, benoit,benoit3,wang4}.

Define the profile of the good substitution variable $v(t)$ as $g(t):= e^{it \Lambda}v(t)$.  From the linear dispersion estimates (\ref{highdecay}) and (\ref{lowdecay}) in Lemma \ref{lineardecay}, to prove the sharp decay rate, it would be sufficient to prove that the $Z_1$-norm of the profile $g(t)$ doesn't grow in time.

 Now, we  explain the main ideas of how to prove $Z_1$ norm doesn't grow under the assumption that the $Z_2$ norm only grows appropriately.

   To get around the difficulties in the  High $\times$ High type interaction, we put the very high weighted in the definition of $Z_1$ norm. As a result, the High $\times$ High type interaction is not a   issue.  So, the real issue is the High $\times$ Low type interaction, e.g., $|\eta|\ll |\xi|$.  As a typical example of the threshold case in the   High $\times$ Low type interaction, we consider the case when $|\eta|\sim 1/t, |\xi|\sim 1 $. For this case,  the high order weighted norm is actually not helpful as we can't do integration by parts in $\eta$. Also,  the volume of support of $\eta$ in this scenario is not sufficient because the size of the Fourier transform of the profile  is ``$ t$'', which is very big, and the size of symbol is $1$, which is not helpful. 

 To get around this difficulty, we use the hidden structure inside the capillary waves system (\ref{waterwaves}). Recall (\ref{conservedmom}). We know that $\widehat{h}(t, 0)$ is conserved over time. A simple Fourier analysis shows that $\widehat{\psi}(t,0) \sim t$. Those two facts 
 motivate  us to expect that the source of trouble is $\widehat{\psi}(t, \eta)$ but not  $\widehat{h}(t, \eta)$.  We expect that $\widehat{h}(t, \eta)$   behaves better than $\widehat{g}(t, \eta)$ when $|\eta|\lesssim 1/t$. As a matter of fact, we do have a better estimate for $\widehat{h}(t, \eta)$, see (\ref{eqn400}) in Lemma \ref{Linftyxi}, which says that $\widehat{h}(t, \eta)$ is bounded above by $t^{2\delta}$ when $\eta\sim 1/t.$ Recall again (\ref{e10}) and (\ref{e11}), we know that there is at least one derivative in front of the velocity potential ``$\psi(t)$''. Therefore, it contributes the smallness of $\eta\sim 1/t$, when `` $\psi(t)$'' has the frequency $\eta$. To sum up, either the symbol contributes the smallness of $\eta$  or the input of smaller frequency is the height ``$h(t)$''. In whichever case, the threshold case when $|\eta|\sim 1/t, |\xi|\sim 1 $ is not a issue.

For the non-threshold case, we do integration by parts in $\eta$ once to take the advantage of the gap between the threshold case and the non-threshold case. For the High $\times$ High type interaction, the gap is created by the extra $2^{\alpha k}$ we  put in the definition of $Z_1$ norm. For the High $\times$ Low type and Low $\times$ High type interactions, the gap is created by the observation we made in above discussion.  The gain from the gap is more than the loss from the growth of the high order weighted norm.  As a result, the $Z_1$ norm of the profile doesn't grow in time.
 \vo

\noindent $\bullet$\quad \textit{The high order weighted norm.\quad}   The estimate of the high order weighted norm is very essential and  will consume most of this paper. Without being too technical, we explain some essential ideas that make it possible to control the growth of the high order weighted norm. 	

 Note that we used the vector field  ``$L:=x\cdot \nabla +2$'' (equivalently, `` $-\xi \cdot \nabla_\xi$'' on the Fourier side) twice in the definition of $Z_2$ norm (\ref{highorderweightnorm}). The    vector field ``$L   $'' does very inconvenient to use. As every time it hits the   phase, it creates a burden of extra ``$t$''.  Use this vector field twice is a nightmare.

 On the one hand,   there is no better option because there is no scaling vector field associated with the capillary waves system (\ref{waterwaves}); on the other hand,   we observe that  the equation satisfied by  ``$ L \big(e^{it\Lambda} u(t)\big)$'',  still has a favorable structure inside. We remark that, without this observation, it is not so clear how one can go beyond the almost global life span of the solution even though it is already a very difficult problem to show the almost global existence.

When  the vector field  $L$ indeed hits the phase, we observe  a very useful hidden structure. As  as an example,  we consider the High $\times$ Low type interaction.  Note that (see (\ref{eqn1301}) and (\ref{eqn951}) for more details),
\be\label{e80}
\xi \cdot \nabla_\xi \Phi^{+, \nu}(\xi, \eta) = \mathcal{O}(1)\Phi^{+, \nu}(\xi, \eta) + \mathcal{O}(|\eta|^2), \quad \textup{when\,\,} |\eta|\ll |\xi|.
\ee

For the first term on the right hand side of (\ref{e80}), we can do integration by parts in time once. The main intuition behind is that the first term  is very small when the frequency is very close to the time resonance set. For the second term of (\ref{e80}), we gain extra smallness of $|\eta|$, which is very helpful in all cases. As a result, the case when the vector field ``$L$'' hits the phase is actually controllable. 

We remark that if one do integration by parts in $\eta$ directly instead of using decomposition (\ref{e80}), then there is a problem when we are forced to put the input with smaller frequency (i.e., $\eta$) in $L^\infty_x$ space. The symbol only contributes the smallness of $|\eta|$. And after doing integration by parts in `` $\eta$ '', there is no symmetric structure available to make it possible to gain more. As a result,  the insufficient smallness of symbol is not sufficient to guarantee $1/t$ decay rate. However, this difficulty is no longer a issue for the second part of the decomposition (\ref{e80}), which is of size $|\eta|^2.$ As the decay rate of $1+\alpha$ derivatives of solution is sharp, two derivatives are very sufficient.

Another difficulty  is when the vector fields hit  the input with relatively bigger frequency. A rough $L^2-L^\infty$ type estimate is not sufficient to  close the argument since we are forced to put the input with smaller frequency in $L^\infty$, which only decays at $(1+t)^{-1/2+\delta}$ in the worst scenario as there is no good derivative in front.

 To get around this difficulty, we use the hidden symmetry between two inputs. As a result, we can gain one degree of smallness of the smaller frequency, which is still not sufficient. Fortunately, this one degree of smallness is of ``$\xi\cdot\eta$'' type, which assembles the structure of the phase itself. A similar decomposition as in  (\ref{e80}) also holds for the resulted symbol after utilizing the hidden symmetries. Hence, we can do integration by parts in time for one part and have a better symbol   for the other part.   See (\ref{e100}), (\ref{eqn611}), and (\ref{eqn470}) for more details.

We remark that we   don't always gain $1/t$ decay rate  from  doing integration by parts in time  since the new introduced input  may only have $(1+t)^{-1/2+\delta}$ decay rate in the worst scenario. For example, after doing integration by parts in time, we have the following term after $\p_t$ hits the input $\widehat{\Omega^2 g}(t,\xi-\eta)$,
\[
\int_{\R^2}\int_{\R^2} e^{i t\Phi^{+,+}(\xi, \eta)} \overline{\widehat{\Omega^2 g}(t, \xi)} e^{i t\Phi^{+,+}(\xi-\eta, \sigma)} \widehat{  \Omega^2 g}(t, \xi-\eta-\sigma) \widehat{g}(t,\eta) \tilde{q}_{+,+}(\xi-\eta-\sigma,\sigma)
\]
 \be\label{e120}
\times  \widehat{g}(t,\sigma)  \psi_k(\xi)\psi_{k_2}(\eta) \psi_{k_2'}(\sigma)  \psi_{k_1}(\xi-\eta) d\eta d \xi, \quad k_2',k_2\leq k_1-10.
 \ee
  For this case, the total decay rate from the  $L^2-L^2-L^\infty-L^\infty$ type estimate  is only $(1+t)^{-1+2\delta}$ in the worst scenario, which is still not sufficient to close the argument.

 To get around this issue, we first identify the worst scenario and then use the hidden symmetry   to see a good cancellation inside the symbol  for the  worst scenario. The resulted symbol  contributes a smallness of  $2^{\max\{k_2,k_2'\}}$. Hence the decay rate of   $ e^{-it\Lambda} P_{k_2} g$ or $ e^{-it\Lambda} P_{k_2'} g$ is improved from $(1+t)^{-1/2+\delta}$ to $(1+t)^{-1+\delta}$.  As a result, now $L^2-L^2-L^\infty-L^\infty$ type estimate is sufficient. If without this symmetry, i.e., without the smallness of  $2^{\max\{k_2,k_2'\}}$, it is not so clear how to close the argument for the case when $|\eta|, |\sigma|\approx (1+t)^{-1/2}$, $|\xi|\sim 1$ and $\xi \cdot \eta, \xi \cdot \sigma \approx 1/t$.

\subsection{ The outline of this paper}\label{notation} 
In section \ref{prelim}, we will introduce notations and some basic lemmas that will be used constantly. In section \ref{energyestimatesection}, we   prove a new type of energy estimate by using  the method of  paralinearization and symmetrization and paying special attentions to the  low frequency part. In section \ref{setupweightednorms}, we identify a good substitution variable to carry out the estimate of weighted norms. In section \ref{loworderweight},  we prove that the low order weighted norm doesn't grow over time under    the assumptions that the high order weight norm only grows appropriately and  a good control of the remainder term is available. In section \ref{highorderweighted}, we prove that the high order weighted norm only grows appropriately under the assumption that we have a good control on the remainder term. In section \ref{reminderestimatefixed}, we first prove some weighted norm estimates for a fixed time, which were took for granted in section \ref{loworderweight} and \ref{highorderweighted} , and then estimate the reminder terms by using a fixed point type argument. 

\vo
\noindent \textbf{Acknowledgement\quad } Part of this work was done when I visited Tsinghua University during the Summer 2016.    I would like to  thank  Prof. Pin Yu for the invitation and the warm hospitality during the stay.

\section{Preliminary}\label{prelim}
 For any two numbers $A$ and $B$, we use  $A\lesssim B$ and $A\ll B$ to denote  $A\leq C B$ and $A\leq c B$ respectively, where $C$ is an absolute constant and $c$ is a sufficiently small absolute constant. We use $A\sim B$ to denote the case when $A\lesssim B$ and $B\lesssim A$. 
 
Throughout this paper, we   abuse  the  notation of $``\Lambda"$.  When  there is no lower script in $\Lambda$, then  $\Lambda:=|\nabla|^{3/2}\sqrt{\tanh(|\nabla|)}$, which is the linear operator associated for the system  (\ref{waterwaves}). When there is a lower script $p$ in $\Lambda$ where $p\in \mathbb{N}_{+}$, then  we use $\Lambda_{p}(\mathcal{N})$ to denote the $p$-th order terms of the nonlinearity $\mathcal{N}$  if a Taylor expansion of $\mathcal{N}$ is available. Also, we use notation $\Lambda_{\geq p}[\mathcal{N}]$ to denote the $p$-th  and higher orders terms. More precisely, $\Lambda_{\geq p}[\mathcal{N}]:=\sum_{q\geq p}\Lambda_{q}[\mathcal{N}]$. For example, $\Lambda_{2}[\mathcal{N}]$ denotes the quadratic term of $\mathcal{N}$ and $\Lambda_{\geq 2}[\mathcal{N}]$  denotes the quadratic and higher order terms of $\mathcal{N}$. If there is no special annotation, then Taylor expansions  are in terms of $h$ and $\psi$.

We  fix an even smooth function $\tilde{\psi}:\R \rightarrow [0,1]$ supported in $[-3/2,3/2]$ and equals to $1$ in $[-5/4, 5/4]$. For any $k\in \mathbb{Z}$, we define
\[
\psi_{k}(x) := \tilde{\psi}(x/2^k) -\tilde{\psi}(x/2^{k-1}), \quad \psi_{\leq k}(x):= \tilde{\psi}(x/2^k)=\sum_{l\leq k}\psi_{l}(x), \quad \psi_{\geq k}(x):= 1-\psi_{\leq k-1}(x),
\]
and use  $P_{k}$, $P_{\leq k}$ and $P_{\geq k}$ to denote the projection operators  by the Fourier multipliers $\psi_{k},$ $\psi_{\leq k}$ and $\psi_{\geq k }$ respectively. We   use  $f_{k}(x)$ to abbreviate $P_{k} f(x)$ very often.  We use both $\widehat{f}(\xi)$ and $\mathcal{F}(f)(\xi)$ to denote the Fourier transform of $f$, which is defined as follows, 
\[
\mathcal{F}(f)(\xi)= \int e^{-ix \cdot \xi} f(x) d x.
\]
We use $\mathcal{F}^{-1}(g)$ to denote the inverse Fourier transform of $g(\xi)$. For an integer $k\in\mathbb{Z}$, we use $k_{+}$ to denote $\max\{k,0\}$ and  use $k_{-}$ to denote $\min\{k,0\}$. For two well defined functions $f(x)$ and $g(x)$ and  a bilinear form  $Q(f,g)$, we  use the convention that the symbol $q(\cdot, \cdot)$ of $Q(\cdot, \cdot)$  is defined in the following sense throughout this paper,
\begin{equation}
\mathcal{F}[Q(f,g)](\xi)=\frac{1}{4\pi^2} \int_{\R^2} \widehat{f}(\xi-\eta)\widehat{g}(\eta)q(\xi-\eta, \eta) d \eta. 
\end{equation}
Very similarly, for a trilinear form $C(f, g, h)$, its symbol $c(\cdot, \cdot, \cdot)$ is defined in the following sense, 
\[
\mathcal{F}[C(f,g,h)](\xi) = \frac{1}{16\pi^4} \int_{\R^2}\int_{\R^2} \widehat{f}(\xi-\eta)\widehat{g}(\eta-\sigma) \widehat{h}(\sigma) c(\xi-\eta, \eta-\sigma, \sigma) d \eta d \sigma.
\]
Define a class of symbol and its associated norms as follows,
\[
\mathcal{S}^\infty:=\{ m: \mathbb{R}^4\,\textup{or}\, \mathbb{R}^6 \rightarrow \mathbb{C}, m\,\textup{is continuous and }  \quad \| \mathcal{F}^{-1}(m)\|_{L^1} < \infty\},
\]
\[
\| m\|_{\mathcal{S}^\infty}:=\|\mathcal{F}^{-1}(m)\|_{L^1}, \quad \|m(\xi,\eta)\|_{\mathcal{S}^\infty_{k,k_1,k_2}}:=\|m(\xi, \eta)\psi_k(\xi)\psi_{k_1}(\xi-\eta)\psi_{k_2}(\eta)\|_{\mathcal{S}^\infty},
\]
\[
 \|m(\xi,\eta,\sigma)\|_{\mathcal{S}^\infty_{k,k_1,k_2,k_3}}:=\|m(\xi, \eta,\sigma)\psi_k(\xi)\psi_{k_1}(\xi-\eta)\psi_{k_2}(\eta-\sigma)\psi_{k_3}(\sigma)\|_{\mathcal{S}^\infty}.
\]
\begin{lemma}\label{Snorm}
For $i\in\{1,2,3\}, $ if $f:\mathbb{R}^{2i}\rightarrow \mathbb{C}$ is a smooth function and $k_1,\cdots, k_i\in\mathbb{Z}$, then the following estimate holds,
\begin{equation}\label{eqn61001}
\| \int_{\mathbb{R}^{2i}} f(\xi_1,\cdots, \xi_i) \prod_{j=1}^{i} e^{i x_j\cdot \xi_j} \psi_{k_j}(\xi_j) d \xi_1\cdots  d\xi_i \|_{L^1_{x_1, \cdots, x_i}} \lesssim \sum_{m=0}^{i+1}\sum_{j=1}^i 2^{m k_j}\|\p_{\xi_j}^m f\|_{L^\infty} .
 \end{equation}
\end{lemma}

\begin{lemma}\label{multilinearestimate}
Assume that $m$, $m'\in S^\infty$, $p, q, r, s \in[1, \infty]$ , then the following estimates hold for  well defined functions $f(x), g(x)$, and $h(x)$, 
\begin{equation}\label{productofsymbol}
\| m\cdot m'\|_{S^\infty} \lesssim \| m \|_{S^\infty}\| m'\|_{S^\infty},
\end{equation}
\begin{equation}\label{bilinearesetimate}
\Big\| \mathcal{F}^{-1}\big[\int_{\R^2} m(\xi, \eta) \widehat{f}(\xi-\eta) \widehat{g}(\eta) d \eta\big]\Big\|_{L^p} \lesssim \| m\|_{\mathcal{S}^\infty}\| f \|_{L^q}\| g \|_{L^r}  \quad \textup{if}\,\,\, \frac{1}{p} = \frac{1}{q} + \frac{1}{r},
\end{equation}
\begin{equation}\label{trilinearesetimate}
\Big\| \mathcal{F}^{-1}\big[\int_{\R^2}\int_{\R^2} m'(\xi, \eta,\sigma) \widehat{f}(\xi-\eta) \widehat{h}(\sigma) \widehat{g}(\eta-\sigma)  d \eta d\sigma\big] \Big\|_{L^{p}} \lesssim \|m'\|_{\mathcal{S}^\infty} \| f \|_{L^q}\| g \|_{L^r} \| h\|_{L^s},\,\, \end{equation}
where $ \displaystyle{\frac{1}{p} = \frac{1}{q} + \frac{1}{r} + \frac{1}{s}}.$
\end{lemma}

\begin{definition}
Given $\rho\in \mathbb{N}_{+}, \rho \geq 0$ and $m\in \R$, we use $\Gamma^{m}_{\rho}(\R^2)$  to denote the space of locally bounded functions $a(x,\xi)$ on $\R^2\times (\R^2/\{0\})$, which are $C^{\infty}$ with respect to $\xi$ for $\xi\neq 0 $. Moreover, they satisfy the following estimate, 
\[
\forall |\xi|\geq 1/2, \| \p_{\xi}^{\alpha} a(\cdot, \xi)\|_{W^{\rho, \infty}}\lesssim_{\alpha} (1+|\xi|)^{m-|\alpha|}, \quad \alpha\in \mathbb{N}^2,
\]
where $W^{\rho, \infty}$ is the usual Sobolev space.  For symbol $a\in \Gamma^{m}_{\rho}$, we can define its norm as follows,
\[
M^{m}_{\rho}(a):= \sup_{|\alpha|\leq 2+\rho} \sup_{|\xi|\geq 1/2} \| (1+|\xi|)^{|\alpha|-m}\p_{\xi}^{\alpha} a (\cdot, \xi)\|_{W^{\rho, \infty}}.
\]

\end{definition}

\begin{definition}
\begin{enumerate}
\item[(i)] We use   $\dot{\Gamma}^{m}_{\rho}(\R^2)$ to denote the subspace of $\Gamma^{m}_{\rho}(\R^2)$, which consists of symbols that are homogeneous of degree $m$ in $\xi.$ \\
\item[(ii)] If $a= \displaystyle{\sum_{0\leq j < \rho} a^{(m-j)}}$, where $a^{(m-j)}\in \dot{\Gamma}^{m-j}_{\rho-j}(\R^2)$, then we say $a^{(m)}$ is the principal symbol of $a$.\\
\item[(iii)] An operator $T$ is said to be of order $m$, $m\in \R$, if for all $\mu\in\R$, it's bounded from $H^{\mu}(\R^2)$ to $H^{\mu-m}(\R^2)$. We use $S^{m}$ to denote the set of all operators of order m.
\end{enumerate}
\end{definition}

For  
$a, f\in L^2$ and pseudo differential operator $\tilde{a}(x,\xi)$, we define the operator $T_{a} f$ and $T_{\tilde{a}} f$ as follows,
 \begin{equation}\label{eqn1001}
T_a f = \mathcal{F}^{-1}[\int_{\R} \widehat{a}(\xi-\eta) \theta(\xi-\eta, \eta)\widehat{f}(\eta) d \eta]
,\,\, T_{\tilde{a}} f = \mathcal{F}^{-1} [ \int_{\R} \mathcal{F}_x(\tilde{a})(\xi-\eta,\eta) \theta(\xi-\eta,\eta)\widehat{f}(\eta) d \eta  ],
\end{equation}
where the cut-off function is defined as follows,
\be\label{lowhighcutoff}
\theta(\xi-\eta, \eta) = \left\{\begin{array}{ll}
1 & \textup{when}\,\,|\xi-\eta|\leq 2^{-10} |\eta| \\
0 & \textup{when}\,\, |\xi-\eta| \geq 2^{10} |\eta| .\\
\end{array}\right.
\ee

\begin{lemma}\label{adjoint}
Let $m\in \R$ and $\rho >0$ and let $a\in \Gamma^{m}_{\rho}(\R^d)$, if we denote $(T_{a})^{\ast}$ as the adjoint operator of $T_a$ and denote $\bar{
a}$ as the complex conjugate of $a$, then we know that, $(T_{a})^{\ast}- T_{a^{\ast}}$ is of order $m-\rho$, where
\[
a^{\ast} = \sum_{|\alpha|< \rho} \frac{1}{i^{|\alpha|} \alpha !} \p_{\xi}^{\alpha} \p_{x}^{\alpha} \bar{a}.
\]
Moreover, the operator norm of $(T_{a})^{\ast} - T_{a^{\ast}}$ is bounded by $M^{m}_{\rho}(a).$
\end{lemma}
\begin{proof}
See \cite{alazard1}[Theorem 3.10].
\end{proof}
 \begin{lemma}\label{composi}
Let $m\in \R$ and $\rho >0,$ if given symbols $a\in \Gamma_{\rho}^{m}(\R^d)$ and $b\in\Gamma_{\rho}^{m'}(\R^d)$,  we can define
\[
a\sharp b = \sum_{|\alpha|< \rho} \frac{1}{i^{|\alpha|} \alpha!} \p_{\xi}^{\alpha} a \p_{x}^{\alpha}b,
\]
then for all $\mu\in\R$, there exists a constant $K$ such that 
\begin{equation}\label{eqn700000}
\| T_a T_b - T_{a\sharp b}\|_{H^{\mu}\rightarrow H^{\mu-m-m'+\rho}} \leq K M^{m}_{\rho}(a) M^{m'}_{\rho}(b).
\end{equation}
\end{lemma}

The following lemma on the  $L^\infty$  estimate of the linear solution holds, 
\begin{lemma}\label{lineardecay}
For $f\in L^1(\R^2)$, we have the following $L^\infty$ type estimates:
\begin{equation}\label{highdecay}
\| e^{it \Lambda} P_{k} f\|_{L^\infty} \lesssim (1+|t|)^{-1} 2^{k/2} \| f\|_{L^1},  \quad \textup{if $k\geq 0$}.
\end{equation}
\begin{equation}\label{lowdecay}
\| e^{it \Lambda} P_{k} f\|_{L^\infty} \lesssim (1+|t|)^{-\frac{1+\theta}{2}}2^{\frac{(1-\theta)k}{2} } \| f\|_{L^1},\quad 0\leq \theta \leq 1, \quad \textup{if $k\leq 0$}.
\end{equation}

\end{lemma}
\begin{proof}
After  checking the expansion of the phase, see (\ref{expansion1}), we can apply the main result in \cite{Guoz}[Theorem 1:(a)\&(b)] directly to derive above results.
\end{proof}
\begin{lemma}
The quadratic terms of $G(h)\psi$ is given as follows, 
\be\label{quadraticterm}
\Lambda_2[G(h)\psi]= -\nabla \cdot ( h \nabla \psi) - \d \tanh(\d)( h \d \tanh(\d) \psi).
\ee
\end{lemma}
\begin{proof}
See   \cite{wang2}[Lemma 3.4].
\end{proof}

\section{The energy estimate}\label{energyestimatesection}

Our bootstrap assumption is stated as follows, 
\be\label{bootstrapassumption}
\sup_{t \in [0, T]} (1+t)^{-\delta} \| ( \tilde{\Lambda} h(t), \psi(t)\|_{H^{N_0}}+  (1+t) \| (  h(t),  \psi)(t)\|_{W^{6,1+\alpha}}  \lesssim \epsilon_1:=\epsilon_0^{5/6},
\ee
where $ \tilde{\Lambda}:=  \d^{1/2} (\tanh \d)^{-1/2}$ and the function space  $W^{6,1+\alpha}$ was defined in (\ref{Linftyspace}).

The goal of this section is to prove that the energy of solution only grow appropriately. More precisely, we will use the paralinearization and symmetrization method to prove the following proposition,  
\begin{proposition}
Under the bootstrap assumption \textup{(\ref{bootstrapassumption})}, the following energy estimate holds for any $t\in[0,T]$,
\be\label{energyestimate}
\| (\tilde{\Lambda}  h(t),\psi(t))\|_{H^{N_0}}^2  \lesssim \epsilon_0^2 +   \int_0^t  \| ( \tilde{\Lambda} h(s),\psi(s))\|_{H^{N_0}}^2 \big( \| (h, \psi)\|_{W^{6,1+\alpha}}+\| (h, \psi)\|_{W^{6,0}}\| (h, \psi)\|_{W^{6,1}}   \big) d s .
\ee
\end{proposition}

We separate this section into three parts:  (i) Firstly, we introduce   main results and briefly explain   main ideas of the paralinearization process for the capillary waves system (\ref{waterwaves}). (ii) Secondly, with the highlighted structures of losing derivative inside the system (\ref{waterwaves}), we symmetrize the system (\ref{waterwaves}) such that it doesn't lose derivatives during energy estimate. (iii) Lastly, we use the  symmetrized system to prove the new   energy estimate (\ref{energyestimate}).

%cN2UK7Km5j%

\subsection{Paralinearization of the full system}
Most of this section  has been studied in details    in \cite{wang2}. Here we only briefly introduce the related main results and the main ideas behind those results. Please refer to \cite{wang2} for more detailed discussions.

To perform the paralinearization process,
 we need some basic estimates of the Dirichlet-Neumann operator, which are obtained from analyzing the velocity potential inside the water region ``$\Omega(t)$".

  Recall that the velocity potential satisfies the following Laplace equation with two boundary conditions as follows, 
 \be\label{harmonice}
 \Delta \phi = 0, \quad \phi\big|_{\Gamma(t)}
=\psi(t), \quad \p_{\vec{n}} \phi\big|_{\Sigma} =0. 
\ee
To simplify analysis, we map the water region ``$\Omega(t)$" into the strap $\mathcal{S}:= \R\times [-1,0]$ by doing change of coordinates as follows, 
\[
(x, y)\longrightarrow (x, z ), \quad z:=\frac{y-h(t, x)}{1+h(t, x)}.
\]
We define $\varphi(t,z): =\phi(t, z+h(t,x))$. From  (\ref{harmonice} ), we have
\be\label{harmonic2}
 P\varphi:=[ \Delta_{x}+ \tilde{a}\p_z^2 +  \tilde{b}\cdot \nabla \p_z + \tilde{c}\p_z ] \varphi=0  , \quad \varphi\big|_{z=0}=\psi , \quad \p_z \varphi\big|_{z=-1}=0,
\ee
where
\begin{equation}\label{coeff}
\tilde{a}= \frac{(y+1)^2|\nabla h|^2}{(1+h)^4}  + \frac{1}{(1+h)^2}=\frac{1+(z+1)^2|\nabla h|^2}{(1+h)^2},
\end{equation}
\begin{equation}\label{coeff1}
\tilde{b}=- 2 \frac{(y+1)\nabla h}{(1+h)^2} = \frac{-2(z+1)\nabla h}{1+h},\quad \tilde{c}= \frac{-(z+1)\Delta_{x} h}{(1+h)} + 2\frac{(z+1)|\nabla h|^2}{(1+h)^2}. 
\end{equation}
\begin{equation}\label{DN1}
G(\,h)\psi = [-\nabla\,h\cdot\nabla \phi + \p_y\phi]\big|_{y=\,h} = \frac{1+|\nabla \,h|^2}{1+\,h} \p_z \varphi \big|_{z=0} -\nabla\psi \cdot \nabla \,h. 
\end{equation}
Hence, to study the Dirichlet-Neumann operator, it is sufficient to study the only nontrivial part of $G(h)\psi$, which is $\p_z \varphi\big|_{z=0}.$

From (\ref{harmonic2}), we can derive the following fixed point type formulation for $\nabla_{x,z}\varphi$, which provides a good way to analyze and estimate the Dirichlet-Neumann operator in the small data regime. More precisely, we have
\[\nabla_{x,z}\varphi = \Bigg[ \Big[ \frac{e^{-(z+1)\d}+ e^{(z+1)\d}}{e^{-\d} + e^{\d}}\Big]\nabla\psi ,  \frac{e^{(z+1)\d}- e^{-(z+1)\d}}{e^{-\d}+e^{\d}} \d \psi\Bigg] + \]
\[
+ [\mathbf{0}, g_1(z)]+\int_{-1}^{0} [K_1(z,s)-K_2(z,s)-K_3(z,s)](g_2(s)+\nabla \cdot g_3(s))  ds \]
\begin{equation}\label{fixedpoint}
+\int_{-1}^{0} K_3(z,s)\d\textup{sign($z-s$)}g_1(s)  -\d [K_1(z,s) +K_2(z,s)]g_1(s)\, d  s,
\end{equation}
where
\begin{equation}\label{equation300}
K_1(z,s):=\Big[ \frac{\nabla}{2\d}\frac{e^{-z\d}- e^{z\d} }{e^{-\d} + e^{\d}}e^{(s-1)\d} +\frac{\nabla}{2\d}e^{(z+s)\d} ,  -\h \frac{e^{z\d}+e^{-z\d}}{e^{-\d} +e^{\d}}e^{(s-1)\d} + \frac{1}{2}e^{(z+s)\d}  \Big],
\end{equation}
\begin{equation}\label{equation301}
K_2(z,s):= \Big[ \frac{\nabla}{2\d}\frac{e^{-z\d}- e^{z\d} }{e^{-\d} + e^{\d}}e^{-(s+1)\d}\,\, , \,\,
  -\h \frac{e^{z\d}+e^{-z\d}}{e^{-\d} +e^{\d}} e^{-(s+1)\d} \Big],
\end{equation}
\begin{equation}\label{equation302}
K_3(z,s)= \Big[  \frac{\nabla}{2\d}e^{-|z-s|\d} \,\, , \,\, \frac{1}{2}e^{-|z-s|\d}\textup{sign($s-z$)}\Big].
\end{equation}

\begin{equation}\label{eqn12}
g_1(z) =  \frac{2\,h+\,h^2 - (z+1)^2 |\nabla\,h|^2}{(1+\,h)^2} \p_z \varphi +\frac{(z+1)\nabla\,h\cdot \nabla\varphi}{1+\,h},\quad g_1(-1)=0,
\end{equation}
\begin{equation}\label{eqn14}
g_2(z) =\frac{(z+1)|\nabla \,h|^2 \p_z\varphi}{(1+\,h)^2}  - \frac{\nabla \,h \cdot \nabla\varphi}{1+\,h} ,\quad g_3(z)= \frac{(z+1)\nabla \,h \p_z\varphi}{1+\,h}.
\end{equation}
From (\ref{fixedpoint}), it is sufficient to derive the following $L^2$ type and $L^\infty$ type estimates for $\nabla_{x,z}\varphi$, which are the very first step and also  very essential. 
\begin{lemma}\label{Sobolevestimate}
Under the smallness condition \textup{(\ref{bootstrapassumption})},  the following  estimates hold for $\nabla_{x,z}\varphi$,
\begin{equation}\label{eqn2200}
 \|\nabla_{x,z}\varphi\|_{L^\infty_z H^k}  \lesssim \| \nabla\psi\|_{H^k} + \| \,h\|_{H^{k+1}} \| \nabla\psi\|_{\widetilde{W^0}},\quad 
\end{equation}
\begin{equation}\label{eqn2201}
\| \nabla_{x}\varphi\|_{L^\infty_z \widetilde{W^\gamma}} \lesssim \| \nabla \psi \|_{\widetilde{W^{\gamma}}} , \quad \| \p_z \varphi\|_{L^\infty_z \widetilde{W^{\gamma}}}\lesssim\|   \psi\|_{\widehat{W}^{\gamma,1+\alpha}}+  \| \,h\|_{\widetilde{W^{\gamma+1}}} \| \nabla \psi \|_{\widetilde{W^\gamma}},
\end{equation}
\begin{equation}\label{eqn2203}
\|\Lambda_{\geq 2}[\nabla_{x,z}\varphi]\|_{L^\infty_z\widetilde{W^{\gamma}}} \lesssim \| \nabla\psi\|_{\widetilde{W^\gamma}}\| \,h\|_{\widetilde{W^{\gamma+1}}},
\end{equation}
\begin{equation}\label{eqn2202}
\| \Lambda_{\geq 2}[\nabla_{x,z}\varphi]\|_{L^\infty_z H^k} \lesssim  \| \,h\|_{\widetilde{W^1}}\| \d\psi\|_{H^k} + \| \nabla \psi\|_{\widetilde{W^0}}\| \,h\|_{H^{k+1}},
\end{equation}
\[
 \| \Lambda_{\geq 2}[\nabla_{x,z}\varphi]\|_{L^\infty_z L^2} \lesssim \big( \| (h, \psi)\|_{W^{6,1+\alpha}} +\|(h, \psi)\|_{W^{6,1}} \| (h, \psi)\|_{W^{6,0}} \big)\|(h, \psi)\|_{H^2},
\]
where $k\leq k'-1$ and $1\leq \gamma \leq \gamma'-1$. In above estimates, the range of $z$ for the $L^\infty_z$ norm is restricted in $[-1,0].$
\end{lemma}
\begin{proof}
Thanks to the small date regime, above estimates can be obtained from the fixed point type formulation in (\ref{fixedpoint}).   With  minor modifications, it is easy to see  that  the proof of  above estimates is almost same as in the proof of Lemma 3.3 in \cite{wang2}.   
\end{proof}

During the paralinearization process,    we usually omit good error terms, which do not lose derivative and have good structures inside. To this end, we define the equivalence relation ``$\ta$'' as follows for    $k \geq 0$,
\[
A\thickapprox B, \quad \textup{if and only if $A-B$ is a good error term in the sense of (\ref{gooderror})},
\]
\begin{equation}\label{gooderror}
\|\textup{good error term}\|_{H^{k}}\lesssim_{k} \big[\|( h,  \psi)\|_{ {W}^{6,1+\alpha}} + \| ( h,  \psi)\|_{ {W}^{6,0}} \| (\,h,  \psi)\|_{ {W}^{6,1}}\big]
\big(\|   h \|_{H^k}  +\|  \psi \|_{H^{(k-1)_{+}}}   \big).
\end{equation}

As a result of  paralinearization, we have a good decomposition  of the Dirichlet-Neumann operator as follows, 
\begin{lemma}\label{paraDN}	
Under the smallness condition, we have
\be\label{paralinearization1}
G(h)\psi \approx   T_{\lambda} \omega - T_{ {V}}\cdot \nabla \,h ,\quad \omega:= \psi-T_{B} h, 
 \ee
 \[
 B \overset{\text{abbr}}{=} B(\,h)\psi =\frac{G(\,h)\psi + \nabla \,h \cdot \nabla \psi}{1+|\nabla \,h|^2},\quad V\overset{\text{abbr}}{=} V(\,h)\psi= \nabla \psi - B \nabla \,h,
 \]
\[ \lambda =\lambda^{(1)} + \lambda^{(0)}, \quad \lambda^{(1)} := \sqrt{(1+|\nabla  h|^2) |\xi|^2 -(\nabla h \cdot \xi)^2}, \]
\[ \lambda^{(0)}= \frac{1+|\nabla h|^2}{2\lambda^{(1)}} \big ( \nabla \cdot \big( \frac{\lambda^{(1)} + i \nabla h \cdot\xi }{ 1+ |\nabla h |^2} \nabla h\big)  + i \nabla_\xi \lambda^{(1)} \cdot \nabla\big( \frac{\lambda^{(1)} + i \nabla h \cdot\xi }{ 1+ |\nabla h |^2}\big)\big),
\]
where ``$\omega$'' is the so-called good unknown variable and $\lambda^{(1)}$ and $\lambda^{(0)}$ are the principle symbol and sub-principle  of the Dirichlet-Neumann operator respectively.
\end{lemma}
\begin{proof}
The detailed proof of above Lemma can be found in \cite{alazard1,wang2}. Only minor modifications are required.

 \end{proof}

As a result of paralinearization, the following good decomposition for the mean curvature $H(h)$ holds, 
\begin{lemma}\label{paraMC}
Under the smallness condition, we have
\[
H(h) \approx - T_{l} h ,\quad  l= l^{(2)}+ l^{(1)},\]
\[
l^{(2)} = (1+|\nabla h|^2)^{-1/2}\Big(|\xi|^2 - \frac{\big(\nabla h \cdot \xi \big)^2}{1+|\nabla h|^2}   \Big),\quad 
l^{(1)} = \frac{-i}{2} ( \nabla_x \cdot \nabla_\xi) h^{(2)}.
\]
\end{lemma}
\begin{proof}
See \cite{alazard1}.
\end{proof}
For other terms inside the nonlinearity of  $\p_t \psi$ in (\ref{waterwaves}),   the following Lemma holds.
\begin{lemma}\label{parave}
Under the smallness condition, we have
\be
\h |\nabla \psi|^2 - \h \frac{(\nabla\,h\cdot \nabla \psi + G(\,h)\psi)^2}{1+|\nabla \,h|^2}\ta T_{V}\cdot \nabla \omega - T_{B}G(\,h)\psi.  
\ee
\end{lemma}
\begin{proof}
See \cite{wang2}. 
\end{proof}

\subsection{Symmetrization of the full system}
In this subsection, we use the results we obtained in the paralinearization  process to find out the good substitution variables such that the system of equations satisfied by the good substitution variables has requisite symmetries inside.

Recall (\ref{waterwaves}) and results in Lemmas  \ref{paraDN}, \ref{paraMC}, and \ref{parave}, we have
\be\label{goodstructure}
\left\{ \begin{array}{l}
\p_t h \ta T_{\lambda} \omega - T_{V}\cdot \nabla h \\
\\
\p_t \psi \ta -T_l h +T_{B}G(h)\psi - T_{V}\cdot \nabla \omega.\\
\end{array}\right.
\ee
The symmetrization process, which is only relevant at the high frequency part, is same as what Alazard-Burq-Zully did in \cite{alazard1}.   We first state the main results and then briefly explain main ideas behind. 

The good substitution variables are given as follows,
\be\label{goodunknown}
U_1= \tilde{\Lambda}( h + T_{p|\xi|^{-1/2}-1 } h)  , \quad U_2 = \omega+ T_{q -1}\omega, \quad \omega = \psi -T_{B} h,
\ee
where
\be\label{solution1}
p= p^{(1/2)} + p^{(-1/2)},\quad q=  ( 1+|\nabla h |^2)^{-1/2} ,
\ee
\be\label{solution2}
p^{(1/2)}= (1+|\nabla h|^2)^{-5/4} \sqrt{\lambda^{(1)}}  , \quad p^{(-1/2)} = \frac{1}{\gamma^{(3/2)}} \big[ q  l^{(1)} - \gamma^{(1/2)} p^{(1/2)} + i \nabla_\xi \gamma^{(3/2)} \cdot \nabla_x p^{(1/2)} \big],
\ee
\be\label{gamma}
\gamma = \sqrt{l^{(2)} \lambda^{(1)}} + \sqrt{ \frac{l^{(2)}}{\lambda^{(1)}}   }\frac{\textup{Re} \lambda^{(0)}}{2} - \frac{i}{2}\big( \nabla_\xi\cdot \nabla_x\big) \sqrt{l^{(2)} \lambda^{(1)}} - |\xi|^{3/2}.
\ee
Note that, in the sense of losing derivatives, $U_1$ and $U_2$ are equivalent to $T_p h $ and $T_q \omega$. Here, we pulled out and emphasized the leading linear terms.

From the bootstrap assumption (\ref{bootstrapassumption}) and estimates in Lemma \ref{Sobolevestimate}, the following estimate holds, 
\be\label{comparableenergy}
\|  U_1 - \tilde{\Lambda} h\|_{H^{N_0}} + \| U_2 - \psi\|_{H^{N_0}} \lesssim (\| h \|_{W^{6,1}} + \|\psi\|_{W^{6,1}}) \| (\tilde{\Lambda} h,\psi)\|_{H^{N_0 }} \lesssim \epsilon_0.
\ee
Therefore, we know that the  energy of $(U_1, U_2)$  is comparable with the energy of $(\tilde{\Lambda} h , \psi)$. It would be sufficient to control the energy of $(U_1, U_2)$.

From (\ref{goodstructure}) and (\ref{goodunknown}), we can derive the following the system of equations satisfied by $U_1$ and $U_2$ as follows, 
\be\label{symmetriedvariable}
\left\{ \begin{array}{l}
\p_t U_1= \Lambda U_2 + T_{\gamma} U_2 - T_{V}\cdot \nabla U_1 + \mathfrak{R}_1\\
\\
\p_t U_2 = -\Lambda U_1 -T_{\gamma} U_1 - T_{V}\cdot \nabla U_2 + \mathfrak{R}_2,\\
\end{array}\right.
\ee
where $\mathfrak{R}_1$ and $\mathfrak{R}_2$ are good error terms in the sense of (\ref{gooderror}).  More precisely, we have
\be\label{remainderestimate2}
\| \mathfrak{R}_1\|_{H^{N_0}} + \| \mathfrak{R}_2 \|_{H^{N_0}} \lesssim_{N_0}   \big(\| (h, \psi) \|_{W^{6,1+\alpha}} + \|(h,\psi)\|_{W^{6,1}}\|(h,\psi)\|_{W^{6,0}}\big) \| (\tilde{\Lambda} h,\psi)\|_{H^{N_0 }}   .
\ee

The idea of symmetrization process  is straightforward. We are trying to find out symbols $p(x,\xi)$ and $q(x, \xi)$ such that the system of equations satisfied by $T_p h $ and $T_q \omega$ is symmetric as in (\ref{symmetriedvariable}). As $\lambda \in \Gamma^1_5$ and $l \in \Gamma^2_5$, naturally, we are looking for   $p\in \Gamma^{1/2}_5$,  $q\in \Gamma^{0}_5$ and  $\lambda\in \Gamma^{3/2}_5$ . 

In order to symmetrize the system of equations satisfied by $(U_1, U_2)$,  the following conditions for $p(x,\xi)$,  $q(x, \xi)$, and  $\lambda(x, \xi)$ have to be satisfied, 
\be\label{condition1}
T_{p} T_{\lambda} \sim T_{\gamma + |\xi|^{3/2}} T_{q}, 
\ee
\be\label{condition2}
   T_{q } T_{l} \sim T_{\gamma+ |\xi|^{3/2}} T_{p},
\ee
\be\label{condition3}
T_{\gamma} \sim (T_{\gamma})^{\ast}, 
\ee
where the equivalence relation ``$\sim$'' is defined in the following sense for any $k \geq 0$,
\[
T_{a_1}\sim T_{a_2}, \quad \textup{iff}\, \| T_{a_1} f -T_{a_2} f\|_{H^k} \lesssim_{k} \big( \| (h, \psi)\|_{W^{6,1+\alpha}}  + \| (h, \psi)\|_{W^{6,1}} \| (h, \psi)\|_{W^{6,0}} \big) \| f\|_{H^k}.
\]
Conditions (\ref{condition1}) and (\ref{condition2}) follow directly from the definitions of $U_1$ and $U_2$ and the highlighted principle symbols in    
(\ref{goodstructure}).  Condition (\ref{condition3}) follows from  avoiding losing derivatives during energy estimate.  From Lemma \ref{adjoint}, we have	
\be\label{adjointgamma}
\big(T_{\gamma} \big)^{\ast} \sim T_{\lambda^{\ast}}, \quad \lambda^{\ast}=  \gamma^{(3/2)} + \overline{\gamma^{(1/2)}} + \frac{1}{i}\nabla_\xi \cdot \nabla_x \gamma^{(3/2)}. 
\ee 
 Hence, (\ref{condition3}) can be reformulated as follows, 
 \be\label{condition3reformulate}
T_{\gamma} \sim T_{\gamma^{(3/2)} + \overline{\gamma^{(1/2)}} + \frac{1}{i}\nabla_\xi \cdot \nabla_x \gamma^{(3/2)}}.
 \ee
By using Lemma \ref{composi}, we can derive several equations about the principle symbols and sub-principle symbols of $p(x,\xi)$, $q(x,\xi)$, and $\gamma(x, \xi)$  from (\ref{condition1}), (\ref{condition2}), and (\ref{condition3reformulate}). By solving those equations, one can see that  the principle symbols and sub-principle symbols of $p(x,\xi)$, $q(x,\xi)$, and $\gamma(x, \xi)$    are given as in (\ref{solution1}), (\ref{solution2}), and (\ref{gamma}).  For more detailed computations, please refer to \cite{alazard1}[subsection 4.2].
\subsection{Energy estimate}

We define the energy as follows, 
\[
E_{N_0}(t):=  \|U_1\|_{L^2}^2+\|U_2\|_{L^2}^2 +  \| U_1^{N_0}\|_{L^2}^2 +   \| U_2^{N_0}\|_{L^2}^2 ,
\]
where
\[
U_1^{N_0} = T_{\beta} U_1, \quad U_2^{N_0} = T_{\beta} U_2,\quad \beta:=\big(\gamma^{(3/2)} +|\xi|^{3/2}\big)^{2N_0/3}.
\]
Note that   
\[
\p_\xi \beta \p_x \big( \gamma^{(3/2)} +|\xi|^{3/2} \big)= \p_\xi\big( \gamma^{(3/2)} +|\xi|^{3/2} \big)  \p_x \beta. 
\]
Hence, very importantly, the operator as follows is an operator of order zero, 
\[
T_{\beta}  T_{\gamma + |\xi|^{3/2}} - T_{\gamma + |\xi|^{3/2}} T_{\beta} .
\]
\begin{remark}
We choose to use the variable $T_\beta U_i$ instead of using $\d^{N_0} U_i$ to   estimate the high order Sobolev norm because the commutator $[T_{|\xi|^{N_0}}, T_{\gamma}]$ is of order $1/2$, which will cause the loss of derivatives. The idea of using the good variables $T_{\beta} U_1$ and  $T_{\beta} U_2$ comes from the work of Alazard-Burq-Zuily \cite{alazard1}.
\end{remark}
From (\ref{gamma}) and the smallness assumption (\ref{bootstrapassumption}), it's easy to see that
\[
E_{N_0}\sim  \| U_1\|_{H^{N_0}}^2+  \| U_1\|_{H^{N_0}}^2\sim \| h \|_{H^{N_0+1/2}}^2 + \| \psi\|_{H^{N_0}}^2. 
\]

Therefore, we can derive the system of equations satisfied by $U_1^{N_0}$ and $U_2^{N_0}$ as follows, 
\be\label{symmetriedvariablehigh}
\left\{ \begin{array}{l}
\p_t U_1^{N_0}= \Lambda  U_1^{N_0} + T_{\gamma}  U_2^{N_0} - T_{V}\cdot \nabla  U_1^{N_0} + \mathfrak{R}_1^{N_0}\\
\\
\p_t U_2^{N_0} = -\Lambda U_1^{N_0} -T_{\gamma}  U_2^{N_0} - T_{V}\cdot \nabla  U_2^{N_0} + \mathfrak{R}_2^{N_0},\\
\end{array}\right.
\ee
where the good remainder  terms $\mathfrak{R}_1^{N_0} $  and  $\mathfrak{R}_1^{N_0} $ satisfy the following estimate, 
\be\label{eqnnn121}
\|\mathfrak{R}_1^{N_0} \|_{L^2}  + \|\mathfrak{R}_2^{N_0} \|_{L^2} \lesssim_{N_0}   \big(\| (h, \psi) \|_{W^{6,1+\alpha}} + \|(h,\psi)\|_{W^{6,1}}\|(h,\psi)\|_{W^{6,0}}\big) \| (\tilde{\Lambda} h,\psi)\|_{H^{N_0 }}.
\ee
From 
(\ref{symmetriedvariable}), (\ref{remainderestimate2}),  (\ref{symmetriedvariablehigh}), and (\ref{eqnnn121}), we have
\[
\Big|\frac{d}{d t} E_{N_0}(t) \Big|\lesssim \| (U_1, U_2)\|_{H^{N_0}}  \| (\mathfrak{R}_1, \mathfrak{R}_2, \mathfrak{R}_1^{N_0}, \mathfrak{R}_2^{N_0})\|_{L^2}   \]
\[ +    \Big|\int_{\R^2}  U_1 \big( -T_{V}\cdot \nabla U_1\big)  +   U_2   \big(-T_{V}\cdot \nabla U_2\big)+ U_1^{N_0} \big( -T_{V}\cdot \nabla U_1^{N_0}\big)  +  U_2^{N_0}   \big(-T_{V}\cdot \nabla U_2^{N_0}\big)  d x \Big|	\]
\[
+ \Big| \int_{\R^2}    U_1   (  T_{\lambda} U_2 ) -   U_2  ( T_{\lambda} U_1 ) +   U_1^{N_0}   (  T_{\lambda} U_2^{N_0} ) -   U_2^{N_0}  ( T_{\lambda} U_1^{N_0} ) \Big|
\]
\[
\lesssim_{N_0}  \big(\| (h, \psi) \|_{W^{6,1+\alpha}} + \|(h,\psi)\|_{W^{6,1}}\|(h,\psi)\|_{W^{6,0}}\big)\|(U_1, U_2)\|_{H^{N_0}}^2  
\]
\[
+ \Big| \int_{\R^2}    U_1   (  T_{\lambda} -(T_{\lambda} )^{\ast}) U_2 )  +   U_1^{N_0}    (  T_{\lambda} -(T_{\lambda} )^{\ast})  U_2^{N_0}    d x    \Big|
\]
\begin{equation}\label{eqn8110000}
\lesssim \big(\| (h, \psi) \|_{W^{6,1+\alpha}} + \|(h,\psi)\|_{W^{6,1}}\|(h,\psi)\|_{W^{6,0}}\big)\|(U_1, U_2)\|_{H^{N_0}}^2 .
\end{equation}
Hence finishing with the proof. In above estimate, we used the following facts,  
\be\label{estimateofgamma1}
\Lambda_{1}[\gamma] = |\xi|^{1/2}\big(\frac{1}{2}\Delta h- \frac{\xi}{|\xi|}\cdot \nabla_x(\nabla h \cdot \frac{\xi}{|\xi|}) ), 
\ee
\be\label{estimateofgamma2}
M^{0}_{5}(\Lambda_{\geq 2}[\gamma] - \Lambda_{\geq 2}[\gamma^{\ast}]) \lesssim \| h\|_{W^{6,1}}^2.
\ee

The   equality  in  (\ref{estimateofgamma1})  can be derived from the explicit formula of $\gamma$ in (\ref{gamma}). Note that $\Lambda_{1}[\gamma]$ only depends on the second derivative of $h$, which explains why we can gain $(1+\alpha)$ derivatives at the low frequency part  for the input putted in $L^\infty$. The   estimate in  (\ref{estimateofgamma2}) is a direct consequence of (\ref{adjointgamma}).

\section{The set-up of the  weighted norm estimates}\label{setupweightednorms}

 Recall the capillary waves system (\ref{waterwaves}) and the quadratic terms of nonlinearities in (\ref{e10}) and (\ref{e11}).  To avoid losing derivatives for the quadratic terms, we define  
\be\label{modified}
\tilde{\psi}:=\psi-T_{\d \tanh\d\psi} h, 	
\ee
  which is the linear part and quadratic part of our good unknown variable ``$\omega$'' in (\ref{paralinearization1}).

From (\ref{e10}) and (\ref{e11}),   direct computations give us  the following equalities, 
\[
\Lambda_{\leq 2}[\p_t h ] = \d\tanh\d \tilde{\psi}+  \d \tanh \d \big(T_{\d \tanh \d \tilde{\psi}} h  \big)
\]
\be\label{eqn483}
- \nabla\cdot(h \nabla \tilde{\psi}) -  \d \tanh \d( h  \d \tanh \d\tilde{\psi}),
\ee
\be\label{eqn484}
\Lambda_{\leq 2}[\p_t \tilde{\psi}] = \Delta h - \h |\nabla \tilde{\psi}|^2 + \h |\d\tanh \d \tilde{\psi} |^2 - T_{\d\tanh \d \tilde{\psi}}  \d\tanh \d \tilde{\psi} -T_{\d\tanh \d \Delta h} h.
\ee
 We remark that the Taylor expansions in (\ref{eqn483}) and (\ref{eqn484})   are in terms of $h$ and $\tilde{\psi}$. In later contexts, the Taylor expansions    are all in terms of $h$ and $\tilde{\psi}$. 

Define $u= \tilde{\Lambda} h + i \tilde{\psi}   $. Recall that   $\tilde{\Lambda}=\d^{ 1/2} \big(\tanh\d\big)^{-1/2}$. Very naturally, we have
\begin{equation}\label{eqn1200}
h = \tilde{\Lambda}^{-1}\big( \frac{u+\bar{u}}{2}\big), \quad \tilde{\psi} =   c_{+} u + c_{-} \bar{u},
\end{equation}
where  $c_{+}:= -i/2$ and $c_{-}:=i/2$. Hence, from (\ref{waterwaves}), (\ref{eqn483}), and (\ref{eqn484}),  we can derive the equation satisfied by $u$  as follows, 
\[
(\p_t + i\Lambda)u = \sum_{\mu, \nu\in \{+,-\}} Q_{\mu, \nu}(u^{\mu}, u^{\nu}) + \sum_{\tau, \kappa,\iota\in \{+,-\}} C_{\tau, \kappa,\iota}(u^{\tau}, u^{\kappa}, u^{\iota}  ) 
\]
\begin{equation}\label{complexversion}
+ \sum_{\mu_1, \mu_2,\nu_1, \nu_2\in \{+,-\}} D_{\mu_1, \mu_2,\nu_1, \nu_2}(u^{\mu_1}, u^{\mu_2}, u^{\nu_1} , u^{\nu_2} ) +  \mathcal{R}, 
\end{equation}
where  $\mathcal{R}$ denotes the quintic and higher order terms.

 We  give the    detailed formulation of quadratic terms  $Q_{\mu, \nu}(\cdot, \cdot)$  here, as the structures inside the quadratic terms are very essential in the whole argument. 

 From (\ref{eqn483}), (\ref{eqn484}), and (\ref{eqn1200}), we have,
\[
Q_{\mu, \nu}(u^{\mu}, u^{\nu})= -\frac{c_{\nu}}{2}\tilde{\Lambda}  \p_x(\tilde{\Lambda}^{-1} u^{\mu} \p_x  u^{\nu}) -\frac{c_{\nu}}{2} \tilde{\Lambda}  \d\tanh \d( \tilde{\Lambda}^{-1}u^{\mu} \d\tanh\d  u^{\nu} - T_{\d\tanh\d  u^{\nu}} \tilde{\Lambda}^{-1}u^{\mu}  )
\]
\[
+ \frac{i c_{\mu}c_{\nu}}{2} \big[-\nabla   u^{\mu}\cdot\nabla  u^{\nu} + \d\tanh\d   u^{\mu} \d\tanh \d    u^{\nu} - T_{ \d\tanh\d   u^{\mu}}\d\tanh \d    u^{\nu}
\]
\begin{equation}\label{quadraticterms}
- T_{\d\tanh \d    u^{\nu}}\d\tanh\d   u^{\mu}\big] - \frac{i  }{4}  T_{\d\tanh \d \Delta u^{\nu}} u^{\mu}.
\end{equation}

Define the profile of the solution $u(t) $ as $f(t):= e^{i t\Lambda} u(t)$.  From (\ref{complexversion}), we have
\[
\p_t \widehat{f}(t, \xi) = \sum_{ (\mu, \nu)\in \{+, -\}} \int_{\R^2 } e^{i t\Phi^{\mu, \nu}(\xi, \eta)}  {q}_{\mu, \nu}(\xi-\eta, \eta) \widehat{f^{\mu}}(t, \xi-\eta) \widehat{f^{\nu}}(\eta) d \eta
\]
\[
+  \sum_{\tau, \kappa,\iota\in \{+,-\}} \int_{\R^2} \int_{\R^2} e^{i t\Phi^{\tau,\kappa, \iota}(\xi, \eta, \sigma)}   {c}_{\tau,\kappa, \iota}(\xi-\eta, \eta-\sigma, \sigma) \widehat{f^{\tau}}(t, \xi-\eta) \widehat{f^{\kappa}}(t,\eta-\sigma) \widehat{f^{\iota}}(t,\sigma) d \eta d \sigma 
\]
\[ 
 +   \sum_{\mu_1, \mu_2,\nu_1, \nu_2\in \{+,-\}}  \int_{\R^2} \int_{\R^2} e^{i t\Phi^{\mu_1, \mu_2,\nu_1, \nu_2}(\xi, \eta, \sigma,\kappa)}   {d}_{\mu_1, \mu_2,\nu_1, \nu_2}(\xi-\eta, \eta-\sigma, \sigma-\kappa, \kappa) \widehat{f^{\mu_1}}(t, \xi-\eta)\]
\be\label{duhamel}
\times \widehat{f^{\mu_2}}(t,\eta-\sigma) \widehat{f^{\nu_1}}(t,\sigma-\kappa) \widehat{f^{\nu_2}}(t, \kappa)d \eta d \sigma d \kappa + e^{it\Lambda(\xi)} \widehat{\mathcal{R}}(t, \xi) ,
\ee
where  the phases $\Phi^{\mu, \nu}(\xi, \eta)$,  $\Phi^{\tau,\kappa, \iota}(\xi, \eta, \sigma)$, and $\Phi^{\mu_1, \mu_2,\nu_1, \nu_2}(\xi, \eta, \sigma,\kappa)$ are defined as follows,
\begin{equation}\label{phases}
\Phi^{\mu,\nu}(\xi, \eta)= \Lambda(|\xi|)- \mu \Lambda(|\xi-\eta|) -\nu \Lambda(|\eta|), \quad   \Lambda(|\xi|):=|\xi|^{3/2}\sqrt{\tanh |\xi|},
\end{equation}
\be\label{phaseofcubic}
\Phi^{\tau,\kappa, \iota}(\xi, \eta,\sigma)= \Lambda(|\xi|)- \tau \Lambda(|\xi-\eta|) -\kappa \Lambda(|\eta-\sigma|) - \iota \Lambda(|\sigma|), 
\ee
\be\label{phaseofquartic}
\Phi^{\mu_1, \mu_2,\nu_1, \nu_2}(\xi, \eta, \sigma,\kappa)=\Lambda(|\xi|)- \mu_1 \Lambda(|\xi-\eta|) -\mu_2 \Lambda(|\eta-\sigma|) - \nu_1 \Lambda(|\sigma-\kappa|)- \nu_2 \Lambda(| \kappa|).  
\ee
From (\ref{quadraticterms}),  we write explicitly the  symbol ${q}_{\mu, \nu}(\xi-\eta, \eta)$ of $Q_{\mu, \nu}(u^{\mu}, u^{\nu})$  as follows, 
\[
 {q}_{\mu, \nu}(\xi-\eta, \eta)= \Big(    \frac{c_{\nu} \tilde{\lambda} (  | \xi|^2)}{ 2\tilde{\lambda} ( |\xi- \eta|^2)} \big(\xi\cdot \eta - |\xi||\eta|\tanh(|\xi|)\tanh(|\eta|)\big)  +  \frac{i c_{\mu}c_{\nu}}{2}\big((\xi-\eta) \cdot \eta+ |\xi-\eta||\eta| \]
 \[ \times \tanh(|\xi-\eta|) \tanh(|\eta|)\big)  \Big)    \tilde{\theta
}(\eta, \xi-\eta ) + \Big(  \frac{c_{\mu} \tilde{\lambda} (  | \xi|^2)}{ 2\tilde{\lambda}(|\eta|^2)}  \big( (\xi-\eta)\cdot \xi- |\xi-\eta||\xi| \tanh(|\xi|)\tanh(|\xi-\eta|)\big)
\]
\be\label{symmetricsymbol}
  + \frac{c_{\nu} \tilde{\lambda} (  | \xi|^2)}{ 2\tilde{\lambda}(|\xi-\eta|^2)}\xi \cdot \eta   + i c_{\mu}c_{\nu} (\xi-\eta)\cdot \eta + \frac{i}{4} |\eta|^2(\tanh |\eta| )^2 \Big) \theta(\eta, \xi-\eta),
\ee
where
 \[  \tilde{\lambda}(  x):= |\xi|^{ 1/4}( \tanh(\sqrt{|\xi|}))^{-1/2}, \quad \tilde{\lambda}(\xi)\approx 1 + \frac{x}{6}, \quad |\xi|\ll 1,
 \]
 \be
\tilde{\theta}(\eta, \xi-\eta):= 1-\theta(\eta, \xi-\eta)-\theta(\xi-\eta, \eta).
 \ee

Note that, we switched the roles of $\xi-\eta$ and $\eta$  in (\ref{symmetricsymbol}) when $|\xi-\eta|\ll |\eta|$. As a result, \emph{ $|\eta|$ is always relatively smaller than $|\xi-\eta|$ inside the symbol  $ {q}_{\mu, \nu}(\xi-\eta, \eta)$, $(\mu, \nu) \in \{+,-\}.$} 

From Lemma \ref{Snorm} and the detailed formula in (\ref{symmetricsymbol}), the following  rough estimate of the size of symbol $ {q}_{\mu, \nu}(\xi-\eta, \eta)$ holds, 
\be\label{symbolquadraticrough}
\|  {q}_{\mu, \nu}(\xi-\eta, \eta)\psi_k(\xi)\psi_{k_1}(\xi-\eta)\psi_{k_2}(\eta)\|_{\mathcal{S}^\infty} \lesssim 2^{2k_{1}}.
\ee

 It turn out  that it is very essential to identify the hidden symmetry inside the cubic terms, which will be very help in later high order weighted norm estimate of the cubic terms. To this end, we prove the following lemma, which  shows the leading part of the symbols of cubic terms, which has the requisite symmetric structure.
\begin{lemma}\label{symbolcubicandquartic}
After writing the cubic term $\Lambda_{3}[B(h)\psi]$ in terms of $u$ and $\bar{u}$ via \textup{(\ref{eqn1200})}, we do dyadic decompositions for all inputs and rearrange inputs such that the following unique decomposition holds
\[
\Lambda_{3}[B(h)\psi] =\sum_{\mu, \nu, \tau\in\{+,-\}}  {C'}_{\mu, \nu, \tau}(u^{\mu}, u^{\nu}, u^{\tau}),
\]
and moreover the first input $u^{\mu}$ of cubic term $  {C'}_{\mu, \nu, \tau}(u^{\mu}, u^{\nu}, u^{\tau})$ has the largest frequency among three inputs. 
The following estimates hold for the  symbol $ {c'}_{\mu, \nu,\tau}(\xi, \eta, \sigma)$ of the cubic term $C'_{\mu, \nu, \tau}(u^{\mu}, u^{\nu}, u^{\tau})$, 
\be\label{symbolcubic}
\|  {c'}_{\mu, \nu, \tau}(\xi, \eta, \sigma)\psi_{k_1}(\xi-\eta)\psi_{k_2}(\eta-\sigma)\psi_{k_3}(\sigma) \|_{\mathcal{S}^\infty}\lesssim 2^{ 2k_1 + 2k_{1,+}}.
\ee
\be\label{removebulk}
\| \big({c'}_{\mu, \nu, \tau}(\xi, \eta, \sigma)-\frac{c_{\mu}}{4} d(\xi)  \big)\psi_{k_1}(\xi-\eta)\psi_{k_2}(\eta-\sigma)\psi_{k_3}(\sigma) \|_{\mathcal{S}^\infty} \lesssim 2^{\max\{ k_2,k_3\}  + 3k_{1,+}  }, \textup{if}\,\, k_2,k_3\leq k_1-10.
\ee
where  the detailed formula of $d(\xi)$ is given in \textup{(\ref{eqn630})}. Moreover, the following  rough estimate holds for the  symbol of quartic terms  $\Lambda_{4}[B(h)\psi]$,  
\be\label{symbolquartic}
\| d_{\mu_1,\nu_1,\mu_2,\nu_2}(\xi, \eta, \sigma,\kappa)\psi_{k_1}(\xi-\eta)\psi_{k_2}(\eta-\sigma)\psi_{k_3}(\sigma-\kappa)\psi_{k_4}(\kappa) \|_{\mathcal{S}^\infty}\lesssim 2^{ 2\max\{k_1,\cdots,k_4\}+ 3\max\{k_1, \cdots,k_4\}_{+}}.
\ee

\end{lemma}
\begin{proof}
The detailed formulas of symbols of cubic terms and quartic terms can be derived from iterating the fixed point type formulation of $\nabla_{x,z}\varphi$ in  (\ref{fixedpoint}). To prove (\ref{symbolcubic}) and (\ref{symbolquartic}), it is sufficient to prove the corresponding estimates  for $ \Lambda_{3}[ g_i(z)] $ and $ \Lambda_{4}[ g_i(z)] $ for 	$i \in \{1,2,3\}$.   From (\ref{eqn12}) and (\ref{eqn14}),  we have
\[
\Lambda_{2}[g_1(z)] = 2 h \Lambda_1[\p_z \varphi] + (z+1) \nabla \,h \cdot \Lambda_{1}[ \nabla\varphi], \]
\[ \Lambda_{2}[g_2(z)] =-\nabla \,h \cdot \Lambda_{1}[ \nabla\varphi], \quad  \Lambda_{2}[g_3(z)] =(z+1)\nabla h \Lambda_{1}[ \p_z \varphi].
\]
Recall that
\[
\Lambda_1[\nabla_{x,z} \varphi] = \Bigg[ \Big[ \frac{e^{-(z+1)\d}+ e^{(z+1)\d}}{e^{-\d} + e^{\d}}\Big]\nabla\psi ,  \frac{e^{(z+1)\d}- e^{-(z+1)\d}}{e^{-\d}+e^{\d}} \d \psi\Bigg].
\]
Hence, we know that there are  two derivatives inside $\Lambda_{2}[\nabla_{x,z}\varphi(z)]$ at the low frequency part. We can keep doing this process to check the minimal and maximal numbers of derivatives inside $\Lambda_{3}[\nabla_{x,z}\varphi]$. Again, from (\ref{eqn12}) and (\ref{eqn14}), we have
\[
\Lambda_{3}[g_1(z)]= 2 h \Lambda_{2}[\p_z \varphi] + \big(-3h^2-(z+1)^2 |\nabla h|^2 \big) \Lambda_{1}[\p_z \varphi] -(z+1) h \nabla h \cdot \nabla \varphi,
\]
\[
\Lambda_{3}[g_2(z)]=(z+1)|\nabla h|^2 \Lambda_{1}[\p_z \varphi] + h \nabla h \cdot \Lambda_{1}[\nabla \varphi]- \nabla h \cdot\Lambda_{2}[\nabla \varphi],
\]
\[
\Lambda_{3}[g_3(z)]=(z+1)\nabla h \Lambda_{2}[\p_z \varphi] - (z+1) h \nabla h \Lambda_{1}[\p_z \varphi].
\]
Recall that there are at least two derivatives inside $\Lambda_{2}[\nabla_{x,z}\varphi]$, hence we know that there are at least two derivatives and at most four derivatives in total inside $\Lambda_{3}[\nabla_{x,z}\varphi]$.

Very similarly, we can see that there are at least two derivatives and at most five derivatives inside $\Lambda_4[\nabla_{x,z}\varphi]$. To sum up,  our desired estimates (\ref{symbolcubic}) and  (\ref{symbolquartic}) hold. 

To prove (\ref{removebulk}), we only need to identify the bulk terms that with all derivatives hit on the input that has the largest frequency. Recall   (\ref{eqn12})  and (\ref{eqn14}). The bulk term  only appears in $g_1(z)$, which is $
\displaystyle{T_{   \frac{2h + h^2}{(1+h)^2} }\p_z \varphi}. 
$
 Recall (\ref{fixedpoint}). It is easy to identify  the problematic part of $\Lambda_{2}[\p_z\varphi(z)]$  as follows, 
\[
 \int_{-1}^0 \big( - e^{-|z-s|\d}- e^{(z+s)\d} +   \frac{e^{z\d}+e^{-z\d}}{e^{-\d} +e^{\d}} (e^{(s-1)\d}+ e^{-(s+1)\d}) \big)\]
 \[\times  \frac{e^{s+1} \d - e^{-(s+1)\d}}{ e^{\d}+e^{-\d}} \d^2 \big( T_h  \psi\big)  ds 	+ 2   \frac{e^{(z+1)\d}- e^{-(z+1)\d}}{e^{-\d}+e^{\d}} \d (T_h \psi).
\]
Very similarly, we can identify the symbol of the problematic part $C(h,h,\psi)$ of $\Lambda_{3}[\p_z \varphi ]\big|_{z=0}$ as follows, 
\[
C(h,h,\psi) = \mathcal{F}^{-1}\big[\int_{\R^2} \int_{\R^2} \widehat{\psi}(\xi-\eta)\widehat{h}(\eta-\sigma) \widehat{h}(\sigma) d(\xi)\theta(\sigma,\xi) \theta(\eta-\sigma, \xi) d \eta d \sigma \big],
\]
where
 \[
d(\xi):= 2  \int_{-1}^0 \int_{-1}^0  \big(    \frac{( e^{-(z+1)|\xi|}-e^{(z+1)|\xi|} )}{e^{-|\xi|} +e^{|\xi|}}  \big) \frac{e^{(s+1)|\xi|} - e^{-(s+1) |\xi|}}{ e^{|\xi|}+e^{-|\xi|}} 
 \big(   \frac{(e^{z|\xi|}+e^{-z|\xi|})(e^{(s-1)|\xi|}+ e^{-(s+1)|\xi|})}{e^{-|\xi|} +e^{|\xi|}}  \]
 \be\label{eqn630} 
- e^{-|z-s||\xi|}- e^{(z+s)|\xi|}   \big)  |\xi|^3 ds 	 d z 
-\int_{-1}^0 \big(  \frac{( e^{-(s+1)|\xi|}-e^{(s+1)|\xi|} ) }{e^{-|\xi|} +e^{|\xi|}} \big)^2   |\xi|^2  ds + \tanh(|\xi|)|\xi| .
\ee
 \end{proof}

\subsection{The  good substitution variable}\label{goodvariable}
In this subsection, we will find a good substituted variable ``$v(t)$''  to carry out the  analysis of the dispersion estimate. There are two criteria to be met for  ``$v(t)$''. Firstly, the $L^\infty$-norm of $v(t)$ has to be comparable with the $L^\infty$-norm of $u(t)$, otherwise, it is not helpful  to the $L^\infty$ estimate of $u(t)$. Lastly, the equation satisfied by ``$v(t)$'' has good structures inside. By that, we mean the following two properties hold:  (i)  `` $\nabla_\xi\, \textup{phase}$''  is small when the frequency  second largest frequency of inputs is much smaller than the largest frequency of inputs, (ii) symbol vanishes when the frequency is near  the space resonance set but far away from the time resonance set.

Our good substituted variable is defined as follows, 
\[
v(t)= u(t) + \sum_{\mu, \nu\in\{+,-\}} A_{\mu, \nu}( u^{\mu}(t), u^{\nu}(t)) + \sum_{  \tau, \kappa, \iota  \in\{+,-\}}  B_{ \tau, \kappa, \iota}(  u^\tau(t), u^\kappa(t),  {u}^\iota(t)) 
\]
\be\label{normalformatransfor}
+  \sum_{ \mu_1, \mu_2,\nu_1, \nu_2 \in \{+,-\} } E_{\mu_1, \mu_2,\nu_1, \nu_2}(u^{\mu_1}(t), u^{\mu_2}(t), u^{\nu_1}(t), u^{\nu_2}(t) ),
\ee
where quadratic terms $A_{\mu, \nu}(\cdot, \cdot)$, cubic terms $ B_{ \tau, \kappa, \iota}(\cdot, \cdot, \cdot)$, and quartic terms $ E_{\mu_1, \mu_2,\nu_1, \nu_2}(\cdot, \cdot, \cdot,\cdot)$ are to be determined.

From the equation satisfied by $u(t)$ in (\ref{complexversion}) and definition of $v(t)$ in (\ref{normalformatransfor}). We can compute the equation satisfied by ``$v(t)$''. We  substitute $u(t)$ by $v(t)$ inside the nonlinearity of $\p_t v(t)$  via (\ref{normalformatransfor}) several times such that, up to quartic terms, they are all in terms of $v(t)$.  As a result, we have
\[
(\p_t + i \Lambda) v = \sum_{\mu, \nu \in \{+,-\}} \tilde{Q}_{\mu, \nu}(v^{\mu}, v^{\nu}) + \sum_{\tau, \kappa, \iota\in\{+,-\}}\tilde{C}_{\tau, \kappa, \iota}(v^{\tau}, v^{\kappa}, v^{\iota}) \]
\be\label{realvariableweighted}
+ \sum_{ \mu_1, \mu_2,\nu_1, \nu_2 \in \{+,-\} } \tilde{D}_{\mu_1, \mu_2,\nu_1, \nu_2}(v^{\mu_1}(t), v^{\mu_2}(t), v^{\nu_1}(t), v^{\nu_2}(t) ) + \mathcal{R}_1(t),
\ee
where $\mathcal{R}_1(t)$ is the quintic and higher order terms. The quadratic terms are given as follows, 
\be\label{normalformquadratic}
\tilde{Q}_{\mu, \nu}(v^{\mu},v^{\nu})= Q_{\mu, \nu}(v^{\mu}, v^{\nu}) + i \Lambda ( A_{\mu, \nu}(v^{\mu}, v^{\nu}))  -i \mu A_{\mu, \nu}( \Lambda v^{\mu}, v^{\nu}) - i \nu A_{\mu, \nu}(  v^{\mu}, \Lambda v^{\nu}).
\ee

 Because we need to identify the symmetric structure inside the cubic terms, we also give the detailed formula of $\tilde{C}_{\tau, \kappa, \iota}(v^{\tau}, v^{\kappa}, v^{\iota}) $ here.  	
\[
\tilde{C}_{\tau, \kappa, \iota}(v^{\tau}, v^{\kappa}, v^{\iota}) := \widehat{C}_{\tau, \kappa, \iota}(v^{\tau}, v^{\kappa}, v^{\iota})  + i \Lambda( B_{ \tau, \kappa, \iota}( v^{\tau},    v^\kappa , v^{\iota})  )  -i\tau   B_{ \tau, \kappa, \iota}(\Lambda v^{\tau},    v^\kappa , v^{\iota})  
 \]
 \be\label{e98709}
-i\kappa B_{ \tau, \kappa, \iota}( v^{\tau},    \Lambda v^\kappa , v^{\iota})  -i\iota   B_{ \tau, \kappa, \iota}(  v^{\tau},    v^\kappa , \Lambda v^{\iota}),
\ee
\emph{ where the cubic term   $\widehat{C}_{\tau, \kappa, \iota}(v^{\tau}, v^{\kappa}, v^{\iota}) $ is the unique cubic term associated with the following equality, such that 
  $v^{\tau}$ has the largest frequency among inputs $v^{\tau}, v^{\kappa},$ and $ v^{\iota}$ after we rearrange the inputs,}
\[
\sum_{ \tau, \kappa, \iota \in  \{+,-\}}\widehat{C}_{\tau, \kappa, \iota}(v^{\tau}, v^{\kappa}, v^{\iota}) = 	\sum_{\tau,\kappa, \iota\in\{+,-\} } C_{\tau,\kappa, \iota}(v^{\tau},v^{\kappa}, v^{\iota}) + \sum_{\mu, \nu,\mu_1, \nu_1\in\{+,-\}} A_{\mu, \nu}( P_{\mu}[Q_{\mu_1, \nu_1}(v^{\mu_1},v^{\nu_1})], v^{\nu} )\]
\be\label{modifiedcubiciteration}
+ A_{\mu, \nu}(  v^{\nu},P_{\nu}[Q_{\mu_1, \nu_1}(v^{\mu_1}, v^{\nu_1})] ) - \tilde{Q}_{\mu, \nu}\big( P_{\mu}\big(A_{\mu_1, \nu_1}(v^{\mu_1}, u^{\nu_1})\big), v^{\nu}	\big) - \tilde{Q}_{\mu, \nu}\big(  v^{\mu},  P_{\nu}\big(A_{\mu_1, \nu_1}(v^{\mu_1}, v^{\nu_1})\big)	\big).
\ee
For the quartic terms, we have
\[
\tilde{D}_{\mu_1, \mu_2,\nu_1, \nu_2}(v^{\mu_1}(t), v^{\mu_2}(t), v^{\nu_1}(t), v^{\nu_2}(t) )= \widehat{D}_{\mu_1, \mu_2,\nu_1, \nu_2}(v^{\mu_1}(t), v^{\mu_2}(t), v^{\nu_1}(t), v^{\nu_2}(t) )
\]
\[
 + i \Lambda(E_{\mu_1, \mu_2,\nu_1, \nu_2}(v^{\mu_1}(t), v^{\mu_2}(t), v^{\nu_1}(t), v^{\nu_2}(t) ))- i\mu_1 E_{\mu_1, \mu_2,\nu_1, \nu_2}( \Lambda  v^{\mu_1}(t), v^{\mu_2}(t), v^{\nu_1}(t), v^{\nu_2}(t) )\]
 \[ - i\mu_2 E_{\mu_1, \mu_2,\nu_1, \nu_2}(  v^{\mu_1}(t),  \Lambda v^{\mu_2}(t), v^{\nu_1}(t), v^{\nu_2}(t) )- i\nu_1 E_{\mu_1, \mu_2,\nu_1, \nu_2}(  v^{\mu_1}(t),  v^{\mu_2}(t),  \Lambda v^{\nu_1}(t), v^{\nu_2}(t) )
\]
\be
- i\nu_2 E_{\mu_1, \mu_2,\nu_1, \nu_2}(  v^{\mu_1}(t),  v^{\mu_2}(t),  v^{\nu_1}(t),  \Lambda v^{\nu_2}(t) ),
\ee
 \emph{where   $\widehat{D}_{\mu_1, \mu_2,\nu_1, \nu_2}( v^{\mu_1}, v^{\mu_2}, v^{\nu_1} , v^{\nu_2})$ is the unique decomposition associated with the quartic terms  such that $v^{\mu_1}$ has the largest frequency among  the four inputs}. The detail formula of $\widehat{D}_{\mu_1, \mu_2,\nu_1, \nu_2}(\cdot, \cdot,\cdot, \cdot)$
   can be  obtained explicitly from $A_{\mu, \nu}(\cdot, \cdot),$ $ B_{\tau, \kappa, \iota}(\cdot, \cdot, \cdot),$ $ Q_{\mu, \nu}(\cdot, \cdot, \cdot)$,   $C_{\tau, \kappa, \iota}(\cdot, \cdot, \cdot)$, and the quartic terms  $ D_{\mu_1, \mu_2,\nu_1, \nu_2}(\cdot, \cdot, \cdot, \cdot)$ in (\ref{complexversion}).  Since the detailed formulas are not necessary in later argument, for simplicity, we omit detail formulas here.

In the following context, we discuss which part of frequencies that we want to remove by the normal form transformation we used in (\ref{realvariableweighted}).

Firstly, we consider  the quadratic terms. When  $|\eta|\ll |\xi|$, $\mu=-$ or $|\xi|\ll |\eta|$, $\mu\nu=+$, we know that    ``$\nabla_\xi\,\textup{phase}$ '' is very big, which is not what we want,  and the size of phase is relatively big. Also note that, when $(\xi, \eta)$ lies inside a small neighborhood of $(\xi, \xi/2)$(the space resonance set), ``$\nabla_\eta\,\textup{phase}$ '' is very small, which is also not what we want,  and the size of phase is also relatively big. To cancel out those parts, it would be sufficient to choose our the symbol $a_{\mu, \nu}(\cdot, \cdot)$ of the bilinear operator $A_{\mu, \nu}(\cdot, \cdot)$   as follows, 
\[
a_{\mu, \nu}(\xi-\eta, \eta)= \sum_{k_2 \in \mathbb{Z}}   \frac{i q_{\mu, \nu}(\xi-\eta, \eta)}{  \Phi^{\mu, \nu}(\xi  , \eta)}\psi_{k_2}(\eta)\big(  \psi_{\leq k_2-5}(\xi-2\eta)\psi_{\leq k_2+4}(\xi-\eta)\psi_{\geq k_2-5}(\xi)\]
\be\label{e87900} 
 + \mathbf{1}_{\{-\}}(\mu) \psi_{\geq k_2+5}(\xi-\eta)  + \mathbf{1}_{\{+\}}(\mu\nu) \psi_{\leq k_2-5} (\xi)\psi_{\leq k_2+4}(\xi-\eta)   ),
\ee
where $\mathbf{1}_S(\cdot)$ denotes the characteristic function of  set $S$.

Now, we proceed to consider   the cubic terms. Note that  we have the  following scenarios such that we are close to the space resonance set but not the time resonance set, i.e., the phase is of size $|\xi|^{2}(1+|\xi|)^{-1/2}$, which is   highly oscillating,
\begin{enumerate}
	\item[$\bullet$] When $\tau=-$  and $|\eta|, |\sigma|\ll |\xi|.$
\item[$\bullet$]  When $\eta$ is very close to $\xi/2$ (space resonance in $\eta$ set) and $\sigma \ll  |\eta|$ .
\item[$\bullet$] When $\eta$ is very close to $2\xi/3$ and $\sigma$ is very close to $\xi/3$, i.e., $(\xi-\eta, \eta-\sigma, \sigma)$ is close to $(\xi/3, \xi/3, \xi/3)$,  which is the space resonance in $\eta$ and $\sigma$ set.
\item[$\bullet$] When $(\xi-\eta,\eta-\sigma,\sigma)$ is very close to $(-\tau\xi, -\kappa\xi, -\iota\xi)$, which is the space resonance in $\eta$ and $\sigma$ set, where $(\tau, \kappa, \iota)\in \widetilde{S}:=  \{ (+,-,-), (-,+,-), (-,-,+)\}$. See subsubsection (\ref{allcomparable}) for more details.
\end{enumerate}

 To cancel out those parts of frequencies, we choose the symbol $b_{  \tau, \kappa,\iota}(\cdot, \cdot, \cdot)$ of the trilinear operator $B_{ \tau, \kappa, \iota}(\cdot, \cdot, \cdot)$   as follows, 
\[
b_{  \tau, \kappa,\iota}(\xi-\eta, \eta-\sigma, \sigma)= \frac{ i \widehat{c}_{ \tau, \kappa,\iota}(\xi-\eta, \eta-\sigma, \sigma)}{\Phi^{ \tau, \kappa,\iota}(\xi, \eta, \sigma)} \sum_{k\in \mathbb{Z}}  \psi_k(\xi) \big(  \mathbf{1}_{\widetilde{S} } ((\tau,\kappa, \iota)) \psi_{\leq k-10}((1+\tau)\xi-\eta)\psi_{\leq k-10}(\sigma+\iota \xi)
\]
\be\label{eqn200}
 +\psi_{\leq k-10}(\eta-2\xi/3)\psi_{\leq k-10}(\sigma-\xi/3) + \psi_{\leq k-10}(\eta-\xi/2)\psi_{\leq k-10}(\sigma)   +   \mathbf{1}_{\{-\}}(\tau)  \psi_{\leq k-10}(\eta-\sigma)\psi_{\leq k-10}( \sigma)\big),
\ee 
where $\widehat{c}_{\tau, \kappa, \iota}(\cdot, \cdot, \cdot)$  is the  associated symbol  of cubic term  $\widehat{C}_{\tau, \kappa,\iota}(\cdot, \cdot, \cdot)$ in (\ref{modifiedcubiciteration}).

Very similarly, for the quartic terms, we cancel out the case when $|\eta|,|\sigma|,|\kappa|\ll |\xi|$, $\mu_1=-$ and the case when $|\eta-\xi/2|, |\sigma|,|\kappa| \ll |\xi|$ by choosing the symbol $e_{  \mu_1, \mu_2,\nu_1, \nu_2}(\cdot, \cdot, \cdot)$ of the quartic term $E_{\mu_1, \mu_2,\nu_1, \nu_2}(\cdot, \cdot, \cdot,\cdot)$   as follows, 
\[
e_{  \mu_1, \mu_2,\nu_1, \nu_2}(\xi-\eta, \eta-\sigma, \sigma-\kappa, \kappa)= \frac{ i \widehat{d}_{\mu_1, \mu_2,\nu_1, \nu_2}(\xi-\eta, \eta-\sigma, \sigma-\kappa, \kappa)} {\Phi^{\mu_1, \mu_2,\nu_1, \nu_2}(\xi, \eta, \sigma,\kappa)}\sum_{k\in \mathbb{Z}}\psi_k(\xi)
\]
\be\label{quarticsymbolnormalform}
\times \big( \psi_{\leq k-10}(\eta-\xi/2) \psi_{\leq k-10}(\sigma-\kappa)\psi_{\leq k-10}( \kappa)  +    \mathbf{1}_{\{-\}}(\mu_1) \psi_{\leq k-10}(\eta)\psi_{\leq k-10}(\sigma-\kappa)\psi_{\leq k-10}( \kappa) \big),
\ee
 where  $\widehat{d}_{\mu_1, \mu_2, \nu_1, \nu_2}(\cdot, \cdot, \cdot, \cdot)$  is the associated symbol of quartic term $\widehat{ D}_{\mu_1, \mu_2, \nu_1, \nu_2}(\cdot, \cdot, \cdot,\cdot)$. 

 Because of the rearrangement of inputs inside  $\tilde{C}_{\tau, \kappa,\iota}(\cdot, \cdot, \cdot)$  and $\tilde{ D}_{\mu_1, \mu_2, \nu_1, \nu_2}(\cdot, \cdot, \cdot,\cdot)$, we know that   \emph{the following estimate holds inside the support of symbol $ \tilde{c}_{ \tau, \kappa,\iota}(\xi-\eta, \eta-\sigma, \sigma)$},
 \[
|\xi-\eta|\gtrsim |\eta-\sigma| \gtrsim |\sigma|.
\]
and  \emph{the following estimate holds inside the support of symbol $ \tilde{d}_{\mu_1, \mu_2,\nu_1, \nu_2}(\xi-\eta, \eta-\sigma, \sigma-\kappa, \kappa)$},
 \[
|\xi-\eta|\gtrsim |\eta-\sigma| \gtrsim |\sigma-\kappa|\gtrsim |\kappa|.
\]

  Define the profile of $v(t)$ as $g(t):=e^{i t\Lambda} v(t)$. From above discussion and (\ref{realvariableweighted}), we have 
\[
\p_t g (t, \xi) \psi_k(\xi) = \sum_{  \mu, \nu \in \{+, -\}} \sum_{k_1,k_2\in \mathbb{Z}} B^{\mu, \nu}_{k,k_1,k_2}(t,\xi)  + \sum_{  \tau, \kappa, \iota  \in \{+, -\}}  \sum_{k_3\leq k_2\leq k_1} T^{\tau, \kappa,\iota}_{k,k_1,k_2,k_3}(t, \xi)
\]
\be\label{realduhamel}
 + \sum_{\mu_1, \mu_2,\nu_1, \nu_2 \in \{+,-\}} \sum_{k_4\leq k_3\leq k_2\leq k_1}  K^{\mu_1, \mu_2,\nu_1, \nu_2}_{k,k_1,k_2,k_3,k_4}(t, \xi ) + e^{it\Lambda(\xi)} \widehat{\mathcal{R}_{1}}(t, \xi)\psi_k(\xi),
\ee
where  
\be\label{eqn650}
B^{\mu, \nu}_{k,k_1,k_2}(t,\xi):=\int_{\R^2 } e^{i t\Phi^{\mu, \nu}(\xi, \eta)}  \tilde{q}_{\mu, \nu}(\xi-\eta, \eta)  \widehat{g^{\mu}_{k_1}}(t, \xi-\eta) \widehat{g^{\nu}_{k_2}}(\eta) \psi_k(\xi) d \eta,
\ee
\be\label{eqn440}
T^{\tau, \kappa,\iota}_{k,k_1,k_2,k_3}(t, \xi) = \int_{\R^2} \int_{\R^2} e^{i t \Phi^{\tau,\kappa, \iota}(\xi, \eta, \sigma)}   \tilde{d}_{\tau,\kappa, \iota}(\xi-\eta, \eta-\sigma, \sigma) \widehat{g^{\tau}_{k_1}}(t, \xi-\eta) \widehat{g^{\kappa}_{k_2}}(t,\eta-\sigma) \widehat{g^{\iota}_{k_3}}(t,\sigma) \psi_k(\xi) d \eta d \sigma,	
\ee
\[
K^{\mu_1, \mu_2,\nu_1, \nu_2}_{k,k_1,k_2,k_3,k_4}(t, \xi ) = \int_{\R^2} \int_{\R^2} e^{i t \Phi^{\mu_1, \mu_2,\nu_1, \nu_2}(\xi, \eta, \sigma,\kappa)}   \tilde{e}_{\mu_1, \mu_2,\nu_1, \nu_2}(\xi-\eta, \eta-\sigma, \sigma-\kappa, \kappa)
\]
\be\label{eqn441}
\times  \widehat{g^{\mu_1}_{k_1}}(t, \xi-\eta)  \widehat{g^{\mu_2}_{k_2}}(t,\eta-\sigma) \widehat{g^{\nu_1}_{k_3}}(t,\sigma-\kappa) \widehat{g^{\nu_2}_{k_4}}(t, \kappa) \psi_k(\xi) d \eta d \sigma d\kappa,
\ee
where
\[
\tilde{q}_{\mu, \nu}(\xi-\eta, \eta)= \sum_{k_2\in \mathbb{Z}} q_{\mu, \nu}(\xi-\eta, \eta) \psi_{k_2}(\eta) \big(   \psi_{\geq k_2-9}(\xi-2\eta)\psi_{\leq k_2+4}(\xi-\eta)\psi_{\geq k_2-5}(\xi)\]
\be\label{eqn1}
 +\frac{1+\mu}{2} \psi_{ \geq k_2+5}(\xi-\eta)  + \frac{ (1-\mu\nu)  }{2}\psi_{\leq k_2-5} (\xi) \psi_{\leq k_2+4}(\xi-\eta)  \big),
\ee
\be\label{eqn1643}
\tilde{d}_{ \tau, \kappa, \iota }(\xi-\eta, \eta-\sigma, \sigma)= \tilde{c}_{ \tau, \kappa, \iota }(\xi-\eta, \eta-\sigma, \sigma)+ i b_{ \tau, \kappa, \iota }(\xi-\eta, \eta-\sigma,\sigma) \Phi^{ \tau, \kappa, \iota }(\xi, \eta, \sigma),
\ee
\[
\tilde{e}_{\mu_1, \mu_2,\nu_1, \nu_2}(\xi-\eta, \eta-\sigma, \sigma-\kappa, \kappa) =  \tilde{d}_{\mu_1, \mu_2,\nu_1, \nu_2}(\xi-\eta, \eta-\sigma, \sigma-\kappa, \kappa)\]
\be\label{quarticsymbolnormalform2}
+ i e_{\mu_1, \mu_2,\nu_1, \nu_2}(\xi-\eta, \eta-\sigma, \sigma-\kappa, \kappa) \Phi^{\mu_1, \mu_2,\nu_1, \nu_2	}(\xi, \eta, \sigma,\kappa).
\ee
Recall that we rearranged the inputs in the construction of good substitution. 
   Hence, in (\ref{realduhamel}), we have $k_3\leq k_2 \leq k_1 $ and $k_4\leq k_3\leq k_2\leq k_1$. 

 From Lemma \ref{Snorm}, (\ref{symbolquadraticrough}), and the fact that phases  are all of size $\max\{|\xi|, |\eta|\}^{2}(1+ \max\{|\xi|, |\eta|\})^{-1/2}$ in the support of symbols of normal form transformations, the following estimate holds, 
\[
\|a_{\mu, \nu}(\xi-\eta, \eta)\psi_k(\xi)\psi_{k_1}(\xi-\eta)\psi_{k_2}(\eta)\|_{\mathcal{S}^\infty } + \|b_{\tau, \kappa, \iota}(\xi-\eta, \eta)\psi_k(\xi)\psi_{k_1}(\xi-\eta)\psi_{k_2}(\eta-\sigma)\psi_{k_3}(\sigma)\|_{\mathcal{S}^\infty}
\] 
\be\label{normaformsize} 
+ \|e_{  \mu_1, \mu_2,\nu_1, \nu_2}(\xi-\eta, \eta-\sigma, \sigma-\kappa, \kappa)\psi_k(\xi)\psi_{k_1}(\xi-\eta)\psi_{k_2}(\eta-\sigma)\psi_{k_3}(\sigma-\kappa)\psi_{k_4}(\kappa)\|_{\mathcal{S}^\infty } \lesssim 2^{k_{1,+}} .  
\ee

From (\ref{eqn1}), it is easy to verify that the following identity holds when $|\eta|\leq 2^{-10} |\xi|$,  
\be\label{eqn900}
\tilde{q}_{-, \nu}(\xi-\eta, \eta)=0, \quad \tilde{q}_{+, \nu}(\xi-\eta, \eta)= q_{+, \nu}(\xi-\eta,\eta).
\ee  
Moreover, when $|\xi|\ll |\eta|$, we have
\be\label{eqn902}
\tilde{q}_{\mu, \mu}(\xi-\eta, \eta)=0, \quad \quad \mu\in\{+,-\}. 
\ee

Recall (\ref{symmetricsymbol}). From the explicit formula, we can identify the leading part ``$c(\xi)$'' of $q_{+, \nu}(\xi-\eta, \eta)$ in the case when $|\eta|\ll |\xi|$ as follows, 
\be\label{eqn939}
c(\xi):= \frac{c_{+}}{2} \tilde{\lambda}(|\xi|^2)|\xi|^2\big(1- \tanh(|\xi|)^2 \big).
\ee
After taking $c(\xi)$ out of the symbol $\tilde{q}_{\mu, \nu}(\cdot, \cdot)$, it behaves better. More precisely, from Lemma \ref{Snorm}, the following estimate holds,
\be\label{eqn932}
\| \big( \tilde{q}_{+, \nu}(\xi-\eta, \eta)- c(\xi)\big) \psi_{k_1}(\xi-\eta) \psi_{k_2}(\eta) \|_{\mathcal{S}^\infty}\lesssim 2^{k_2+k_1}, \quad \textup{if $k_2\leq k_1-10$}.
\ee

In later high order weighted norm estimate, we will also need to use the hidden symmetry inside the symbol   $\tilde{d}_{ \tau, \kappa, \iota }(\xi-\eta, \eta-\sigma, \sigma)$ when $|\sigma|, |\eta|\ll |\xi|$. To this end, we identify the leading  symbol inside $\tilde{d}_{ \tau, \kappa, \iota }(\xi-\eta, \eta-\sigma, \sigma)$ first.   From   (\ref{eqn200}) and (\ref{eqn1643}), we know that we only have to consider the case when $\tau=+$ and the leading part of $\tilde{d}_{\tau, \kappa, \iota}(\xi-\eta, \eta-\sigma, \sigma)$ is same as the leading part of $\tilde{c}_{\tau, \kappa, \iota}(\xi-\eta, \eta-\sigma, \sigma)$. Recall (\ref{modifiedcubiciteration}) and (\ref{e87900}). It is easy to verify the following estimate holds when $ k_2, k_3\leq k_1-10,$
\be\label{e88987}
\| (\tilde{d}_{+,   \kappa,\iota}(\xi-\eta, \eta-\sigma, \sigma) - e(\xi) )\psi_{k_2}(\eta-\sigma) \psi_{k_3}(\sigma)\psi_{k_1}(\xi-\eta)\|_{\mathcal{S}^\infty} \lesssim 2^{\max\{k_2,k_3\}+k_1+4k_{1,+}}. 
\ee
where
\be\label{e88988}
e(\xi):= \frac{c_{+}}{4}d (\xi) -\frac{i c(\xi)^2}{\Lambda(|\xi|)},
\ee
 where ``$d(\xi)$'' is defined in (\ref{eqn630}).  The first part of $e(\xi)$ comes from the cubic term $ C_{\tau,\kappa, \iota}(u^{\tau},u^{\kappa}, u^{\iota})$ in  (\ref{modifiedcubiciteration}), see   (\ref{removebulk}) in Lemma  \ref{symbolcubicandquartic}. The second part of  $e(\xi)$ comes from   the composition of quadratic terms  and the normal form transformation in  (\ref{modifiedcubiciteration}).

Lastly, we consider the symbol of quartic terms. 
The precise formulas of  symbols $\tilde{e}_{\mu_1, \mu_2,\nu_1, \nu_2}(\xi-\eta, \eta-\sigma, \sigma-\kappa, \kappa)$  is not so important here. It is good enough to know that it satisfies a similar estimate  as in  
(\ref{symbolquartic}). The proof of this claim follows easily from estimate (\ref{symbolquartic}) in Lemma   \ref{symbolcubicandquartic}.

\subsection{Analysis of phases}
Recall that $  \Lambda(|\xi|):=|\xi|^{3/2}\sqrt{\tanh |\xi|}$. Let    $\lambda( x):= \Lambda(\sqrt{x})$.  The following approximation holds when $|\xi|$ is very close to zero, 
\be\label{expansion1}
 \Lambda(|\xi|)\approx |\xi|^2- \frac{1}{6}|\xi|^4, \quad \lambda(  |\xi|) \approx  |\xi| - \frac{1}{6} |\xi|^2, \quad |\xi|\ll 1. 
\ee
 Recall (\ref{phases}), we have the following expansion when $|\eta|\ll |\xi|,$
\[
\Phi^{+,\nu}(\xi, \eta)= \Lambda(|\xi|) - \Lambda(|\xi-\eta|) - \nu \Lambda(|\eta|)
\]
\be\label{degeneratephase}
= \lambda(|\xi|^2)  - \lambda(|\xi|^2-2\xi\cdot \eta + |\eta|^2) - \nu \lambda(|\eta|^2)= 2\lambda'(|\xi|^2) \xi \cdot \eta + \mathcal{O}(|\eta|^2).
\ee

When $|\xi|\ll |\eta|$, the following approximation holds for the phase $\Phi^{\mu, \nu}(\xi, \eta)$ when $\nu=-\mu$, 
\[
\Phi^{\mu, -\mu}(\xi, \eta)= \lambda(|\xi|^2) - \mu\big( \lambda(|\xi-\eta|^2) - \lambda(|\eta|^2) \big)
\]
\be\label{eqn920}
=  \lambda(|\xi|^2) - \mu\big( \lambda(|\xi|^2-2\xi \cdot \eta +|\eta|^2) - \lambda(|\eta|^2) \big)= 2\mu \lambda'(|\eta|^2) \xi\cdot \eta  + \mathcal{O}(|\xi|^2).
\ee

Note that, when $\eta$ is not very close to $\xi/2$ (space resonance set), e.g., $|\eta-\xi/2|\geq 2^{-10} |\xi|$,  the following estimates hold ,
\be\label{eqn928}
|\nabla_\eta \Phi^{\mu, \nu}(\xi, \eta)| = 2\big| \mu \lambda'(|\xi-\eta|^2) (\xi-\eta) - \nu \lambda'(|\eta|^2) \eta   \big| \gtrsim |\xi| \big( |\xi-\eta|+|\eta|+1\big)^{-1/2},
\ee
 \be\label{eqn929}
|\nabla_\eta \Phi^{\mu, \nu}(\xi, \eta)|  + |\nabla_\xi \Phi^{\mu, \nu}(\xi, \eta)| \lesssim  \max\{|\xi|, |\eta|\}(  |\xi| +|\eta| +1)^{-1/2}.
 \ee
Suppose that $|\eta|\sim 2^{k_2},  |\xi|\sim 2^{k}, |\xi-\eta|\sim 2^{k_1}, k_2\leq k_1 $. From (\ref{eqn928}),  it is easy to verify that the following estimate holds inside the support of $\tilde{q}_{\mu, \nu}(\xi-\eta, \eta)$, 
 \be\label{hitsphasedenominator}
2^{k_2}\frac{\nabla_\eta^2  \Phi^{\mu, \nu}(\xi, \eta) }{ | \nabla_\eta   \Phi^{\mu, \nu}(\xi, \eta)| } + 2^{k}\frac{\nabla_\eta\nabla\xi  \Phi^{\mu, \nu}(\xi, \eta) }{ | \nabla_\eta   \Phi^{\mu, \nu}(\xi, \eta)| } \lesssim 1.
  \ee

\subsection{The set-up of $Z$-norm estimate}
Recall the   $Z_1$-normed space and the  $Z_2$-normed space we defined in (\ref{loworderweightnrom}) and (\ref{highorderweightnorm}). The spatial localization  function $\varphi_j^k(x)$ used there is defined as follows, 
  \begin{equation}\label{spatiallocalization}
 \varphi_j^{k}(x):=\left\{\begin{array}{cc}
\tilde{\psi}_{(-\infty, -k]}(x) & \textup{if}\, k + j =0 \,\textup{and}\, k \leq 0,\\
\tilde{\psi}_{(-\infty, 0]}(x) & \textup{if}\,  j =0 \,\textup{and}\, k \geq 0,\\
\tilde{\psi}_j(x) & \textup{if}\,  k+j \geq 1 \,\textup{and}\, j \geq 1.\\
\end{array}\right.
 \end{equation}

We first show that the $L^\infty$ norm of $v(t)$ and the $L^\infty$ norm of $u(t)$ are comparable.  More precisely, the following Lemma holds,
 \begin{lemma}\label{lemmacomparabledecay}
Under the bootstrap assumption \textup{(\ref{bootstrapassumption})}, we have
\be\label{comparabledecay}
\sup_{t\in[0,T]}(1+t) \|v(t)-u(t)\|_{W^{6,1+\alpha}} + \| v(t)-u(t)\|_{H^{N_0-10	}}\lesssim \epsilon_0.
\ee
\end{lemma}
\begin{proof}
From the $L^\infty-L^\infty$ type bilinear estimate and (\ref{normaformsize}),  the following estimate holds after putting the input with the largest frequency in $L^\infty$, 
\[
\| v(t)-u(t)\|_{W^{6, 1+\alpha}} \lesssim \| u(t)\|_{W^{6,1+\alpha}}^{4/3 }\| u(t)\|_{H^{N_0}}^{2/3}\lesssim (1+t)^{-6/5}\epsilon_1^2\lesssim (1+t)^{-6/5}\epsilon_0,
\]
\[
 \| v(t)-u(t)\|_{H^{N_0-10	}} \lesssim \| u(t)\|_{H^{N_0}} \| u(t)\|_{W^{4,0}} \lesssim \epsilon_1^2\lesssim \epsilon_0.
\]
\end{proof}

As a result of above Lemma,  it would be sufficient   to prove the improved $L^\infty$ type estimate for $v(t)$. Recall the definitions of $Z_1$ norm and $Z_2$ norm in (\ref{loworderweightnrom}) and (\ref{highorderweightnorm}),  we expect that the $Z_1$ norm of the profile $g(t)$ doesn't grow and the $Z_2$ norm of the profile only grows appropriately. 
Hence, we make the   bootstrap assumption as follows for some $T'\in(0,T]$,
\be\label{smallness}
\sup_{t\in[0,T']} 
    (1+t)   \| e^{-i t \Lambda} g(t)\|_{W^{6,1+\alpha}}+   \| g (t)\|_{Z_1}+ (1+t)^{-\tilde{\delta}} \| g(t)\|_{Z_2} \lesssim \epsilon_1:=\epsilon_0^{5/6}.
\ee
As a direct consequence, we  derive the following estimates  for any $t\in[2^{m-1}, 2^m]\subset[0,T'],$
\[
 \| e^{-i t \Lambda} g_k(t)\|_{L^\infty} \lesssim \min\{2^{-m-(1+\alpha)k -6k_{+}} , 2^{-m+\tilde{\delta}m -k} \} \epsilon_1,\quad \| g_{k}(t)\|_{L^2} \lesssim 2^{-(N_0-10)k_{+}+\delta m }, 
\] 
 \[
 \| g_{k,j}\|_{L^2}\lesssim \min\{ 2^{-j-(1+\alpha)k -8k_{+}} , 2^{-2j-2k + \tilde{\delta} m }\}\epsilon_1,\quad \| g_k\|_{L^1} \lesssim \sum_{j\geq -k_{-}} 2^{k+j} \| g_{k,j}\|_{L^2} \lesssim 2^{  \tilde{\delta} m  }\epsilon_1,
\]

To close the argument, it would be sufficient if we could prove the following  estimate,
\be\label{desired1}
\sup_{t_1, t_2\in[2^{m-1}, 2^m]}  \| g(t_2)- g(t_1)\|_{Z_1}\lesssim 2^{-\delta m }\epsilon_0,
\ee
\be\label{desired2}
 \sup_{t_1, t_2\in[2^{m-1}, 2^m]} \| g (t_2)\|_{Z_2}^2- \|g (t_1)\|_{Z_2}^2\lesssim 2^{ 2 \tilde{\delta} m }\epsilon_0.
\ee

\section{The improved estimate of the low order weighted norm }\label{loworderweight}

In this subsection, we mainly prove (\ref{desired1}) under the 
bootstrap assumption (\ref{smallness}). Recall (\ref{realduhamel}).  In the first subsection, we will estimate the quadratic terms $B^{\mu, \nu}_{k,k_1,k_2}(\xi, \eta)$ in details. In the second subsection, we will handle the cubic terms $T^{\tau, \kappa,\iota}_{k,k_1,k_2,k_3}(t, \xi )$ and quartic terms $K^{\mu_1, \mu_2,\nu_1, \nu_2}_{k,k_1,k_2,k_3,k_4}(t, \xi )   $ together  as the methods we  will use are very similar. For the quintic and higher order remainder term  $\widehat{\mathcal{R}_{1}}(t, \xi)$, estimate (\ref{eqnj878}) in Lemma \ref{remaindertermweightednorm} is very sufficient.

\subsection{The $Z_1$-norm estimate of quadratic terms: when $|k_1-k_2|\leq 10$} 
Note that, from the $L^2\rightarrow L^1$ type Sobolev embedding and $L^2-L^2$ type estimate, the following rough estimate holds for any $\mu, \nu\in\{+, -\}$,
\[
\|  \mathcal{F}^{-1}[\int_{t_1}^{t_2} B^{\mu, \nu}_{k,k_1,k_2}(t, \xi) d t]\|_{B_{k,j}} \lesssim 2^{(2+\alpha) k + m + j+ 2k_1 + 10 k_{1,+}} \| g_{k_1}\|_{L^2} \| g_{k_2}\|_{L^2} \]
\[\lesssim 2^{(2+\alpha) k + m + j+ (2-2\alpha)k_1 -(N_0-12) k_{1,+}}\epsilon_0.
\]
From above estimate, we can first rule out the case when $k\leq -(1+\delta)(m+j)/ (2+\alpha)$   or $ k_1\leq  -(1+\delta)(m+j)/ (4-\alpha)$ or $k_1\geq (m+j)/(N_0-30)$. As a result, it is sufficient to consider the case when $k$ and $k_1$ are restricted in the following range, 
\be\label{ruledoutregion}
 -(1+\delta)(m+j)/ (2+\alpha)\leq k \leq k_1\leq (m+j)/( N_0-30),\quad 
k_1 \geq -(1+\delta)(m+j)/ (4-\alpha).
\ee

Recall (\ref{eqn650}). We  do spatial localizations for two inputs and have the   decomposition as follows, 
\be\label{eqn310}
  B^{\mu, \nu}_{k,k_1, k_2  }(t, \xi)=\sum_{j_1\geq -k_{1,-}, j_2\geq -k_{2,-}} B^{\mu, \nu,j_1,j_2}_{k,k_1 ,k_2 }(t, \xi),
 \ee
 \be\label{spatiallocalized}
 B^{\mu, \nu,j_1,j_2}_{k,k_1 ,k_2 }(t, \xi)=  \int_{\R^2 } e^{i t\Phi^{\mu, \nu}(\xi, \eta)}  \tilde{q}_{\mu, \nu}(\xi-\eta, \eta) \widehat{g^{\mu}_{k_1,j_1}}(t, \xi-\eta) \widehat{g^{\nu}_{k_2,j_2}}(\eta)\psi_k(\xi) d \eta.
\ee
 {where we used the abbreviation $g_{k_1,j_1}:=P_{[k-2,k+2]}\big[ \varphi_{j_1}^{k_1}(x) P_{k_1} g]$}. We will also use this abbreviation throughout this paper.

After using the inverse Fourier transform, we have
 \be\label{inversefourier}
 \mathcal{F}^{-1}[ B^{\mu, \nu,j_1,j_2}_{k,k_1 ,k_2 }(t, \xi)](x)= \int_{\R^2\times \R^2 } e^{i x \cdot \xi + i t\Phi^{\mu, \nu}(\xi, \eta)}  \tilde{q}_{\mu, \nu}(\xi-\eta, \eta) \widehat{g^{\mu}_{k_1,j_1}}(t, \xi-\eta) \widehat{g^{\nu}_{k_2,j_2}}(\eta) \psi_k(\xi) d \eta d \xi.
 \ee

\noindent $\bullet$\quad When $j \geq (1+\delta)\max\{m + k_{1 }, -k_{-}\}+  2\tilde{\delta}m  $.\quad We first consider the case  when $\min\{j_1, j_2\}\geq  j-\delta j -\delta m $, the following estimate holds,
\[
\sum_{\min\{j_1, j_2\}\geq  j-\delta j -\delta m  } \|  \mathcal{F}^{-1}[ \int_{t_1}^{t_2} B^{\mu, \nu}_{k_1,j_1,k_2,j_2}(t,\xi)  d t ]\|_{B_{k,j}} \lesssim \sum_{\min\{j_1, j_2\} \geq  j-\delta j -\delta m  }  2^{(2+\alpha)k + m + j + 10k_{+}+  2k_1}  
\]
\[
\times \| g_{k_1,j_1}\|_{L^2} \| g_{k_2,j_2}\|_{L^2} \lesssim 2^{(2+\alpha)k+m+\tilde{\delta}m + 10\delta m  -(2-2\delta)j + (2-2\alpha)k_1- 6k_{1,+}}\epsilon_0 \lesssim 2^{-2\delta m -2\delta j}\epsilon_0.
\]
Now we proceed to consider the case   $\min\{j_1, j_2\}\leq  j-\delta j -\delta m$.   For this case, we do integration by parts in ``$\xi$'' for (\ref{inversefourier}) many times to see rapidly decay.  
Note that the following estimate holds from (\ref{eqn929}), 
 \be\label{hitsphase}
|\nabla_\xi\big(x\cdot\xi + t \Phi^{\mu, \nu}(\xi, \eta)\big)|= \big|x + t\nabla_\xi \Phi^{\mu, \nu}(\xi, \eta) \big|\varphi_k^j(x)\sim 2^{j}.
 \ee
If $j_2=\min\{j_1,j_2\}$, then we do change of variables to switch the role of $\xi-\eta$ and $\eta$.  As a result, the following estimate holds, 
\[
|\nabla_\xi\big(x\cdot\xi + t \Phi^{\mu, \nu}(\xi, \xi-\eta)\big)|= \big|x + t\nabla_\xi \Phi^{\mu, \nu}(\xi, \xi-\eta) \big|\varphi_k^j(x)\sim 2^{j}.
\]
In whichever case, by doing integration by parts in $\xi$ once, we gain $2^{-j}$ by paying the price of at most $ \max\{2^{\min\{j_1,j_2\}}, 2^{-k}\} $. Hence, the net gain of doing integration by parts in ``$\xi$'' once is at least $2^{-\delta m -\delta j}$. After doing this process many times, we can see rapidly decay.

\noindent $\bullet$\quad When   $j\leq (1+\delta) \max\{m + k_{1 }, -k_{-}\}+  2\tilde{\delta}m$.  As $j$ is bounded from above now, 
from (\ref{ruledoutregion}), we have the following upper bound and lower bound for $k$ and $k_1$,
\be\label{rangehh1}
 -m/(1+\alpha/3)\leq k\leq k_1\leq 2 \beta m,\quad j\leq  \max\{m + k_{1 }, -k_{-}\}+  3\tilde{\delta}m, \quad \beta:=1/(N_0-50),
\ee

Hence, it would be sufficient to consider fixed $k$ and $k_1$  inside the range (\ref{rangehh1}), as  there are at most $m^3$ cases to consider, which is only a logarithmic loss. 

After doing integration by parts in $\eta$ many times, we can rule out the case when $\max\{j_1, j_2\} \leq m + k_{-} - 3\beta m $. It remains to consider the case when $\max\{j_1, j_2\} \geq m + k_{-} - 3\beta m $. From $L^2-L^\infty$ type bilinear estimate in Lemma \ref{multilinearestimate}, the following estimate holds after putting the input with  the maximum spatial concentration in $L^2$ and the other input in $L^\infty$, 
\[
\sum_{\max\{j_1, j_2\} \geq m + k_{-} -3 \beta m}   \| \mathcal{F}^{-1}[ \int_{t_1}^{t_2} B^{\mu, \nu}_{k_1,j_1,k_2,j_2}(t,\xi) d t ] \|_{B_{k,j}}\]
\[\lesssim 2^{(1+\alpha)k + m + j +2k_1 + 10 k_{+}-m- (1+\alpha) k_1 } \min\{ 2^{- m-k_{-}-(1+\alpha)k_1+6\beta m 	}, 2^{ -2k_1 - 2(m + k_{-} -3 \beta m) +\tilde{\delta}m	} \} \epsilon_1^2\]
\be\label{e890}
\lesssim \min\{2^{\alpha k +12k_{+}+ (1-2\alpha)k_1 +10\beta m},2^{-(1-\alpha)k+12k_{+} -\alpha k _1-m+10\beta m } \}\epsilon_0
\lesssim 2^{-10\delta m }\epsilon_0.
\ee

\subsection{The $Z_1$-norm estimate of quadratic terms: when $k_2\leq k_1- 10$}
 Recall (\ref{eqn900}). \emph{For the case we are considering, we have $\mu=+$.}   Recall (\ref{eqn939}) and (\ref{eqn932}). It motivates us to    split the symbol ``$\tilde{q}_{+, \nu}(\xi, \eta)$'' into two parts as follows,
 \[
 \tilde{q}_{+, \nu}(\xi-\eta, \eta) =   q^1_{+, \nu}(\xi-\eta, \eta) + q^2_{+, \nu}(\xi-\eta, \eta), 
 \] 
\be\label{symboldecomposition}
  q_{+, \nu}^1(\xi-\eta, \eta)= c(\xi), \quad q_{+, \nu}^2(\xi-\eta, \eta)= q_{+, \nu}(\xi-\eta, \eta)-    c(\xi).
\ee
Hence, we do the decomposition as follows, 
\[
 \sum_{  \nu \in\{+,-\}} \int_{t_1}^{t_2}B^{+, \nu}_{k,k_1, k_2  }(t, \xi)  d t =  \sum_{i=1,2} I^{  i}_{k,k_1, k_2}, \]
 \[
I^{  i}_{k,k_1, k_2}= \sum_{  \nu\in\{+,-\}}    \int_{t_1}^{t_2} \int_{\mathbb{R}^2}  e^{i t\Phi^{+, \nu}(\xi, \eta)}   q^i_{\mu, \nu}(\xi-\eta, \eta)   \widehat{g_{k_1}^{ }}(t,\xi-\eta)  \widehat{g^\nu_{k_2}}(t,\eta) \psi_k(\xi) d \eta d t,\quad i=1,2.\]
Recall (\ref{symboldecomposition}). Since $ q^1_{\mu, \nu}(\xi-\eta, \eta) $ actually doesn't depend on the sign ``$\nu$'', we have
 \[
I^{1}_{k,k_1,k_2}= 2\int_{t_1}^{t_2} \int_{\mathbb{R}^2}   e^{i t(\Lambda(|\xi|)-\Lambda(|\xi-\eta|)}  c(\xi)   \widehat{g_{k_1}^{ }}(t,\xi-\eta) \widehat{\textup{Re}(v) }(t,\eta)  \psi_{k_2}(\eta) \psi_k(\xi) d \eta d t.  
\] 
From (\ref{normalformatransfor})  and estimate (\ref{eqn400}) in Lemma \ref{Linftyxi}, the following estimate holds after using the volume of the support of ``$\eta$'',
\[
\| I^{  1}_{k_1,k_2}\|_{B_{k,j}} \lesssim \sup_{t\in [2^{m-1},2^m]} 2^{(3+\alpha)k+ m + j+ 10 k_{+}} \| g_{k_1}(t)\|_{L^2} 2^{2k_2}  \| \widehat{\textup{Re}(v)}(t, \xi)\psi_{k_2}(\xi)\|_{L^\infty_\xi}
\]
\[
\lesssim 2^{(3+\alpha)k+ m + \delta m + j +2k_2 -(N_0-30)k_+}( \| \widehat{h}(t, \xi)\psi_{k_2}(\xi)\|_{L^\infty_\xi} + \| u\|_{H^{10}}^2+  \| u\|_{H^{10}}^3 +  \| u\|_{H^{10}}^4 ) 
\]
\be\label{eqn302}
\lesssim 2^{(3+\alpha)k+ 2m +10\delta m + j + 3k_2 -(N_0-30)k_+} \epsilon_0+ 2^{(3+\alpha)k + 3m+10\delta m  + j +4k_2 -(N_0-30)k_+}  \epsilon_0.
\ee
 
Now we proceed to estimate $I^{  2}_{k_1,k_2}$. Recall (\ref{symboldecomposition}) and (\ref{eqn932}).  From the $L^2-L^\infty$ type  bilinear  estimate (\ref{bilinearesetimate}) in Lemma \ref{multilinearestimate} and $L^\infty\rightarrow L^2$ type Sobolev embedding, we have
\[
\| I^{ 2}_{k_1,k_2}\|_{B_{k,j}} \lesssim \sup_{t\in [2^{m-1},2^m]} 2^{(2+\alpha)k+ m + j +k_2+k_1 +10 k_{+}} \| g_{k_1}(t)\|_{L^2} \| e^{i t\Lambda} g_{k_2}(t)\|_{L^\infty}
\]
\be\label{eqn303}
\lesssim  2^{(3+\alpha)k -(N_0-10)k_+ + m + j + 	2k_2 +2\delta m } \epsilon_0.
\ee

To sum up, from (\ref{eqn302}) and (\ref{eqn303}), we can rule out the  case when 
$
k_2\lesssim -(1+5\delta)\max\{ (m+j)/2, (3m +j)/4  \}  
$  or $k\geq 4(m+j)/(N_0-40)$.  Now, we only need to consider the case when $k_2$ is restricted in the following range, 
\be\label{lowerboundk2}
 -(1+5\delta)\max\{ (m+j)/2, (3m +j)/4 \} 	\leq k_2 \leq k\leq  (3m+j)/(N_0-40).
\ee

Very similar to what we did in the case when $|k_1-k_2|\leq 10$, we separate into two cases based on the size of ``$j$'' as follows.

\noindent $\bullet$ When $j\geq (1+\delta)\max\{m + k ,-k_{-}\} +10\delta m$. We first consider the case   when $\min\{j_1, j_2\}\leq j-\delta j - \delta m $. Same as we considered in the High $\times$ High type interaction, we also do  integration by parts in $\xi$ many times to see rapidly decay. 
Now, we proceed to consider
 the case when $\min\{j_1, j_2\} \geq j-\delta j -\delta m $. From $L^2-L^\infty$ type bilinear estimate and $L^\infty\rightarrow L^2$ type Sobolev embedding, we have
\[
\| I_{k_1,k_2}^1\|_{B_{k,j}}\lesssim \sum_{ \min\{j_1, j_2\} \geq j-\delta j -\delta m}2^{(1+\alpha)k + 10 k_{+}+ m + j + 2k_1  + k_2}\| g_{k_1,j_1}\|_{L^2} \| g_{k_2,j_2}\|_{L^2}
\]
\[
\lesssim 2^{(1+\alpha)k  +k_2+ (1+50\beta)m  -(1-50\beta)j}  2^{-j/2-k_2/2} \epsilon_1^2\lesssim 2^{-\beta m } \epsilon_0.
\]

\noindent $\bullet$\quad    When $j\leq (1+\delta) \max\{m + k ,-k_{-}\} +10\delta m $. For this case,  whether $j_1$ is less than $j_2$ makes a difference.

\noindent $\oplus$\quad  If $j_1 \leq j_2$.\quad   For this case, we don't need to do change of coordinates to switch the role between $\xi-\eta$ and  $\eta$. Note that $|\nabla_\xi \Phi^{+, \nu}(\xi, \eta)|\lesssim |\eta|$, hence we can do better for $j$. More precisely, we can rule out the case when $j\geq \max\{m + k_2 , -k_{-}\} +100\beta m$ and $j_1\leq j-\delta m$ by doing integration by parts in $\xi$ many times.  If $j\geq  \max\{m + k_2 , -k_{-}\} +100\beta m $ and  $j-\delta m \leq j_1\leq j_2$, then the following estimate holds after using the $L^2-L^\infty$ type bilinear estimate and $L^\infty\rightarrow L^2$ type Sobolev embedding,
\[
\sum_{j-\delta m \leq j_1\leq j_2}  \| \mathcal{F}^{-1}[ \int_{t_1}^{t_2} B^{+, \nu}_{k_1,j_1,k_2,j_2}(t,\xi) d t ] \|_{B_{k,j}} \lesssim \sum_{j-\delta m \leq j_1\leq j_2} 2^{(1+\alpha)k + 10k_{+} + m + j + 2k_1} \| g_{k_1, j_1}\|_{L^2} 
\]
\[
 \times 2^{k_2} \| g_{k_2,j_2}\|_{L^2} \lesssim 2^{(1+\alpha)k  +k_2+ (1+50\beta)m  -(1-50\beta)j}  2^{-25\beta j -25\beta k_2 } \epsilon_1^2\lesssim 2^{-\beta m } \epsilon_0.
\]

It remains to consider the case when  $j\leq \max\{m + k_2 , -k_{-}\} +100\beta m $. When $k_{-}+k_2\leq -m  +\beta m $, it is easy to see our desired estimate holds from (\ref{eqn302}) and (\ref{eqn303}). Hence, we only have to consider the case when $k_{-}+k_2\geq -m  +\beta m $. For this case, we have $j\leq m +k_2 +100\beta m $.  Recall (\ref{lowerboundk2}), we know that $k_2 \geq -4m/5-30\beta m. $

 Same as in (\ref{eqn310}), we also do spatial localizations for two inputs.  After doing integration by parts in ``$\eta$'' many times, we can rule out the case when $j_2 \leq   m + k_{1,-}-10\delta m.$ Therefore, it remains to consider the case when $j_2\geq m +k_{1,-}-10\delta m$. After putting $g_{k_2,j_2}$ in  $L^2$ and  putting $g_{k_1, j_1}$  in  $L^\infty$, we have 
\[
\sum_{ j_2\geq \max\{m + k_{1,-}-10\delta m, j_1\} } \| \mathcal{F}^{-1}[ \int_{t_1}^{t_2} B^{+, \nu}_{k_1,j_1,k_2,j_2}(t,\xi) d t ] \|_{B_{k,j}} \lesssim \sum_{ j_2\geq \max\{m +  k_{1,-}-10\delta m, j_1\} } 2^{(1+\alpha) k+10 k_+  }
\]
\[
 \times  2^{2k_1+ m+ j }\sup_{t\in[2^{m}, 2^{m+1}] }\| e^{-it \Lambda} g_{k_1,j_1}(t)\|_{L^\infty} \| g_{k_2, j_2}(t)\|_{L^2}\lesssim 2^{-m-k_2+150\beta m }\epsilon_1^2 \lesssim 2^{-\beta m }\epsilon_0.
\]

\noindent $\oplus$\quad If $-k_2\leq j_2\leq j_1$. We first consider the case when $k_1+k_2\leq -4m/5 $.  From the $L^2-L^\infty$ type bilinear estimate and $L^\infty\rightarrow L^2$ type Sobolev embedding, the following estimate holds,
\[
\sum_{ j_2\leq j_1}  \| \mathcal{F}^{-1}[ \int_{t_1}^{t_2} B^{+, \nu}_{k_1,j_1,k_2,j_2}(t,\xi) d t ] \|_{B_{k,j}} \lesssim \sum_{j_2 \leq j_1} 2^{(1+\alpha)k +10k_{+}+ m + j +2k_1} \| g_{k_1, j_1}\|_{L^2} 2^{k_2} \|g_{k_2,j_2}\|_{L^2}
\]
\[
\lesssim \sum_{-k_2 \leq j_1} 2^{2m+ (4+\alpha)k_1 +k_2} 2^{-2k_1 -2j_1+ 50\beta m } \epsilon_1^2 \lesssim 2^{(2+\alpha)k +3k_2+2m+ 50\beta m }\epsilon_1^2  \lesssim 2^{- \beta m }\epsilon_0.
\]

It remains to consider the case when $k_1+k_2 \geq -4m/5 $. For this case, we do integration by parts in $\eta$ many times to rule out the case when $j_1 \leq m+k_{1,-}-10\delta m $. For the case when $j_1 \geq m + k_{1,-}-10\delta m $, the following estimate holds from the $L^2-L^\infty$ type bilinear estimate,
\[
\sum_{  j_1\geq \max\{j_1, m+k_{1,-}-10\delta m \}}  \| \mathcal{F}^{-1}[ \int_{t_1}^{t_2} B^{+, \nu}_{k_1,j_1,k_2,j_2}(t,\xi) d t ] \|_{B_{k,j}} \lesssim \sum_{  j_1\geq \max\{j_1, m+k_{1,-}-10\delta m \}}2^{(1+\alpha)k + 10k_{+}  } \]
\[\times  2^{m + j +2k_1} \sup_{t\in[2^{m-1},2^{m}]}\| g_{k_1, j_1}(t)\|_{L^2}   \| e^{-i t\Lambda} g_{k_2,j_2}(t)\|_{L^\infty}  \lesssim 2^{-m  -(1+\alpha)k_2 + 50\beta m  }\epsilon_1^2\lesssim 2^{- \beta m }\epsilon_0.
\]
 
\begin{lemma}\label{Linftyxi}
Under the bootstrap assumption \textup{(\ref{bootstrapassumption})}, the following estimate holds for $t \in[2^{m-1}, 2^{m}]\subset [0,T]$, $m \in \mathbb{N}$ and $k\in \mathbb{Z}, k\leq 0$,
\be\label{eqn400}
\| \widehat{h}(t,\xi)\psi_k(\xi)\|_{L^\infty_\xi} \lesssim 2^{2\delta m }\big( 2^{2k + 2m} + 2^{k+ m}\big)\epsilon_0.
\ee
\end{lemma}
\begin{proof}
Recall (\ref{duhamel}), it is easy to see the following estimate holds for any $t\in[2^{m-1}, 2^{m}]$ and $k\leq 0,$
\be\label{eqn420}
\| \widehat{f}(t,\xi)\psi_k(\xi)\|_{L^\infty_\xi}\lesssim \epsilon_0 + \int_0^t \|  {f}(s)\|_{H^{10}}^2  d s \lesssim 2^{m+2\delta m } \epsilon_0.
\ee
 Recall the  equation satisfied by height ``$h(t)$'' in    (\ref{waterwaves}) and the Taylor expansion (\ref{e10}), we have
\[
\p_t  \widehat{h}(t, \xi)  =|\xi| \tanh(|\xi|)\widehat{\psi}(t, \xi) + \mathcal{F}[\Lambda_2[G(h)\psi]](\xi) + \mathcal{F}[\Lambda_{\geq 3}[G(h)\psi]](\xi).
\]
Hence, from $L^2-L^2$ type bilinear estimate (\ref{bilinearesetimate}) in Lemma \ref{multilinearestimate}  and (\ref{eqn420}),  the following estimate holds for any $k\leq 0$, 
\[
\| \widehat{h}(t, \xi)\psi_k(\xi)\|_{L^\infty_\xi} \lesssim \epsilon_0 + \int_0^t 2^{2k } \| \widehat{\psi}(s, \xi)\psi_k(\xi)\|_{L^\infty_\xi} d  s  + \int_0^t 2^{k } \| h(s)\|_{H^{10}} \| \psi(s)\|_{H^{10}} d  s 
\]
\be\label{improvedestimate}
\lesssim \epsilon_0 + \int_0^t 2^{2k } \| \widehat{f}(s, \xi)\psi_k(\xi)\|_{L^\infty_\xi} d  s  + \int_0^t 2^{k } \| f(s)\|_{H^{10}}^2 d  s \lesssim 2^{2\delta m }\big( 2^{2k + 2m} + 2^{k+ m}\big)\epsilon_0.
\ee
 
\end{proof}

\subsection{The $Z_1$ estimates of cubic terms and quartic terms.} 
The main goal of this subsection is to prove the following Proposition,
\begin{proposition}\label{eqq298}
Under the bootstrap assumption \textup{(\ref{smallness})}, the following estimate holds, 
\be\label{Z1estimatecubicandquartic}
  \sum_{k_3\leq k_2\leq k_1} \| \mathcal{F}^{-1}[T^{\tau, \kappa,\iota}_{k,k_1,k_2,k_3}(t, \xi) ]\|_{Z_1}+   \sum_{k_4\leq k_3\leq k_2\leq k_1} \| \mathcal{F}^{-1}[ K^{\mu_1, \mu_2,\nu_1, \nu_2}_{k,k_1,k_2,k_3,k_4}(t, \xi )]\|_{Z_1} \lesssim 2^{-m-\beta m}\epsilon_0,
\ee
where $T^{\tau, \kappa,\iota}_{k,k_1,k_2,k_3}(t, \xi)$ and $ K^{\mu_1, \mu_2,\nu_1, \nu_2}_{k,k_1,k_2,k_3,k_4}(t, \xi ) $ are defined in \textup{(\ref{eqn440})} and \textup{(\ref{eqn441})} respectively. 
\end{proposition}

   Same as before, we can do integration by parts in ``$\xi$'' many times to rule out the case when $j
  \geq (1+\delta)\max\{m +k_1, -k_{-}\} + 2 \tilde{\delta} m $.\emph{ Hence, in the rest of this section, we restrict ourself to the case when  $j\leq (1+\delta)\max\{m +k_1, -k_{-}\} + 2 \tilde{\delta} m $}.

  From the $L^2-L^\infty-L^\infty$ type trilinear estimate in Lemma \ref{multilinearestimate}, the following estimate holds, 
  \[
\| \mathcal{F}^{-1}[T^{\tau, \kappa,\iota}_{k,k_1,k_2,k_3}(t, \xi)]\|_{B_{k,j}}\lesssim \sup_{t\in[2^{m-1}, 2^m]} 2^{(1+\alpha) k  	+j +2k_1+ 2k_{1,+}+10 k_{+}} \| e^{-it \Lambda} g_{k_1}\|_{L^\infty} \| g_{k_2}\|_{L^2} 
  \] 
\be\label{roughestimatecubic}
\times \| e^{-it \Lambda} g_{k_3}\|_{L^\infty}  \lesssim \min\{2^{(1+\alpha)k +2k_1 + k_3 +20\beta m  }, 2^{(1+\alpha)k +3 k_1-(N_0-30)k_{1,+}+ k_3+  m  +  \beta m}\} \epsilon_0.
\ee
  \[
\| \mathcal{F}^{-1}[K^{\mu_1, \mu_2,\nu_1, \nu_2}_{k,k_1,k_2,k_3,k_4}(t, \xi )]\|_{B_{k,j}}\lesssim \sup_{t\in[2^{m-1}, 2^m]} 2^{(1+\alpha) k +10 k_+   +j +2k_1+ 2k_{1,+}} \| e^{-it \Lambda} g_{k_1}\|_{L^\infty} \|e^{-it \Lambda} g_{k_2}\|_{L^2} 
  \] 
\be\label{roughestimatequartic}
\times \| g_{k_3}\|_{L^2} \| e^{-it \Lambda} g_{k_4}\|_{L^\infty} \lesssim  2^{(1+\alpha)k  +k_4 +20\beta m  }\min\{2^{ 2k_1 - m/2  }, 2^{ 3k_1-(N_0-30)k_{1,+} +  m/2    } \}\epsilon_0.
\ee

From the rough estimate  (\ref{roughestimatecubic}), we can rule out the case when $ 
k_3   \leq  -m - 30\beta m 
$, or $k_1\geq 2\beta m$ or $k\leq -m/(1+\alpha/2)$ for the cubic terms. From the rough estimate
  (\ref{roughestimatequartic}), we can rule out the case when 
  $k_4\leq -m/2-30\beta m$ or $k_1\geq 2\beta m$ or $k\leq -m/(2+\alpha )$ for the quartic terms. 

Therefore,  it is sufficient to consider fixed $k$, $k_1$, $k_2$, $k_3$ for cubic terms and  fixed $k$, $k_1$, $k_2$, $k_3$, $k_4$ for quartic terms in the following ranges respectively, 
\be\label{restrictedrange}
 \noindent (\textup{Cubic terms})\quad -m - 30\beta m \leq k_3 \leq k_2\leq k_1\leq 2\beta m, \quad  -m/(1+\alpha/2)\leq k \leq 2\beta m ,
\ee
\be\label{restrictedrangequartic}
  \noindent  (\textup{Quartic terms})\quad -m/2 - 30\beta m \leq k_4\leq   k_3 \leq k_2\leq k_1\leq 2\beta m, \quad  -m/(2+\alpha )\leq k \leq 2\beta m .
\ee
\subsubsection{ When $k_2 \leq k_1 -10$}\label{lowhighcubic}  
Recall the normal form transformation   we did in subsection \ref{goodvariable}.  As $k_2\leq k_1-10$, the case when ``$\tau=-$'' is canceled out. \emph{Hence,  we only have to consider the case ``$\tau=+$''}.

 Note that, the following estimate holds for the derivatives of phase, 
\[
|\nabla_\xi  \Phi^{+, \kappa, \iota}(\xi, \eta,\sigma)|=|\nabla_\xi  \Phi^{+, \mu_2, \nu_1, \nu_2}(\xi, \eta,\sigma)|= \big| \Lambda'(|\xi|)\frac{\xi}{|\xi|} - \Lambda'(|\xi-\eta|)\frac{\xi-\eta}{|\xi-\eta|} \big| 
\]
\be\label{derivativephase}
\lesssim  \max\{2^{k_1}\angle(\xi, \xi-\eta), |\xi|-|\xi-\eta|\}\sim |\eta| \lesssim 2^{k_2 }.
\ee
\be\label{derivativephaseeta}
|\nabla_\eta  \Phi^{+, \kappa, \iota}(\xi, \eta,\sigma)|= \big|   \Lambda'(|\xi-\eta|)\frac{\xi-\eta}{|\xi-\eta|} + \kappa \Lambda'(| \eta-\sigma|)\frac{  \eta-\sigma}{|  \eta-\sigma|}  \big|  \sim 2^{k_1 -k_{1,+}/2}.
\ee

After doing spatial localizations for the inputs $\widehat{g_{k_1}}(\cdot)$ and $\widehat{g_{k_2}}(\cdot)$, we have the decomposition as follows, 
\[
T^{\tau, \kappa,\iota}_{k,k_1,k_2,k_3}(t, \xi)= \sum_{j_1\geq - k_{1,-}, j_2\geq -k_{2,-}}  T^{\tau, \kappa,\iota}_{k_1,j_1,k_2,j_2 }(t,\xi), \quad 
\]
\be\label{eqn455}
T^{\tau, \kappa,\iota}_{k_1,j_1,k_2,j_2 }(t,\xi) =  \int_{\R^2} e^{i t  \Phi^{\tau, \kappa, \iota}(\xi, \eta,\sigma)}  \tilde{d}_{\tau,\kappa, \iota}(\xi-\eta, \eta-\sigma, \sigma)  \widehat{g^{\tau}_{k_1,j_1}}(t  , \xi-\eta) \widehat{g^{\kappa}_{k_2,j_2}}(t, \eta-\sigma) \widehat{g^{\iota}_{k_3}}(t , \sigma) d \sigma  d \eta , 
\ee
\[
K^{\mu_1, \mu_2,\nu_1, \nu_2}_{k,k_1,k_2,k_3,k_4}(t, \xi )= \sum_{j_1\geq - k_{1,-}, j_2\geq -k_{2,-}}  K^{\mu_1, \mu_2,\nu_1, \nu_2}_{k_1,j_1,k_2,j_2 }(t,\xi),\]
\[ K^{\mu_1, \mu_2,\nu_1, \nu_2}_{k_1,j_1,k_2,j_2 }(t,\xi)=  \int_{\R^2} e^{i t  \Phi^{\mu_1, \mu_2, \nu_1, \nu_2}(\xi, \eta,\sigma,\kappa)}  \tilde{e}_{\mu_1, \mu_2, \nu_1, \nu_2}(\xi-\eta, \eta-\sigma, \sigma-\kappa, \kappa) 
\]
\be\label{eqn460}
\times  \widehat{g^{\mu_1}_{k_1,j_1}}(t  , \xi-\eta)  \widehat{g^{\mu_2}_{k_2,j_2}}(t, \eta-\sigma)  \widehat{g^{\nu_1}_{k_3}}(t , \sigma-\kappa) \widehat{g^{\nu_2}_{k_4}}(t , \kappa) d \kappa d \sigma  d \eta  . 
\ee

 Recall (\ref{derivativephase}). By doing integration by parts in ``$\xi$'' many times, we can rule out the case when $j\geq \max\{m +  k_2 , -k_{1,-}\}+ \beta m$ and $j_1 \leq j-\delta m $. For the case when  $j\geq \max\{m +  k_2 , -k_{1,-}\}+ \beta m $ and $j_1 \geq j-\delta m $, the following estimate holds from $L^2-L^\infty-L^\infty$ type trilinear estimate in Lemma \ref{multilinearestimate}, 
\[
\| \sum_{j_1 \geq j-\delta m }\mathcal{F}^{-1}[T^{\tau, \kappa,\iota}_{k_1,j_1,k_2,j_2}(t, \xi) ]\|_{B_{k,j}} \lesssim  \sum_{j_1 \geq j-\delta m }  2^{(1+\alpha)k   + j +2k_1+ 2k_{1,+}+10 k_+ } \| g_{k_1, j_1}(t)\|_{L^2} 
\]
\[
\times  2^{k_2  }  \| g_{k_2  }(t)\|_{L^2} \| e^{-i t\Lambda} g_{k_3}(t)\|_{L^\infty} \lesssim 2^{-m/2+30\beta m } 2^{ k_2   -j } \epsilon_0 \lesssim 2^{-3m/2+40\beta m } \epsilon_0.
\]
\[
\| \sum_{j_1 \geq j-\delta m }\mathcal{F}^{-1}[ K^{\mu_1, \mu_2,\nu_1, \nu_2}_{k_1,j_1,k_2,j_2}(t, \xi )]\|_{B_{k,j}} \lesssim  \sum_{j_1 \geq j-\delta m }  2^{(1+\alpha)k + j +2k_1+ 2k_{1,+} +10k_+} \]
\[
\times \| g_{k_1, j_1}(t)\|_{L^2} 
 2^{k_2  }  \| g_{k_2  }(t)\|_{L^2} \| e^{-i t\Lambda} g_{k_3}(t)\|_{L^\infty}\| e^{-i t\Lambda} g_{k_4}(t)\|_{L^\infty}   \lesssim 2^{-2m +40\beta m } \epsilon_0.
\]

Therefore, it remains  to consider the case when  $j\leq \max\{m +  k_2 , -k_{1,-}\}+ \beta m $. From the $L^2-L^\infty-L^\infty-L^\infty$ type multilinear  estimate, we have
\[
\|  \mathcal{F}^{-1}[K^{\mu_1, \mu_2,\nu_1, \nu_2}_{k,k_1,k_2,k_3,k_4}(t, \xi ) ]\|_{B_{k,j}}\lesssim 2^{(1+\alpha)k + 10k_+  +j + 2k_1+ 2k_{1,+}} \| e^{-it\Lambda} g_{k_1}\|_{L^\infty} \| e^{-it\Lambda} g_{k_2}\|_{L^\infty}
\]
\[
\times \| e^{-it\Lambda} g_{k_3}\|_{L^\infty} \| g_{k_4}\|_{L^2} \lesssim 2^{-3m/2 + 40\beta m} \epsilon_0.	
\]

Now we proceed to estimate the cubic terms ``$T^{+, \kappa,\iota}_{k,k_1,k_2,k_3}(t, \xi)$''. 
 If   $k_1 +k_2  \leq -m/2 -12\beta m $, then the following estimate holds from the $L^2-L^\infty-L^\infty$ type trilinear estimate (\ref{trilinearesetimate}) in Lemma \ref{multilinearestimate} and $L^\infty\rightarrow L^2$ type Sobolev embedding,
\[
 \|\mathcal{F}^{-1}[T^{+, \kappa,\iota}_{k,k_1,k_2,k_3}(t, \xi)]\|_{B_{k,j}} \lesssim  2^{(1+\alpha)k + 10k_+  + j + 2k_1+ 2k_{1,+} } \| e^{-it \Lambda} g_{k_1 }(t)\|_{L^\infty}2^{k_2 }   \| g_{k_2 }(t)\|_{L^2} \| g_{k_3}(t)\|_{L^2}
\]
\[
   \lesssim  2^{2 k_1 + 2k_2 +20\beta m }\epsilon_0 + 2^{k_1+2k_2+20\beta m} \epsilon_0 \lesssim 2^{-m-\beta m }\epsilon_0.
\]

Now, we proceed to consider the case when $k_1 +k_2 \geq -m/2-12\beta m $. Recall (\ref{derivativephaseeta}). By doing integration by parts in ``$\eta$'' many times, we can   rule out the case when $\max\{j_1  , j_2 \}\leq m+k_{1,-}  - \beta m $.   For the case  when $\max\{j_1 ,j_2 \}\geq m+k_{1,-}  - \beta m $, the following estimate holds from $L^2-L^\infty-L^\infty$ type trilinear estimate (\ref{trilinearesetimate}) in Lemma \ref{multilinearestimate}, 
\[
 \sum_{\max\{j_1 ,j_2 \}\geq m+k_{1,-} - \beta m }\|\mathcal{F}^{-1}[T^{\tau, \kappa,\iota}_{k_1,j_1,k_2,j_2 }(t,\xi) ]\|_{B_{k,j}} \lesssim   \sum_{j_1 \geq \max\{m+k_{1,-} - \beta m ,j_2 \}} 2^{(1+\alpha)k +10 k_+ + 2k_1+  j   }
 \]
 \[
 \times   \| g_{k_1 , j_1 }(t)\|_{L^2}     \| e^{-it \Lambda }g_{k_2 ,j_2  }(t)\|_{L^\infty} \| e^{-i t\Lambda} g_{k_3}(t)\|_{L^\infty} +  \sum_{j_2 \geq \max\{m+k_{1,-} - \beta m ,j_1 \}} 2^{(1+\alpha)k +10k_+ + j +2k_1  } \]
 \be\label{integrationbypartsinkappa}
 \times     \| g_{k_2  , j_2 }(t)\|_{L^2}   \| e^{-it \Lambda }g_{k_1 ,j_1  }(t)\|_{L^\infty} \| e^{-i t\Lambda} g_{k_3}(t)\|_{L^\infty} \lesssim 2^{-5m/2+50\beta m - k_2} \epsilon_0 \lesssim 2^{-m-\beta  m} \epsilon_0.
 \ee
\subsubsection{When $ k_1 -10\leq k_2 \leq k_1 $ and $k_3\leq k_2 -10$} 
The estimate of quartic terms is straightforward, from $L^2-L^\infty-L^\infty-L^\infty$ type mutilinear estimate, we have, 
\[
\|  \mathcal{F}^{-1}[K^{\mu_1, \mu_2,\nu_1, \nu_2}_{k,k_1,k_2,k_3,k_4}(t, \xi ) ]\|_{B_{k,j}} \lesssim    2^{(1+\alpha)k+ 10k_+  +j + 2k_1} \| e^{-it \Lambda} g_{k_1}(t)\|_{L^\infty }  \| e^{-it \Lambda} g_{k_2}(t)\|_{L^\infty } \]
\be\label{eqn450}
\times \| e^{-it \Lambda} g_{k_3}(t)\|_{L^\infty } \| g_{k_4}(t)\|_{L^2}  \lesssim 2^{-3m/2 +40\beta  m}\epsilon_0.
\ee

Now, it remains to  estimate  the cubic terms ``$T^{\tau, \kappa,\iota}_{k,k_1,k_2,k_3}(t, \xi)$''.  Recall the normal form transformation we did in subsection \ref{goodvariable}. Note that    the case when $\eta$ is close to $\xi/2$ is canceled, see (\ref{eqn200}). Hence, the following estimate always holds for the case we are considering, 
\be\label{derivativephase3}
 |\nabla_\eta \Phi^{\tau,\kappa, \iota}(\xi, \eta, \sigma)|\gtrsim 2^{k-k_{1,+}/2}.
\ee

 After putting $g_{k_3}$ in $L^2$ and the other two inputs in $L^\infty$, the following estimate holds from the $L^2-L^\infty-L^\infty$ type trilinear estimate (\ref{trilinearesetimate}) in Lemma \ref{multilinearestimate} when $k\leq -2\beta m $, 
\[
 \|\mathcal{F}^{-1}[T^{\tau, \kappa,\iota}_{k,k_1,k_2,k_3}(t, \xi)]\|_{B_{k,j}} \lesssim 2^{(1+\alpha)k  + j + 2k_1+2k_{1,+}} \| e^{-it \Lambda} g_{k_1 }(t)\|_{L^\infty}  \| e^{-it \Lambda} g_{k_2 }(t)\|_{L^2} \| g_{k_3}(t)\|_{L^2}
\]
\[
  \lesssim  \max\{2^{\alpha k -2m + 2\beta m  }, 2^{(1+\alpha) k -m  +\beta m }\} \epsilon_1^3 \lesssim 2^{-m-\beta m }\epsilon_0.
\]

Hence, it remains to consider the case when $k\geq -2\beta m $. Recall (\ref{derivativephase3}).  We can rule out the case when $\max\{j_1,j_2\}\leq   m+k_{-}   - 3\beta m $ by doing integration by parts in ``$\eta$'' many times. Hence, we only have to consider the case when $\max\{j_1,j_2\}\geq   m+k_{-}   - 3\beta m$ . From the $L^2-L^\infty-L^\infty$ type estimate (\ref{trilinearesetimate}) in Lemma \ref{multilinearestimate}, the following estimate holds, 
\[
\sum_{\max\{j_1,j_2\}\geq   m+k_{-}   - 3\beta m } \| \mathcal{F}^{-1}[  T^{\tau, \kappa,\iota}_{k_1,j_1,k_2,j_2 }(t,\xi)]\|_{B_{k,j}} \lesssim   \sum_{\max\{j_1,j_2\}\geq  m+k_{-}   - 3\beta m} 2^{(1+\alpha)k+10k_{+}+j + 2k_1}\]
\[
\times   \| g_{k_1,j_1}\|_{L^2} 2^{k_2} \| g_{k_2,j_2}\|_{L^2}  \| e^{-it \Lambda} g_{k_3}\|_{L^\infty}
\lesssim 2^{-3m/2+50\beta m }\epsilon_0.
 \]

\subsubsection{When $ k_1 -10\leq k_2  \leq k_1  $ and $k_2\leq k_3 -10\leq  k_2 $}\label{allcomparable}
 Note that the estimate (\ref{eqn450}) still holds as the size of $k_3$ plays little role there. Hence, we only have to estimate the cubic term  ``$T^{+, \kappa,\iota}_{k,k_1,k_2,k_3}(t, \xi)$''for this case.
Define
\be\label{eqn719}
 \mathcal{S}_1:=  \{(+,-,-),(-,+,+) \}, \quad  \mathcal{S}_2:=  (+,-,+), (-,+,-)\},
\ee
\be\label{eqn720}
\mathcal{S}_3:= \{(+,+,-),(-,-,+)\},\quad  \mathcal{S}_4:= \{(+,+,+),(-,-,-) \}.
\ee
Recall (\ref{phaseofcubic}), we have
\[ \nabla_\eta \Phi^{\tau , \kappa, \iota}(\xi, \eta, \sigma) =   -\tau \Lambda'(|\xi-\eta|) \frac{ \eta-\xi}{| \eta-\xi|}- \kappa \Lambda'(|\eta-\sigma|) \frac{\eta-\sigma}{|\eta-\sigma|},
\]
\[
\nabla_\sigma \Phi^{\tau , \kappa, \iota}(\xi, \eta, \sigma)  = -\kappa \Lambda'(|\eta -\sigma|) \frac{\sigma-\eta}{| \sigma-\eta|} - \iota \Lambda'(|\sigma|) \frac{ \sigma}{|\sigma|}.
\]
Correspondingly, the space resonance set in ``$\eta$'' and in ``$\sigma$'' is defined as follows, 
\[
\mathcal{R}_{\tau, \kappa, \iota}:=\{(\xi, \eta, \sigma): \nabla_\eta \Phi^{\tau , \kappa, \iota}(\xi, \eta, \sigma)=\nabla_\sigma \Phi^{\tau , \kappa, \iota}(\xi, \eta, \sigma)=0 \} \]
\[= \{ (\xi, \eta, \sigma): \xi  = \big((1+\tau \kappa)(1+\kappa \iota) -\tau \kappa \big) \sigma,   \eta= (1+\kappa \iota) \sigma \},\quad \tau, \kappa, \iota\in\{+,-\}.
\]
More specifically, we have
\[
\mathcal{R}_{\tau, \kappa, \iota}= \{(\xi, \eta, \sigma): \xi  =  \sigma,   \eta= 2\sigma  \}, \quad (  \xi-\eta, \eta-\sigma, \sigma)\big|_{  \mathcal{R}_{\tau, \kappa, \iota}}=(-\xi, \xi, \xi),  \quad  {(\tau, \kappa, \iota)\in \mathcal{S}_1}, \]
 \[
 \mathcal{R}_{\tau, \kappa, \iota}= \{  (\xi, \eta, \sigma): \xi  =   \sigma,   \eta= 0  \},\quad (  \xi-\eta, \eta-\sigma, \sigma)\big|_{ \mathcal{R}_{\tau, \kappa, \iota}}=(\xi, -\xi, \xi), \quad  {(\tau, \kappa, \iota)\in \mathcal{S}_2}, \]
 \[
 \mathcal{R}_{\tau, \kappa, \iota}= \{ (\xi, \eta, \sigma): \xi  = - \sigma,   \eta=0  \},\quad (  \xi-\eta, \eta-\sigma, \sigma)\big|_{ 	\mathcal{R}_{\tau, \kappa, \iota}}=(\xi, \xi, -\xi), \quad  {(\tau, \kappa, \iota)\in \mathcal{S}_3}, \]
\[
\mathcal{R}_{\tau, \kappa, \iota}= \{ (\xi, \eta, \sigma): \xi  =  3\sigma,   \eta= 2\sigma  \},\quad (  \xi-\eta, \eta-\sigma, \sigma)\big|_{ 	\mathcal{R}_{\tau, \kappa, \iota}}=(\xi/3, \xi/3, \xi/3), \quad  {(\tau, \kappa, \iota)\in \mathcal{S}_4}. \]

\noindent $\bullet$\quad When $  (\tau, \kappa, \iota)\in \mathcal{S}_1\cup \mathcal{S}_2\cup \mathcal{S}_3 $. Note that, after changing of variables, those three cases are symmetric. Hence, it is sufficient to estimate the case when $ (\tau, \kappa, \iota) \in \mathcal{S}_1$ in details.  We do change of coordinates for ``$ T^{\tau, \kappa,\iota}_{k,k_1,k_2,k_3}(t, \xi)$'' as follows,
\[
T^{\tau, \kappa,\iota}_{k,k_1,k_2,k_3}(t, \xi)=  \int_{\R^2} e^{i t  \widetilde{\Phi}^{\tau, \kappa, \iota}(\xi, \eta, \sigma)} \widetilde{c}(\xi, 2\xi+\eta+\sigma,\xi+\sigma)   \widehat{g^{\tau }_{k_1 }}(t  , -\xi-\eta-\sigma)  \widehat{g^{\kappa}_{k_2 }}(t, \xi+\eta ) \widehat{g^{\iota}_{k_3}}(t ,  \xi+\sigma) d \sigma  d \eta  , 
\]
where  the phase $\widetilde{\Phi}^{\tau, \kappa, \iota}(\xi, \eta, \sigma)$ is defined as follows, 
\be\label{eqn724}
\widetilde{\Phi}^{\tau, \kappa, \iota}(\xi, \eta, \sigma):= \Lambda(|\xi|)-\tau \Lambda(|\xi+\eta+\sigma|)-\kappa \Lambda(|\xi+\eta |)- \iota \Lambda(|\xi+\sigma|), \quad (\tau, \kappa, \iota)\in \mathcal{S}_1.
\ee

We localize both $\eta$ and $\sigma$ around zero (the space resonance set) with a well chosen threshold   and 
 decompose the cubic term as follows, 
\be\label{cubicdecomposition24}
T^{\tau, \kappa,\iota}_{k,k_1,k_2,k_3}(t, \xi)= \sum_{l_1, l_2 \geq \bar{l}_{\tau}}C^{\tau , l_1, l_2},\quad C^{\tau , l_1, l_2}= \sum_{j_1 \geq -k_{1,-} , j_2 \geq -k_{2,-}  } C^{\tau , l_1, l_2}_{j_1  ,  j_2  }, 
\ee
where the thresholds $\bar{l}_{-}:= -2m/5-10\beta m $ and $\bar{l}_{+}:=  k_{-}-10 $ and  $ C^{\tau , l_1, l_2}_{j_1  ,  j_2  }$ is defined as follows,
 \[    C^{\tau , l_1, l_2}_{j_1 ,  j_2 }:= \int_{\R^2} e^{i t  \widetilde{\Phi}^{\tau, \kappa, \iota}(\xi, \eta, \sigma)}  \widetilde{c}(\xi, 2\xi+\eta+\sigma,\xi+\sigma)  \widehat{g^{\tau }_{k_1,j_1 }}(t  , -\xi-\eta-\sigma) \widehat{g^{\kappa}_{k_2,j_2}}(t, \xi+\eta ) \]
\be\label{cubicdecompose}
\times  \widehat{g^{\iota}_{k_3}}(t ,  \xi+\sigma)  \varphi_{l_1;\bar{l}_{\tau }}(\eta) \varphi_{l_2;\bar{l}_{\tau  }}(\sigma) d \sigma  d \eta,
\ee
where the cutoff function $\varphi_{l;\bar{l}}(\cdot)$ with the threshold $\bar{l}$ is defined as follows, 
\be\label{thresholdcutoff}
\varphi_{ {l};\bar{l}}(x):= \left\{ 
\begin{array}{ll}
\psi_{\leq \bar{l}}(x) &   \textup{if} \,  l =\bar{l} \\
\psi_l(x) &   \textup{if} \,  l >\bar{l}.  \\
\end{array}\right. 
\ee
 
$\oplus$ If $\tau=+$, i.e., $(\tau, \kappa, \iota)=(+,-,-)$. \quad Recall the normal form transformation that we did in subsection \ref{goodvariable} , see (\ref{normalformatransfor}) and (\ref{eqn200}). For the case we are considering, i.e., $(\tau, \kappa, \iota)\in \widetilde{S}$, we already canceled out the case  when    $\max\{l_1, l_2\} = \bar{l}_{+}$. Hence it would be sufficient   to  consider the case when $\max\{l_1, l_2\} > \bar{l}_{-}$. By symmetry, we might assume that $l_2 = \max\{l_1, l_2\} > \bar{l}_{+}:=k_{-}-10 .$ For this case, we take the advantage of the fact that $\nabla_\eta  \widetilde{\Phi}^{\tau, \kappa, \iota}(\xi, \eta, \sigma) $ is relatively big. More precisely, we have
 \be\label{eqn723}
 \big|\nabla_\eta  \widetilde{\Phi}^{+,-,-}(\xi, \eta, \sigma)\big| = \big|  \Lambda'(|\xi+\eta+\sigma|) \frac{\xi+\eta+\sigma}{|\xi+\eta+\sigma|} -  \Lambda'(|\xi+\eta  |) \frac{\xi+\eta }{|\xi+\eta |} \big|   \gtrsim 2^{l_2 }.
\ee
Hence, we can do integration by parts in ``$\eta$'' many times to rule out the case when $\max\{j_1 , j_2  \} \leq m + k_{-} - \beta m $. From the $L^2-L^\infty-L^\infty$  type trilinear estimate (\ref{trilinearesetimate}) in Lemma \ref{multilinearestimate} and the $L^\infty\rightarrow L^2$ type Sobolev embedding,   the following estimate holds, 
\[
\sum_{\max\{j_1, j_2 \} \geq m + k_{-} - \beta m  } \| \mathcal{F}^{-1}[ C^{+,l_1, l_2}_{  j_1, j_2}] \|_{B_{k,j}}  \lesssim \sum_{\max\{j_1, j_2 \} \geq m +   k_{-} - \beta m  }  2^{(1+\alpha)k +10k_{+}   + j +2k_1 } 
\]
\[
\times  \| e^{-it \Lambda} g_{k_3}(t)\|_{L^\infty} \| g_{k_2, j_2 }(t)\|_{L^2} 2^{k_2} \| g_{k_1 , j_1 }(t)\|_{L^2} \lesssim 2^{-3m/2+40\beta m } \epsilon_0.
\]

$\oplus$ If $\tau =-$, i.e., $(\tau, \kappa, \iota)=(-,+,+)$ . As before, by symmetry, we might assume that $l_2 = \max\{l_1, l_2\} $. Recall (\ref{eqn724}).  We have
\[
\big|\nabla_\xi  \widetilde{\Phi}^{-,+,+}(\xi, \eta, \sigma)\big|   = \Big|\Lambda'(|\xi|) \frac{\xi}{|\xi|} + \Lambda'(|\xi+\eta|) \frac{\xi+\eta+\sigma}{|\xi+\eta+\sigma|}  - \Lambda'(|\xi+\eta |) \frac{\xi+\eta }{|\xi+\eta |} -\Lambda'(|\xi+\sigma|) \frac{\xi+\sigma}{|\xi+\sigma| }\Big|, 
\]
\[
\big|\nabla_\eta  \widetilde{\Phi}^{-,+,+}(\xi, \eta, \sigma)\big|= \Big|\Lambda'(|\xi+\eta+\sigma|) \frac{\xi+\eta+\sigma}{|\xi+\eta+\sigma|}  - \Lambda'(|\xi+\eta |) \frac{\xi+\eta }{|\xi+\eta |}\Big|
\]
Now, it is easy to see that 
\be\label{eqn739}
\big |\nabla_\xi  \widetilde{\Phi}^{-,+,+}(\xi, \eta, \sigma) \big|   \lesssim  2^{l_2},\quad \big|\nabla_\eta  \widetilde{\Phi}^{-,+,+}(\xi, \eta, \sigma)\big| \gtrsim 2^{l_2-k_{+}/2}.
\ee
Hence, we can first rule out the case when $j \geq m + l_2 +2\beta m  $ by doing integration by parts in ``$\xi$'' many times. It would be sufficient to consider the case when $j \leq m + l_2 +2\beta m  $.

 We first consider the case when $\max\{l_1,l_2\} = \bar{l}_{-}=-2m/5-10\beta m $.  After using the volume of supports in $\eta$ and $\sigma$, the following estimate holds, 
\[
\| \mathcal{F}^{-1}[C^{-, \bar{l}_{-}, \bar{l}_{-}}]\|_{B_{k, j}} \lesssim  2^{(1+\alpha)k +10k_{+}   + j + 2k_1 } 2^{4\bar{l}}\| g_{k_1 }(t)\|_{L^2} \| g_{k_2 }(t)\|_{L^1}\| g_{k_3 }(t)\|_{L^1}\]
\[\lesssim 2^{5 \bar{l}+ m +30\beta m } \epsilon_1^3\lesssim 2^{-m-\beta m } \epsilon_0.
\]

Now, we proceed to consider the case when $\max\{l_1, l_2\}> \bar{l}_{-}=-2m/5-10\beta m$.   For this case, we do integration by parts in $\eta$ many times to rule out the case when $\max\{j_1, j_2 \}\leq m  + l_2-4\beta m $. From  the $L^2-L^\infty-L^\infty$ type trilinear estimate (\ref{trilinearesetimate}) in Lemma \ref{multilinearestimate}, the following estimate holds when $\max\{j_1, j_2 \}\geq m  + l_2-4\beta m $, 
\be\label{generalestimate3}
\sum_{\max\{j_1, j_2  \} \geq m + l_2 -4\beta m  }\|\mathcal{F}^{-1}[ C^{-, l_1, l_2}_{j_1,j_2}] \|_{B_{k,j }} \lesssim  2^{(1+\alpha)k+10k_{+}   + j + 2k_1 } \| e^{-it \Lambda}g_{k_3 }(t)\|_{L^\infty} 
 \ee
\[\times \big[  \sum_{j_2 \geq \max\{m + l_2 -4\beta m ,j_1 \}}  \| g_{k_2,j_2}(t)\|_{L^2} \| e^{-it \Lambda} g_{k_1,j_1}(t)\|_{L^\infty} +  \sum_{j_1 \geq \max\{m + l_2 -4\beta m ,j_2 \}}   \| e^{-it \Lambda}g_{k_2 ,j_2 }(t)\|_{L^\infty}\]
\[\times  \| g_{k_1 ,j_1 }(t)\|_{L^2} \big]
 \lesssim 2^{-2m-l_2-m/2+40\beta m } \epsilon_0\lesssim 2^{-m-\beta m } \epsilon_0.
\]

 \noindent$\bullet$\quad When $(\tau, \kappa,\iota) \in \mathcal{S}_4$. \quad Very similarly, we localize around the space resonance set ``$(\xi/3,\xi/3,\xi/3)$'' by doing change of variables    for   ``$T^{\tau, \kappa,\iota}_{k_1,k_2,k_3}(t,\xi)$'' as follows, 
\[
T^{\tau, \kappa,\iota}_{k_1,k_2,k_3}(t,\xi)=  \int_{\R^2} e^{i t  \widehat{\Phi}^{\tau, \kappa, \iota}(\xi, \eta, \sigma)} \widetilde{c}(\xi, 2\xi/3+\eta+\sigma, \xi/3 +\sigma ) \widehat{g^{\tau }_{k_1 }}(t  ,\xi/3 -\eta-\sigma)\]
\[\times    \widehat{g^{\kappa}_{k_2 }}(t, \xi/3+\eta  ) \widehat{g^{\iota}_{k_3}}(t ,  \xi/3 +\sigma ) d \sigma  d \eta , 
\]
where  the phase $ \widehat{\Phi}^{\tau, \kappa, \iota}(\xi, \eta, \sigma)$ is defined as follows, 
\[
\widehat{\Phi}^{\tau, \kappa, \iota}(\xi, \eta, \sigma):= \Lambda(|\xi|)-\tau \Lambda(|\xi/3-\eta-\sigma|)-\kappa \Lambda(|\xi/3+\eta )- \iota \Lambda(|\xi/3+\sigma|), \quad (\tau, \kappa, \iota)\in \mathcal{S}_4.
\]

Recall the normal form transformation that we did in subsection \ref{goodvariable}. The symbol around a neighborhood of $(\xi/3, \xi/3,\xi/3)$ is removed, see (\ref{eqn200}) and (\ref{eqn1643}). Hence, the following decomposition holds,
\[
T^{\tau, \kappa,\iota}_{k,k_1,k_2,k_3}(t, \xi) = \sum_{i=1,2} T^{\tau, \kappa,\iota}_{k_1,k_2,k_3;i}(t,\xi),\quad  T^{\tau, \kappa,\iota}_{k_1,k_2,k_3;1}(t,\xi)= \sum_{j_1\geq -k_{1,-}, j_2\geq -k_{2,-}}  T^{\tau, \kappa,\iota}_{k_1,j_1,k_2,j_2;1}(t,\xi), \]
\[
  T^{\tau, \kappa,\iota}_{k_1,k_2,k_3;2}(t,\xi)= \sum_{j_1\geq -k_{1,-}, j_3\geq -k_{3,-}}  T^{\tau, \kappa,\iota}_{k_1,j_1,k_3,j_3;2}(t,\xi),
\]
\[
T^{\tau, \kappa,\iota}_{k_1,j_1,k_2,j_2;1}(t,\xi)=  \int_{\R^2} e^{i t  \widehat{\Phi}^{\tau, \kappa, \iota}(\xi, \eta, \sigma)}\psi_{\geq k-20}(2\eta+\sigma) \widetilde{c}(\xi, 2\xi/3+\eta+\sigma, \xi/3+\sigma)
\]
 \be\label{e1020}
\times  \widehat{g^{\tau }_{k_1,j_1}}(t  ,\xi/3-\eta-\sigma) \widehat{g^{\kappa}_{k_2,j_2}}(t, \xi/3+\eta ) \widehat{g^{\iota}_{k_3}}(t ,  \xi/3+\sigma)  d \sigma  d \eta   ,
\ee
\[
T^{\tau, \kappa,\iota}_{k_1,j_1,k_3,j_3;2}(t,\xi)=  \int_{\R^2} e^{i t  \widehat{\Phi}^{\tau, \kappa, \iota}(\xi, \eta, \sigma)}\widetilde{c}(\xi, 2\xi/3+\eta+\sigma, \xi/3+\sigma)  \psi_{\geq k-20}(2\sigma+\eta)\psi_{\leq k-20}(2\eta+	\sigma)
\]
 \be\label{e1021}
\times  \widehat{g^{\tau }_{k_1,j_1 }}(t  ,\xi/3-\eta-\sigma) \widehat{g^{\kappa}_{k_2  }}(t, \xi/3+\eta ) \widehat{g^{\iota}_{k_3,j_3}}(t ,  \xi/3+\sigma) d \sigma  d \eta  .
\ee
The estimates of  ``$T^{\tau, \kappa,\iota}_{k_1,k_2,k_3;1}(t,\xi)$'' and ``$T^{\tau, \kappa,\iota}_{k_1,k_2,k_3;2}(t,\xi)$'' are very similar. For simplicity, we only estimate  $T^{\tau, \kappa,\iota}_{k_1,k_2,k_3;1}(t,\xi)$ in details here. 
 For this case, note that ``$2\eta +\sigma$'' has a good upper bound. Hence, the size of $\nabla_\eta  \widehat{\Phi}^{\tau, \kappa, \iota}(\xi, \eta, \sigma)$ is bounded from blow by $2^{k-k_{+}/2}$. 

 Hence, by doing integration by parts many times in ``$\eta$'', we can rule out the case when $\max\{j_1 ,j_2\}\leq m + k_{-} -2\beta m $. For the case when $\max\{j_1,j_2\}\geq m + k_{-} -2\beta m $, a similar estimate as in (\ref{generalestimate3}) holds,
\[
\sum_{\max\{j_1,j_2\}\geq m + k_{-} -2\beta m} \| \mathcal{F}^{-1}[ T^{\tau, \kappa,\iota}_{k_1,j_1,k_2,j_2;1}(t,\xi)] \|_{B_{k,j }} \lesssim  \textup{R.H.S. of (\ref{generalestimate3}) }  \lesssim 2^{-m-\beta  m} \epsilon_0.
\]

\section{The improved estimate of the high order weighted norm}\label{highorderweighted}
Our main goal in this section is to prove (\ref{desired2}) under the smallness assumption (\ref{smallness}). Recall that $L:=x\cdot\nabla +2$ and $\Omega:=x^{\perp}\cdot \nabla$ and the $Z_2$ norm is defined in (\ref{highorderweightnorm}). Define
\[
  \hat{\Omega}_\xi:= - \xi^{\perp}\cdot \nabla_\xi,\quad d_{\Omega}:=0,\quad \xi_{\Omega}:=-\xi^{\perp},\quad \hat{L}_\xi:=  - \xi\cdot \nabla_\xi , \quad d_{L}:=-2, \quad \xi_{L}:= -\xi.
\]
\[
\chi_k^1:=\{(k_1,k_2): |k_1-k_2|\leq 10, k\leq k_1+10\}, \quad \chi_k^2:=\{(k_1,k_2): k_2\leq k_1- 10, |k_1-k|\leq 10\}.
\]
Note that, 
\[
\hat{\Omega}_\xi\widehat{g}(t, \xi)= \widehat{\Omega g}(t, \xi), \quad \hat{L}_\xi\widehat{g}(t, \xi)= \widehat{L g}(t, \xi),\quad 
\]
\be\label{eqn940} 
\|g(t)\|_{Z_2}    \sim \sum_{\Gamma^1_\xi, \Gamma^2_\xi\in \{\hat{\Omega}_\xi, \hat{L}_\xi\}}  \|  \Gamma^1_\xi\Gamma^2_\xi \widehat{g} (t,\xi) \|_{L^2}  + \| \Gamma^1_\xi \widehat{g} (t,\xi) \|_{L^2}.
\ee

Therefore, to close the argument, it would be sufficient to prove the following desired estimate for any $\Gamma_\xi, \Gamma^1_\xi, \Gamma^2_\xi \in \{\hat{L}_\xi, \hat{\Omega}_\xi \} $( correspondingly, $\Gamma, \Gamma^1, \Gamma^2\in\{L, \Omega\}$) and any $t_1,t_2\in[2^{m-1}, 2^{m}]$, 
 
\be\label{eqn1000}
\Big| \textup{Re}\big[ \int_{t_1}^{t_2} \int_{\R^2} \overline{\Gamma_\xi  \widehat{g}(t, \xi )} \Gamma_\xi  \p_t \widehat{g}(t, \xi ) d \xi d t \big]\Big| + \Big| \textup{Re}\big[ \int_{t_1}^{t_2} \overline{\Gamma^1_\xi \Gamma^2_\xi \widehat{g}(t, \xi )} \Gamma^1_\xi \Gamma^2_\xi \p_t \widehat{g}(t, \xi ) d \xi d t \big]\Big|  \lesssim 2^{2\tilde{\delta}m }.
\ee

\emph{The estimate of the first part of the left hand side of (\ref{eqn1000}) is similar and also much easier than the second part. Hence, for simplicity, we only estimate the second part in details.} 

Recall (\ref{realduhamel}) and (\ref{eqn650}).   From the  direct computations, we have the following identity for the    quadratic terms,
 \[
  \int_{t_1}^{t_2}  \int_{\R^2}\overline{\Gamma^1_\xi \Gamma^2_\xi \widehat{g_k}(t, \xi )} \Gamma^1_\xi \Gamma^2_\xi  B^{\mu, \nu}_{k,k_1,k_2}(t,\xi)  d \xi d t =  \int_{t_1}^{t_2}  \int_{\R^2}\overline{\Gamma^1_\xi \Gamma^2_\xi \widehat{g_k}(t, \xi )} e^{i t\Phi^{\mu, \nu}(\xi, \eta)}\Big[ \Gamma^1_\xi \Gamma^2_\xi \big(\tilde{q}_{\mu, \nu}(\xi-\eta, \eta)   
\]
\[
\times\widehat{g_{k_1}^{\mu}}(t, \xi-\eta) \big)  \widehat{g^{\nu}_{k_2}}(t, \eta)+ \sum_{{l,n}=\{1,2\}}  it \big(\Gamma^l_\xi \Phi^{\mu, \nu}(\xi, \eta)\big) \Gamma^n_\xi \big(\tilde{q}_{\mu, \nu}(\xi-\eta, \eta) \widehat{g_{k_1}^{\mu}}(t, \xi-\eta) \big)  \widehat{g^{\nu}_{k_2}}(t, \eta)  
\]
\be\label{eqn711}
-  t^2    \Gamma^1_\xi \Phi^{\mu, \nu}(\xi, \eta) \Gamma^2_\xi \Phi^{\mu, \nu}(\xi, \eta)  \tilde{q}_{\mu, \nu}(\xi-\eta, \eta) \widehat{g_{k_1}^{\mu}}(t, \xi-\eta)    \widehat{g^{\nu}_{k_2}}(t, \eta) \Big]d \eta d \xi d t,
\ee

Since  the formulation   (\ref{eqn711})  lacks the requisite symmetry, in order to utilize the hidden symmetry, we need to   make one of the inputs to be the same type as the output $  \Gamma^1_\xi \Gamma^2_\xi \widehat{g_k}(t, \xi )$. Hence, we split  $\Gamma_\xi^i$,  $i\in\{1,2\}$, into two parts as follows,  
\[
\Gamma_\xi^i \widehat{g}(t, \xi-\eta) = \Gamma_{\xi-\eta}^i \widehat{g}(t, \xi-\eta)-\Gamma_\eta^i\widehat{g}(t, \xi-\eta). 
\]
We do integration by parts in ``$\eta$'' in (\ref{eqn711}) to move the derivative in front of $\Gamma_\eta^i\widehat{g}(t, \xi-\eta) $ around. As a result, the following identity holds,   
\be\label{eqn1006}
 \textup{Re}\Big[\int_{t_1}^{t_2} \int_{\R^2} \overline{\Gamma^1_\xi \Gamma^2_\xi \widehat{g}_k(t, \xi)}  \Gamma^1_\xi \Gamma^2_\xi B^{\mu, \nu}_{k,k_1,k_2}(t,\xi) d \xi d t\Big] = \sum_{i=1,2,3,4} \textup{Re}[P_{k,k_1,k_2}^i],
\ee
where
\[
P_{k,k_1,k_2}^1:=  \sum_{\{l,n\}=\{1,2\}} \int_{t_1}^{t_2} \int_{\R^2} \int_{\R^2 } \overline{\widehat{\Gamma^1 \Gamma^2 g_k}(t, \xi)} e^{i t\Phi^{\mu, \nu}(\xi, \eta)} it   (\Gamma_\xi^l +\Gamma_\eta^l )\Phi^{\mu, \nu}(\xi, \eta) \]
\[
\times \big[\tilde{q}_{\mu, \nu}(\xi-\eta, \eta)  \big(  \widehat{\Gamma^n  g_{k_1}^{\mu}}(t, \xi-\eta)\widehat{g_{k_2}^\nu}(t, \eta) +   \widehat{g_{k_1}^{\mu}}(t, \xi-\eta) \widehat{\Gamma^n g_{k_2}^\nu}(t, \eta)\big) + (\Gamma_\xi^n +\Gamma_\eta^n + d_{\Gamma^n} )\tilde{q}_{\mu, \nu}(\xi-\eta, \eta)\]
\be\label{eqn1010} 
 \times \widehat{   g_{k_1}^{\mu}}(t, \xi-\eta)\widehat{g_{k_2}^\nu}(t, \eta) \big] d \eta d \xi d t,
\ee
\[
P_{k,k_1,k_2}^2:= - \int_{t_1}^{t_2} \int_{\R^2} \int_{\R^2 } \overline{\widehat{\Gamma^1 \Gamma^2 g_k}(t, \xi)} e^{i t\Phi^{\mu, \nu}(\xi, \eta)}  t^2 (\Gamma_\xi^1 +\Gamma_\eta^1 ) \Phi^{\mu, \nu}(\xi, \eta) (\Gamma_\xi^2 +\Gamma_\eta^2 ) \Phi^{\mu, \nu}(\xi, \eta)\]
\be\label{eqn1011}
\times  \tilde{q}_{\mu, \nu}(\xi-\eta, \eta)  \widehat{g_{k_1}^{\mu}}(t, \xi-\eta) \widehat{  g_{k_2}^\nu}(t, \eta)   d \eta d \xi d t,
\ee
\[
P_{k,k_1,k_2}^3:=  \int_{t_1}^{t_2} \int_{\R^2} \int_{\R^2 } \overline{\widehat{\Gamma^1 \Gamma^2 g_k}(t, \xi)} e^{i t\Phi^{\mu, \nu}(\xi, \eta)} \Big(    \tilde{q}_{\mu, \nu}(\xi-\eta, \eta) \big(\widehat{  \Gamma^1   \Gamma^2 g_{k_1}^{\mu}}(t, \xi-\eta)\widehat{g_{k_2}^\nu}(t, \eta) \]
  \[+ \widehat{  g_{k_1}^{\mu}}(t, \xi-\eta)\widehat{\Gamma^1   \Gamma^2 g_{k_2}^\nu}(t, \eta )\big)   + (\Gamma_\xi^1 +\Gamma_\eta^1 + d_{\Gamma^1} )(\Gamma_\xi^2 +\Gamma_\eta^2 + d_{\Gamma^2} )\tilde{q}_{\mu, \nu}(\xi-\eta, \eta) \widehat{  g_{k_1}^{\mu}}(t, \xi-\eta) \widehat{  g_{k_2}^\nu}(t, \eta)\]
\be\label{eqn1012}
  + (\Gamma_\xi^l +\Gamma_\eta^l + d_{\Gamma^l} )\tilde{q}_{\mu, \nu}(\xi-\eta, \eta)\big(\widehat{ \Gamma^n g_{k_1}^{\mu}}(t, \xi-\eta) \widehat{  g_{k_2}^\nu}(t, \eta) + \widehat{ g_{k_1}^{\mu}}(t, \xi-\eta) \widehat{ \Gamma^n  g_{k_2}^\nu}(t, \eta) \big) \Big) d \eta d\xi d t, 
\ee
\[
P_{k,k_1,k_2}^4:= \sum_{j_1\geq -k_{1,-}, j_2\geq -k_{2,-}}P_{k,k_1,k_2}^{4,j_1,j_2}, \quad  P_{k,k_1,k_2}^{4,j_1,j_2}:= \sum_{ \{l,n\}=\{1,2\}} \int_{t_1}^{t_2} \int_{\R^2} \int_{\R^2 } \overline{\widehat{\Gamma^1 \Gamma^2 g_k}(t, \xi)} \]
\be\label{eqn1013}
\times  e^{i t\Phi^{\mu, \nu}(\xi, \eta)}   \tilde{q}_{\mu, \nu}(\xi-\eta, \eta) \widehat{ \Gamma^l g_{k_1,j_1}^{\mu}}(t, \xi-\eta) \widehat{\Gamma^n g_{k_2,j_2}^\nu}(t, \eta) d \eta d\xi dt. 
\ee

Note that, the following equalities hold  when $|\eta|\ll |\xi|$ and $\mu=+$, 
\[
\big(\hat{L}_\xi + \hat{L}_\eta\big) \Phi^{\mu, \nu}(\xi, \eta)=    -2\xi \cdot \big( \lambda'(|\xi|^2) \xi -    \lambda'(|\xi-\eta|^2)  (\xi-\eta)\big) -2\eta \cdot \big( -\lambda'(|\xi-\eta|^2)  (\eta-\xi)
\]
\be\label{eqn945}
- \nu \lambda'(|\eta|)\eta\big) = -4\big(   \lambda'(|\xi|^2) +  \lambda''(|\xi|^2)|\xi|^2  \big) \xi\cdot \eta + \mathcal{O}(|\eta|^2),
\ee
\[
(\hat{\Omega}_\xi + \hat{\Omega}_\eta ) \Phi^{\mu, \nu}(\xi, \eta)= -2\xi^{\perp}\cdot \big(\lambda'(|\xi|^2) \xi  -\mu\lambda(|\xi-\eta|^2)(\xi-\eta)\big)\]
\be\label{eqn989}
- 2\eta^{\perp} \cdot \big( -\mu\lambda(|\xi-\eta|^2)(\eta-\xi)-\nu \lambda'(|\eta|^2)\eta\big) = - 2\mu \lambda'(|\xi-\eta|^2)\big( \xi^{\perp}\cdot \eta + \eta^{\perp}\cdot \xi)=0.
\ee
Meanwhile, the following equalities hold when   $|\xi|\ll |\eta|$ and $\mu\nu=-,$
\[
\big(\hat{L}_\xi + \hat{L}_\eta\big) \Phi^{\mu, \nu}(\xi, \eta)=  -2\lambda'(|\xi|^2) |\xi|^2 + \mu 2\lambda'(|\xi-\eta|^2) \xi\cdot(\xi-\eta) +2\mu \lambda'(|\xi-\eta|^2) \eta\cdot (\eta-\xi) 
\]
\be\label{eqn944}
+2 \nu\lambda'(|\eta|^2) |\eta|^2= -4\mu\big( \lambda'(|\eta|^2)  +  \lambda''(|\eta|^2)|\eta|^2 \big) \xi \cdot \eta + \mathcal{O}(|\xi|^2),
\ee
\be\label{eqn990}
(\hat{\Omega}_\xi + \hat{\Omega}_\eta ) \Phi^{\mu, \nu}(\xi, \eta)=-2\mu \lambda'(|\xi-\eta|^2)\big( \xi^{\perp}\cdot \eta + \eta^{\perp}\cdot \xi)=0.
\ee
Hence,  from (\ref{degeneratephase}) and (\ref{eqn945}),  the following identity holds when $|\eta|\ll |\xi|$ and $\mu=+$, 
\be\label{eqn1301}
\big(\hat{L}_\xi + \hat{L}_\eta\big) \Phi^{\mu, \nu}(\xi, \eta) = \tilde{c}(\xi-\eta)\Phi^{\mu, \nu}(\xi, \eta) + \mathcal{O}(|\eta|^2),\quad \tilde{c}(\xi):= -\frac{ 2 \lambda''(|\xi|^2)|\xi|^2 + 2\lambda'(|\xi|^2) }{ \lambda'(|\xi|^2) }.
\ee
Meanwhile,
from (\ref{eqn920}) and (\ref{eqn944}), the following identity holds when $|\xi|\ll |\eta|$ and  $\mu\nu=-,$
\be\label{eqn951}
\big(\hat{L}_\xi + \hat{L}_\eta\big) \Phi^{\mu, \nu}(\xi, \eta) = \tilde{c}(\xi-\eta)\Phi^{\mu, \nu}(\xi, \eta) + \mathcal{O}(|\xi|^2) .
\ee
  
\subsection{The $Z_2$-norm estimate of quadratic terms: when $|k_1-k_2|\leq 10$ and $k\leq k_1-10$.}  Recall the normal form transformation we did in subsection \ref{goodvariable}. For the case we are considering, \emph{we have $\mu\nu=-$.}

Recall 
(\ref{eqn1006}). For the High $\times$ High type interaction that we are considering, the estimates of ``$ P_{k,k_1,k_2}^3$'' (see (\ref{eqn1013})) and ``$ P_{k,k_1,k_2}^4$'' (see(\ref{eqn1012}))  are much easier. Hence we estimate them first. From the $L^2-L^\infty$ type bilinear estimate (\ref{bilinearesetimate}) in Lemma \ref{multilinearestimate}, we have
\[
\big| \sum_{|k_1-k_2|\leq 10, k\leq k_1+5}  P_{k,k_1,k_2}^3\big| \lesssim \sup_{t_1,t_2\in[2^{m-1},2^m]}  \sum_{|k_1-k_2|\leq 10} 2^{m+2k_1} \| P_{\leq k_1 +5} \Gamma^1 \Gamma^2 g(t)\|_{L^2}   \big( \|e^{-it \Lambda}  g_{k_1}(t)\|_{L^\infty} 
\]
\be\label{eqn998}
+  \|e^{-it \Lambda}  g_{k_2}(t)\|_{L^\infty} \big)  \big(\sum_{l,m\in\{1,2\}}\| \Gamma^1 \Gamma^2 g_{k_m}(t)\|_{L^2}    + \| \Gamma^l g_{k_m}\|_{L^2}  + \| g_{k_m}(t)\|_{L^2} \big) \lesssim 2^{2\tilde{\delta} m }\epsilon_0^2.
\ee
The estimate of $P_{k,k_1,k_2}^4 $ is similar but slightly different. The spatial concentrations of inputs play a role. We  put
the input with smaller spatial concentration ``$j$'' in $L^\infty$ and the other input in $L^2$. As a result, we have
\[
\big| \sum_{|k_1-k_2|\leq 10, k\leq k_1+5}  P_{k,k_1,k_2}^4	\big| \lesssim \sup_{t_1,t_2\in[2^{m-1},2^m]} \sum_{|k_1-k_2|\leq 10}  \sum_{ \{l,m\}=\{1,2\}} 2^{m+2k_1}   \| P_{\leq k_1 +5} \Gamma^1 \Gamma^2 g(t)\|_{L^2} 
\] 
\[
\times \big( \sum_{j_1 \geq j_2}\| \Gamma^l g_{k_1,j_1} \|_{L^2} \| e^{-it \Lambda} \Gamma^n g_{k_2,j_2}\|_{L^\infty} +   \sum_{j_2 \geq j_1} \| e^{-it \Lambda}\Gamma^l g_{k_1,j_1}\|_{L^\infty}\|\Gamma^n g_{k_2,j_2} \|_{L^2} \big)\lesssim \sum_{j_2} 2^{2k_1+2j_2}
\]
\be\label{eqn999}
\times \| g_{k_2,j_2}(t)\|_{L^2} \big(\sum_{j_1\geq j_2} 2^{ 2 \tilde{\delta} m -j_1}\epsilon_1 \big) + \sum_{j_1} 2^{2k_1+2j_1} \| g_{k_1,j_1}(t)\|_{L^2} \big(\sum_{j_2\geq j_1} 2^{ 2\tilde{\delta} m -j_2} \epsilon_1 \big)\lesssim 2^{2\tilde{\delta}m }\epsilon_0^2.
\ee
Note that, in above estimate we used the following simple facts, 
\[
 \| \Gamma^l g_{k ,j } \|_{L^2}\lesssim  2^{-k -j +\tilde{\delta} m }\epsilon_1, \quad \| e^{-it \Lambda}\Gamma^l g_{k ,j }\|_{L^\infty}\lesssim 2^{-m+ k  +2j }  \| g_{k ,j }\|_{L^2}.
\] 
 
Now our main goal is reduced to estimate ``$P_{k,k_1,k_2}^1$'' (see (\ref{eqn1010})) and ``$P_{k,k_1,k_2}^2$'' (see (\ref{eqn1011})). From the $L^2-L^\infty$ type bilinear estimate (\ref{bilinearesetimate}) in Lemma \ref{multilinearestimate}, the following estimate holds, 
\[
|P_{k,k_1,k_2}^1| + |P_{k,k_1,k_2}^2| \lesssim  ( 2^{2m +k+3k_1} + 2^{3m+2k+4k_1} )\| \Gamma^1 \Gamma^2 g_{k}\|_{L^2}\big(\sum_{i=1,2} \| g_{k_i}(t)\|_{L^2} +  2^{k_1}\| \nabla_\xi \widehat{g}_{k_i}(t, \xi)\|_{L^2}  \big)
\]
\be\label{eqn1069}
 \times  \big(\sum_{i=1,2} \| e^{-i t\Lambda} g_{k_i}(t)\|_{L^\infty}   \big)\lesssim  2^{ \tilde{\delta}m  +\delta m}(2^{m+k+k_{1,-}- 15k_{1,+}} + 2^{2m+2k+2k_{1,-}-14k_{1,+}})\epsilon_0^2 .
\ee

From above rough estimate (\ref{eqn1069}), we can rule out  the case when $k+k_{1,-}\leq -m +\tilde{\delta} m/3$ or $k_1\geq m/5$. Note that there are only   $m^2$ cases left to consider, which is only a logarithmic loss.  \emph{For the rest of this subsection, we restrict ourself to the case when $k$ and $k_1$ are fixed such that $k+k_{1,-}\geq -m +\tilde{\delta }  m/3$ and $k_1\leq m/5$}.

\subsubsection{ The estimate of $P_{k,k_1,k_2}^1$}\label{highhighinteraction1}  Recall (\ref{eqn1010}) and (\ref{eqn990}). We know that the integral inside $P_{k,k_1,k_2}^1$ actually vanishes when $\Gamma^l=\hat{\Omega}_\xi$. Hence, we only need to consider the case when $\Gamma^l_\xi=\hat{L}_\xi$. 

 Due to the $l^2$ type structure with respect to $k$ (corresponding to the output frequency $|\xi|\sim 2^k$)  of the $Z_2$-normed space (see (\ref{highorderweightnorm})), after doing integration by parts in $\eta$, we confront a summability issue with respect to $k$. 

 To get around this difficulty, we use the good decomposition of  ``$(\hat{L}_\xi + \hat{L}_\eta) \Phi^{\mu, \nu}(\xi, \eta)$''. Recall (\ref{eqn951}).  To take the advantage of this decomposition,    we  decompose $P_{k,k_1,k_2}^1$ into two parts and have the following estimate,  
\be\label{eqn1019}
|P_{k,k_1,k_2}^1|\leq \sum_{\Gamma\in\{L,\Omega\}} | \Gamma_{k,k_1,,k_2}^{1 ,1}| +|  \Gamma_{k,k_1,,k_2}^{1,2}|,  
\ee
where
\[
 \Gamma_{k,k_1,,k_2}^{1,i}:=  \int_{t_1}^{t_2} \int_{\R^2} \int_{\R^2 } \overline{\widehat{\Gamma^1 \Gamma^2 g_k}(t, \xi)} e^{i t\Phi^{\mu, \nu}(\xi, \eta)} it  \tilde{q}^{i}_{\mu, \nu} (\xi-\eta, \eta)  \big[\tilde{q}_{\mu, \nu}(\xi-\eta, \eta)  \big(  \widehat{\Gamma  g_{k_1}^{\mu}}(t, \xi-\eta)  \]
\be\label{eqn1405}
\times   \widehat{g_{k_2}^\nu}(t, \eta)+   \widehat{g_{k_1}^{\mu}}(t, \xi-\eta) \widehat{\Gamma  g_{k_2}^\nu}(t, \eta)\big) + (\Gamma_\xi  +\Gamma_\eta + d_\Gamma )\tilde{q}_{\mu, \nu}(\xi-\eta, \eta)  \widehat{   g_{k_1}^{\mu}}(t, \xi-\eta)\widehat{g_{k_2}^\nu}(t, \eta) \big] d \eta d \xi d t, 
\ee
where
\be\label{eqn1120}
 \tilde{q}^{1}_{\mu, \nu} (\xi-\eta, \eta)=  \tilde{c}(\xi-\eta)\Phi^{\mu, \nu}(\xi, \eta), \quad  \tilde{q}^{ 2}_{\mu, \nu} (\xi-\eta, \eta):=(\hat{L}_\xi + \hat{L}_\eta) \Phi^{\mu, \nu}(\xi, \eta)- \tilde{c}(\xi-\eta)\Phi^{\mu, \nu}(\xi, \eta).
\ee

For $\Gamma_{k,k_1,k_2}^{1,1}$, we do integration by parts in time. As a result, we have
\be\label{eqn1425}
\Gamma_{k,k_1,k_2}^{1,1}= \sum_{i=1,2}\widetilde{\Gamma}_{k,k_1,k_2}^{1,i},\quad \widetilde{\Gamma}_{k,k_1,k_2}^{1,1}=\sum_{j_1\geq -k_{1,-}, j_2\geq-k_{2,-}} \widetilde{\Gamma}_{k,k_1,k_2}^{j_1,j_2,1,1}, \ee
\[  \widetilde{\Gamma}_{k,k_1,k_2}^{j_1,j_2,1,1}:= -\int_{t_1}^{t_2} \int_{\R^2} \int_{\R^2 } \overline{\widehat{\Gamma^1 \Gamma^2 g_k}(t, \xi)} e^{i t\Phi^{\mu, \nu}(\xi, \eta)} \tilde{c}(\xi-\eta)\big[\tilde{q}_{\mu, \nu}(\xi-\eta, \eta) \big(  \widehat{g_{k_2,j_2}^\nu}(t, \eta) \widehat{\Gamma   g_{k_1,j_1}^{\mu}}(t, \xi-\eta)  \]
\[     +   \widehat{g_{k_1,j_1}^{\mu}}(t, \xi-\eta) \widehat{\Gamma  g_{k_2,j_2}^\nu}(t, \eta)\big) + (\Gamma_\xi  +\Gamma_\eta  + d_\Gamma)\tilde{q}_{\mu, \nu}(\xi, \eta)   \widehat{   g_{k_1,j_1}^{\mu}}(t, \xi-\eta) \widehat{g_{k_2,j_2}^\nu}(t, \eta) \big] d \eta d \xi d t
\]
\[
 + \sum_{i=1,2} (-1)^i \int_{\R^2} \int_{\R^2 } \overline{\widehat{\Gamma^1 \Gamma^2 g_k}(t_i, \xi)} e^{i t_i \Phi^{\mu, \nu}(\xi, \eta)} t_i \tilde{c}(\xi-\eta)\big[\tilde{q}_{\mu, \nu}(\xi-\eta, \eta)    \big(  \widehat{\Gamma   g_{k_1,j_1}^{\mu}}(t_i, \xi-\eta) \widehat{g_{k_2j_2}^\nu}(t_i, \eta)   \]
\be\label{eqn1025}
  +   \widehat{g_{k_1,j_1}^{\mu}}(t_i, \xi-\eta) \widehat{\Gamma  g_{k_2,j_2}^\nu}(t_i, \eta)\big) + (\Gamma_\xi  +\Gamma_\eta  + d_\Gamma)\tilde{q}_{\mu, \nu}(\xi-\eta, \eta)   \widehat{   g_{k_1,j_1}^{\mu}}(t_i, \xi-\eta)\widehat{g_{k_2,j_2}^\nu}(t_i, \eta) \big] d \eta d \xi,
\ee
\[
\widetilde{\Gamma}_{k,k_1,k_2}^{1,2} = -\int_{t_1}^{t_2} \int_{\R^2} \int_{\R^2 } e^{i t\Phi^{\mu, \nu}(\xi, \eta)}  t   \tilde{c}(\xi-\eta)   \big[\tilde{q}_{\mu, \nu}(\xi-\eta, \eta)  \p_t \big( \overline{\widehat{\Gamma^1 \Gamma^2 g_k}(t, \xi)}\big(  \widehat{\Gamma   g_{k_1}^{\mu}}(t, \xi-\eta)   \widehat{g_{k_2}^\nu}(t, \eta) \]
\be\label{eqn1026}
+   \widehat{g_{k_1}^{\mu}}(t, \xi-\eta) \widehat{\Gamma  g_{k_2}^\nu}(t, \eta)\big) \big) + (\Gamma_\xi  +\Gamma_\eta + d_\Gamma )\tilde{q}_{\mu, \nu}(\xi-\eta, \eta) \p_t \big( \overline{\widehat{\Gamma^1 \Gamma^2 g_k}(t, \xi)} \widehat{   g_{k_1}^{\mu}}(t, \xi-\eta)\widehat{g_{k_2}^\nu}(t, \eta)\big) \big] d \eta d \xi d t.
\ee

 Since the symbol of  $\Gamma_{k,k_1,k_2}^{1,2}$ contributes the smallness of $|\xi|^2$ instead of $|\xi|$, which makes the summability with respect to $k$ not a issue any more. For $\Gamma_{k,k_1,k_2}^{1,2}$, we are safely to do integration by parts in ``$\eta$''. More precisely, the following Lemma holds,
\begin{lemma}\label{errorestimate}
Under the bootstrap assumption \textup{(\ref{smallness})}, the following estimate holds,
\be\label{eqn112900}
 \sum_{k\leq k_1+2, |k_1-k_2|\leq 10}\sum_{\Gamma\in\{L, \Omega\}} \big|  \Gamma_{k,k_1,,k_2}^{1,2}   \big| \lesssim  2^{2\tilde{\delta} m  }\epsilon_0^2. 
\ee
\end{lemma}
\begin{proof}
Recall (\ref{eqn1120}) and (\ref{eqn951}).  From  Lemma \ref{Snorm}, the following estimate holds, 
\be\label{eqn1121}
\| \tilde{q}_{\mu, \nu}^2(\xi-\eta, \eta) \psi_{k}(\xi) \psi_{k_1}(\xi-\eta)\psi_{k_2}(\eta)\|_{\mathcal{S}^\infty}\lesssim 2^{2k},\quad k\leq k_1-10.
\ee
After doing integration by parts in ``$\eta$'' once, the following estimate holds, 
\[
| \Gamma_{k,k_1,,k_2}^{1,2} | \lesssim   \| \Gamma^1 \Gamma^2 g_k(t)\|_{L^2}  2^{m+k+k_1+k_{1,+}}\big(\sum_{i=0,1,2}2^{i k_1}\| \nabla_\xi^i \widehat{g_{k_1}}(t, \xi)\|_{L^2} +2^{i k_1} \| \nabla_\xi^i \widehat{g_{k_2}}(t, \xi)\|_{L^2}  \big)
\]
\[
\times \big(\sum_{i=1,2} \| e^{-it \Lambda} g_{k_i}\|_{L^\infty} \big)+ \sum_{j_1\geq j_2} 2^{ k+3k_1+k_{1,+}}\| \Gamma^1 \Gamma^2 g_k(t)\|_{L^2}    2^{2j_2} \| g_{k_2,j_2}\|_{L^2} 2^{j_1}\| g_{k_1,j_1}\|_{L^2}+  \sum_{j_2\geq j_1}2^{ k+3k_1+k_{1,+}}  \]
\be
\times\| \Gamma^1 \Gamma^2 g_k(t)\|_{L^2}  2^{2j_1} \| g_{k_2,j_2}\|_{L^2} 2^{j_2}\| g_{k_1,j_1}\|_{L^2} \lesssim 2^{(1-\alpha)k-10k_{1,+} +2\tilde{\delta}m} \epsilon_0^2,
\ee
which is very sufficient to derive our desired estimate (\ref{eqn112900}).
\end{proof}

\begin{lemma}
Under the bootstrap assumption \textup{(\ref{smallness})}, the following estimate holds,
\be
 \sum_{\Gamma\in\{L,\Omega\}} \big|  \widetilde{\Gamma}_{k,k_1,k_2}^{1,1} \big| \lesssim 2^{9\tilde{\delta} m/5}\epsilon_0^2.
\ee
\end{lemma}
\begin{proof}
Recall (\ref{eqn1425}) and  (\ref{eqn1025}). For $ \widetilde{\Gamma}_{k,k_1,k_2}^{1,1}$, we do integration by parts in ``$\eta$'' once. As a result,  the following estimate holds from the $L^2-L^\infty$ type bilinear estimate (\ref{bilinearesetimate}) in Lemma \ref{multilinearestimate},
\[
  |\widetilde{\Gamma}_{k,k_1,k_2}^{1,1}| \lesssim   2^{k_{1,+}} \| \Gamma^1 \Gamma^2 g_{k}\|_{L^2}\Big[2^{ -k+ k_1}\big(\sum_{i=0,1,2} 2^{i k_1}\| \nabla_\xi^i \widehat{g}_{k_1}(t, \xi)\|_{L^2}+ 2^{i k_1}\| \nabla_\xi^i \widehat{g}_{k_2}(t, \xi)\|_{L^2}\big)
\]
\[
  \times \big(\sum_{i=1,2} \| e^{-i t\Lambda} g_{k_i}(t)\|_{L^\infty}   \big) + \sum_{j_1 \geq j_2} 2^{-k+3k_1 +j_1 } \| e^{-it \Lambda}\mathcal{F}^{-1}\big[ \nabla_\xi \widehat{g}_{k_2,j_2} \big]\|_{L^\infty} \| g_{k_1,j_1}\|_{L^2} 
 \]
 \[+  \sum_{j_2 \geq j_1} 2^{-k+3k_1 +j_2 } \| e^{-it \Lambda}\mathcal{F}^{-1}\big[ \nabla_\xi \widehat{g}_{k_1,j_1} \big]\|_{L^\infty} \| g_{k_2,j_2}\|_{L^2} \Big] \lesssim 2^{-m-k-k_1+2\tilde{\delta} m +\delta m }\epsilon_0^2\lesssim 2^{9\tilde{\delta} m/5}\epsilon_0^2. 
\]
In above estimate, we used the fact that $k+k_{1,-}\geq -m + \tilde{\delta}m /3$.
\end{proof}
\begin{lemma}
Under the bootstrap assumption \textup{(\ref{smallness})}, the following estimate holds,
\be
\sum_{\Gamma\in\{L,\Omega\}} \big|   \widetilde{\Gamma}_{k,k_1,k_2}^{1,2} \big| \lesssim  2^{9\tilde{\delta} m/5}\epsilon_0^2.
\ee
\end{lemma}
\begin{proof}
Since`` $\p_t$'' can hit every input inside $\widetilde{\Gamma}^{1,2}_{k,k_1,k_2}$, which creates many terms. We  put terms that have similar structures together and    split $\widetilde{\Gamma}^{1,2}_{k,k_2,k_2}$ into five parts  as follows, 
\be\label{eqn1467}
\widetilde{\Gamma}^{1,2}_{k,k_1,k_2}= \sum_{i=1,2,3,4,5 } \widehat{\Gamma}^{i}_{k,k_1,k_2}, \quad \widehat{\Gamma}^{1}_{k,k_1,k_2} =- \int_{t_1}^{t_2} \int_{\R^2} \int_{\R^2 } e^{i t\Phi^{\mu, \nu}(\xi, \eta)}  t  \tilde{c}(\xi-\eta)  \big[  \tilde{q}_{\mu, \nu}(\xi-\eta, \eta)    
 \ee
\[ \times\big(  \widehat{\Gamma   g_{k_1}^{\mu}}(t, \xi-\eta)  \widehat{g_{k_2}^\nu}(t, \eta)  + \widehat{g_{k_1}^{\mu}}(t, \xi-\eta)  \widehat{\Gamma  g_{k_2}^\nu}(t, \eta)\big) +  (\Gamma_\xi  +\Gamma_\eta + d_{\Gamma} )\tilde{q}_{\mu, \nu}(\xi-\eta, \eta)   \widehat{   g_{k_1}^{\mu}}(t , \xi-\eta)\widehat{g_{k_2}^\nu}(t , \eta) \big]  \]
\be\label{eqn1031}
 \times \overline{ \big(\p_t \widehat{\Gamma^1 \Gamma^2 g_k}(t, \xi) -   \sum_{\nu'\in\{+,-\}} \sum_{(k_1',k_2')\in \chi_k^2} \widetilde{B}^{+, \nu'}_{k,k_1',k_2'}(t, \xi) \big)}d\eta d \xi d t, 
\ee
where $\widetilde{B}^{+, \nu'}_{k,k_1',k_2'}(t, \xi)$ is defined in (\ref{eqn892}),
\[
 \widehat{\Gamma}^{2}_{k,k_1,k_2} =   \sum_{  \nu'\in\{+,-\}} \sum_{(k_1',k_2')\in \chi_k^2}- \int_{t_1}^{t_2} \int_{\R^2} \int_{\R^2} \int_{\R^2 } \overline{ e^{it \Phi^{+, \nu'}(\xi, \kappa)} \widehat{\Gamma^1\Gamma^2 g_{k_1'}}(t, \xi-\kappa) \widehat{g}_{k_2'}(t, \kappa) \tilde{q}_{\mu',\nu'}(\xi-\kappa, \kappa) } \]
 \[\times  t  e^{i t\Phi^{\mu, \nu}(\xi, \eta)} \tilde{c}(\xi-\eta)  \big[  \tilde{q}_{\mu, \nu}(\xi-\eta, \eta) \big(  \widehat{\Gamma   g_{k_1}^{\mu}}(t, \xi-\eta)   \widehat{g_{k_2}^\nu}(t, \eta) +   \widehat{g_{k_1}^{\mu}}(t, \xi-\eta)  \widehat{\Gamma g_{k_2}^\nu}(t, \eta)\big)    \]
\be\label{eqn1030} 
 +  (\Gamma_\xi+\Gamma_\eta  + d_\Gamma)\tilde{q}_{\mu, \nu}(\xi-\eta, \eta)  \widehat{   g_{k_1}^{\mu}}(t , \xi-\eta)\widehat{g_{k_2}^\nu}(t , \eta) \big] d \kappa d\eta d \xi d t, 
\ee
\[
\widehat{\Gamma}^{3}_{k,k_1,k_2}  = -\int_{t_1}^{t_2} \int_{\R^2} \int_{\R^2 } \overline{\widehat{\Gamma^1 \Gamma^2 g_k}(t, \xi)}e^{i t\Phi^{\mu, \nu}(\xi, \eta)}  t    \tilde{c}(\xi-\eta)   \big[\tilde{q}_{\mu, \nu}(\xi-\eta, \eta)   \big(  \widehat{\Gamma  g_{k_1}^{\mu}}(t, \xi-\eta)\p_t \widehat{g_{k_2}^\nu}(t, \eta)   \]
\[
+  \p_t\widehat{g_{k_1}^{\mu}}(t, \xi-\eta) \widehat{\Gamma  g_{k_2}^\nu}(t, \eta)+ \widehat{\Gamma \Lambda_{\geq 3}[ \p_t g_{k_1}^{\mu}]}(t, \xi-\eta) \widehat{g_{k_2}^\nu}(t, \eta)   +  \widehat{g_{k_1}^{\mu}}(t, \xi-\eta) \widehat{\Gamma \Lambda_{\geq 3}[ \p_t g_{k_2}^\nu]}(t, \eta)    \big)   
\]
\be\label{eqn1040}
   + (\Gamma_\xi  +\Gamma_\eta + d_\Gamma )\tilde{q}_{\mu, \nu}(\xi-\eta, \eta)  \p_t \big(   \widehat{   g_{k_1}^{\mu}}(t, \xi-\eta)\widehat{g_{k_2}^\nu}(t, \eta) \big)\big] d \eta d \xi d t,
\ee
 \be\label{e9809}
\widehat{\Gamma}^{i}_{k,k_1,k_2}  =  \sum_{k_1',k_2'\in \mathbb{Z}} \Gamma_{k,k_1,k_2}^{k_1', k_2' ;i-3}, \quad \Gamma_{k,k_1,k_2}^{k_1', ,k_2' ;i-4}:= \sum_{j_1'\geq -k_{1,-}', j_2'\geq -k_{2,-}'  } \Gamma_{k,k_1,k_2}^{k_1',j_1',k_2',j_2';i-4}, \quad i\in\{ 4,5 \}, 
\ee
\[
\Gamma_{k,k_1,k_2}^{k_1',j_1',k_2',j_2';1}:=  \sum_{\tau, \iota\in\{+,-\}} -\int_{t_1}^{t_2} \int_{\R^2} \int_{\R^2} \int_{\R^2 } \overline{\widehat{\Gamma^1 \Gamma^2 g_k}(t, \xi)}e^{i t\Phi^{\mu, \nu}(\xi, \eta)}  t   \tilde{c}(\xi-\eta)    \tilde{q}_{\mu, \nu}(\xi-\eta, \eta)      \]
\[\times  \big(  P_{\mu}[e^{i t\Phi^{\tau, \iota}(\xi-\eta, \sigma)}\tilde{q}_{\tau, \iota}(\xi-\eta-\sigma, \sigma)    \widehat{g_{k_2',j_2'}^{\iota}}(t, \sigma) \Gamma_{\xi-\eta} \widehat{   g_{k_1',j_1'}^{\tau}}(t, \xi-\eta-\sigma) ] \widehat{g_{k_2}^\nu}(t, \eta) \]
\be\label{eqn1039}
 +  \widehat{g_{k_1}^{\mu}}(t, \xi-\eta) P_{\nu}[ e^{it \Phi^{\tau,\iota}(\eta, \sigma)}  \tilde{q}_{\tau, \iota}( \eta-\sigma, \sigma)    \Gamma_\eta \widehat{ g_{k_1',j_1'}^\tau}(t, \eta-\sigma) \widehat{g_{k_2',j_2'}^{\iota}}(t,\sigma)]\big) d\sigma  d \eta d \xi d t.
\ee
 \[
\Gamma_{k,k_1,k_2}^{k_1',j_1',k_2',j_2';2}  =  \sum_{\tau, \iota\in\{+,-\}}  - \int_{t_1}^{t_2} \int_{\R^2} \int_{\R^2} \int_{\R^2 } \overline{\widehat{\Gamma^1 \Gamma^2 g_k}(t, \xi)}e^{i t\Phi^{\mu, \nu}(\xi, \eta)}  it^2 \tilde{c}(\xi-\eta)    \tilde{q}_{\mu, \nu}(\xi-\eta, \eta)      \]
\[\times  \big(  P_{\mu}[e^{i t\Phi^{\tau, \iota}(\xi-\eta, \sigma)} \Gamma_{\xi-\eta} \Phi^{\tau, \iota}(\xi-\eta, \sigma) \tilde{q}_{\tau, \iota}(\xi-\eta-\sigma, \sigma)    \widehat{g_{k_2',j_2'}^{\iota}}(t, \sigma)   \widehat{   g_{k_1',j_1'}^{\tau}}(t, \xi-\eta-\sigma) ] \widehat{g_{k_2}^\nu}(t, \eta) \]
\be\label{eqn1060} 
+   \widehat{g_{k_1}^{\mu}}(t, \xi-\eta) P_{\nu}[ e^{it \Phi^{\tau,\iota}(\eta, \sigma)} \Gamma_\eta \Phi^{\tau,\iota}(\eta, \sigma) \tilde{q}_{\tau, \iota}( \eta-\sigma, \sigma)     \widehat{ g_{k_1',j_1'}^\tau}(t, \eta-\sigma) \widehat{g_{k_2',j_2'}^{\iota}}(t,\sigma)]\big) d\sigma  d \eta d \xi d t.
\ee
 
Recall (\ref{eqn1031}). For $\widehat{\Gamma}^{1}_{k,k_2,k_2}$, we do integration by parts in ``$\eta$'' once. From (\ref{eqn730}) in Lemma \ref{derivativeL2estimate2} and the $L^2-L^\infty$ type bilinear   estimate (\ref{bilinearesetimate}) in Lemma \ref{multilinearestimate}, the  following estimate holds, 
\[
   |\widehat{\Gamma}^{1}_{k,k_1,k_2} | \lesssim \sup_{t\in [2^{m-1}, 2^m]}  2^{-k+5k_{1,+} +  k_1}\big( 2^{ \tilde{\delta} m + \delta m  } + 2^{3\tilde{\delta} m +  k } \big)  \big( (\sum_{i=0,1,2} 2^{i k_1} \| \nabla_\xi^i  \widehat{g_{k_1}}(t, \xi) \|_{L^2} ) \| e^{-it\Lambda}  g_{k_2}(t)\|_{L^\infty}
\]
\[+  (\sum_{i=0,1,2} 2^{i k_1} \| \nabla_\xi^i  \widehat{g_{k_2}}(t, \xi) \|_{L^2} ) \| e^{-it\Lambda}  g_{k_1}\|_{L^\infty} + \sum_{j_1 \geq j_2} 2^{-m + 2j_2 +j_1+k_1+k_2}\| g_{k_1,j_1}(t)\|_{L^2} \| g_{k_2,j_2}(t)\|_{L^2}   \]
\[+ \sum_{j_2 \geq j_1} 2^{-m + 2j_1 +j_2+k_1+k_2}\| g_{k_1,j_1}(t)\|_{L^2} \| g_{k_2,j_2}(t)\|_{L^2}  \big) \lesssim 2^{-m+ 2\tilde{\delta}m  + \delta m -k-k_1}\epsilon_0^2 +2^{-m + 4\beta m } \epsilon_0^2 \lesssim  2^{9\tilde{\delta} m/5  }\epsilon_0^2.
\]
In above estimate, we used the fact that $k+k_1\geq -m + \tilde{\delta}m/3.$

Recall (\ref{eqn1030}). For $ \widehat{\Gamma}^{2}_{k,k_2,k_2}$, we do integration by parts in `` $\eta$'' once. Recall that $|k_1'-k|\leq 10$.  The loss of $2^{-k}$ from integration by parts in $\eta$ is compensated by the smallness of $2^{2k_1'}$ from the symbol $\tilde{q}_{\mu',\nu'}(\xi-\kappa, \kappa) $.  As a result, the following estimate holds from the $L^2-L^\infty  $ type bilinear estimate (\ref{bilinearesetimate}) in Lemma \ref{multilinearestimate}
\[
 |\widehat{\Gamma}^{2}_{k,k_1,k_2}| \lesssim \sum_{k_2'\leq k-10}2^{ m+  k+k_1}    \| \Gamma^1\Gamma^2  g_{k_1'}\|_{L^2} \| e^{-it \Lambda} g_{k_2'}(t)\|_{L^\infty}\big(\sum_{i=0,1,2} 2^{i k_1} \| \nabla_\xi^i \widehat{g}_{k_1}(t, \xi) \|_{L^2}  \]
\[+ 2^{i k_1} \| \nabla_\xi^i \widehat{g}_{k_2}(t, \xi) \|_{L^2} \big) \big( \sum_{i=1,2} 2^{k_1}\| e^{-it\Lambda} \mathcal{F}^{-1}[\nabla_\xi \widehat{g}_{k_i}]\|_{L^\infty}  +\| e^{-it\Lambda}  {g}_{k_i} \|_{L^\infty}   \big) \lesssim 2^{-m/2+\beta m }\epsilon_0^2.
\]

Now, we proceed to estimate $\widehat{\Gamma}^{3}_{k,k_2,k_2}$. Recall (\ref{eqn1040}). From estimate (\ref{eqn751}) in Lemma  \ref{derivativeL2estimate1}, estimate (\ref{Z1estimatecubicandquartic})  in Proposition \ref{eqq298}, (\ref{eqnj878}) in Lemma \ref{remaindertermweightednorm},  and the $L^2-L^\infty$ type bilinear estimate (\ref{bilinearesetimate}) in Lemma \ref{multilinearestimate}, the following estimate holds, 
\[
  |\widehat{\Gamma}^{3}_{k,k_1,k_2} | \lesssim \sum_{l=1,2}  2^{2m+(2-\alpha)k_1} \| \Gamma^1 \Gamma^2 g_k\|_{L^2}\Big[ \big(\|  \p_t \widehat{g}_{k_l}(t, \xi) - \sum_{\mu, \nu\in\{+,-\}} \sum_{(k_1',k_2')\in \chi_{k_l}^1} B^{\mu, \nu}_{k_l, k_1',k_2'}(t, \xi)\|_{L^2}  \]
  \[
  \times  \| e^{-it \Lambda} \Gamma g_{k_{3-l}}\|_{L^\infty} +  \|  \Gamma   g_{k_{3-l}}\|_{L^2}  \sum_{(k_1',k_2')\in \chi_{k_l}^1}\| e^{-it \Lambda}\mathcal{F}^{-1}[  B^{\mu, \nu}_{k_l, k_1',k_2'}(t, \xi)]\|_{L^\infty}
  \]
\[ 
 + \| \Lambda_{\geq 3}[\p_t g_{k_l}]\|_{Z_1} \|  e^{-it \Lambda} g_{k_{3-l}} \|_{L^\infty}\Big] 
\lesssim 2^{3\tilde{\delta} m/2   }\epsilon_0^2.
\]

Lastly,  we   estimate $\widehat{\Gamma}_{k,k_1,k_2}^4$ and  $\widehat{\Gamma}_{k,k_1,k_2}^5$. Recall (\ref{e9809}). Based on the size of difference between $k_1'$ and $k_2'$ and the size of $k_{1,-}'+k_2$, we split into three cases as follows, 

$\oplus$\quad If $|k_1'-k_2'|\leq 10$.\quad  For this case,  we know that $\nabla_\sigma \Phi^{\tau,\iota}(\cdot, \cdot) $ is bounded from blow by $2^{k_{1,-}-k_{1,+}'}$. Hence, to take advantage of this fact, we do integration by parts in $\sigma$ once for $\Gamma_{k,k_1,k_2}^{k_1', k_2' ;1}$ and do integration by parts in ``$\sigma$'' twice for   $\Gamma_{k,k_1,k_2}^{k_1', k_2' ;2}$. As a result, the following estimate holds from the $L^2-L^\infty$ type bilinear estimate (\ref{bilinearesetimate}) in Lemma \ref{multilinearestimate},
\[
\sum_{|k_1'-k_2'|\leq 10} \sum_{|k_1-k_2|\leq 10 } \sum_{i=1,2}|\Gamma_{k,k_1,k_2}^{k_1', k_2' ;i}|  \lesssim  \sum_{|k_1'-k_2'|\leq 10} \sum_{|k_1-k_2|\leq 10 } 2^{m+k_1+k_1' +2k_{1,+}'} \| \Gamma^1\Gamma^2 g_k\|_{L^2}\]
\[\times \big( \sum_{i=0,1,2} 2^{i k_1'} \| \nabla_\xi^i \widehat{g_{k_1'}}(t, \xi)\|_{L^2} + 2^{i k_1'} \| \nabla_\xi^i \widehat{g_{k_2'}}(t, \xi)\|_{L^2}  \big)  \big( \sum_{i=1,2} \| e^{-it \Lambda} g_{k_i}\|_{L^\infty}\big)
\]  
\[
\times \big( \sum_{i=1,2} 2^{k_1}\| e^{-it\Lambda} \mathcal{F}^{-1}[\nabla_\xi \widehat{g}_{k_i'}]\|_{L^\infty}  +\| e^{-it\Lambda}  {g}_{k_i'} \|_{L^\infty}   \big)\lesssim 2^{-\beta m }\epsilon_0^2.
\]
 
 $\oplus$\quad If $k_2'\leq k_1'-10$ and $k_{1,-}'+ k_2' \leq -19m/20$. \quad  Note that $|k_1'-k_1|\leq 10$. For this case,  we use the same strategy that we used in the estimates  (\ref{eqn302}) and (\ref{eqn303}).  From estimate  (\ref{eqn400}) in Lemma \ref{Linftyxi}, we have
\[
 \sum_{k_2'\leq k_1-10, k_2'+k_{1,-}\leq -9m/10} \sum_{i=1,2}|\Gamma_{k,k_1,k_2}^{k_1', k_2' ;i}|  \lesssim \sum_{k_2'\leq k_1-10, k_2'+k_{1,-}\leq -9m/10}  \| \Gamma^1\Gamma^2 g_k\|_{L^2}   \big(  2^{3k_2'} \| \widehat{g}_{k_2'}(t, \xi)\|_{L^\infty_\xi} 
\]
\[
+2^{k_1'+2k_2'} \| \widehat{\textup{Re}[v]}(t, \xi)\psi_{k_2'}(\xi)\|_{L^\infty_\xi}  \big)\big(\sum_{i=1,2}\| e^{-it \Lambda} g_{k_i}\|_{L^\infty} \big)   \big(2^{2m+2k_1+k_1'}\| \Gamma   g_{k_1'}\|_{L^2} + 2^{3m+2k_1+2k_1'+k_2'}\| g_{k_1'}(t)\|_{L^2}\big)\]
\[
\lesssim \sum_{k_2'\leq k_1-10, k_2'+k_{1,-}\leq -9m/10} 2^{3\tilde{\delta}m + 2m+2k_1+3k_2'}(1+ 2^{ 2  m+ 2k_{1,-}+2k_2'})\epsilon_0^2 \lesssim 2^{-\beta m }\epsilon_0^2.
\]

$\oplus$\quad If $k_2'\leq k_1'-10$ and $k_{1,-}+k_2'\geq -19m/20$.\quad   For this case, we do integration by parts in $\sigma$ many times to rule out the case when $\max\{j_1',j_2'\} \leq m +k_{1,-}- \beta m $. From    the $L^2-L^\infty-L^\infty$ type trilinear estimate (\ref{trilinearesetimate}) in Lemma \ref{multilinearestimate}, the following estimate holds when  $\max\{j_1',j_2'\} \geq m +k_{1,-}- \beta m $, 
\[
  \sum_{i=1,2} \sum_{\max\{j_1',j_2'\}\geq m +k_{1,-}- \beta m}|\Gamma_{k,k_1,k_2}^{k_1', j_1', k_2',j_2' ;i}|  \lesssim  \Big[ \sum_{ j_1'\geq \max\{j_2', m +k_{1,-}- \beta m \} }\big( 2^{ m+j_1'+5k_1} + 2^{2m+5k_1+k_2'}\big)
\]
\[
  \times \| g_{k_1',j_1'}\|_{L^2} \| g_{k_2',j_2'}\|_{L^1} \big( \sum_{i=1,2} \| e^{-it \Lambda} g_{k_i}\|_{L^\infty}\big) + \sum_{ j_2'\geq \max\{j_1', m +k_{1,-}- \beta m \} }  \big( 2^{ m+j_2'+5k_1} + 2^{2m+5k_1+k_2'}\big)
\]
\be\label{eqn1073}
\times  \| g_{k_2',j_2'}\|_{L^2} \| g_{k_1',j_1'}\|_{L^1} \big( \sum_{i=1,2} \| e^{-it \Lambda} g_{k_i}\|_{L^\infty}\big)\Big] \| \Gamma^1\Gamma^2 g_k\|_{L^2} \lesssim 2^{-m-k_2'+10\beta m}\epsilon_0^2\lesssim 2^{-\beta m}\epsilon_0^2.
\ee
Note that, we only have at most $m^4$ cases to consider. Hence, estimate (\ref{eqn1073} ) is very  sufficient.
\end{proof}

\subsubsection{ The estimate of $P_{k,k_1,k_2}^2$} Recall (\ref{eqn1011}),  (\ref{eqn990}) and (\ref{eqn951}). We know that $P_{k,k_1,k_2}^2$ vanishes except when \emph{$\Gamma^1=\Gamma^2=L$}. Very similar to what we did in (\ref{eqn1019}).   We decompose ``$P_{k,k_1,k_2}^2$'' it into two parts and have the following estimate, 
\[
|P_{k,k_1,k_2}^2| \leq  |\widetilde{P }_{k,k_1,k_2}^1| + |\widetilde{P }_{k,k_1,k_2}^2|,
\]
where $\widetilde{P }_{k,k_1,k_2}^i$,  $i\in\{1,2\}$, is defined as follow,
\be\label{eqn1133}
 \widetilde{P}_{k,k_1,k_2}^i=  -\int_{t_1}^{t_2} \int_{\R^2} \int_{\R^2 } \overline{\widehat{ LL g_k}(t, \xi)} e^{i t\Phi^{\mu, \nu}(\xi, \eta)}  t^2  \widehat{q}^i_{\mu, \nu}(\xi, \eta)    \widehat{g_{k_1}^{\mu}}(t, \xi-\eta) \widehat{  g_{k_2}^\nu}(t, \eta)   d \eta d \xi d t,  
 \ee
 \be\label{eqn1501}
   \widehat{q}^1_{\mu, \nu}(\xi, \eta)=  \tilde{q}_{\mu, \nu}(\xi-\eta, \eta)(\hat{L}_\xi+ \hat{L}_\eta)\Phi^{\mu, \nu}(\xi, \eta) \tilde{c}(\xi-\eta)\Phi^{\mu, \nu}(\xi, \eta), 
\ee
\be\label{eqn1132}
 \widehat{q}^2_{\mu, \nu}(\xi, \eta)= \tilde{q}_{\mu, \nu}(\xi-\eta, \eta)(\hat{L}_\xi+ \hat{L}_\eta)\Phi^{\mu, \nu}(\xi, \eta) \big((\hat{L}_\xi+ \hat{L}_\eta)\Phi^{\mu, \nu}(\xi, \eta)-\tilde{c}(\xi-\eta)\Phi^{\mu, \nu}(\xi, \eta)\big).
\ee
For $\widetilde{P}_{k,k_1,k_2}^1 $, we do integration by parts in time once. As a result, we have
\[
\widetilde{P }_{k,k_1,k_2}^1= \sum_{i=1,2,3,4,5} \widehat{P  }_{k,k_1,k_2}^{i},\]
\[ \widehat{P }_{k,k_1,k_2}^{1}= \sum_{i=1,2} (-1)^i \int_{\R^2} \int_{\R^2 } \overline{\widehat{LL  g_k}(t_i, \xi)} e^{i t_i\Phi^{\mu, \nu}(\xi, \eta)} i t_i^2  \widehat{p}^1_{\mu, \nu}(\xi, \eta)  \widehat{  g_{k_2}^\nu}(t_i, \eta)    \widehat{g_{k_1}^{\mu}}(t_i, \xi-\eta)    d \eta d\xi\]
\be\label{eq3}
 - \int_{t_1}^{t_2} \int_{\R^2} \int_{\R^2 } \overline{\widehat{LL g_k}(t , \xi)} e^{i t \Phi^{\mu, \nu}(\xi, \eta)} i  2 t  \widehat{p}^1_{\mu, \nu}(\xi, \eta)  \widehat{g_{k_1}^{\mu}}(t, \xi-\eta) \widehat{  g_{k_2}^\nu}(t, \eta)   d \eta d \xi  d t.
\ee
\[
\widehat{P }_{k,k_1,k_2}^{2} =  -\int_{t_1}^{t_2} \int_{\R^2} \int_{\R^2 } \overline{ \big(\p_t \widehat{LL g_k}(t, \xi) -   \sum_{\nu\in\{+,-\}} \sum_{(k_1',k_2')\in \chi_k^2} \widetilde{B}^{+, \nu}_{k,k_1',k_2'}(t, \xi) \big)}\]
\be\label{eq4}
\times e^{i t \Phi^{\mu, \nu}(\xi, \eta)} i t^2  \widehat{p}^1_{\mu, \nu}(\xi, \eta) \widehat{g_{k_1}^{\mu}}(t, \xi-\eta) \widehat{  g_{k_2}^\nu}(t, \eta)   d \eta d \xi  d t,
\ee
\[
  \widehat{P }_{k,k_1 ,k_2  }^{3}:=\sum_{j_1\geq -k_{1,-}, j_2\geq -k_{2,-}} \widehat{P }_{k,k_1,k_2}^{3,j_1,j_2},\quad   
\widehat{P }_{k,k_1,k_2}^{3,j_1,j_2}=  \sum_{  \nu'\in\{+,-\}} \sum_{(k_1',k_2')\in \chi_k^2}-\int_{t_1}^{t_2} \int_{\R^2} \int_{\R^2} \int_{\R^2 }e^{i t\Phi^{\mu, \nu}(\xi, \eta)}  i t^2   \]
\be\label{eqn1090} 
 e^{-it\Phi^{+, \nu'}(\xi, \kappa)}\overline{    \widehat{LL g_{k_1'}}(t, \xi-\kappa) \widehat{g^{\nu'}_{k_2'}}(t, \kappa) \tilde{q}_{+,\nu'}(\xi-\kappa, \kappa) } \widehat{p}^1_{\mu, \nu}(\xi, \eta)\widehat{g_{k_1,j_1}^{\mu}}(t, \xi-\eta) \widehat{  g_{k_2,j_2}^\nu}(t, \eta)  d\kappa d \eta  d\xi d t, 
\ee
\[
\widehat{P }_{k,k_1,k_2}^{4} := -\int_{t_1}^{t_2} \int_{\R^2} \int_{\R^2 } e^{i t\Phi^{\mu, \nu}(\xi, \eta)}    i    t^2 \widehat{p}^1_{\mu, \nu}(\xi, \eta)    \overline{\widehat{\Gamma^1 \Gamma^2 g_k}(t, \xi)}   \big(  \widehat{   \Lambda_{\geq 3}[\p_t  g^{\mu}_{k_1 }]}(t, \xi-\eta)   \widehat{g_{k_2}^\nu}(t, \eta) \]
\be\label{eqq326}
 +   \widehat{g_{k_1}^{\mu}}(t, \xi-\eta) \widehat{  \Lambda_{\geq 3}[\p_t  g^{\nu}_{k_2 }]}(t, \eta)\big)      d \eta d \xi d t,
\ee
 \be\label{eqn1110}
\widehat{P }_{k,k_1,k_2}^{5} = \sum_{k_1',k_2'\in \mathbb{Z}} \widehat{P'}_{k,k_1,k_2}^{k_1',k_2'}, \quad \widehat{P'}_{k,k_1,k_2}^{k_1',k_2' } = \sum_{j_1'\geq -k_{1,-}', j_2'\geq -k_{2,-}'}\widehat{P'}_{k,k_1,k_2}^{k_1',j_1', k_2',j_2' }, 
\ee
\[
\widehat{P'}_{k,k_1,k_2}^{k_1',j_1', k_2',j_2'}:=\sum_{\mu', \nu'\in\{+,-\}} -\int_{t_1}^{t_2} \int_{\R^2} \int_{\R^2 } \overline{\widehat{\Gamma^1 \Gamma^2 g_k}(t , \xi)} e^{i t \Phi^{\mu, \nu}(\xi, \eta)} i    t^2 \widehat{p}^1_{\mu, \nu}(\xi, \eta) \]
\[\times \big[   P_{\mu}[ e^{it \Phi^{\mu', \nu'}(\xi-\eta, \sigma)} \tilde{q}_{\mu', \nu'}(\xi-\eta-\sigma, \sigma) \widehat{g_{k_1',j_1'}^{\mu'}}(t, \xi-\eta-\sigma) \widehat{g_{k_2',j_2'}^{\nu'}}(t, \sigma)   ]\psi_{k_1}(\xi-\eta)\widehat{g^{\nu}_{k_2}}(t, \eta) \]
\be\label{eqn1096} 
+ \widehat{g_{k_1}^{\mu}} (t, \xi-\eta) P_{\nu} [ e^{it \Phi^{\mu', \nu'}(\eta, \sigma)} \tilde{q}_{\mu', \nu'}(\eta-\sigma, \sigma) \widehat{g_{k_1',j_1'}^{\mu'}}(t, \eta-\sigma) \widehat{g_{k_2',j_2'}^{\nu'}}(t, \sigma) ] \psi_{k_2}(\eta)   \big] d \sigma d\eta d \xi d t,
 \ee
 where
 \be\label{eqn1100}
\widehat{p}^1_{\mu, \nu}(\xi, \eta) = \tilde{q}_{\mu, \nu}(\xi-\eta, \eta)(\hat{L}_\xi+ \hat{L}_\eta)\Phi^{\mu, \nu}(\xi, \eta) \tilde{c}(\xi-\eta).
 \ee
From Lemma \ref{Snorm} and (\ref{eqn1301}), the following estimate holds,
\be\label{eqn1101}
\| \widehat{p}^1_{\mu, \nu}(\xi, \eta)\psi_k(\xi)\psi_{k_1}(\xi-\eta)\psi_{k_2}(\eta)\|_{\mathcal{S}^\infty}\lesssim 2^{k+3k_1},\quad k\leq k_1-10.
\ee

Similar to the estimate (\ref{eqn112900}) in Lemma  \ref{errorestimate},   the error term of decomposition can be handled very easily because of the extra smallness of $|\xi|$. More precisely, we have
\begin{lemma}\label{errorestimate2}
Under the bootstrap assumption \textup{(\ref{smallness})}, the following estimate holds,
\be\label{eq10}
 \sum_{k\leq k_1+2, |k_1-k_2|\leq 10} \big|   \widetilde{P }_{k,k_1,k_2}^2  \big| \lesssim  2^{2\tilde{\delta} m  }\epsilon_0^2. 
\ee
\end{lemma}
\begin{proof}
Recall (\ref{eqn1133}),  (\ref{eqn1132}) and (\ref{eqn951}). The following estimate holds from Lemma \ref{Snorm}, 
\be\label{eqn1134}
\|  \widehat{q}^2_{\mu, \nu}(\xi, \eta) \psi_k(\xi) \psi_{k_1}(\xi-\eta)\psi_{k_2}(\eta)\|_{\mathcal{S}^\infty} \lesssim 2^{3k+3k_1}.
\ee
After doing integration by parts in ``$\eta$'' twice, the following estimate holds, 
\[
 \sum_{k\leq k_1+2, |k_1-k_2|\leq 10} \big|   \widetilde{P }_{k,k_1,k_2}^2  \big| \lesssim   \sum_{k\leq k_1+2, |k_1-k_2|\leq 10}  2^{m+k+k_1+k_{1,+}} \| \Gamma^1 \Gamma^2 g_k(t)\|_{L^2}\big(\sum_{i=0,1,2}2^{i k_1} \| \nabla_\xi^i \widehat{g_{k_1}}(t, \xi)\|_{L^2}
\]
\[
 +2^{i k_1} \| \nabla_\xi^i \widehat{g_{k_2}}(t, \xi)\|_{L^2}  \big) \big(\sum_{i=1,2} \| e^{-it \Lambda} g_{k_i}\|_{L^\infty} \big)+    \sum_{k\leq k_1+2, |k_1-k_2|\leq 10}\| \Gamma^1 \Gamma^2 g_k(t)\|_{L^2} \Big( \sum_{j_1\geq j_2} 2^{ k+3k_1+k_{1,+}} \]
\be
\times 2^{2j_2}  \| g_{k_2,j_2}\|_{L^2}   2^{j_1}\| g_{k_1,j_1}\|_{L^2} +   \sum_{j_2\geq j_1} 2^{ k+3k_1+k_{1,+}} 2^{2j_1} \| g_{k_2,j_2}\|_{L^2} 2^{j_2}\| g_{k_1,j_1}\|_{L^2} \Big)\lesssim 2^{ 2\tilde{\delta}m} \epsilon_0^2.
\ee
\end{proof}
The rest of this subsection is devoted to prove the following Lemma, 
\begin{lemma}
Under the bootstrap assumption \textup{(\ref{smallness})}, the following estimate holds,
\be
\big|\widetilde{P }_{k,k_1,k_2}^1\big| \lesssim \sum_{i=1,2,3,4,5}  \big|   \widehat{P  }_{k,k_1,k_2}^{i} \big| \lesssim  2^{9\tilde{\delta} m/5}\epsilon_0^2. 
\ee
\end{lemma}
\begin{proof}
Recall (\ref{eq3}) and (\ref{eq4}). For $\widehat{P }_{k,k_1,k_2}^1$ and $\widehat{P }_{k,k_1,k_2}^2$ , we do integration by parts in $\eta$ twice. As a result, the following estimate holds from the $L^2-L^\infty$ type bilinear estimate (\ref{bilinearesetimate}) in Lemma \ref{multilinearestimate} and estimate (\ref{eqn730}) in Lemma \ref{derivativeL2estimate2} , 
\[
 \sum_{i=1,2}|\widehat{P}_{k,k_1,k_2}^i |   \lesssim  2^{k_1 + 6 k_{1,+}+\delta m }\big(  2^{\tilde{\delta} m -k} + 2^{3\tilde{\delta}  m  } \big) \Big[\big(\sum_{i=0,1,2} 2^{i k_1}\| \nabla_\xi^i \widehat{g}_{k_1}(t, \xi)\|_{L^2} + 2^{i k_1}\| \nabla_\xi^i \widehat{g}_{k_2}(t, \xi)\|_{L^2}\big) 
\]
\[
 \times  \big(\sum_{i=1,2} \| e^{-i t\Lambda} g_{k_i}(t)\|_{L^\infty}   \big)+ \sum_{j_1 \geq j_2} 2^{2k_1 +j_1 } \| e^{-it \Lambda}\mathcal{F}^{-1}\big[ \nabla_\xi \widehat{g}_{k_2,j_2} \big]\|_{L^\infty} \| g_{k_1,j_1}\|_{L^2} +  \sum_{j_2 \geq j_1} 2^{2k_1 +j_2 }
 \]
 \[  \times \| e^{-it \Lambda}\mathcal{F}^{-1}\big[ \nabla_\xi \widehat{g}_{k_1,j_1} \big]\|_{L^\infty} \| g_{k_2,j_2}\|_{L^2} \Big]\lesssim 2^{-m-k-k_1+2\tilde{\delta} m +\delta m }\epsilon_0^2 + 2^{-m + 4\beta m }\epsilon_0^2\lesssim 2^{9\tilde{\delta} m/5   }\epsilon_0^2. 
\]
 In above estimate, we used the fact that $k+k_1\geq -m + \tilde{\delta} m /3.$

Now, we proceed to estimate $\widehat{P }_{k,k_1,k_2}^3$.  Recall (\ref{eqn1090}).  Note that $(k_1',k_2')\in \chi_k^2$, i.e., $|k_1'-k|\leq 10$. Hence the  symbol $\tilde{q}_{+,\nu'}(\xi-\kappa, \kappa)$ contributes the smallness of ``$2^{2k}$''.
 By doing integration by parts in ``$\eta$'' many times, we can rule out the case  when $\max\{j_1,j_2\}\leq m +k_{-}-k_{1,+}- \beta m $.  From the $L^2-L^\infty$ type bilinear estimate (\ref{bilinearesetimate}) in Lemma \ref{multilinearestimate},  the following estimate holds when    $\max\{j_1,j_2\}\geq  m +k_{-}-k_{1,+}- \beta m $
\[
\sum_{k_2'\leq k_1'-10}  \sum_{\max\{j_1,j_2\}\geq m +k_{-}-k_{1,+}- \beta m } | \widehat{P'}_{k,k_1,j_1,k_2,j_2}^{3}| \lesssim  \sum_{k_2'\leq k_1'-10}  \sum_{\max\{j_1,j_2\}\geq m +k_{-}-k_{1,+}- \beta m } 2^{3m+ 3k+3k_1}  \]
\[\times \| LL g_{k_1'}\|_{L^2}\| e^{-it \Lambda} g_{k_2'}\|_{L^\infty}     \big( \sum_{j_1 \geq \max\{j_2, m +k_{-}-k_{1,+}- \beta m  \}} \| e^{-it \Lambda} g_{k_2,j_2}\|_{L^\infty} \| g_{k_1,j_1}\|_{L^2}\]
\[ + \sum_{j_2 \geq \max\{j_1, m +k_{-}-k_{1,+}- \beta m \}}  \| g_{k_2,j_2}\|_{L^2}   \| e^{-it \Lambda} g_{k_1,j_1}\|_{L^\infty} \big) \lesssim 2^{-m/2 +10\beta m }\epsilon_0^2.
\]

Now, we proceed to estimate   $\widehat{P }_{k,k_1,k_2}^4$. Recall (\ref{eqq326}) and (\ref{eqn1101}).  For this case, we do integration by parts in ``$\eta$'' once. As a result, the following estimate holds from estimate (\ref{Z1estimatecubicandquartic}) in Proposition (\ref{eqq298}), estimate (\ref{eqnj878}) in Lemma (\ref{remaindertermweightednorm}), and $L^2-L^\infty$ type bilinear estimate (\ref{bilinearesetimate}) in Lemma \ref{multilinearestimate}, 
\[
| \widehat{P }_{k,k_1,k_2}^4| \lesssim \sup_{t\in [2^{m-1}, 2^m]} 2^{2m +(2-\alpha)k_1 +k_{1,+}} \| \Gamma^1\Gamma^2 g_{k}(t)\|_{L^2} \big(   \| \Lambda_{\geq 3}[\p_t g_{k_1}]\|_{Z_1} \| e^{-it \Lambda} g_{k_2}(t)\|_{L^\infty}
\]
\[+    \| \Lambda_{\geq 3}[\p_t g_{k_2}]\|_{Z_1} \| e^{-it \Lambda} g_{k_1}(t)\|_{L^\infty} \big) \lesssim 2^{-\beta m +2\tilde{\delta}m}\epsilon_0^2\lesssim 2^{-\delta m }\epsilon_0^2.
\]

Lastly, we proceed to estimate   $\widehat{P }_{k,k_1,k_2}^5$. Recall (\ref{eqn1110}) and (\ref{eqn1096}).  We first consider the case when $|k_1'-k_2'|\leq 10$. By doing integration by parts in ``$\sigma $'' many times, we can rule out the case when $\max\{j_1',j_2'\} \leq m +k_{1,-}-k_{1,+}'- \beta m $. By using the $L^2-L^\infty-L^\infty$ type trilinear estimate (\ref{trilinearesetimate}) in Lemma \ref{multilinearestimate}, the following estimate holds when $\max\{j_1',j_2'\} \geq m +k_{1,-}-k_{1,+}'- \beta m $,
\[
  \sum_{|k_1'-k_2'|\leq 10, k_1\leq k_1'+10}\sum_{ \max\{j_1',j_2'\} \geq m +k_{1,-}-k_{1,+}'- \beta m} | \widehat{P'}_{k,k_1,k_2}^{k_1',j_1', k_2',j_2'}| \lesssim  \sum_{|k_1'-k_2'|\leq 10, k_1\leq k_1'+10}  2^{3m+k+3k_1+2k_1'}  \]
\[\times \| LL g_k\|_{L^2}\big(\sum_{i=1,2} \| e^{-it \Lambda} g_{k_i}(t)\|_{L^\infty} \big)\big(\sum_{j_1'\geq \max\{ j_2',   m +k_{1,-}-k_{1,+}'- \beta m\} } \| g_{k_1',j_1'}\|_{L^2} \| e^{-it \Lambda} g_{k_2',j_2'}\|_{L^\infty}   \]
\[ + \sum_{j_2'\geq \max\{ j_1',  m +k_{1,-}-k_{1,+}'- \beta m\} } \| g_{k_2',j_2'}\|_{L^2} \| e^{-it \Lambda} g_{k_1',j_1'}\|_{L^\infty} \big) \lesssim 2^{-m +10\beta m }\epsilon_0^2.
\]
 
 It remains to consider the case when $k_2'\leq k_1'-10$. We split it into four cases based on the size of $k_1'+k_2'$ and whether $k$ is greater than $k_2'$ as follows, 

$\oplus$ If  $k_{1,-}'+k_2'\leq -19m/20$ and $k\leq k_2'+20$.\quad  By using the same strategy that we used in the estimates  (\ref{eqn302}) and (\ref{eqn303}), the following estimate holds from  estimate (\ref{eqn400}) in Lemma \ref{Linftyxi}, 
\[
 \sum_{k_2'\leq k_1'-10, |k_1-k_1'|\leq 20}  | \widehat{P'}_{k,k_1,k_2}^{k_1',  k_2' }| \lesssim  \sum_{k_2'\leq k_1'-10, |k_1-k_1'|\leq 20} 2^{3m+ k+4k_1 }\| LL g_k\|_{L^2}\| g_{k_1'}(t)\|_{L^2}
 \]
\[
\times \big(  \| e^{-it \Lambda} g_{k_1}(t)\|_{L^\infty} + \| e^{-it \Lambda} g_{k_2 }(t)\|_{L^\infty} \big) 
 \big(  2^{k_1'+2k_2'} \|\widehat{\textup{Re}[v]}(t, \xi)\psi_{k_2'}(\xi)\|_{L^\infty_\xi} +  2^{3k_2'} \| \widehat{g}_{k_2'}(t, \xi)\|_{L^\infty_\xi}\big)\]
 \[\lesssim 2^{3m+2\tilde{\delta} m+ 4k_2'+3k_1-15 k_{1,+}} (1+ 2^{m+k_1+k_2'}) \lesssim 2^{-\beta m}\epsilon_0^2.
\]

$\oplus$ If   $k_{1,-}'+k_2'\leq -19m/20$ and $k \geq k_2'+20 $.\quad For the case we are considering, we have $|\sigma|\ll |\xi|\ll |\eta| $. Hence,  the following estimate holds,
\[
|\nabla_\eta\big(\Phi^{\mu,\nu}(\xi, \eta)+\nu(\Phi^{\mu', \nu'})(\eta, \sigma) \big) |+ |\nabla_\eta\big(\Phi^{\mu,\nu}(\xi, \eta)+\mu(\Phi^{\mu', \nu'})(\xi-\eta, \sigma) \big) |
\]  
\be\label{eqn760}
 \gtrsim |\xi-\sigma|(1+|\eta|)^{-1/2}\sim2^{k-k_{1,+}/2}.
\ee
To take advantage of this fact, we do integration by parts in ``$\eta$'' once. As a result, the following estimate holds    from  estimate (\ref{eqn400}) in Lemma \ref{Linftyxi},
\[
 | \widehat{P'}_{k,k_1,k_2}^{k_1',  k_2' }|\lesssim 2^{2m + 3k_1 +k_{1,+}}\| LL g_k\|_{L^2}\big(\sum_{i=1,2} \| e^{-it \Lambda} g_{k_i}(t)\|_{L^\infty} + 2^{k_1} \| e^{-it \Lambda} \mathcal{F}^{-1}[\nabla_\xi \widehat{g_{k_i}}(t, \xi)]\|_{L^\infty} \big)\]
 \[\times \big( \sum_{i=1,2} \| g_{k_1'}(t)\|_{L^2} + 2^{k_1}\| \nabla_\xi \widehat{g}_{k_i}(t, \xi)\|_{L^2}+  2^{k_1}\| \nabla_\xi \widehat{g}_{k_1'}(t, \xi)\|_{L^2}\big)  
\]\[
\times\big(  2^{k_1'+2k_2'} \| \widehat{\textup{Re}[v]}(t, \xi)\psi_{k_2'}(\xi)\|_{L^\infty_\xi} +  2^{3k_2'} \| \widehat{g}_{k_2'}(t, \xi)\|_{L^\infty_\xi}\big)\lesssim 2^{-\beta m}\epsilon_0^2.
\]

$\oplus$ If  $k_{1,-}'+k_2'\geq -19m/20$ and $k\leq k_2'+20$.\quad By doing integration by parts in ``$\sigma$'' many times, we can rule out the case when $\max\{j_1',j_2'\}\leq m +k_{1,-}- \beta m$. From the $L^2-L^\infty-L^\infty$ type trilinear estimate (\ref{trilinearesetimate}) in Lemma \ref{multilinearestimate}, the following estimate holds when $\max\{j_1',j_2'\}\geq m +k_{1,-}- \beta m$, 
\[
\sum_{ \max\{j_1',j_2'\}\geq  m +k_{1,-}- \beta m }| \widehat{P'}_{k,k_1,k_2}^{k_1',j_1', k_2',j_2'}| \lesssim 2^{3m+k+3k_1+2k_1'} \| LL g_k\|_{L^2}\big(\sum_{i=1,2} \| e^{-it \Lambda} g_{k_i}(t)\|_{L^\infty} \big)  \]
\[\times \big(\sum_{j_1'\geq \max\{ j_2', m +k_{1,-} - \beta m\} } \| g_{k_1',j_1'}\|_{L^2} \| e^{-it \Lambda} g_{k_2',j_2'}\|_{L^\infty}  + \sum_{j_2'\geq \max\{ j_1', m  +k_{1,-}- \beta m\} } \| g_{k_2',j_2'}\|_{L^2} \]
\[\times \| e^{-it \Lambda} g_{k_1',j_1'}\|_{L^\infty} \big) \lesssim 2^{-m -k_2'+ 10\beta m  }\epsilon_0^2\lesssim 2^{-\beta  m }\epsilon_0^2.
\]

$\oplus$ If  $k_{1,-}'+k_2'\geq -19m/20$ and $k\geq k_2'+20$.\quad   By doing integration by parts in ``$\sigma$'' many times, we can rule out the case when $\max\{j_1',j_2'\}\leq m +k_{1,-}-\beta m$.
Now, it remains to consider the case when $\max\{j_1',j_2'\}\geq  m +k_{1,-}-\beta m$.  As $k\geq k_2'+20$, then estimate (\ref{eqn760}) still holds.For this case, we do integration by parts in ``$\eta$'' once. As a result, the following estimate holds from the $L^2-L^\infty-L^\infty$ type estimate, 
\[
\sum_{ \max\{j_1',j_2'\}\geq m +k_{1,-}-\beta m}| \widehat{P'}_{k,k_1,k_2}^{k_1',j_1', k_2',j_2'}| \lesssim 2^{2m + 4k_1}\| LL g_k\|_{L^2}\big(\sum_{i=1,2}  2^{k_1} \| e^{-it \Lambda} \mathcal{F}^{-1}[\nabla_\xi \widehat{g_{k_i}}(t, \xi)]\|_{L^\infty}   \]
\[+ \| e^{-it \Lambda} g_{k_i}(t)\|_{L^\infty}\big)\big( \sum_{j_1'\geq \max\{j_2', m+k_{1,-} - \beta m \}} 2^{j_1'} \| g_{k_1',j_1'}\|_{L^2} \| e^{-it \Lambda} g_{k_2',j_2'}\|_{L^\infty} 
 \]
 \[ +\sum_{j_2'\geq \max\{j_1', m +k_{1,-}- \beta m\}} 2^{-m+2j_1'} \|  g_{k_1',j_1'}(t)\|_{L^2} \| g_{k_2',j_2'}\|_{L^2}\big)  \lesssim  2^{-m/2+10\beta m }\epsilon_0^2 + 2^{-m -k_2' }\epsilon_0^2\lesssim 2^{-\beta m}\epsilon_0^2.
\] 
\end{proof}
\subsection{The $Z_2$-norm estimate of the quadratic terms: when  $|k_1-k_2|\leq 10$ and $|k-k_1|\leq 10 $.}   
 For this case, the summability with respect to $k$ is no longer a issue, as $|k-k_1|\leq 10$ and we can gain the smallness of $2^{2k_1}$ from the symbol. Hence, it is not necessary any more to use the good decomposition of $(\hat{L}_\xi + \hat{L}_\eta) \Phi^{\mu, \nu}(\xi,\eta).$ Moreover, recall that a small neighborhood of $(\xi, \xi/2)$ is removed, there is no issue to do integration by parts in ``$\eta$''.

  The estimate of $ P_{k,k_1,k_2}^3$ is same as what we did in (\ref{eqn998}). The estimate of $ P_{k,k_1,k_2}^4$ is same as what we did in (\ref{eqn999}). The estimate of $ P_{k,k_1,k_2}^1$ is same as what we did in the estimate of $\Gamma_{k,k_1,k_2}^{1,2}$ in (\ref{eqn112900}).  The estimate of $ P_{k,k_1,k_2}^2$ is same as what we did in the estimate of $\widetilde{P }_{k,k_1,k_2}^2 $ in (\ref{eq10}). We omit details here.

\subsection{The $Z_2$-norm estimate of the quadratic terms: when $k_2\leq k_1-10$.}   
\emph{  Note that, for the case we are considering, we have $\mu=+$  (see  (\ref{eqn900}))}.  We first rule out 
the very high frequency case and very low frequency case. For both cases, we  use the formulation  (\ref{eqn711}).

  We first consider the case when $k_{1 }+k_2\leq -19m/20$.  By using the same strategy that we used in the estimates of (\ref{eqn302}) and (\ref{eqn303}), the following estimate holds from estimate (\ref{eqn400}) in Lemma \ref{Linftyxi}  
\[
\big|  \int_{t_1}^{t_2}  \int_{\R^2}\overline{\Gamma^1_\xi \Gamma^2_\xi \widehat{g}(t, \xi )} \Gamma^1_\xi \Gamma^2_\xi  B^{+, \nu}_{k,k_1,k_2}(t,\xi)  d \xi d t \big| \lesssim \sup_{t\in[2^{m-1}, 2^m]} \| \Gamma_1 \Gamma_2 g_k(t)\|_{L^2}\big(\sum_{i=0,1,2} 2^{i k_1} \| \nabla_\xi^{i}   \widehat{g}_{k_1}(t, \xi)\|_{L^2} \big) 
\]
\[
\times 2^{m+k_1}\big( 1+ 2^{2m+2k_2+2k_1} \big)    \min\big\{  2^{k_1 +2k_2 } \| \widehat{\textup{Re}[v]}_{ }(t, \xi)\psi_{k_2}(\xi)\|_{L^\infty_\xi} 
+  2^{3k_2 } \| \widehat{g}_{k_2  }(t, \xi)\|_{L^\infty_\xi} , 2^{k_1+k_2}\| g_{k_2}(t)\|_{L^2}\big\}\]
\[\lesssim 2^{3\tilde{\delta} m} \min\{2^{ m + 2k_1+ k_2}(1 + 2^{2m+2k_1+2k_2}), 2^{2m+k_1+3k_2}( 1 + 2^{3m+3k_1+3k_2}) \} \lesssim     2^{-\beta m }\epsilon_0^2.
\]

Now, we remove the case when $k_1$ is relatively big. More precisely, we consider the case when  $k_1\geq 5\beta m $ and $k_1+k_2\geq -19m/20$. Recall (\ref{eqn711}). Note that $ {\Gamma}_\xi \widehat{g_{k_1}}(t,\xi-\eta)=  -\xi_{\Gamma}\cdot \nabla_\eta \widehat{g_{k_1}}(t,\xi-\eta)$.  When $ {\Gamma}_\xi$ hits $\widehat{g_{k_1}}(t,\xi-\eta)$,  we do integration by parts in ``$\eta$ ''to move around the derivative $\nabla_\eta$ in front of $\widehat{g_{k_1}}(t,\xi-\eta)$. As a result, the following estimate holds from the $L^2-L^\infty$ type bilinear estimate,
\[
\sum_{k_1\geq 5\beta m, k_2\geq -m-k_1}\big|  \int_{t_1}^{t_2}  \int_{\R^2}\overline{\Gamma^1_\xi \Gamma^2_\xi \widehat{g}(t, \xi )} \Gamma^1_\xi \Gamma^2_\xi  B^{+, \nu}_{k,k_1,k_2}(t,\xi)  d \xi d t \big|\lesssim \sum_{k_1\geq 5\beta m, k_2\geq -m-k_1} \sup_{t\in[2^{m-1}, 2^m]}  \]
\[\times \| \Gamma_1 \Gamma_2 g_k(t)\|_{L^2} \| g_{k_1}(t)\|_{L^2}2^{k_2}\big(2^{-2k_2} \| g_{k_2}\|_{L^2}+2^{- k_2} \| \nabla_\xi \widehat{g}_{k_2}(t)\|_{L^2}
\]
\be\label{e789}
  + \| \nabla_\xi^2\widehat{g}_{k_2}(t)\|_{L^2} \big) \lesssim \sum_{k_1\geq 5\beta m, k_2\geq -m-k_1} 2^{3m +\beta m +4k_1-k_2 -(N_0-30)k_{1,+}} \epsilon_0 \lesssim 2^{-\beta m}\epsilon_0. 
\ee

Hence, for the rest of this subsection,  we restrict ourself to   \emph{the case when $k_{1 }+k_2\geq -19m/20$ and $k_1\leq 5\beta m$}. The symmetric structure inside (\ref{waterwaves}) is very essential for this case. To utilize symmetry, we use the formulation (\ref{eqn1006}).   Although all terms inside $P_{k,k_1,k_2}^i$, $i\in\{1,2,4\},$  can still be handled by the same method, terms inside $P_{k,k_1,k_2}^3$ (see (\ref{eqn1012})) cannot be treated in the same way since now $k_1$ and $k_2$ are not close to each other. We decompose $P_{k,k_1,k_2}^3  $ into three parts as follows, 
\be\label{eqn1170}
P_{k,k_1,k_2}^3 = \sum_{i=1,2,3} Q_{k,k_1,k_2}^i,
\ee
where
\be\label{eqn1300}
Q_{k,k_1,k_2}^1= \int_{t_1}^{t_2} \int_{\R^2} \int_{\R^2 } \overline{\widehat{\Gamma^1 \Gamma^2 g_k}(t, \xi)} e^{i t\Phi^{+, \nu}(\xi, \eta)}     \tilde{q}_{+, \nu}(\xi-\eta, \eta) \widehat{  \Gamma^1   \Gamma^2 g_{k_1} }(t, \xi-\eta)\widehat{g_{k_2}^\nu}(t, \eta) d \eta d \xi d t, 
\ee
\be\label{eqn1230}
Q_{k,k_1,k_2}^2= \int_{t_1}^{t_2} \int_{\R^2} \int_{\R^2 } \overline{\widehat{\Gamma^1 \Gamma^2 g_k}(t, \xi)} e^{i t\Phi^{+, \nu}(\xi, \eta)}     \tilde{q}_{+, \nu}(\xi-\eta, \eta) \widehat{  g_{k_1} }(t, \xi-\eta)\widehat{ \Gamma^1   \Gamma^2 g_{k_2}^\nu}(t, \eta) d \eta d \xi d t, 
\ee
\[
Q_{k,k_1,k_2}^3= \sum_{j_1\geq -k_{1,-},j_2\geq -k_{2,-}}Q_{k,k_1,k_2}^{j_1,j_2, 3}, \quad Q_{k,k_1,k_2}^{j_1,j_2, 3}:=  \int_{t_1}^{t_2} \int_{\R^2} \int_{\R^2 } \overline{\widehat{\Gamma^1 \Gamma^2 g_k}(t, \xi)} e^{i t\Phi^{+, \nu}(\xi, \eta)}\]
\[\times \big[ \sum_{\{l,n\}=\{1,2\}}(\Gamma_\xi^l +\Gamma_\eta^l + d_{\Gamma^l} )\tilde{q}_{+, \nu}(\xi-\eta, \eta)\big(\widehat{ \Gamma^n g_{k_1,j_1}^{ }}(t, \xi-\eta) \widehat{  g_{k_2,j_2}^\nu}(t, \eta) + \widehat{ g_{k_1,j_1}^{ }}(t, \xi-\eta) \widehat{ \Gamma^n  g_{k_2,j_2}^\nu}(t, \eta) \big)   \]
 \be\label{eqn1160}
 +    (\Gamma_\xi^1 +\Gamma_\eta^1 + d_{\Gamma^1} ) (\Gamma_\xi^2 +\Gamma_\eta^2 + d_{\Gamma^2} )\tilde{q}_{ +, \nu}(\xi-\eta, \eta) \widehat{  g_{k_1,j_1}^{ }}(t, \xi-\eta)   \widehat{  g_{k_2,j_2}^\nu}(t, \eta)  \big] d \eta d\xi d t.
\ee

\subsubsection{The estimate of $P_{k,k_1,k_2}^1$}
Recall (\ref{eqn1010}) and  (\ref{eqn989}). Same as in the High $\times$ High type interaction, we know that the integral inside $P_{k,k_1,k_2}^1$ vanishes if $\Gamma^l =\Omega.$ \emph{Hence, we only have to consider the case when $\Gamma^l=L$}. Recall  (\ref{eqn1301}). We know that  similar decompositions as in (\ref{eqn1019}) and (\ref{eqn1120}) also hold. To simplify notations, we still the notations used in subsubsection \ref{highhighinteraction1}.   Hence, to estimate $P_{k,k_1,k_2}^1$, it is sufficient to estimate $\Gamma_{k,k_1,k_2}^{1,1}$ and 
 $\Gamma_{k,k_1,k_2}^{1,2}$. 

We first consider $\Gamma_{k,k_1,k_2}^{1,2}$, which is relatively easier. Since the symbol of  $\Gamma_{k,k_1,k_2}^{1,2}$ contributes the smallness of $|\eta|^2$, which makes the decay rate of the input $g_{k_2}(t)$ sharp when it is putted in $L^\infty$. Hence, a simple integration by parts in $\eta$ is sufficient. More precisely, we have 
 \begin{lemma}
Under the bootstrap assumption \textup{(\ref{smallness})}, the following estimate holds,
\be\label{eqn1129001}
 \sum_{k_2\leq k_1-10, |k_1-k |\leq 10}\sum_{\Gamma\in\{L, \Omega\}} \big|  \Gamma_{k,k_1,,k_2}^{1,2}   \big| \lesssim  2^{2\tilde{\delta} m  }\epsilon_0^2. 
\ee
\end{lemma}
\begin{proof}
Recall (\ref{eqn1120}) and (\ref{eqn1301}).  From  Lemma \ref{Snorm}, the following estimate holds, 
\be\label{eqn112109}
\| \tilde{q}_{+, \nu}^2(\xi , \eta) \psi_{k}(\xi) \psi_{k_1}(\xi-\eta)\psi_{k_2}(\eta)\|_{\mathcal{S}^\infty}\lesssim 2^{2k_2},\quad k_2\leq k_1-10.
\ee
After doing integration by parts in ``$\eta$'' once, the following decomposition holds, 
\[
| \Gamma_{k,k_1,k_2}^{1,2} | \lesssim | \Gamma_{k,k_1,k_2}^{1,2;1} |+ | \Gamma_{k,k_1,k_2}^{1,2;2} |,\quad
\]
where
\[
\Gamma_{k,k_1,,k_2}^{1,2;1} :=  
  \int_{t_1}^{t_2} \int_{\R^2} \int_{\R^2 } \overline{\widehat{\Gamma^1 \Gamma^2 g_k}(t, \xi)} e^{i t\Phi^{+, \nu}(\xi, \eta)} \nabla_\eta \cdot \Big( \frac{\nabla_\eta \Phi^{+, \nu}(\xi, \eta)}{|\nabla_\eta \Phi^{+, \nu}(\xi, \eta) |^2}  \tilde{q}^{2}_{+, \nu} (\xi-\eta, \eta)  \big[\tilde{q}_{+, \nu}(\xi-\eta, \eta)  \]
\[
\times    \big(  \widehat{\Gamma  g_{k_1}^{}}(t, \xi-\eta) \widehat{g_{k_2}^\nu}(t, \eta)+   \widehat{g_{k_1}^{}}(t, \xi-\eta) \widehat{\Gamma  g_{k_2}^\nu}(t, \eta)\big)+ (\Gamma_\xi  +\Gamma_\eta + d_\Gamma )\tilde{q}_{+, \nu}(\xi-\eta, \eta) \widehat{   g_{k_1}^{}}(t, \xi-\eta)\widehat{g_{k_2}^\nu}(t, \eta) \big]\Big)  \]
\[- \overline{\widehat{\Gamma^1 \Gamma^2 g_k}(t, \xi)} e^{i t\Phi^{+, \nu}(\xi, \eta)} \nabla_\eta \cdot \Big( \frac{\nabla_\eta \Phi^{+, \nu}(\xi, \eta)}{|\nabla_\eta \Phi^{+, \nu}(\xi, \eta) |^2} \tilde{q}^{2}_{+, \nu} (\xi-\eta, \eta)   \tilde{q}_{+, \nu}(\xi-\eta, \eta)  \widehat{g_{k_2}^\nu}(t, \eta) \Big) \widehat{\Gamma  g_{k_1}^{}}(t, \xi-\eta)  d \eta d \xi d t, 
\]
 \[
 \Gamma_{k,k_1,k_2}^{1,2;2} := \sum_{j_1\geq -k_{1,-}, j_2\geq -k_{2,-}} \Gamma_{k,k_1,j_1,k_2,j_2}^{1,2;2}, \quad \Gamma_{k,k_1,j_1,k_2,j_2}^{1,2;2} :=    \int_{t_1}^{t_2} \int_{\R^2} \int_{\R^2 } \overline{\widehat{\Gamma^1 \Gamma^2 g_k}(t, \xi)} e^{i t\Phi^{+, \nu}(\xi, \eta)}\]
\[  \times  \widehat{\Gamma  g_{k_1,j_1}^{}}(t, \xi-\eta)   \nabla_\eta \cdot \Big( \frac{\nabla_\eta \Phi^{+, \nu}(\xi, \eta)}{|\nabla_\eta \Phi^{+, \nu}(\xi, \eta) |^2} \tilde{q}^{2}_{+, \nu} (\xi-\eta, \eta)\tilde{q}_{+, \nu}(\xi-\eta, \eta)  \widehat{g_{k_2,j_2}^\nu}(t, \eta) \Big) d \eta d \xi d t.
\]
From the $L^2-L^\infty$ type bilinear estimate, the following estimate holds,
\[
 \sum_{k_2\leq k_1-10, |k_1-k |\leq 10} | \Gamma_{k,k_1,,k_2}^{1,2;1} | \lesssim  \sum_{k_2\leq k_1-10, |k_1-k |\leq 10} \sum_{i=1,2} 2^{m+2k_2}\| \Gamma^1 \Gamma^2 g_{k}\|_{L^2}
\]
\[
 \times \Big[ \big( 2^{2k_1} \| \nabla_\xi^2 \widehat{g}_{k_1}(t, \xi)\|_{L^2} + 2^{k_1} \| \nabla_\xi \widehat{g}_{k_1}(t, \xi)\|_{L^2}  )     \| e^{-it \Lambda} g_{k_2}\|_{L^\infty}+ 2^{   k_1 }  \| e^{-it \Lambda} g_{k_1}\|_{L^\infty} \big( 2^{k_2}\| \nabla_\xi^2 \widehat{g}_{k_2}(t, \xi)\|_{L^2} 
 \]
 \[ + \| \nabla_\xi  \widehat{g}_{k_2}(t, \xi)\|_{L^2} +2^{-k_2} \|    {g}_{k_2}(t )\|_{L^2}\big)+ \sum_{j_1\geq j_2} 2^{ -m+k_2+k_1} 2^{2j_2}\| g_{k_2,j_2}\|_{L^2} 2^{j_1}\| g_{k_1,j_1}\|_{L^2}   \]
  \[
 +\sum_{j_2\geq j_1} 2^{-m+ k_2+k_1}  2^{2j_1}  \| g_{k_2,j_2}\|_{L^2}  2^{j_2}\| g_{k_1,j_1}\|_{L^2}\Big] \lesssim 2^{2\tilde{\delta}m}\epsilon_0^2.
\]
Now, we proceed to estimate  $\Gamma_{k,k_1,,k_2}^{1,2;2}$. By doing integration by parts in $\eta$ many times, we can rule out the case when $\max\{j_1,j_2\}\leq m + k_{1,-}-\beta m $. From the $L^2-L^\infty$ type bilinear estimate, the following estimate holds when $\max\{j_1,j_2\}\geq m + k_{1,-}-\beta m $,
\[
\sum_{  \max\{j_1,j_2\}\geq m + k_{1,-}-\beta m  } \big| \Gamma_{k,k_1,j_1,k_2,j_2}^{1,2;2}  \big|\lesssim  2^{m+2k_2 +k_1} \| \Gamma^1 \Gamma^2 g_{k}\|_{L^2} \big[\sum_{j_1\geq \max\{j_2, m+k_{1,-}-\beta m \}} 2^{k_1+j_1} \]
\[\times \| g_{k_1,j_1}(t)\|_{L^2} \big(2^{-k_2}\| e^{-it\Lambda} g_{k_2,j_2}\|_{L^\infty}  + \| e^{-it\Lambda}\mathcal{F}^{-1}[\nabla_\xi g_{k_2,j_2}(t)]\|_{L^\infty}\big) +  \sum_{j_2\geq \max\{j_1, m+k_{1,-}-\beta m \}} 2^{ j_2} 
 \]
 \[
 \times \| g_{k_2,j_2}(t)\|_{L^2} \| e^{-it\Lambda} \Gamma g_{k_1,j_1}(t)\|_{L^\infty}\big] \lesssim 2^{-m/2	+20\beta m }\epsilon_0^2.
 \]
 Hence finishing the proof.
\end{proof}

Now, we proceed to consider $\Gamma_{k,k_1,k_2}^{1,1}$. Same as in the High  $\times$ High interaction, we do integration by parts in time once. As a result, we have the same formulations as in (\ref{eqn1425}), (\ref{eqn1025}) and (\ref{eqn1026}).
\begin{lemma}
Under the bootstrap assumption \textup{(\ref{smallness})}, the following estimate holds, 
\be\label{eqn1434}
 |\widetilde{\Gamma}_{k,k_1,k_2}^{1,1}  |\lesssim 2^{ -\beta m } \epsilon_0^2.
\ee
\end{lemma}
\begin{proof}

Recall  (\ref{eqn1425}). By doing integration by parts in ``$\eta$'' many times, we can rule out the case when $\max\{j_1,j_2\}\leq m+k_{1,-}- \beta m$. From the $L^2-L^\infty$ type bilinear estimate (\ref{bilinearesetimate}) in Lemma \ref{multilinearestimate}, the following estimate holds when $\max\{j_1,j_2\} \geq m +k_{1,-}- \beta m $.
\[
\sum_{\max\{j_1,j_2\} \geq m +k_{1,-}- \beta m } | \widetilde{\Gamma}_{k,k_1,k_2}^{j_1,j_2,1,1} | \lesssim 2^{m+2k_1} \| \Gamma^1\Gamma^2 g_k\|_{L^2}\big[ \sum_{j_1 \geq \max\{j_2, m +k_{1,-}- \beta m  \}} 2^{k_1+j_1 } \| g_{k_1,j_1}\|_{L^2}\]
\[\times  \big( \| e^{-it \Lambda} g_{k_2,j_2}\|_{L^\infty} + 2^{k_2}\| e^{-it \Lambda}\mathcal{F}^{-1}[\nabla_\xi \widehat{g}_{k_2,j_2} (t)]\|_{L^\infty} \big)+ \sum_{j_2 \geq \max\{j_1, m +k_{1,-}- \beta m  \}} 2^{ k_2+j_2}  \| g_{k_2,j_2}\|_{L^2} 
\]
\[
\times(\| e^{-it \Lambda} g_{k_1,j_1}\|_{L^\infty}+2^{k_1}\| e^{-it \Lambda}\mathcal{F}^{-1}[\nabla_\xi \widehat{g}_{k_1,j_1} (t)]\|_{L^\infty} )  \big]\lesssim 2^{-m-k_2+20\beta m  }\epsilon_0^2\lesssim 2^{ -\beta m } \epsilon_0^2.
\]
\end{proof}

Now, we proceed to estimate $\widetilde{\Gamma}_{k,k_1,k_2}^{1,2} $. Recall (\ref{eqn1026}). Since now $k_1$ and $k_2$ are not comparable, different from the decomposition we did in (\ref{eqn1467}) in the High $\times$ High type interaction, we do decomposition as follows, 
 \be\label{eqn811}
\widetilde{\Gamma}_{k,k_1,k_2}^{1,2}= \sum_{i=1}^7  \widetilde{\Gamma}_{k,k_1,k_2}^{1,2;i }
\ee
\be\label{eqnnn811}
\widetilde{\Gamma}_{k,k_1,k_2}^{1,2;2 }=\sum_{k_2'\leq k_1'+10} \widehat{\Gamma}_{k,k_1,k_2 }^{  k_1' , k_2' ,1} , \quad   \widehat{\Gamma}_{k,k_1,k_2 }^{  k_1' , k_2' ,1} =\sum_{j_2\geq -k_{2,-}, j_1'\geq-k_{1,-}',j_2'\geq-k_{2,-}' }\widehat{\Gamma}_{k,k_1,k_2,j_2}^{  k_1',j_1', k_2',j_2',1},
\ee
\be\label{eqn802}
\widetilde{\Gamma}_{k,k_1,k_2}^{1,2;3 }=\sum_{k_2'\leq k_1'+10} \widehat{\Gamma}_{k,k_1,k_2 }^{  k_1' , k_2' ,2} , \quad \widehat{\Gamma}_{k,k_1,k_2 }^{  k_1' , k_2' ,2} =\sum_{j_1\geq -k_{1,-}, j_1'\geq-k_{1,-}',j_2'\geq-k_{2,-}' }\widehat{\Gamma}_{k,k_1, j_1, k_2 }^{  k_1',j_1', k_2',j_2',2}, 
 \ee
\[
 \widetilde{\Gamma}_{k,k_1,k_2}^{1,2;i}= \sum_{j_1\geq -k_{1,-}, j_2\geq -k_{2,-}}  \widetilde{\Gamma}_{k,k_1,j_1,k_2,j_2}^{ 1,2;i},\quad i\in\{4,5\},
\]

where
\[  \widetilde{\Gamma}_{k,k_1,k_2}^{1,2;1}: =   -  \int_{t_1}^{t_2} \int_{\R^2} \int_{\R^2 } \overline{\widehat{\Gamma^1 \Gamma^2 g_k}(t, \xi)}  e^{i t\Phi^{+, \nu}(\xi, \eta)} \tilde{c}(\xi-\eta)  t\p_t \widehat{g_{k_2}^\nu}(t, \eta)  \big( \tilde{q}_{+, \nu}(\xi, \eta)      \widehat{\Gamma  g_{k_1}}(t, \xi-\eta) 
\]
\be\label{eqn1497}
  + (\Gamma_\xi +\Gamma_\eta+ d_\Gamma )\tilde{q}_{+, \nu}(\xi, \eta) 
   \widehat{   g_{k_1}}(t, \xi-\eta) \big) d\eta d \xi d t,
\ee
which is resulted from the case  when $\p_t $ hits the input ``$\widehat{g_{k_2}}(t, \xi-\eta)$'' in  (\ref{eqn1026}),
\[ \widehat{\Gamma}_{k,k_1,k_2,j_2}^{ k_1',j_1', k_2',j_2',1}:=\sum_{\mu',\nu'\in\{+,-\}}  - \int_{t_1}^{t_2} \int_{\R^2} \int_{\R^2 } \overline{\widehat{\Gamma^1 \Gamma^2 g_k}(t, \xi)}  e^{i t\Phi^{+, \nu}(\xi, \eta)} t \tilde{c}(\xi-\eta)  e^{it \Phi^{\mu', \nu'}(\xi-\eta, \sigma)}    \]
 \[\times  \tilde{q}_{\mu', \nu'}(\xi-\eta-\sigma, \sigma) \psi_{k_1}(\xi-\eta) \widehat{g^{\mu'}_{k_1',j_1'}}(t, \xi-\eta-\sigma) \widehat{g^{\nu'}_{k_2',j_2'}}(t, \sigma)  \big( \tilde{q}_{+, \nu}(\xi-\eta, \eta)  \widehat{\Gamma g_{k_2,j_2}^\nu}(t, \eta) \]
 \be\label{eqn805}
+  (\Gamma_\xi +\Gamma_\eta+ d_\Gamma )\tilde{q}_{+, \nu}(\xi-\eta, \eta)  \widehat{  g_{k_2,j_2}^\nu}(t, \eta)\big)
  d\eta d \xi d t,
\ee
which is resulted from the quartic terms when $\p_t $ hits the input ``$\widehat{g_{k_1}}(t, \xi-\eta)$'' in  (\ref{eqn1026}),
\[
\widehat{\Gamma}_{k,k_1,j_1,k_2}^{ k_1',j_1',k_2',j_2',2}:=  \sum_{\mu', \nu'\in \{+,-\}} - \int_{t_1}^{t_2} \int_{\R^2} \int_{\R^2 } \int_{\R^2 }\overline{\widehat{\Gamma^1 \Gamma^2 g_k}(t, \xi)}  e^{i t\Phi^{+, \nu}(\xi, \eta)} \tilde{c}(\xi-\eta)  t \tilde{q}_{+, \nu}(\xi-\eta, \eta)\widehat{g_{k_1,j_1} }(t, \xi-\eta)
\]
\[
  \times   e^{it\Phi^{\mu',\nu'}(\eta, \sigma)} \big[   \Gamma_\eta \big(\tilde{q}_{\mu', \nu'}(\eta-\sigma, \sigma)\widehat{g^{\mu'}_{k_1',j_1'}}(t, \eta-\sigma)\big) + it  \Gamma_\eta  \Phi^{\mu',\nu'}(\eta, \sigma)\tilde{q}_{\mu', \nu'}(\eta-\sigma, \sigma) \widehat{g^{\mu'}_{k_1',j_1'}}(t, \eta-\sigma)  \big]\]
  \be\label{eqn809}
  \times     \widehat{g^{\nu'}_{k_2',j_2'}}(t, \sigma) d \sigma d\eta d \xi d t.
\ee
 which is resulted from the quartic terms when $\p_t $ hits the input ``$\widehat{\Gamma g_{k_2}}(t,  \eta)$'' in  (\ref{eqn1026}),
 \[
\widetilde{\Gamma}_{k,k_1,j_1,k_2,j_2}^{ 1,2;4}:= -\int_{t_1}^{t_2} \int_{\R^2} \int_{\R^2 } \overline{\widehat{\Gamma^1 \Gamma^2 g_k}(t, \xi)} e^{i t\Phi^{+, \nu}(\xi, \eta)}  t   \tilde{c}(\xi-\eta)     \big[\big(       \widehat{\Gamma  \Lambda_{\geq 3}[\p_t  g }]_{k_1,j_1}(t, \xi-\eta)\widehat{g_{k_2,j_2}^\nu}(t, \eta) \]
\[
  +    \widehat{g_{k_1,j_1} }(t, \xi-\eta) \widehat{\Gamma  \Lambda_{\geq 3}[\p_t  g^{\nu}]_{k_2,j_2}}(t, \eta)+  \widehat{\Lambda_{\geq 3}[\p_t g ]_{k_1,j_1}}(t, \xi-\eta) \widehat{\Gamma  g^{\nu}_{k_2,j_2}}(t, \eta) \big)\tilde{q}_{+, \nu}(\xi-\eta, \eta)    \]
\be\label{eqq546}
+   (\Gamma_\xi +\Gamma_\eta+ d_\Gamma ) \tilde{q}_{+, \nu}(\xi-\eta, \eta)    \widehat{ \Lambda_{\geq 3}[\p_t  g ]_{k_1,j_1}}(t, \xi-\eta)\widehat{g_{k_2,j_2}^{\nu}}(t,  \eta)  \big)      d \eta d \xi d t,
\ee
 which is resulted from the quintic and higher order terms when $\p_t $ hits the inputs ``$g_{k_1}(t)$'', ``$\Gamma g_{k_1}(t)$'', and  ``$\Gamma g_{k_2}(t)$'' in  (\ref{eqn1026}),
\[ 
\widetilde{\Gamma}_{k,k_1,j_1,k_2,j_2}^{1, 2;5} =     -\int_{t_1}^{t_2} \int_{\R^2} \int_{\R^2 }   e^{i t\Phi^{+, \nu}(\xi, \eta)} t\tilde{c}(\xi-\eta) \overline{  \big(\p_t \widehat{\Gamma^1 \Gamma^2 g_k}(t, \xi) -   \sum_{\nu\in\{+,-\}} \sum_{(k_1',k_2')\in \chi_k^2} \widetilde{B}^{+, \nu}_{k,k_1',k_2'}(t, \xi) \big) }
\]
\[  
\times       \big(   \tilde{q}_{+, \nu}(\xi-\eta, \eta)  (    \widehat{  g_{k_1,j_1}}(t, \xi-\eta) \widehat{\Gamma g_{k_2,j_2}^\nu}(t, \eta)+      \widehat{\Gamma  g_{k_1,j_1}}(t, \xi-\eta) \widehat{g_{k_2,j_2}^\nu}(t, \eta)) + (\Gamma_\xi + \Gamma_\eta) \tilde{q}_{+, \nu}(\xi-\eta, \eta) 
\]
 \be\label{eqn812}
 \times \widehat{   g_{k_1,j_1}}(t, \xi-\eta) \widehat{g_{k_2,j_2}^\nu}(t, \eta) \big) d \eta d \xi d t ,
 \ee
which is resulted from the good error terms when $\p_t $ hits  ``$\Gamma^1 \Gamma^2 g_{k }(t)$''   in  (\ref{eqn1026}),
\be\label{eqn820}
\widetilde{\Gamma}_{k,k_1,k_2}^{1,2;6}= \sum_{|k_1'-k_2'|\leq 10}   \widehat{\Gamma}_{k,k_1,k_2;1 }^{ k_1' ,k_2' ,3} + \sum_{k_2'\leq k_1'-10} \widehat{\Gamma}_{k,k_1,k_2;2 }^{ k_1' ,k_2' ,3},
\ee
\be\label{eq399}
   \widehat{\Gamma}_{k,k_1,k_2;i}^{ k_1' ,k_2' ,3}= \sum_{j_1'\geq -k_{1,-}', j_2'\geq -k_{2,-}', j_2\geq -k_{2,-}}  \widehat{\Gamma}_{k,k_1,k_2,j_2;i}^{ k_1',j_1',k_2',j_2',3},\quad i \in \{1,2\},
\ee
\[  \widehat{\Gamma}_{k,k_1,k_2,j_2;1}^{ k_1',j_1',k_2',j_2',3}:= \sum_{\mu', \nu'\in \{+,-\}} - \int_{t_1}^{t_2} \int_{\R^2} \int_{\R^2 }   \overline{\widehat{\Gamma^1 \Gamma^2 g_k}(t, \xi)} e^{i t\Phi^{+, \nu}(\xi, \eta)}\tilde{c}(\xi-\eta)  t \tilde{q}_{+, \nu}(\xi-\eta, \eta)  \widehat{g^{\nu}_{k_2,j_2}} (t, \eta) \]
\[\times \psi_{k_1}(\xi-\eta)    e^{ i t\Phi^{\mu', \nu'}(\xi-\eta, \sigma)}\widehat{g^{\nu'}_{k_2',j_2'}}(t, \sigma) \Big[it\Gamma_{\xi-\eta }\Phi^{\mu', \nu'}(\xi-\eta, \sigma)  \tilde{q}_{\mu', \nu'}(\xi-\eta-\sigma, \sigma)   \]
 \be\label{eq823}
 \times \widehat{g^{\mu'}_{k_1',j_1'}}(t, \xi-\eta-\sigma)  + \Gamma_{\xi-\eta} \big( \tilde{q}_{\mu', \nu'}(\xi-\eta-\sigma, \sigma)  \widehat{g^{\mu'}_{k_1',j_1'}}(t, \xi-\eta-\sigma)\big)  \Big]  d \sigma d \eta d \xi d t ,
\ee
 which is resulted from the quartic terms when $\p_t $ hits the input ``$\widehat{\Gamma g_{k_1}}(t, \xi- \eta)$'' in  (\ref{eqn1026}) and moreover two inputs inside $\Lambda_{2}[ \p_t \widehat{\Gamma g_{k_1}}(t,  \xi-\eta)]$ have comparable sizes of frequencies, 
\[  \widehat{\Gamma}_{k,k_1,k_2,j_2;2}^{ k_1',j_1',k_2',j_2',3}:= \sum_{ \nu'\in \{+,-\}} - \int_{t_1}^{t_2} \int_{\R^2} \int_{\R^2 }   \overline{\widehat{\Gamma^1 \Gamma^2 g_k}(t, \xi)} e^{i t\Phi^{+, \nu}(\xi, \eta)}\tilde{c}(\xi-\eta)  t \tilde{q}_{+, \nu}(\xi-\eta, \eta)  \widehat{g^{\nu}_{k_2,j_2}} (t, \eta) \]
\[\times \psi_{k_1}(\xi-\eta)    e^{ i t\Phi^{+, \nu'}(\xi-\eta, \sigma)}\widehat{g^{\nu'}_{k_2',j_2'}}(t, \sigma) \Big[it\Gamma_{\xi-\eta} \Phi^{+, \nu'}(\xi-\eta, \sigma)  \tilde{q}_{+, \nu'}(\xi-\eta-\sigma, \sigma) \widehat{g^{ }_{k_1',j_1'}}(t, \xi-\eta-\sigma)  \]
 \be\label{eqn823}
   + \Gamma_{\xi-\eta} \big( \tilde{q}_{+, \nu'}(\xi-\eta-\sigma, \sigma)  \widehat{g^{}_{k_1',j_1'}}(t, \xi-\eta-\sigma)\big)   -  \widehat{\Gamma g^{}_{k_1',j_1'}}(t, \xi-\eta-\sigma)  \tilde{q}_{+, \nu'}(\xi-\eta-\sigma, \sigma) \Big]  d\sigma d \eta d \xi d t ,
\ee
 which is resulted from the quartic terms when $\p_t $ hits the input ``$\widehat{\Gamma g_{k_1}}(t, \xi- \eta)$'' in  (\ref{eqn1026}) and moreover two inputs inside $\Lambda_{2}[ \p_t \widehat{\Gamma g_{k_1}}(t,  \xi-\eta)]$ have  different size of frequencies and the bulk term of this scenario is removed, 
\[
\widetilde{\Gamma}_{k,k_1,k_2}^{1,2;7}=   \sum_{k_2'\leq k_1'-10, |k_1-k_1'|\leq 10} \sum_{\nu'\in\{+,-\}} - \int_{t_1}^{t_2} \int_{\R^2}  \int_{\R^2 }  e^{i t\Phi^{+, \nu}(\xi, \eta)}  t \tilde{c}(\xi-\eta) \tilde{q}_{+,\nu}(\xi-\eta, \eta) \Big[ \widehat{    \Gamma g_{k_1}}(t, \xi-\eta)  
\]
\[
\times   \widehat{g^{\nu}_{k_2}}(t,\eta)  \overline{e^{it\Phi^{+, \nu}(\xi, \kappa)} \widehat{\Gamma^1 \Gamma^2 g_{k_1'}}(t, \xi-\kappa) \widehat{g^{\nu'}_{k_2'}(t, \kappa)} \tilde{q}_{+, \nu'}(\xi-\kappa, \kappa)}    + \overline{\widehat{\Gamma^1 \Gamma^2 g_k}(t, \xi)} \widehat{g^{\nu}_{k_2}}(t,\eta)   
\]
\[
\times  e^{i t\Phi^{+, \nu'}(\xi-\eta, \kappa)}  \tilde{q}_{+, \nu'}(\xi-\eta-\kappa, \kappa) \widehat{  \Gamma   g_{k_1'}}(t, \xi-\eta-\kappa) \widehat{g^{\nu'}_{k_2'}(t, \kappa)}  \Big]  d \kappa  d \eta d \xi d t.  
\]
\[
= \sum_{k_2'\leq k_1'-10, |k_1-k_1'|\leq 10} \sum_{\nu'\in\{+,-\}}  - \int_{t_1}^{t_2} \int_{\R^2}  \int_{\R^2 }    e^{i t\Phi^{+, \nu}(\xi, \eta)-it \Phi^{+, \nu'}(\xi,\kappa)}   t   {r}_{k_1,k_1'}^{\nu, \nu'}(\xi,\eta, \kappa)
\]
\[
\times   \widehat{  \Gamma  g_{  }}(t, \xi-\eta)   \widehat{g^{\nu}_{k_2}}(t,\eta)    \overline{  \widehat{\Gamma^1 \Gamma^2 g_{ }}(t, \xi-\kappa) } \widehat{g^{\nu'}_{k_2'}(t,- \kappa)}  d\kappa d \eta  d \xi d t,
\]
 which is resulted from putting the bulk term inside ``$\Lambda_{2}[\p_t\widehat{  \Gamma g_{k_1}}(t, \xi- \eta)]$'' and the bulk term inside ``$\Lambda_{2}[\p_t\widehat{\Gamma^1\Gamma^2 g_{ k}}(t, \xi )]$'' together, and the symbol $ {r}_{k_1,k_1'}^{\nu, \nu'}(\xi,\eta, \kappa)$ is given as follows, 
\[
    {r}_{k_1,k_1'}^{\nu, \nu'}(\xi,\eta, \kappa)=  \tilde{c}(\xi-\eta) \tilde{q}_{+, \nu}(\xi-\eta, \eta)  \overline{\tilde{q}_{+, -\nu'}(\xi-\kappa,\kappa)}  \psi_{k_1'}(\xi-\kappa)\psi_{k_1}(\xi-\eta)\psi_{k}(\xi) \]
 \[+ \tilde{c}(\xi-\eta-\kappa)  \tilde{q}_{+, \nu}(\xi- \kappa-\eta, \eta) {\tilde{q}_{+,  \nu'}(\xi- \eta ,-\kappa)}   \psi_{k_1'}(\xi-\eta)\psi_{k }(\xi-\kappa)\psi_{k_1}(\xi-\eta-\kappa).
\]
Recall (\ref{eqn939}) and (\ref{eqn932}). From Lemma \ref{Snorm}, the following estimate holds, 
\be\label{eqn826}
\|    {r}_{k_1,k_1'}^{\nu, \nu'}(\xi,\eta, \kappa)\psi_{k_2}(\eta)\psi_{k_2'}(\kappa)\|_{\mathcal{S}^\infty} \lesssim 2^{\max\{k_2, k_2'\}+3k_1}.
\ee
After doing spatial localizations for inputs $\Gamma  g_{k_1}$ and $g_{k_2}$ inside $\widetilde{\Gamma}_{k,k_1,k_2}^{1,2;7}$, we have 
\be\label{eqn829}
\widetilde{\Gamma}_{k,k_1,k_2}^{1,2;7}= \sum_{j_1\geq -k_{1,-}, j_2\geq -k_{2,-}} \widetilde{\Gamma}_{k,k_1,j_1,k_2,j_2}^{1,2;7}, 
\ee
\[
  \widetilde{\Gamma}_{k,k_1,j_1,k_2,j_2}^{1,2;7}:=	  \sum_{k_2'\leq k_1'-10, |k_1-k_1'|\leq 10} \sum_{\nu'\in\{+,-\}}   \int_{t_1}^{t_2} \int_{\R^2}  \int_{\R^2 }  e^{i t\Phi^{+, \nu}(\xi, \eta)-it \Phi^{+, \nu'}(\xi,\kappa)}  t   {r}_{k_1,k_1'}^{\nu, \nu'}(\xi,\eta, \kappa)
\]
\be\label{eqn830}
\times   \widehat{  \Gamma  g_{k_1,j_1}}(t, \xi-\eta)   \widehat{g^{\nu}_{k_2,j_2}}(t,\eta)    \overline{  \widehat{\Gamma^1 \Gamma^2 g_{k_1'}}(t, \xi-\kappa) } \widehat{g^{\nu'}_{k_2'}(t,- \kappa)}   d \eta d \kappa d \xi d t.
\ee

\begin{lemma}
Under the bootstrap assumption \textup{(\ref{smallness})}, the following estimate holds,
\be
  |  \widetilde{\Gamma}_{k,k_1,k_2}^{1,2}  | \lesssim \sum_{i=1}^7 \widetilde{\Gamma}_{k,k_1,k_2}^{1,2;i} |\lesssim 2^{-\beta  m } \epsilon_0^2.
\ee
\end{lemma}
\begin{proof}
From estimate (\ref{eqn52}) in  Lemma \ref{derivativeL2estimate1} and the $L^2-L^\infty$ type bilinear estimate, we have
\[
|   \widetilde{\Gamma}_{k,k_1,k_2}^{1,2;1}| \lesssim \sum_{i=1,2} 2^{2m+2k_1} \| \Gamma^1\Gamma^2 g_k\|_{L^2} \| \p_t \widehat{g_{k_2}}(t, \xi) \|_{L^2}\big(\| e^{-it\Lambda} \Gamma^i g_{k_1}\|_{L^\infty}+\| e^{-it\Lambda}   g_{k_1}\|_{L^\infty}\big)
\]
\[
\lesssim 2^{m+ 2\tilde{\delta}m}(2^{-21m/20} + 2^{-2m-k_2+2\tilde{\delta}m})\epsilon_0^2\lesssim 2^{-\beta m }\epsilon_0^2.
\]
 
Now, we proceed to estimate $ \widetilde{\Gamma}_{k,k_1,k_2}^{1,2;2} $. Recall (\ref{eqnnn811}) and (\ref{eqn805}). We split into two cases as follows based on the  size of difference between $k_1'$ and $k_2'$.

$\oplus$\quad If $|k_1'-k_2'|\leq 5$.\quad   Note that $k_1'\geq k_1-5\geq k_2+5$. 
By doing integration by parts in ``$\eta$'' many times, we can rule out the case when $\max\{j_1',j_2\}\leq m+k_{1,-}'-\beta m $. Hence, it would be sufficient to consider the case when $\max\{j_1',j_2\}\geq m+k_{1,-}'-\beta m $.   From the $L^2-L^\infty-L^\infty$ type  trilinear estimate (\ref{trilinearesetimate}) in Lemma \ref{multilinearestimate}, the following estimate holds,
\[
\sum_{|k_1'-k_2'|\leq 5}\big| \sum_{\max\{j_1',j_2\}\geq m+k_{1,-}'-\beta m  }  \widehat{\Gamma}_{k,k_1,k_2,j_2}^{  k_1',j_1', k_2',j_2',1} \big| \lesssim  \sum_{|k_1'-k_2'|\leq 5}   2^{2m+2k_1+2k_1'} \| \Gamma^1\Gamma^2 g_k(t)\|_{L^2 }  
\]
\[
\times \Big[\sum_{j_1'\geq \max\{  j_2,  m+k_{1,-}'-\beta m \} }  \| g_{k_1',j_1'}\|_{L^2}  \| e^{-it\Lambda} g_{k_2'}\|_{L^\infty}(  \| e^{-it\Lambda} g_{k_2,j_2}\|_{L^\infty} +  \| e^{-it\Lambda} \Gamma^n g_{k_2,j_2}\|_{L^\infty}) \]
\[+ \sum_{j_2\geq \max\{  j_1',  m+k_{1,-}'-\beta m  \} }  \| e^{-it \Lambda} g_{k_1',j_1'}\|_{L^\infty}
 2^{k_2+j_2}\| g_{k_2,j_2}\|_{L^2} \| e^{-it \Lambda} g_{k_2'}(t)\|_{L^\infty}\Big]\]
 \[\lesssim 2^{-m-k_2+20\beta m } \epsilon_0^2\lesssim 2^{-2\beta m }\epsilon_0^2.
\]

$\oplus$ If $k_2'\leq k_1'-5$.\quad For this case we have $|k_1-k_1'|\leq 2$ and $k_1'\geq k_2 + 5$. If moreover $k_1+k_2'\leq -9m/10$, then the following estimate holds from estimate (\ref{eqn400}) in Lemma \ref{Linftyxi}, 
\[
\sum_{k_2'\leq \min\{-9m/10-k_1, k_1-10\}}  |\widehat{\Gamma}_{k,k_1,k_2}^{  k_1',k_2',1} | \lesssim \sum_{k_2'\leq \min\{-9m/10-k_1, k_1-10\}} 2^{2m+3k_1}\| \Gamma^1 \Gamma^2 g_{k}(t)\|_{L^2}\|e ^{-it \Lambda} g_{k_1'}\|_{L^\infty}\]
\[\times \big( \| g_{k_2}\|_{L^2} + \| \Gamma^n g_{k_2}\|_{L^2} \big)\big( 2^{3k_2'}\| \widehat{g}_{k_2'}(t,\xi)\|_{L^\infty_\xi} + 2^{k_1+2k_2'} \| \widehat{\textup{Re}[v]}_{ }(t,\xi)\psi_{k_2'}(\xi)\|_{L^\infty_\xi}  \big)\lesssim 2^{-2\beta m }\epsilon_0^2.
\]

Lastly, if $k_1+k_2'\geq -9m/10$, we can do integration by part in ``$\sigma$'' many times  to rule out the case when $\max\{j_1',j_2'\} \leq  m+k_{1,-}-\beta m $.  Also, by doing integration by parts in ``$\eta$'' many times, we can rule out the case when $\max\{j_1',j_2 \}\leq m+k_{1,-}-\beta m $. Hence, it would be sufficient to consider the case when $\max\{j_1',j_2\} \geq  m+k_{1,-}-\beta m $ and $\max\{j_1',j_2'\} \geq  m+k_{1,-}-\beta m$.  As a result, either $j_1'\geq m +k_{1,-}-\beta m $ or $j_1'\leq m +k_{1,-}-\beta m $ and $j_2,j_2'\geq m +k_{1,-}-\beta m  $. 

From the  $L^2-L^\infty-L^\infty$ type  trilinear estimate (\ref{trilinearesetimate}) in Lemma \ref{multilinearestimate}, the following estimate holds, 
\[
\big|\sum_{ \max\{j_1',j_2'\}, \max\{j_1',j_2\} \geq  m +k_{1,-}-\beta m }  \widehat{\Gamma}_{k,k_1,k_2,j_2}^{  k_1',j_1', k_2',j_2',1} \big| \lesssim  2^{2m+4k_1 } \| \Gamma^1\Gamma^2 g_k(t)\|_{L^2 }\big( \sum_{j_1'\geq m +k_{1,-}-\beta m }  \| g_{k_1',j_1'}\|_{L^2} \]
\[\times \| e^{-it\Lambda} g_{k_2',j_2'}\|_{L^\infty} \big(\| e^{-it\Lambda} g_{k_2 ,j_2 }\|_{L^\infty} +\| e^{-it\Lambda} \Gamma g_{k_2 ,j_2 }\|_{L^\infty}\big)+ \sum_{j_2',j_2 \geq m +k_{1,-}-\beta m  }  2^{2k_2+j_2}\|   g_{k_2 ,j_2 }\|_{L^2} \]
\[\times  \| g_{k_2',j_2'}\|_{L^2} \| e^{-it\Lambda} g_{k_1',j_1'}\|_{L^\infty} \big) \lesssim 2^{-m-k_{2}' + 20\beta m }\epsilon_0^2\lesssim 2^{-2\beta m }\epsilon_0^2.
\]

Now, we proceed to estimate $ \widetilde{\Gamma}_{k,k_1,k_2}^{1,2;3} $. Recall (\ref{eqn802}) and (\ref{eqn809}).  We split into two cases as follows based on the  size of difference between $k_1'$ and $k_2'$. 

$\oplus$ If $|k_1'-k_2'|\leq 10$. \quad Note that $k_1'\geq k_2-5$. By doing integration by parts in ``$\sigma$'', we can rule out the case when $\max\{j_1',j_2'\} \leq m +k_{2,-}-k_{1,+}'-\beta m $.  From the  $L^2-L^\infty-L^\infty$ type  trilinear estimate (\ref{trilinearesetimate}) in Lemma \ref{multilinearestimate}, the following estimate holds when  $\max\{j_1',j_2'\} \geq m +k_{2,-}-k_{1,+}'-\beta m $,
\[	
\sum_{|k_1'-k_2'|\leq 10}\big|\sum_{\max\{j_1',j_2'\} \geq  m +k_{2,-}-k_{1,+}'-\beta m} \widehat{\Gamma}_{k,k_1,j_1,k_2}^{  k_1',j_1',k_2',j_2',2}\big| \lesssim \sup_{t\in[2^{m-1},2^m]}  \sum_{|k_1'-k_2'|\leq 10}  2^{2m+2k+2k_1'}\| \Gamma^1\Gamma^2 g_k(t)\|_{L^2 }\]
\[\times\| e^{-it \Lambda} g_{k_1}(t)\|_{L^\infty} \big(\sum_{ j_1'\geq\{j_2',  m +k_{2,-}-k_{1,+}'-\beta m\} } (2^{k_1'+j_1'} + 2^{m+k_2+k_1'}) \| g_{k_1',j_1'}(t)\|_{L^2} \| e^{-it \Lambda}g_{k_2',j_2'}(t)\|_{L^\infty}  \]
\[  +\sum_{ j_2'\geq\{j_1', m +k_{2,-}-k_{1,+}'-\beta m\} } (2^{k_1'+j_1'}+ 2^{m+k_2+k_1'})   \| g_{k_2',j_2'}(t)\|_{L^2} 2^{-m}\| g_{k_1',j_1'}(t)\|_{L^1} \big) \]
\[\lesssim 2^{-m-k_{2}  + 20\beta m }\epsilon_0^2\lesssim 2^{-2\beta m }\epsilon_0^2.
\]

$\oplus$ If $k_2'\leq k_1'- 10$. \quad For this case, we have $k_2-2\leq k_1'\leq k_2+2\leq k_1-5$. By doing integration by parts in ``$\eta$'', we can rule out the case when $\max\{j_1 ,j_1'\} \leq m +k_{1,-}-\beta m $.  From the  $L^2-L^\infty-L^\infty$ type  trilinear estimate (\ref{trilinearesetimate}) in Lemma \ref{multilinearestimate}, the following estimate holds when  $\max\{j_1 ,j_1'\} \geq m +k_{1,-}-\beta m $,
\[
\sum_{k_2'\leq k_1'-10} \big|\sum_{\max\{j_1 ,j_1'\} \geq m +k_{1,-}-\beta m } \widehat{\Gamma}_{k,k_1,j_1,k_2}^{ k_1',j_1',k_2',j_2',2}\big| \lesssim \sup_{t\in[2^{m-1},2^m]} \sum_{k_2'\leq k_1'-10} 2^{2m+2k+2k_1'} \| \Gamma^1\Gamma^2 g_k(t)\|_{L^2 }
\]
\[
\times\| e^{-it \Lambda} g_{k_2'}(t)\|_{L^\infty}\big( \sum_{j_1 \geq \max\{j_1' ,m +k_{1,-}-\beta m \}} (2^{k_1'+j_1'} + 2^{m+k_2+k_1'})\| g_{k_1,j_1}(t)\|_{L^2} 2^{-m} \|  g_{k_1',j_1'}(t)\|_{L^1} \]
\[+\sum_{j_1'\geq \max\{j_1 ,m +k_{1,-}-\beta m \}} (2^{k_1'+j_1'} + 2^{m+k_2+k_1'})\| g_{k_1',j_1'}(t)\|_{L^2}  \| e^{-it \Lambda} g_{k_1 ,j_1 }(t)\|_{L^\infty}  \big)\lesssim 2^{-m/2+20\beta m }\epsilon_0^2.
\]

Now, we proceed to estimate $ \widetilde{\Gamma}_{k,k_1,k_2}^{1,2;4} $ and  $ \widetilde{\Gamma}_{k,k_1,k_2}^{1,2;5} $ . Recall (\ref{eqn811}), (\ref{eqq546}), and (\ref{eqn812}). By doing integration by parts in `` $\eta$'', we can rule out the case when $\max\{j_1,j_2\} \leq m+k_{1,-}-\beta m $. From estimate (\ref{eqn730}) in Lemma \ref{derivativeL2estimate2}, (\ref{eqq425}) in Lemma \ref{Z2normcubicandhigher}, and the $L^2-L^\infty$ type bilinear estimate (\ref{bilinearesetimate}) in Lemma \ref{multilinearestimate}, the following estimate holds when $\max\{j_1,j_2\} \geq   m+k_{1,-}-\beta m$, 
\[
\sum_{i=4,5}\sum_{ \max\{j_1,j_2\} \geq  m+k_{1,-}-\beta m } |  \widetilde{\Gamma}_{k,k_1,j_1,k_2,j_2}^{1,2;i} | \lesssim 2^{m +2k_1+ \beta m  } \Big[\sum_{j_1\geq \max\{j_2,  m+k_{1,-}-\beta m\} }    (\| e^{-it \Lambda} g_{k_2,j_2}(t)\|_{L^\infty}     \]
\[+\| e^{-it \Lambda} \Gamma g_{k_2,j_2}(t)\|_{L^\infty}  \big) 2^{k_1+j_1}\big(2^{6 k_{+}}\|  g_{k_1,j_1}(t)\|_{L^2} +2^{m }\| \Lambda_{\geq 3}[\p_t  g(t)]_{k_1,j_1}\|_{L^2} \big) + 2^{m+k_2} \| g_{k_1,j_1}\|_{L^2} 2^{k_2+j_2} \]
\[
\times  \| \Lambda_{\geq 3}[\p_t  g(t)]_{k_2,j_2}\|_{L^2}   +  \sum_{j_2 \geq \max\{j_1,  m+k_{1,-}-\beta m \} } 2^{k_1+j_1}\big( 2^{6k_{+}}\| g_{k_1,j_1}\|_{L^2} + 2^{m }\| \Lambda_{\geq 3}[\p_t  g(t)]_{k_1,j_1}\|_{L^2} \big)   \]
\[\times  2^{k_2} \|g_{k_2,j_2}\|_{L^2}   +    2^{k_2 +j_2 }  \big(  2^{  6k_{+}}  \|g_{k_2,j_2}\|_{L^2} + 2^{m }\| \Lambda_{\geq 3}[\p_t  g(t)]_{k_2,j_2}\|_{L^2} \big)\| e^{-it\Lambda} g_{k_1,j_1}(t)\|_{L^\infty} \Big]	\]
\[\lesssim 2^{-m +40\beta  m -k_2}\epsilon_0^2 \lesssim 2^{-2\beta m }\epsilon_0^2.
\]
 
Now, we proceed to estimate $ \widetilde{\Gamma}_{k,k_1,k_2}^{1,2;6} $. Recall (\ref{eqn820}) and (\ref{eq399}). We split into three cases based on the difference between $k_1'$ and $k_2'$ and the size of $k_1'+k_2'$. 

$\oplus$ If $|k_1'-k_2'|\leq 10$, i.e., we are estimating $ \widehat{\Gamma}_{k,k_1,k_2;1 }^{ k_1' ,k_2' ,3} $. \quad   Note that we have $k_1'\geq k_1-5$. Recall (\ref{eq823}).  By doing integration by parts in ``$\sigma$'' many times, we can rule out the case when $\max\{j_1', j_2'\} \leq m + k_{1,-}-k_{1,+}'-\beta m . $ From the $L^2-L^\infty-L^\infty$ type trilinear estimate (\ref{trilinearesetimate}) in Lemma \ref{multilinearestimate}, the following estimate holds when $\max\{j_1', j_2'\} \geq m + k_{1,-}-k_{1,+}'-\beta m $,
\[
\big|\sum_{ \max\{j_1', j_2'\} \geq m + k_{1,-}-k_{1,+}'-\beta m   }   \widehat{\Gamma}_{k,k_1,k_2,j_2;1}^{ k_1',j_1',k_2',j_2',3} \big| \lesssim 2^{2m +2k + 2k_1'}\| \Gamma^1\Gamma^2 g_{k}(t)\|_{L^2}  \| e^{-it \Lambda} g_{k_2}(t)\|_{L^\infty} \]
\[\times \big(\sum_{j_1'\geq \max\{j_2', m + k_{1,-}-k_{1,+}'-\beta m \}} ( 2^{m+k_1+k_2'}+ 2^{k_2'+j_1'} )  \| g_{k_1',j_1'}\|_{L^2} \| e^{-it \Lambda} g_{k_2',j_2'}\|_{L^\infty} 
\]
\[  
+ \sum_{j_2'\geq \max\{j_1', m + k_{1,-}-k_{1,+}'-\beta m \}} ( 2^{m+k_1+k_2'}+ 2^{k_2'+j_1'} ) \| g_{k_2',j_2'}\|_{L^2} 2^{-m}\| g_{k_1',j_1'}\|_{L^1} \big) \lesssim 2^{-m/2+20\beta m }\epsilon_0^2.
\]

Now, we proceed to consider the case when $k_2'\leq k_1'-10$, i.e., we are estimating  $ \widehat{\Gamma}_{k,k_1,k_2;2  }^{ k_1' ,k_2' ,3} $. For this case, we have $|k_1-k_1'|\leq 2$.

$\oplus$ If $k_2'\leq k_1'- 10$ and  $k_1'+k_2'\leq -19m/20$.\quad For this case, we have $|k_1'-k_1|\leq 5$.   From estimate (\ref{eqn400}) in Lemma \ref{Linftyxi}, we have
\[
| \widehat{\Gamma}_{k,k_1,k_2;2}^{ k_1' ,k_2' ,3} |\lesssim 2^{2m +3k_1} \| \Gamma^1\Gamma^2 g_k(t)\|_{L^2} \| e^{-it \Lambda} g_{k_2}(t)\|_{L^\infty} \big( (2^{m+k_2'+k_1 } + 1) \| g_{k_1'}\|_{L^2} 
\]
\[
+\sum_{i=1,2} 2^{k_2'} \|  \nabla_\xi \widehat{g}_{k_1'}(t, \xi)\|_{L^2}\big) \big( 2^{3k_2'}\| \widehat{g}_{k_2'}(t)\|_{L^\infty_\xi} + 2^{k_1+2k_2'} \| \widehat{\textup{Re}[v]}_{}(t,\xi)\psi_{k_2'}(\xi)\|_{L^\infty_\xi}  \big)\lesssim 2^{-2\beta m }\epsilon_0^2.
\]

$\oplus$ If $k_2'\leq k_1'- 10$ and  $k_1'+k_2'\geq -19m/20$.\quad Recall (\ref{eqn823}). By doing integration by parts in ``$\sigma$'' many times, we can rule out the case when $\max\{j_1',j_2'\}\leq m +k_{1,-}-\beta m $. By doing integration by parts in ``$\eta$'' many times, we can rule out the case when $\max\{j_1',j_2\}\geq m +k_{1,-}-\beta m$. Therefore, we only need to consider the case when $\max\{j_1',j_2'\}\geq m +k_{1,-}-\beta m $ and $\max\{j_1',j_2  \}\geq m +k_{1,-}-\beta m $. In other words, either $j_1'\geq m +k_{1,-}-\beta m $ or $j_2', j_2\geq m +k_{1,-}-\beta m$. From the $L^2-L^\infty-L^\infty$ type trilinear estimate (\ref{trilinearesetimate}) in Lemma \ref{multilinearestimate}, the following estimate holds,
\[
\big|\sum_{ \max\{j_1', j_2'\}, \max\{j_1',j_2\} \geq m + k_{1,-}-\beta m  }   \widehat{\Gamma}_{k,k_1,k_2,j_2;2}^{  k_1',j_1',k_2',j_2',3} \big| \lesssim 2^{2m +4k  }\| \Gamma^1\Gamma^2 g_{k}(t)\|_{L^2} 
\]
\[
\times \big(\sum_{j_1'\geq   m +k_{1,-}-\beta m } ( 2^{m+k_1+k_2'}+ 2^{k_2'+j_1'} ) \| g_{k_1',j_1'}\|_{L^2} \| e^{-it \Lambda} g_{k_2',j_2'}\|_{L^\infty} \| e^{-it \Lambda} g_{k_2,j_2}(t)\|_{L^\infty} \]
\[
+\sum_{j_2',j_2\geq m+   k_{1,-}-\beta m }   ( 2^{m+k_1+k_2'} \| e^{-it \Lambda} g_{k_1',j_1'}\|_{L^\infty}  + 2^{k_2' } \| e^{-it \Lambda} \mathcal{F}^{-1}[\nabla_\xi \widehat{g_{k_1',j_1'}}]\|_{L^\infty}  )\|   g_{k_2',j_2'}\|_{L^2}   \]
\[\times  2^{k_2}\|   g_{k_2,j_2}(t)\|_{L^2} \big)\lesssim 2^{-m-k_2 +30\beta m }\epsilon_0^2\lesssim 2^{-2\beta m}\epsilon_0^2.
\]

Lastly, we estimate $ \widetilde{\Gamma}_{k,k_1,k_2}^{ 1,2;7} $. Recall (\ref{eqn829}) and (\ref{eqn830}). By doing integration by parts in ``$\eta$'' many times, we can rule out the case  when $\max\{j_1,j_2\} \leq m +k_{1,-}-\beta m $. From the $L^2-L^\infty-L^\infty$ type trilinear estimate (\ref{trilinearesetimate}) in Lemma \ref{multilinearestimate} and (\ref{eqn826}), the following estimate holds when $\max\{j_1,j_2\} \geq m +k_{1,-}-\beta m $,
\[
\sum_{ \max\{j_1,j_2\} \geq  m +k_{1,-}-\beta m   } |\widetilde{\Gamma}_{k,k_1,j_1,k_2,j_2}^{1,2;7} | \lesssim \sum_{k_2'\leq k_1-10}  2^{2m+\max\{k_2, k_2'\}+3k_1} \| \Gamma^1 \Gamma^2 g_k\|_{L^2} \| e^{-it \Lambda }g_{k_2'}\|_{L^\infty} \]
\[\times \big( \sum_{j_1\geq \max\{j_2,  m +k_{1,-}-\beta m \}}   \|\Gamma g_{k_1,j_1}\|_{L^2} \| e^{-it \Lambda} g_{k_2,j_2}\|_{L^\infty} + 	\sum_{j_2\geq \max\{j_1, m+k_{1,-}-\beta m \}}   \| e^{-it \Lambda} \Gamma g_{k_1,j_1}\|_{L^\infty}  \]
\be\label{eqn841}
\times\| g_{k_2,j_2}\|_{L^2} \big)\lesssim \sum_{k_2'\leq k_2}  2^{20\beta m } \| e^{-it \Lambda }g_{k_2'}\|_{L^\infty} \epsilon_1^2+ \sum_{k_2\leq k_2'\leq k_1-10} 2^{-m-k_2+20\beta m }\epsilon_0^2
  \lesssim 2^{-2\beta m }\epsilon_0^2.
\ee
 
\end{proof}

\subsubsection{The estimate of $P_{k,k_1,k_2}^2$}

Recall (\ref{eqn1011}) and (\ref{eqn989}).  Note that $P_{k,k_1,k_2}^2$ vanishes except when $\Gamma^1=\Gamma^2= L $. Hence, we only have to consider the case when $\Gamma^1=\Gamma^2=L$. We decompose it into two parts as follows, 
\[
P_{k,k_1,k_2}^2=\sum_{i=1,2}P_{k,k_1,k_2}^{2,i}, \quad P_{k,k_1,k_2}^{2,i} =  -\int_{t_1}^{t_2} \int_{\R^2} \int_{\R^2 }  \overline{\widehat{LL g_k}(t, \xi)}e^{i t\Phi^{+, \nu}(\xi, \eta)}  t^2   \widehat{q}^i_{+, \nu}(\xi , \eta) \]
\[\times  	   \widehat{g_{k_1}}(t, \xi-\eta) \widehat{  g_{k_2}^\nu}(t, \eta) d \eta d \xi d t,\quad i \in \{1,2\},
\]
where $\widehat{q}^i_{+, \nu}(\xi-\eta, \eta)$ is defined (\ref{eqn1501}) and (\ref{eqn1132}). 

For ``$ P_{k,k_1,k_2}^{2,1} $'', we do integration by parts in time once. As a result, we have
\be\label{eq5648}
P_{k,k_1,k_2}^{2,1} = \sum_{i=1,2,3,4,5} \widetilde{P}_{k,k_1,k_2}^{i}  , \quad  \widetilde{P}_{k,k_1,k_2}^{1}= \sum_{j_1\geq -k_{1,-}, j_2\geq -k_{2,-}} \widetilde{P}^{ j_1,j_2, 1}_{k,k_1,  k_2 }, 
\ee
\[
\widetilde{P}^{ j_1,j_2, 1}_{k,k_1,  k_2 }=\sum_{i=1,2}(-1)^i \int_{\R^2} \int_{\R^2 }  \overline{\widehat{\Gamma^1 \Gamma^2 g_k}(t_i, \xi)}  e^{i t_i \Phi^{+, \nu}(\xi, \eta)} 
  i t_i^2  \widehat{p}^1_{+, \nu}(\xi , \eta) \widehat{g_{k_1,j_1}}( t_i, \xi-\eta)  \widehat{  g_{k_2,j_2}^\nu}( t_i, \eta) d \eta   \]
 \[ -  \int_{t_1}^{t_2} \int_{\R^2} \int_{\R^2 }  \overline{\widehat{\Gamma^1 \Gamma^2 g_k}(t, \xi)}  e^{i t  \Phi^{+, \nu}(\xi, \eta)}    i 2 t  	 \widehat{p}^1_{+, \nu}(\xi , \eta) \widehat{g_{k_1,j_1}}( t  , \xi-\eta) \widehat{  g_{k_2,j_2}^\nu}( t , \eta)  
-   e^{i t  \Phi^{+, \nu}(\xi, \eta)}    i t^2 \]
\be\label{e340} 
\times  \widehat{p}^1_{+, \nu}(\xi , \eta)	   \widehat{g_{k_1,j_1 }}( t  , \xi-\eta) \widehat{  g_{k_2,j_2}^\nu}( t , \eta)\overline{  \big(\p_t \widehat{\Gamma^1 \Gamma^2 g_k}(t, \xi) -   \sum_{\nu\in\{+,-\}} \sum_{(k_1',k_2')\in \chi_k^2} \widetilde{B}^{+, \nu}_{k,k_1',k_2'}(t, \xi) \big) }  d \eta d \xi d t,
\ee

\be\label{eqn872}
\widetilde{P}^{  2 }_{k,k_1,k_2} = \sum_{k_2'\leq k_1' +10} \widehat{P}^{  2,k_1', k_2' }_{k,k_1,k_2},\quad  \widehat{P}^{ 2,k_1', k_2' }_{k,k_1,k_2}=\sum_{j_1\geq -k_{1,-},j_1'\geq -k_{1,-}',j_2'\geq -k_{2,-}'}\widehat{P}^{  2,k_1', j_1', k_2',j_2' }_{k,k_1,j_1, k_2},
\ee

\[\widehat{P}^{  2,k_1', j_1', k_2',j_2' }_{k,k_1,j_1, k_2}:=\sum_{\mu',\nu'\in\{+,-\}} - \int_{t_1}^{t_2} \int_{\R^2} \int_{\R^2 }  \overline{\widehat{\Gamma^1 \Gamma^2 g_k}(t, \xi)}  e^{i t  \Phi^{+, \nu}(\xi, \eta)}  i t^2  	 \widehat{p}^1_{+, \nu}(\xi , \eta)	   \widehat{g_{k_1, j_1 }}( t  , \xi-\eta)  \]
\be\label{eqn869}
\times P_{\nu}\big[ e^{it \Phi^{\mu', \nu'}(\eta, \sigma)} \tilde{q}_{\mu',\nu'}(\eta-\sigma, \sigma) \widehat{g_{k_1',j_1'}^{\mu'}}(t, \eta-\sigma)  \widehat{g_{k_2',j_2'}^{\nu'}}(t,   \sigma) \big]  d \eta d \xi d t, 
\ee

\be\label{eqn880}
\widetilde{P}^{ 3 }_{k,k_1,k_2} = \sum_{|k_1'-k_2'|\leq 10} \widehat{P}^{ 3,k_1',k_2' 	 }_{k,k_1,k_2},\quad \widehat{P}^{ 3,k_1',k_2' 	 }_{k,k_1,k_2}=\sum_{j_1'\geq-k_{1,-}', j_2'\geq -k_{2,-}'} \widehat{P}^{3,k_1',j_1',k_2',j_2'  	 }_{k,k_1,k_2},
\ee
\[
\widehat{P}^{3,k_1',j_1',k_2',j_2'  	 }_{k,k_1,k_2}= \sum_{\mu',\nu'\in\{+,-\}} - \int_{t_1}^{t_2} \int_{\R^2} \int_{\R^2 }  \overline{\widehat{\Gamma^1 \Gamma^2 g_k}(t, \xi)}  e^{i t  \Phi^{+, \nu}(\xi, \eta)}  i t^2 	 \widehat{p}^1_{+, \nu}(\xi , \eta)	   \]
\be\label{eqn881}
\times e^{it\Phi^{\mu',\nu'}(\xi-\eta,\sigma)} \tilde{q}_{\mu',\nu'}(\xi-\eta-\sigma, \sigma) \widehat{g_{k_1',j_1'}^{\mu'}}( t  , \xi-\eta-\sigma) \widehat{g_{k_2',j_2'}^{\nu'}}(t, \sigma) \widehat{  g_{k_2 }^\nu}( t , \eta) d \eta d \xi d t.
\ee

\[
\widetilde{P}_{k,k_1,k_2}^{ 4}=  \sum_{j_1\geq -k_{1,-}, j_2\geq -k_{2,-}} \widetilde{P}^{4}_{k,k_1,j_1, k_2,j_2} , \quad  \widetilde{P}^{4}_{k,k_1,j_1, k_2,j_2}:=- \int_{t_1}^{t_2} \int_{\R^2} \int_{\R^2 } e^{i t\Phi^{+, \nu}(\xi, \eta)}  it^2 	\widehat{p}^1_{+, \nu}(\xi , \eta)      \]
\be\label{eqq878}
\times \overline{\widehat{\Gamma^1 \Gamma^2 g_k}(t, \xi)} \big(  \widehat{  \Lambda_{\geq 3}[\p_t  g^{}]_{k_1,j_1}}(t, \xi-\eta)   \widehat{g_{k_2,j_2}^\nu}(t, \eta)+   \widehat{g_{k_1,j_1}^{}}(t, \xi-\eta) \widehat{   \Lambda_{\geq 3}[\p_t  g^{\nu}]_{k_2,j_2}}(t, \eta)\big)      d \eta d \xi d t,
\ee
\[
\widetilde{P}_{k,k_1,k_2}^{ 5}=   \sum_{k_2'\leq k_1'-10, |k_1-k_1'|\leq 10} \sum_{\nu'\in\{+,-\}} - \int_{t_1}^{t_2} \int_{\R^2}  \int_{\R^2 }  e^{i t\Phi^{+, \nu}(\xi, \eta)} i t^2 	\widehat{p}^1_{+, \nu}(\xi , \eta) \Big[ \widehat{     g_{k_1}}(t, \xi-\eta)  \widehat{g^{\nu}_{k_2}}(t,\eta) 
\]
\[
\times \overline{e^{it\Phi^{+, \nu'}(\xi, \kappa)} \widehat{\Gamma^1 \Gamma^2 g_{k_1'}}(t, \xi-\kappa) \widehat{g^{\nu'}_{k_2'}(t, \kappa)} \tilde{q}_{+, \nu'}(\xi-\kappa, \kappa)}    + \overline{\widehat{\Gamma^1 \Gamma^2 g_k}(t, \xi)} \widehat{g^{\nu}_{k_2}}(t,\eta)   
\]
\[
\times e^{i t\Phi^{+, \nu'}(\xi-\eta, \kappa)} \tilde{q}_{+, \nu'}(\xi-\eta-\kappa, \kappa) \widehat{    g_{k_1'}}(t, \xi-\eta-\kappa) \widehat{g^{\nu'}_{k_2'}(t, \kappa)}  \Big] d \eta d \kappa d \xi d t.  
\]
\[
= \sum_{k_2'\leq k_1'-10, |k_1-k_1'|\leq 10} \sum_{\nu'\in\{+,-\}}   - \int_{t_1}^{t_2} \int_{\R^2}  \int_{\R^2 }  e^{i t\Phi^{+, \nu}(\xi, \eta)-it \Phi^{+, \nu'}(\xi,\kappa)}  it^2 \widetilde{r}_{k_1,k_1'}^{\nu, \nu'}(\xi , \eta,\kappa)
\]
\[
\times   \widehat{    g }(t, \xi-\eta)   \widehat{g^{\nu}_{k_2}}(t,\eta)    \overline{  \widehat{\Gamma^1 \Gamma^2 g }(t, \xi-\kappa) } \widehat{g^{\nu'}_{k_2'}(t,- \kappa)}   d \eta d \kappa d \xi d t,
\]

where the symbol`` $\widehat{p}^1_{\mu, \nu}(\xi , \eta)$'' is defined in (\ref{eqn1100}) and the symbol $\widetilde{r}_{k_1,k_1'}^{\nu, \nu'}(\xi , \eta,\kappa)$  is defined as follows,
\[
\widetilde{r}_{k_1,k_1'}^{\nu, \nu'}(\xi , \eta,\kappa)= \widehat{p}^1_{+, \nu}(\xi , \eta) \overline{\tilde{q}_{+, -\nu'}(\xi-\kappa, \kappa) }\psi_{k_1'}(\xi-\kappa)\psi_{k_1}(\xi-\eta)\psi_{k}(\xi)  \]
\[+  \widehat{p}^1_{+, \nu}(\xi-\kappa , \eta) \tilde{q}_{+,  \nu'} (\xi-\eta , -\kappa) \psi_{k_1'}(\xi-\eta)\psi_{k_1}(\xi-\eta-\kappa)\psi_{k }(\xi-\kappa).
\]
Recall (\ref{eqn1100}), (\ref{eqn939}) and (\ref{eqn932}). From Lemma \ref{Snorm}, the following estimate holds,
\be\label{eqn860}
\| \widetilde{r}_{k_1,k_1'}^{\nu, \nu'}(\xi , \eta,\kappa)\psi_{k_2}(\eta)\psi_{k_2'}(\kappa)\|_{\mathcal{S}^\infty} \lesssim 2^{ \max\{k_2,k_2'\}+k_2+4k_1 }.
\ee
After doing spatial localizations for inputs $\widehat{g}_{k_1}(t)$ and $\widehat{g}_{k_2}(t)$ in $\widetilde{P}_{k,k_1,k_2}^{ 5}$,  the following decomposition holds, 
\be\label{eqn888}
\widetilde{P}_{k,k_1,k_2}^{5}=\sum_{k_2'\leq k_1'-10, |k_1-k_1'|\leq 10} \widehat{P}_{k,k_1,k_2}^{  5,k_1',k_2' }, \quad \widehat{P}_{k,k_1,k_2}^{  5,k_1',k_2' }=\sum_{j_1 \geq-k_{1,-},j_2\geq -k_{2,-}}\widehat{P}_{k,k_1,j_1,k_2,j_2}^{5,k_1',k_2' },
\ee
\[
\widehat{P}_{k,k_1,j_1,k_2,j_2}^{ 5,k_1',k_2' }= \sum_{\nu'\in\{+,-\}}   -\int_{t_1}^{t_2} \int_{\R^2}  \int_{\R^2 }  e^{i t\Phi^{+, \nu}(\xi, \eta)-it \Phi^{+, \nu'}(\xi,\kappa)}  i t^2 \widetilde{r}_{k_1,k_1'}^{\nu, \nu'}(\xi , \eta,\kappa)  \widehat{    g_{k_1,j_1}}(t, \xi-\eta)   
\]
\be\label{eqn870}
\times \widehat{g^{\nu}_{k_2,j_2}}(t,\eta)    \overline{  \widehat{\Gamma^1 \Gamma^2 g_{k_1'}}(t, \xi-\kappa) } \widehat{g^{\nu'}_{k_2'}(t,- \kappa)}   d \eta d \kappa d \xi d t.
\ee
\begin{lemma}
Under the bootstrap assumption \textup{(\ref{smallness})}, the following estimate holds,  
\be
  |   {P}_{k,k_1,k_2}^{2,1}  | \lesssim   \sum_{i=1,2,3,4,5} | \widetilde{P}_{k,k_1,k_2}^{i}| \lesssim 2^{-\beta  m } \epsilon_0^2.
\ee
\end{lemma}
\begin{proof}

We first estimate $\tilde{P}^{  1}_{k,k_1,k_2}$. Recall  (\ref{eq5648}) and (\ref{e340}).  By doing integration  by parts in ``$\eta$'' many times, we can rule out the case when $\max\{j_1,j_2\}\leq m +k_{1,-}-\beta m $. From the $L^2-L^\infty$ type bilinear estimate (\ref{bilinearesetimate}) in Lemma \ref{multilinearestimate}, (\ref{eqn730}) in Lemma \ref{derivativeL2estimate2},,   the following estimate holds when  $\max\{j_1,j_2\}\geq m +k_{1,-}-\beta m  $,
\[
\sum_{ \max\{j_1,j_2\}\geq m +k_{1,-}-\beta m} | \widetilde{P}^{ j_1,j_2, 1}_{k,k_1,  k_2 }| \lesssim  2^{2m + 2\tilde{\delta} m +k_2+3k_1 + 6k_{+}} \big( \sum_{j_1\geq \max\{j_2,m +k_{1,-}-\beta m \}} \| e^{-it \Lambda}g_{k_2,j_2} \|_{L^\infty}  \]
\[ \times \| g_{k_1,j_1}\|_{L^2}   +  \sum_{j_2\geq \max\{j_1,m +k_{1,-}-\beta m \}} \| g_{k_2,j_2}\|_{L^2} \| e^{-it \Lambda}g_{k_1,j_1} \|_{L^\infty}\big)\lesssim 2^{-m-k_2 + 50\beta m }\epsilon_0^2\lesssim 2^{-2\beta m }\epsilon_0^2.
\]

Now we proceed to estimate $\tilde{P}^{  2}_{k,k_1,k_2}$. Recall (\ref{eqn872}) and (\ref{eqn869}). Based on the size of the difference between $k_1'$ and $k_1$, we split into two cases as follows,

$\oplus$ If $k_1'\geq k_1-5$. \quad For this case, we have  $k_1'\geq k_2+5$ and $|k_1'-k_2'|\leq 5$. By doing integration by parts in ``$\sigma$'', we can rule out the case when $\max\{j_1',j_2'\} \leq m +k_{2,-}-k_{1,+}'- \beta m $.  From the  $L^2-L^\infty-L^\infty$ type  trilinear estimate (\ref{trilinearesetimate}) in Lemma \ref{multilinearestimate}, the following estimate holds when  $\max\{j_1',j_2'\} \geq   m +k_{2,-}-k_{1,+}'- \beta m  $,
\[
\sum_{k_1'\geq k_1-5}\big| \sum_{\max\{j_1',j_2'\} \geq m +k_{2,-}-k_{1,+}'- \beta m }\widehat{P}^{  2,k_1', j_1', k_2',j_2' }_{k,k_1,j_1, k_2} \big| \lesssim \sum_{|k_1'-k_2'|\leq 5} 2^{3m+k_2+3k_1+2k_1'}\| e^{-it \Lambda} g_{k_1}(t)\|_{L^\infty}\]
\[ \big( \sum_{j_1'\geq \max\{j_2', m +k_{2,-}-k_{1,+}'- \beta m\}}\| g_{k_1',j_1'}(t)\|_{L^2} \| e^{-it \Lambda} g_{k_2',j_2'}\|_{L^\infty} + \sum_{j_2'\geq \max\{j_1', m +k_{2,-}-k_{1,+}'- \beta m\}}\| g_{k_2',j_2'}(t)\|_{L^2}\]
\[\times \| e^{-it \Lambda} g_{k_1',j_1'}\|_{L^\infty}    \big)\| \Gamma^1\Gamma^2 g_k\|_{L^2}\lesssim 2^{-m-k_2 +30\beta m }\epsilon_0^2 \lesssim 2^{-2\beta m }\epsilon_0^2.
\]

$\oplus$ If $k_1'\leq  k_1-5$. \quad For this case, we do integration by parts in ``$\eta$ '' many times to rule out the case when $\max\{j_1',j_1\}\leq m +k_{1,-}-\beta m$.  From the  $L^2-L^\infty-L^\infty$ type  trilinear estimate (\ref{trilinearesetimate}) in Lemma \ref{multilinearestimate}, the following estimate holds when  $\max\{j_1',j_1\}\geq m +k_{1,-}-\beta m$,
\[
\sum_{k_2'\leq k_1'\leq k_1-5}\big|\sum_{ \max\{j_1',j_1\}\geq m +k_{1,-}-\beta m} \widehat{P}^{  2,k_1', j_1', k_2',j_2' }_{k,k_1,j_1, k_2}\big| \lesssim \sum_{k_2'\leq k_1'\leq k_1-5} 2^{3m+k_2+3k_1 + 2k_1'}\| e^{-it \Lambda} g_{k_2'}(t)\|_{L^\infty}\]
\[ \big( \sum_{j_1'\geq \max\{j_1,m +k_{1,-}-\beta m\}}\| g_{k_1',j_1'}(t)\|_{L^2} \| e^{-it \Lambda} g_{k_1,j_1 }\|_{L^\infty} + \sum_{j_1 \geq \max\{j_1' ,m +k_{1,-}-\beta m\}}\| g_{k_1,j_1}(t)\|_{L^2}\]
\[\times \| e^{-it \Lambda} g_{k_1',j_1'}\|_{L^\infty}    \big)\| \Gamma^1\Gamma^2 g_k\|_{L^2}\lesssim 2^{-m/2+30\beta m}\epsilon_0^2.
\]
 
Now, we proceed to   estimate $\tilde{P}^{ 3}_{k,k_1,k_2}$. Recall (\ref{eqn880}) and (\ref{eqn881}).  Note that $|k_1'-k_2'|\leq 10$ and  ``$\nabla_\sigma \Phi^{\mu',\nu'}(\xi-\eta,\sigma)$'' always has a lower bound, which is $2^{k_1-k_{1,+}'}.$ By doing integration by parts in ``$\sigma$'' many times, we can rule out the case when $\max\{j_1',j_2'\}\leq m +k_{1,-} -k_{1,+}'- \beta m $. From the  $L^2-L^\infty-L^\infty$ type  trilinear estimate (\ref{trilinearesetimate}) in Lemma \ref{multilinearestimate}, the following estimate holds when  $\max\{j_1',j_2'\}\geq  m +k_{1,-}-k_{1,+}' - \beta m$,
\[
\sum_{|k_1'-k_2'|\leq 10}\sum_{ \max\{j_1',j_2'\}\geq  m +k_{1,-}-k_{1,+}' - \beta m} |  \widehat{P}^{3,k_1',j_1',k_2',j_2'  	 }_{k,k_1,k_2}| \lesssim \sum_{|k_1'-k_2'|\leq 10} 2^{3m +k_2+3k_1+2k_1'}\| e^{-it \Lambda} g_{k_2 }(t)\|_{L^\infty}\]
\[ \big( \sum_{j_1'\geq \max\{j_2',  m +k_{1,-} -k_{1,+}'- \beta m\}}\| g_{k_1',j_1'}(t)\|_{L^2} \| e^{-it \Lambda} g_{k_2',j_2' }\|_{L^\infty} + \sum_{j_2' \geq \max\{j_1' ,  m +k_{1,-}-k_{1,+}' - \beta m\}}\| g_{k_2',j_2'}(t)\|_{L^2}\]
\[\times \| e^{-it \Lambda} g_{k_1',j_1'}\|_{L^\infty}    \big)   \| \Gamma^1\Gamma^2 g_k\|_{L^2}\lesssim 2^{-m/2+30\beta m}\epsilon_0^2.
\]

Now, we proceed to estimate $\tilde{P}^{4}_{k,k_1,k_2}$. Recall (\ref{eqq878}).   By doing integration  by parts in ``$\eta$'' many times, we can rule out the case when $\max\{j_1,j_2\}\leq m +k_{1,-}-\beta m $. From the $L^2-L^\infty$ type bilinear estimate (\ref{bilinearesetimate}) in Lemma \ref{multilinearestimate}, estimate (\ref{eqq425}) in Lemma \ref{Z2normcubicandhigher}, and estimate (\ref{L2cubicandhigher}) in Lemma \ref{derivativeL2estimate1}, the following estimate holds when $\max\{j_1,j_2\}\geq m +k_{1,-}-\beta m $,
\[
\sum_{ \max\{j_1,j_2\}\geq m +k_{1,-}-\beta m   }\|\widetilde{P}^{4}_{k,k_1,j_1, k_2,j_2}\|_{L^2}\lesssim \sup_{t\in[2^{m-1}, 2^m]} 2^{3m + k_2+3k_1} \| \Gamma^1 \Gamma^2 g_{k}(t)\|_{L^2}\]
\[\times  \big[\sum_{j_1 \geq \max\{j_2,m +k_{1,-}-\beta m   \} } \| \Lambda_{\geq 3}[\p_t  g^{\mu}]_{k_1,j_1} \|_{L^2}  \| e^{-it\Lambda} g_{k_2,j_2}\|_{L^\infty} + \| g_{k_1,j_1}\|_{L^2} 2^{k_2}\|  \Lambda_{\geq 3}[\p_t g_{k_2}]\|_{L^2}
\]
\[
+ \sum_{j_2 \geq \max\{j_1,m +k_{1,-}-\beta m   \} }   \| \Lambda_{\geq 3}[\p_t  g^{\mu}]_{k_2,j_2} \|_{L^2}  \| e^{-it\Lambda} g_{k_1,j_1}\|_{L^\infty} + 2^{k_2}\| g_{k_2,j_2}(t) \|_{L^2} \| \Lambda_{\geq 3}[\p_t  g^{\mu}_{k_1 }]\|_{L^2}\big]
\]
\[
\lesssim 2^{-m-k_2+40 \beta m }\epsilon_0^2 + 2^{-m/2+40\beta m }\epsilon_0^2\lesssim 2^{-2\beta m }\epsilon_0^2.
\]

Lastly, we estimate $\tilde{P}^{5}_{k,k_1,k_2}$. Recall (\ref{eqn888}) and (\ref{eqn870}).  For the case we are considering, we have $k_2'\leq k_1'-10$ and $|k_1'-k_1|\leq 10 $. By doing integration by parts in ``$\eta$'' many times, we can rule out the case when $\max\{j_1,j_2\} \leq m +k_{1, -}-\beta m $.  From the  $L^2-L^\infty-L^\infty$ type  trilinear estimate (\ref{trilinearesetimate}) in Lemma \ref{multilinearestimate} and estimate (\ref{eqn860}), the following estimate holds when  $\max\{j_1 ,j_2 \}\geq  m +k_{1, -}-\beta m $,
\[
\sum_{k_2'\leq k_1-10}\sum_{\max\{j_1 ,j_2 \}\geq  m +k_{1, -}-\beta m } | \widehat{P}_{k,k_1,j_1,k_2,j_2}^{5,k_1',k_2' }|\lesssim \sum_{k_2'\leq k_1-10}  2^{3m + k_2+\max\{k_2,k_2'\} + 4k_1} \| e^{-it \Lambda} g_{k_2'}(t)\|_{L^\infty}
\]
\[
\times \| \Gamma^1\Gamma^2 g_k\|_{L^2}  \big( \sum_{j_1\geq \max\{j_2,m +k_{1, -}-\beta m\}} \| g_{k_1,j_1}\|_{L^2} \| e^{-it\Lambda} g_{k_2,j_2}\|_{L^\infty} +  \sum_{j_2\geq \max\{j_1, m +k_{1, -}-\beta m\}} \| g_{k_2,j_2}\|_{L^2} \]
\[\times \| e^{-it\Lambda} g_{k_1,j_1}\|_{L^\infty}\big) \lesssim  2^{-m/2+30\beta m } \epsilon_0^2 + 2^{-m-k_2 +30\beta  m}\epsilon_0^2 \lesssim 2^{-2\beta m }\epsilon_0^2.
\]

\end{proof}
  \begin{lemma}
Under the bootstrap assumption \textup{(\ref{smallness})}, the following estimate holds,
\be\label{eqn1129}
 \sum_{k_2\leq k_1-10, |k_1-k |\leq 10}  \big|  P_{k,k_1,,k_2}^{2,2}   \big| \lesssim  2^{2\tilde{\delta} m  }\epsilon_0^2. 
\ee
\end{lemma}

\begin{proof}
Recall  (\ref{eqn1301}) and (\ref{eqn1132}).  From  Lemma \ref{Snorm}, the following estimate holds, 
\be\label{eqn112781}
\| \widehat{q}_{+, \nu}^2(\xi-\eta, \eta) \psi_{k}(\xi) \psi_{k_1}(\xi-\eta)\psi_{k_2}(\eta)\|_{\mathcal{S}^\infty}\lesssim 2^{3k_1+3k_2},\quad k_2\leq k_1-10.
\ee
After doing integration by parts in ``$\eta$'' twice, the following estimate holds, 
\[
 \sum_{k_2\leq k_1-10, |k_1-k |\leq 10}  \big|  P_{k,k_1,,k_2}^{2,2}   \big|\lesssim  \sum_{k_2\leq k_1-10 } \sum_{i=1,2} 2^{m+k_1+3k_2+k_{1,+}}\| \Gamma^1 \Gamma^2 g_{k}\|_{L^2}\Big[ \| e^{-it \Lambda} g_{k_2}\|_{L^\infty} \big(\| \nabla_\xi^2 \widehat{g}_{k_1}(t, \xi)\|_{L^2}   
\]
\[
 +2^{-k_2}\| \nabla_\xi  \widehat{g}_{k_1}(t, \xi)\|_{L^2}\big)  +   \| e^{-it \Lambda} g_{k_1}\|_{L^\infty} \big(\| \nabla_\xi^2 \widehat{g}_{k_2}(t, \xi)\|_{L^2}+ 2^{-k_2}\| \nabla_\xi  \widehat{g}_{k_2}(t, \xi)\|_{L^2} +2^{-2k_{2}} \| g_{k_2}(t)\|_{L^2}\big)
 \]
 \[ + \sum_{j_1\geq j_2} 2^{ -m +2j_2} \| g_{k_2,j_2}\|_{L^2} 2^{j_1}\| g_{k_1,j_1}\|_{L^2}    +\sum_{j_2\geq j_1} 2^{-m+   2j_1}  \| g_{k_1,j_1}\|_{L^2}
  2^{j_2} \| g_{k_2,j_2}\|_{L^2}  \Big] \lesssim 2^{2\tilde{\delta}m}\epsilon_0^2.
\]
 
\end{proof}
\subsubsection{The estimate of $P_{k,k_1,k_2}^3$}
 Recall (\ref{eqn1170}). We first estimate $Q_{k,k_1,k_2}^1$. A crucial  ingredient  of estimating $Q_{k,k_1,k_2}^1$   is utilizing symmetries, which is to switch the role of $\xi$ and $\xi-\eta$ inside $Q_{k,k_1,k_2}^1$. As a result, we have 
\[
\textup{Re}\big[   Q_{k,k_1,k_2}^1 \big] = \textup{Re}\big[  \widetilde{Q}_{k,k_1,k_2}^{ 1}\big], \quad  \widetilde{Q}_{k,k_1,k_2}^{ 1}:=  \int_{t_1}^{t_2} \int_{\R^2} \int_{\R^2 }  \overline{\widehat{\Gamma^1 \Gamma^2 g }(t, \xi)}  e^{i t\Phi^{+, \nu}(\xi, \eta)} \]
\[\times   {p}^{+, \nu}_{k,k_1}(\xi-\eta, \eta) \widehat{  \Gamma^1   \Gamma^2 g }(t, \xi-\eta)\widehat{g_{k_2}^\nu}(t, \eta) d \eta d\xi d t.
\]
where \be\label{e100}
 {p}^{+, \nu}_{k,k_1}(\xi-\eta, \eta)= \frac{\tilde{q}_{+,\nu}(\xi-\eta, \eta)\psi_k(\xi)\psi_{k_1}(\xi-\eta)}{2} + \frac{\overline{\tilde{q}_{+, -\nu}(\xi , -\eta)}\psi_k(\xi-\eta)\psi_{k_1 }(\xi )}{2}.
\ee
From (\ref{symmetricsymbol}) and (\ref{eqn1}), we have
\[
 {p}^{+, \nu}_{k,k_1}(\xi-\eta, \eta)=  {p}^{+, \nu,1}_{k,k_1}(\xi-\eta, \eta)+  {p}^{+, \nu,2}_{k,k_1}(\xi-\eta, \eta)= \mathcal{O}(1) \xi \cdot \eta + \mathcal{O}(|\eta|^2) , 
\]
where $ {p}^{+, \nu,2}_{k,k_1}(\xi-\eta, \eta)= {p}^{+, \nu }_{k,k_1}(\xi-\eta, \eta)- {p}^{+, \nu,1 }_{k,k_1}(\xi-\eta, \eta) $ and the detail formula of ${p}^{+, \nu,1}_{k,k_1}(\xi-\eta, \eta)$ is given as follows,
\[
 {p}^{+, \nu,1}_{k,k_1}(\xi-\eta, \eta)=  c_{+}|\xi|^2 \big(1 -   \tanh^2(|\xi|)  \big) \big( \tilde{\lambda}'(|\xi|^2)\psi_k(\xi)\psi_{k_1}(\xi ) + \hat{\psi}_{k_1} (|\xi|^2)\hat{\psi}_{k}'(|\xi|^2) - \hat{\psi}_{k } (|\xi|^2)\hat{\psi}_{k_1}'(|\xi|^2)\big)\xi \cdot \eta, \quad
\]
where $\hat{\psi}_k(x):=\psi_k(\sqrt{x}) $, which is still a smooth function. Recall (\ref{degeneratephase}), we can rewrite ${p}^{+, \nu,1}_{k,k_1}(\xi-\eta, \eta)$ as follows, 
\[
{p}^{+, \nu,1}_{k,k_1}(\xi-\eta, \eta)=   c_{+} a_{k,k_1}( \xi )  \Phi^{+, \nu}(\xi-\eta, \eta) + \mathcal{O}(|\eta|^2), \]
\[  a_{k,k_1}( \xi ):=\frac{  \big( |\xi|^2 - |\xi|^2 \tanh^2(|\xi|)  \big) \big( \tilde{\lambda}'(|\xi|^2)\psi_k(\xi)\psi_{k_1}(\xi ) +  \hat{\psi}_{k_1} (|\xi|^2)\hat{\psi}_{k}'(|\xi|^2 )- \hat{\psi}_{k } (|\xi|^2)\hat{\psi}_{k_1}'(|\xi|^2)\big)}{2\lambda'(|\xi|^2)}.
\]
To sum up, we can decompose $p_{+,\nu}(\xi-\eta, \eta)$ into two parts as follows, 
\be\label{eqn611}
{p}^{+, \nu}_{k,k_1}(\xi-\eta, \eta)= \sum_{i=1,2} \tilde{p}^{+, \nu,i}_{k,k_1}(\xi-\eta, \eta),\quad \tilde{p}^{+, \nu, 1}_{k,k_1}(\xi-\eta, \eta)=  \frac{-i}{2}  a_{k,k_1}( \xi ) \Phi^{+, \nu}(\xi , \eta),
\ee
where $  \tilde{p}^{+, \nu,2}_{k,k_1}(\xi-\eta, \eta)$ satisfies the following estimate, 
\be\label{eqn470}
\|  \tilde{p}^{+, \nu,2}_{k,k_1}(\xi-\eta, \eta)\|_{\mathcal{S}^\infty_{k, k_1, k_2}} \lesssim 2^{2k_2}, \quad k_2\leq k_1-5.
\ee
Correspondingly, we decompose $\widetilde{Q}_{k,k_1,k_2}^{ 1}$ into two parts as follows, 
\[
\widetilde{Q}_{k,k_1,k_2}^{ 1} = \sum_{i=1,2}\widetilde{Q}_{k,k_1,k_2}^{ 1;i},\quad \widetilde{Q}_{k,k_1,k_2}^{ 1;i}:=
  \int_{t_1}^{t_2} \int_{\R^2}  \int_{\R^2 } \overline{\widehat{\Gamma^1 \Gamma^2 g }(t, \xi)}  e^{i t\Phi^{+, \nu}(\xi, \eta)}  \]
  \[\times  \tilde{p}^{+, \nu,i}_{k,k_1}(\xi-\eta, \eta)   \widehat{( \Gamma^1 \Gamma^2 g)_{ }}(t, \xi-\eta) \widehat{g^{\nu}_{k_2}}(t,\eta)  d \eta d \xi d t,\quad i \in\{1,2\}.
\] 
From   the $L^2-L^\infty$ type bilinear estimate (\ref{bilinearesetimate}) in Lemma \ref{multilinearestimate}, the following estimate holds for $ \widetilde{Q}_{k,k_1,k_2}^{ 1;2}$, 
\[
\sum_{|k-k_1|\leq 10,k_2\leq k_1-10} |\widetilde{Q}_{k,k_1,k_2}^{ 1;2}| \lesssim \sum_{|k-k_1|\leq 10,k_2\leq k_1-10}  2^{2k_2}\| \Gamma^1 \Gamma^2 g_k\|_{L^2}  \| \Gamma^1 \Gamma^2 g_{k_1}\|_{L^2}\]
\be\label{eqn790}
 \times \| e^{-it \Lambda}g_{k_2}\|_{L^\infty}\lesssim 2^{2\tilde{\delta} m}\epsilon_0^2.
\ee
The estimate of $\widetilde{Q}_{k,k_1,k_2}^{ 1;1}$ is slightly more delicate. We put the estimate of this term in the proof of  the following Lemma.
\begin{lemma}
Under the bootstrap assumption \textup{(\ref{smallness})}, the following estimate holds, 
\be
\sum_{|k-k_1|\leq 10,  k_2\leq k_1-10} |\widetilde{Q}_{k,k_1,k_2}^{ 1;1}| \lesssim    2^{2\tilde{\delta} m } \epsilon_0^2.
\ee
\end{lemma}
\begin{proof}
Recall (\ref{eqn611}).
For $\widetilde{Q}_{k,k_1,k_2}^{ 1;1} $, we do integration by parts in time. As a result, we have
\[
\widetilde{Q}_{k,k_1,k_2}^{ 1;1}= \sum_{i=1,2} \widehat{Q}_{k,k_1,k_2}^{ 1;i}, \quad \widehat{Q}_{k,k_1,k_2}^{ 1;1}:=\sum_{i=1,2}(-1)^{i-1} \int_{\R^2}  \int_{\R^2 } \overline{\widehat{\Gamma^1 \Gamma^2 g }(t_i, \xi)}  e^{i t_i \Phi^{+, \nu}(\xi, \eta)}   \widehat{  \Gamma^1 \Gamma^2 g }(t_i, \xi-\eta) \]
  \[\times \frac{  a_{k,k_1}( \xi )}{2}  \widehat{g^{\nu}_{k_2}}(t_i,\eta)  d \eta d \xi  +   \int_{t_1}^{t_2} \int_{\R^2}  \int_{\R^2 } \overline{\widehat{\Gamma^1 \Gamma^2 g }(t, \xi)}  e^{i t\Phi^{+, \nu}(\xi, \eta)}   \frac{  a_{k,k_1}( \xi )}{2}   \widehat{ \Gamma^1 \Gamma^2 g }(t, \xi-\eta) \p_t  \widehat{g^{\nu}_{k_2}}(t,\eta)   d \eta d \xi d t,
\]
 \[
 \widehat{Q}_{k,k_1,k_2}^{ 1;2}:=\int_{t_1}^{t_2} \int_{\R^2}  \int_{\R^2 }  e^{i t\Phi^{+, \nu}(\xi, \eta)}   \frac{  a_{k,k_1}( \xi )}{2}  \p_t \big(  \overline{\widehat{\Gamma^1 \Gamma^2 g }(t, \xi)}  \widehat{  \Gamma^1 \Gamma^2 g }(t, \xi-\eta) \big) \widehat{g^{\nu}_{k_2}}(t,\eta)  d \eta d \xi d t.
 \]
From the $L^2-L^\infty$ type bilinear estimate (\ref{bilinearesetimate}) in Lemma \ref{multilinearestimate}, the following estimate holds, 
\[
 \sum_{|k-k_1|\leq 10,k_2\leq k_1-10}  | \widehat{Q}_{k,k_1,k_2}^{ 1;1}| \lesssim \sum_{|k-k_1|\leq 10,k_2\leq k_1-10} \| \Gamma^1 \Gamma^2 g_k\|_{L^2}  \| \Gamma^1 \Gamma^2 g_{k_1}\|_{L^2}\big( \| e^{-it \Lambda}g_{k_2}\|_{L^\infty}\]
\be\label{eqn1180}
+ 2^{m}  \| e^{-it \Lambda}\p_t g_{k_2}(t)\|_{L^\infty} \big)\lesssim 2^{-m/2+\beta m }\epsilon_0^2.
\ee

Recall the estimate (\ref{eqn730}) in Lemma \ref{derivativeL2estimate2}. It motivates us to do the decomposition as follows,
\[
 \widehat{Q}_{k,k_1,k_2}^{ 1;2}= \widehat{Q}_{k,k_1,k_2}^{ 1;2,1} +  \widehat{Q}_{k,k_1,k_2}^{ 1;2,2}, \]
where
\[  \widehat{Q}_{k,k_1,k_2}^{ 1;2,1}:=\int_{t_1}^{t_2} \int_{\R^2}  \int_{\R^2 }  e^{i t\Phi^{+, \nu}(\xi, \eta)}   \frac{ a_{k,k_1}( \xi )}{2} \widehat{g^{\nu}_{k_2}}(t,\eta)  \big[   \overline{ \big(\p_t \widehat{\Gamma^1 \Gamma^2 g_k}(t, \xi)-   \sum_{(k_1',k_2')\in \chi_{k}^2} \sum_{\nu'\in\{+,-\}} \widetilde{B}_{k,k_1',k_2'}^{+, \nu'}(t, \xi) \big)  }   \]
\[\times  \widehat{  \Gamma^1 \Gamma^2 g }(t, \xi-\eta)    +   \overline{\widehat{\Gamma^1 \Gamma^2 g}(t, \xi)} \big(\p_t \widehat{\Gamma^1 \Gamma^2 g_{k_1}}(t, \xi-\eta)-  \sum_{(k_1',k_2')\in \chi_{k_1}^2} \sum_{\nu'\in\{+,-\}} \widetilde{B}_{k_1,k_1',k_2'}^{+, \nu'}(t, \xi-\eta) \big)  d \eta d \xi d t,
  \]
\[ \widehat{Q}_{k,k_1,k_2}^{ 1;2,2}:= \sum_{k_2'\leq k_1'-10, |k_1-k_1'|\leq 10} \sum_{\nu'\in\{+,-\}} \int_{t_1}^{t_2} \int_{\R^2}  \int_{\R^2 }  e^{i t\Phi^{+, \nu}(\xi, \eta)}   \frac{ a_{k,k_1}( \xi )}{2}  \Big[ \widehat{  \Gamma^1 \Gamma^2 g } (t, \xi-\eta)  \widehat{g^{\nu}_{k_2}}(t,\eta)  
\]
\[
\times \overline{e^{it\Phi^{+, \nu}(\xi, \kappa)} \widehat{\Gamma^1 \Gamma^2 g_{k_1'}}(t, \xi-\kappa) \widehat{g^{\nu'}_{k_2'}(t, \kappa)} \tilde{q}_{+, \nu'}(\xi-\kappa, \kappa)}    + \overline{\widehat{\Gamma^1 \Gamma^2 g  }(t, \xi)} \widehat{g^{\nu}_{k_2}}(t,\eta)   
\]
\[
\times e^{i t\Phi^{+, \nu'}(\xi-\eta, \kappa)} \tilde{q}_{+, \nu'}(\xi-\eta-\kappa, \kappa) \widehat{  \Gamma^1 \Gamma^2 g_{k_1'}}(t, \xi-\eta-\kappa) \widehat{g^{\nu'}_{k_2'}(t, \kappa)}  \Big] d \eta d \kappa d \xi d t. 
\]

From estimate (\ref{eqn730}) in Lemma \ref{derivativeL2estimate2} and the $L^2-L^\infty$ type bilinear estimate (\ref{bilinearesetimate}) in Lemma \ref{multilinearestimate}, we have
\[
  \sum_{|k-k_1|\leq 10,k_2\leq k_1-10} |   \widehat{Q}_{k,k_1,k_2}^{ 1;2,1}|\lesssim   \sum_{|k-k_1|\leq 10,k_2\leq k_1-10} \epsilon_1 2^{50\beta m  }\| \Gamma^1 \Gamma^2 g_k\|_{L^2}\| e^{-it \Lambda}g_{k_2}\|_{L^\infty}\lesssim 2^{-m/2 +60\beta m }\epsilon_0^2.
\]

Lastly, we proceed to estimate $  \widehat{Q}_{k,k_1,k_2}^{ 1;2,2}$. To utilize symmetry,  we do change of variables  for the second part of integration as follows $(\xi, \eta, \kappa)\longrightarrow (\xi-\kappa, \eta, -\kappa)$. As a result, we have
\[ \widehat{Q}_{k,k_1,k_2}^{ 1;2,2}:= \sum_{k_2'\leq k_1'-10, |k_1-k_1'|\leq 10} \sum_{\nu'\in\{+,-\}} \int_{t_1}^{t_2} \int_{\R^2}  \int_{\R^2 }  e^{i t\Phi^{+, \nu}(\xi, \eta)-it \Phi^{+, \nu'}(\xi,\kappa)} \Big[ \widehat{  \Gamma^1 \Gamma^2 g }(t, \xi-\eta)   
\]
\[
\times   \widehat{g^{\nu}_{k_2}}(t,\eta)  \frac{ a_{k,k_1}( \xi )}{2}  \overline{  \widehat{\Gamma^1 \Gamma^2 g_{k_1'}}(t, \xi-\kappa) \widehat{g^{\nu'}_{k_2'}(t, \kappa)} \tilde{q}_{+, \nu'}(\xi-\kappa, \kappa)}    +  \frac{ a_{k,k_1}(\xi-\kappa)}{2}  \overline{\widehat{\Gamma^1 \Gamma^2 g }(t, \xi-\kappa)} \widehat{g^{\nu}_{k_2}}(t,\eta)   
\]
\[
\times  \tilde{q}_{+, \nu'}(\xi-\eta , -\kappa) \widehat{  \Gamma^1 \Gamma^2 g_{k_1'}}(t, \xi-\eta ) \widehat{g^{\nu'}_{k_2'}(t, -\kappa)}  \Big] d \eta d \kappa d \xi d t   
\]
\[
= \sum_{k_2'\leq k_1'-10, |k_1-k_1'|\leq 10} \sum_{\nu'\in\{+,-\}} \h \int_{t_1}^{t_2} \int_{\R^2}  \int_{\R^2 }  e^{i t\Phi^{+, \nu}(\xi, \eta)-it \Phi^{+, \nu'}(\xi,\kappa)}   \widehat{  \Gamma^1 \Gamma^2 g }(t, \xi-\eta)   \widehat{g^{\nu}_{k_2}}(t,\eta) 
\]
\be\label{eqn1186}
\times     \overline{  \widehat{\Gamma^1 \Gamma^2 g }(t, \xi-\kappa) } \widehat{g^{\nu'}_{k_2'}(t,- \kappa)} \tilde{r}_{\nu, \nu'}^{k,k_1'}(\xi,\eta, \kappa)  d \eta d \kappa d \xi d t,
\ee
 where
 \[
\tilde{r}_{\nu, \nu'}^{k,k_1'}(\xi,\eta, \kappa):= {a_{k,k_1}(\xi)}  \overline{\tilde{q}_{+, -\nu'}(\xi-\kappa, \kappa) }\psi_{k_1'}(\xi-\kappa)    +  a_{k,k_1}(\xi-\kappa)  \tilde{q}_{+,  \nu'}(\xi-\eta , -\kappa)\psi_{k_1'}(\xi-\eta),
 \]
 \[
\Phi^{+, \nu}(\xi, \eta)- \Phi^{+,\nu'}(\xi, \kappa) = -\Lambda(\xi-\eta)-\nu\Lambda(\eta)+\Lambda(\xi-\kappa) -\nu'\Lambda(\kappa).
\]
 Recall (\ref{eqn939}) and (\ref{eqn932}). It is easy to verify that the following estimate holds from the Lemma \ref{Snorm},
\be\label{eqn783}
\| \tilde{r}_{\nu, \nu'}(\xi,\eta, \kappa) \psi_{k_2'}(\kappa)\psi_{k_1}(\xi-\eta)\psi_{k_2}(\eta) \|_{\mathcal{S}^\infty}\lesssim 2^{\max\{k_2,k_2'\}+ k_1}.
\ee
From (\ref{eqn783}), and multilinear estimate, the following estimate holds, 
\[
 \sum_{|k-k_1|\leq 10,k_2\leq k_1-10} | \widehat{Q}_{k,k_1,k_2}^{ 1;2,2}  | \lesssim \sum_{|k-k_1|\leq 10,k_2\leq k_1-10}  \sum_{k_2'\leq k_1 -10}  2^{m+ \max\{k_2,k_2'\}+ k_1 }  \| \Gamma^1\Gamma^2 g_k\|_{L^2} 
\]
\[\times  \| \Gamma^1\Gamma^2 g_{k_1}\|_{L^2}\| e^{-it\Lambda} g_{k_2}\|_{L^\infty} \| e^{-it \Lambda} g_{k_2'}\|_{L^\infty}\lesssim 2^{-m/2+30\beta m }\epsilon_0^2.
\]
\end{proof}
 \begin{lemma}
Under the bootstrap assumption \textup{(\ref{smallness})}, the following estimate holds, 
\be
\Big|\sum_{|k-k_1|\leq 10,  k_2\leq k_1-10}   {Q}_{k,k_1,k_2}^{ 2}\Big| + \Big|\sum_{|k-k_1|\leq 10,  k_2\leq k_1-10}   {Q}_{k,k_1,k_2}^{ 3}\Big| \lesssim    2^{2\tilde{\delta} m } \epsilon_0^2.
\ee
\end{lemma}
\begin{proof}
We first estimate $Q_{k,k_1,k_2}^2$. Recall (\ref{eqn1230}). From the $L^2-L^\infty$ type bilinear estimate (\ref{bilinearesetimate}) in Lemma \ref{Snorm}, (\ref{eqn939}) and (\ref{eqn932}), the following estimate holds, 
\[
\big| \sum_{k_2\leq k_1+2, |k-k_1|\leq 10} Q_{k,k_1,k_2}^2\big|\lesssim \sum_{|k-k_1|\leq 10 } 2^{m+2k_1}\| \Gamma^1 \Gamma^2 g_{k}\|_{L^2 } \]
\be\label{eqn1194}
\times \| P_{\leq k_1+2}[ \Gamma^1 \Gamma^2 g  ]\|_{L^2 } \| e^{-it \Lambda} g_{k_1}\|_{L^2} \lesssim 2^{2\tilde{\delta}m}\epsilon_0^2.
\ee

Now, we proceed to consider  $Q_{k,k_1,k_2}^3$. Recall (\ref{eqn1160}). By doing integration by parts in $\eta$ many times, we can rule out the case when $\max\{j_1,j_2\} \leq m+k_{1,-}-\beta m$.  From the $L^2-L^\infty$ type bilinear estimate,  the following estimate holds when $\max\{j_1,j_2\} \geq m+k_{1,-}-\beta m$,
\[
\sum_{\max\{j_1,j_2\} \geq m+k_{1,-}-\beta m} \big|Q_{k,k_1,k_2}^{j_1,j_2, 3}\big| \lesssim\sum_{i=1,2} 2^{m+2k_1} \| \Gamma^1\Gamma^2 g_k\|_{L^2}\big( \sum_{j_1\geq \max\{j_2,m+k_{1,-}-\beta m\}} 2^{k_1+j_1} \| g_{k_1,j_1}\|_{L^2} 
\]
\[
\times (\| e^{-it \Lambda} \Gamma^i g_{k_2,j_2}(t)\|_{L^\infty} + \| e^{-it \Lambda}   g_{k_2,j_2}(t)\|_{L^\infty})  +  \sum_{j_2\geq \max\{j_1,m+k_{1,-}-\beta m\}} 2^{k_2+j_2} \| g_{k_2,j_2}\|_{L^2} \]
\be
 \times (\| e^{-it \Lambda} \Gamma^i g_{k_1,j_1}(t)\|_{L^\infty} + \| e^{-it \Lambda}   g_{k_1,j_1}(t)\|_{L^\infty}) \big)\lesssim 2^{-m+20\beta m  -k_2 } \epsilon_0^2\lesssim 2^{-\beta m }\epsilon_0^2.
\ee
\end{proof}

 \subsubsection{The estimate of $P_{k,k_1,k_2}^4$}
  Recall (\ref{eqn1013}). The estimate of ``$P_{k,k_1,k_2}^4$'' can be summarized as the following Lemma,
\begin{lemma}
Under the bootstrap assumption \textup{(\ref{smallness})}, the following estimate holds,
\be
  |   {P}_{k,k_1,k_2}^{4}  |   \lesssim 2^{-\beta  m } \epsilon_0^2.
\ee
\end{lemma}
\begin{proof}
By doing integration by parts in ``$\eta$ '' many times, we can rule out the case when $\max\{j_1,j_2\} \leq m +k_{1,-}- \beta m $. From the $L^2-L^\infty$ type bilinear estimate,  the following estimate holds when $\max\{j_1,j_2\} \geq  m +k_{1,-}- \beta m $.
 \[
 \sum_{  \max\{j_1,j_2\} \geq m +k_{1,-}- \beta m  } |P_{k,k_1,k_2}^{4,j_1,j_2} | \lesssim 2^{m+2k_1} \| \Gamma^1 \Gamma^2 g_k\|_{L^2}\Big[ \sum_{j_1\geq \max\{j_2, m +k_{1,-}- \beta m \}} 2^{-m+ k_1+ j_1+k_2+2j_2} \]
 \[\times \| g_{k_1,j_1}\|_{L^2} \| g_{k_2,j_2}\|_{L^2}  + \sum_{j_2\geq \max\{j_1, m +k_{1,-}- \beta m \}} 2^{-m+ k_1+ 2j_1+k_2+ j_2}  \| g_{k_1,j_1}\|_{L^2} \| g_{k_2,j_2}\|_{L^2}\Big]
 \]
 \be
 \lesssim 2^{-m-k_2+20\beta m}\epsilon_0^2\lesssim 2^{-\beta  m } \epsilon_0^2.
 \ee
\end{proof}
\subsection{The $Z_2$ norm estimate of cubic terms}
  Recall (\ref{eqn440}) and \emph{the fact that $k_3\leq k_2\leq k_1$}. For any $\Gamma_\xi^1, \Gamma^2_\xi \in \{ \hat{L}_\xi, \hat{\Omega}_\xi \}$, we have
  \[
  \Gamma^1_\xi \Gamma^2_\xi \Lambda_{3}[\p_t \widehat{g}(t, \xi)]\psi_k(\xi) =\sum_{\tau, \kappa, \iota\in\{+,-\}} \sum_{k_3\leq k_2\leq k_1} \sum_{i=1,2,3,4}T^{\tau, \kappa,\iota,i}_{k,k_1,k_2,k_3}(t, \xi), 
  \]
  \[
T^{\tau, \kappa,\iota,i}_{k,k_1,k_2,k_3}(t, \xi) 
  = \sum_{j_1\geq -k_{1,-}, j_2\geq -k_{2,-}, j_3\geq -k_{3,-} } T_{k,k_1, j_1,k_2, j_2,k_3,j_3 }^{\tau, \kappa,\iota, i}(t, \xi), \quad i\in\{ 3,4 \},
  \]
  where
  \[
T^{\tau, \kappa,\iota,1}_{k,k_1,k_2,k_3}(t, \xi)=\int_{\R^2}\int_{\R^2}  e^{i t\Phi^{\tau, \kappa,\iota}(\xi, \eta,\sigma)}  \tilde{d}_{\tau, \kappa, \iota}(\xi-\eta, \eta-\sigma, \sigma) \Gamma^1_\xi \Gamma^2_\xi \widehat{g_{k_1 }^{\tau}}(t, \xi-\eta)\]
  \be\label{eqq3}
\times \widehat{g^{\kappa}_{k_2 }}(t, \eta-\sigma) \widehat{g^{\iota}_{k_3 }}(t,  \sigma) \psi_{k}(\xi) d \eta d\sigma,
  \ee
\[
T^{\tau, \kappa,\iota,2}_{k,k_1,k_2,k_3}(t, \xi)= \int_{\R^2}\int_{\R^2} e^{i t\Phi^{\tau, \kappa,\iota}(\xi, \eta,\sigma)}\big[\Gamma^1_\xi \Gamma^2_\xi \big( \tilde{d}_{\tau, \kappa, \iota}(\xi-\eta, \eta-\sigma, \sigma)\big) \widehat{g_{k_1  }^{\tau}}(t, \xi-\eta)   
 \]
 \be\label{eqn1604}
 +\sum_{\{l,n\}=\{1,2\}}\Gamma_\xi^l \tilde{d}_{\tau, \kappa, \iota}(\xi-\eta, \eta-\sigma, \sigma) \Gamma^n_\xi  \widehat{g_{k_1  }^{\tau}}(t, \xi-\eta) \big]\widehat{g^{\kappa}_{k_2  }}(t, \eta-\sigma) \widehat{g^{\iota}_{k_3  }}(t,  \sigma) \psi_k(\xi)d \eta d\sigma,
  \ee
\[
T^{\tau, \kappa,\iota,3}_{k,k_1,j_1, k_2,j_2, k_3,j_3}(t, \xi)= \int_{\R^2}\int_{\R^2} e^{i t\Phi^{\tau, \kappa,\iota}(\xi, \eta,\sigma)}   e^{i t\Phi^{\tau, \kappa,\iota}(\xi, \eta,\sigma)}  it \big(\Gamma^l_\xi \Phi^{\tau, \kappa,\iota}(\xi, \eta,\sigma)\big)   \]
\be\label{eqn1605}
\times  \Gamma^n_\xi  \big(\tilde{d}_{\tau, \kappa, \iota}(\xi-\eta, \eta-\sigma, \sigma)\widehat{g_{k_1,j_1}^{\tau}}(t, \xi-\eta)  \big) \widehat{g^{\kappa}_{k_2,j_2}}(t, \eta-\sigma)  \widehat{g^{\iota}_{k_3,j_3}}(t,  \sigma)   \psi_k(\xi)d \eta d\sigma,
\ee
\[
T^{\tau, \kappa,\iota,4}_{k,k_1,j_1, k_2,j_2, k_3,j_3}(t, \xi)= -\int_{\R^2}\int_{\R^2}e^{i t\Phi^{\tau, \kappa,\iota}(\xi, \eta,\sigma)}  t^2 \big(  \Gamma^1_\xi  \Phi^{\tau, \kappa,\iota}(\xi, \eta,\sigma)\Gamma^2_\xi  \Phi^{\tau, \kappa,\iota}(\xi, \eta,\sigma) \big)  \]
\be\label{eqn1609}
  \times  \tilde{d}_{\tau, \kappa, \iota}(\xi-\eta, \eta-\sigma, \sigma) \widehat{g_{k_1,j_1}^{\tau}}(t, \xi-\eta)  \widehat{g^{\kappa}_{k_2,j_2}}(t, \eta-\sigma) \widehat{g^{\iota}_{k_3,j_3}}(t,  \sigma)\psi_k(\xi) d \eta d\sigma.
\ee
Therefore, we have
\be\label{eqn1690}
\textup{Re}\big[\int_{t_1}^{t_2}\int_{\R^2} \overline{ \Gamma^1_\xi \Gamma^2_\xi \widehat{g}(t, \xi) } \Gamma^1_\xi \Gamma^2_\xi \Lambda_{3}[\p_t \widehat{g}(t, \xi)]\psi_k(\xi) d \xi d t\big] =\sum_{\tau, \kappa, \iota\in\{+,-\}} \sum_{k_3\leq k_2\leq k_1}  \sum_{i=1,2,3,4}   \textup{Re}\big[ T_{k,k_1,k_2,k_3}^{\tau, \kappa,\iota, i}],
\ee
\be\label{eqq4}
 T_{k,k_1,k_2,k_3}^{\tau, \kappa,\iota, i}= \int_{t_1}^{t_2} \int_{\R^2} \int_{\R^2} \overline{ \Gamma^1_\xi \Gamma^2_\xi \widehat{g}(t, \xi) } T^{\tau, \kappa,\iota,i}_{k,k_1,k_2,k_3}(t, \xi) d \xi d t . 
\ee
 The main goal of this subsection is to prove the following proposition,
\begin{proposition}\label{cubicproposition1}
 Under the bootstrap assumption \textup{(\ref{smallness})}, the following estimates hold,
 \be\label{eqq1}
 \sup_{t_1, t_2\in[2^{m-1}, 2^m]} \big|\sum_{k } \textup{Re}\big[\int_{t_1}^{t_2}\int_{\R^2} \overline{ \Gamma^1_\xi \Gamma^2_\xi \widehat{g}(t, \xi) } \Gamma^1_\xi \Gamma^2_\xi \Lambda_{3}[\p_t \widehat{g}(t, \xi)]\psi_k(\xi) d \xi d t\big] \big|\lesssim 2^{2\tilde{\delta}m}\epsilon_0^2,
 \ee
 \be\label{eqq2}
 \sup_{t \in[2^{m-1}, 2^m]} \big\| \sum_{k_3\leq k_2\leq k_1}  T^{\tau, \kappa,\iota,i}_{k,k_1,k_2,k_3}(t, \xi)\big\|_{L^2} \lesssim 2^{-m + \tilde{\delta}m } \big(1+ 2^{2\tilde{\delta}m +k+5k_{+}} \big)\epsilon_0.
 \ee
 
\end{proposition}

Firstly, we  rule out the very high frequency case and the very low frequency case. 

 Very similar to what we did in the estimate of quadratic terms (see (\ref{e789})), we do integration by parts in $\eta$ to move the derivatives $\nabla_\eta$ of $\nabla_\xi \widehat{g}_{k_1}(t, \xi-\eta) = -\nabla_\eta \widehat{g}_{k_1}(t, \xi-\eta) $ around such that there is no derivatives in front of $\widehat{g}_{k_1}(t, \xi-\eta) $. As a result, the following estimate holds from the $L^2-L^\infty-L^\infty$ type trilinear estimate (\ref{trilinearesetimate}) in Lemma \ref{multilinearestimate} and the $L^\infty\rightarrow L^2$ type Sobolev embedding,
 \[
\sum_{i=1,2,3,4}\|T_{k,k_1,k_2,k_3}^{\tau, \kappa,\iota, i}(t, \xi) \|_{L^2} \lesssim   2^{2m+2k_1 + 6k_{1,+}}   \| g_{k_1}(t)\|_{L^2} 2^{k_2 +k_3} \| g_{k_3}(t)\|_{L^2}  \big(2^{-2k_{2 }} \| g_{k_2}(t)\|_{L^2}
 \]
\be\label{eqn1601}
 + 2^{-k_{2 }} \| \nabla_\xi \widehat{g}_{k_2}(t)\|_{L^2} +  \| \nabla_\xi^2 \widehat{g}_{k_2}(t)\|_{L^2} \big) \lesssim 2^{2m+ \beta m  - (N_0-20) k_{1,+}  } \epsilon_0.
\ee
From above estimate, we can rule out the case when $k_1\geq 4 \beta m $.  It remains to consider the case when $k_1\leq 4\beta m $.

Now, we proceed to rule out the very low frequency case. If either $k\leq -2m$ or $k_3\leq -3m-30\beta m $, then the following estimate holds from the $L^2-L^\infty-L^\infty$ type trilinear estimate,
\[
\sum_{i=1,2,3,4} \|T_{k,k_1,k_2,k_3}^{\tau, \kappa,\iota, i}(t, \xi) \|_{L^2} \lesssim  (1+2^{2m+2k})    2^{ k+ k_3+  4k_{1,+}}\big( 2^{2k_1}\|\nabla_\xi^2 \widehat{g}_{k_1}(t, \xi)\|_{L^2} 
\]
\[
+ 2^{ k_1  }\|\nabla_\xi  \widehat{g}_{k_1}(t, \xi)\|_{L^2} + \| g_{k_1}(t)\|_{L^2} \big)   \|   g_{k_2}(t)\|_{L^2}  \| g_{k_3}(t)\|_{L^2} \lesssim 2^{-m-\beta m}\epsilon_0.
\]
 Therefore, for the rest of this subsection, we restrict ourself to the case when $k, k_1$, $k_2$, and $k_3$ are  in the  range as follows,
\be\label{restrictedrangeforcubicterms}
-3m-30\beta m \leq k_3\leq k_2\leq k_1\leq 4\beta m,\quad -2m \leq k \leq 4\beta m.
\ee
 
Recall (\ref{eqn1604}).  From the $L^2-L^\infty-L^\infty$ type trilinear estimate (\ref{trilinearesetimate}) in Lemma \ref{multilinearestimate}, the following estimate holds, 
 \[
\| T_{k,k_1,  k_2, k_3 }^{\tau, \kappa,\iota, 2}(t, \xi) \|_{L^2}\lesssim 2^{ 2k_1+4k_{1,+}} \big( \| e^{-it \Lambda} g_{k_1}\|_{L^\infty} + \sum_{i=1,2} \| e^{-it \Lambda} \Gamma^i g_{k_1}\|_{L^\infty}  \big)\| e^{-it\Lambda} g_{k_2}\|_{L^\infty} 
 \]
 \[
\times  \| g_{k_3}\|_{L^2}\lesssim 2^{-3m/2 + 50 \beta m }\epsilon_0^2, \Longrightarrow | T_{k,k_1,  k_2, k_3 }^{\tau, \kappa,\iota, 2}|\lesssim 2^{- m/2 + 50 \beta m }\epsilon_0^2.
 \]

Since there are only at most ``$m^4$ ''cases  in the range (\ref{restrictedrangeforcubicterms}),  to prove (\ref{eqq1}) and (\ref{eqq2}), it would be sufficient to prove the following estimate for fixed $k, k_1,k_2,k_3$ in the range (\ref{restrictedrangeforcubicterms}),
 \be\label{desiredcubicZ2estimate}
   \sum_{i=1, 3,4}|\textup{Re}\big[ T_{k,k_1,k_2,k_3}^{\tau, \kappa,\iota, i}\big] |   \lesssim 2^{ 3 \tilde{\delta} m/2 } \epsilon_0^2,\quad    \sum_{i=1, 3,4}\|  T_{k,k_1,k_2,k_3}^{\tau, \kappa,\iota, i}(t, \xi)\|_{L^2} \lesssim  2^{-m + \tilde{\delta}m /2} \big(1+ 2^{2\tilde{\delta}m + k+5k_{+}} \big)\epsilon_0.
 \ee
\begin{lemma}
For fixed $k_1,k_2,k_3$ in the range \textup{ (\ref{restrictedrangeforcubicterms}) },  our desired estimate (\ref{desiredcubicZ2estimate}) holds if moreover $k_2\leq k_1-10$.
\end{lemma}
\begin{proof}
Recall the normal form transformation that we did in subsection \ref{goodvariable}, see  (\ref{eqn200}) and (\ref{eqn1643}). For the case we are considering, which is $k_2\leq k_1-10$, 
\emph{we have ``$\tau=+$'' and $|\nabla_\xi \Phi^{+, \kappa,\iota}(\xi, \eta,\sigma)|\lesssim 2^{k_2}. $} 

We first consider $T_{k_1,k_2,k_3}^{+, \kappa,\iota,1 }$ and $T_{k_1,k_2,k_3}^{+, \kappa,\iota,1 }(t, \xi)$. Recall (\ref{eqq3}) and (\ref{eqq4}). The following estimate
holds from the $L^2-L^\infty-L^\infty$ type trilinear estimate (\ref{trilinearesetimate}) in Lemma
\ref{multilinearestimate},
\[
\| T_{k_1,k_2,k_3}^{+, \kappa,\iota,1 }(t, \xi)\|_{L^2} \lesssim 2^{2k_1+ 4k_{1,+}} \big(2^{2k_1}\| \nabla_\xi^2 \widehat{g_{k_1}}(t, \xi)\|_{L^2}+ 2^{ k_1}\| \nabla_\xi  \widehat{g_{k_1}}(t, \xi)\|_{L^2} + \| g_{k_1}(t)\|_{L^2} \big)\]
\[\times \| e^{-it \Lambda} g_{k_2}\|_{L^\infty} \| e^{-it \Lambda} g_{k_3}\|_{L^\infty} \lesssim 2^{-m + 7\tilde{\delta}m/3 +  k+5k_{+}}\epsilon_0. 
\]

 Since the general $L^\infty$ decay rate is slightly slower than $t^{-1/2 }$, a rough $L^2-L^\infty-L^\infty$ is not sufficient to close the estimate of $T_{k_1,k_2,k_3}^{\tau, \kappa,\iota,1 }$. An essential ingredient is to utilize symmetry such that one of the inputs, which is putted in $L^\infty$, has a derivative in front.  We decompose $T_{k_1,k_2,k_3}^{+, \kappa,\iota,1 }$ into three parts as follows,
\[
T_{k,k_1,k_2,k_3}^{+, \kappa,\iota,1}=\sum_{i=1,2,3}  T_{k,k_1,k_2,k_3}^{+, \kappa,\iota,1;i},\quad  T_{k,k_1,k_2,k_3}^{+, \kappa,\iota,1;1}=    \int_{t_1}^{t_2} \int_{\R^2} \int_{\R^2} \overline{  \widehat{\Gamma^1  \Gamma^2  g}(t, \xi) } e^{i t\Phi^{+, \kappa,\iota}(\xi, \eta,\sigma)}  e(\xi) \]
\[\times  \widehat{ \Gamma^1   \Gamma^2   g_{  }^{  }}(t, \xi-\eta) \psi_{k}(\xi) \psi_{k_1}(\xi-\eta)  \widehat{g^{\kappa}_{k_2 }}(t, \eta-\sigma) \widehat{g^{\iota}_{k_3 }}(t,  \sigma) d \eta d\sigma d \xi d t, 
\]
\[
T_{k,k_1, k_2, k_3   }^{+, \kappa,\iota,1;2}=  \int_{t_1}^{t_2} \int_{\R^2} \int_{\R^2} \overline{  \widehat{\Gamma^1  \Gamma^2  g}(t, \xi) } e^{i t\Phi^{+, \kappa,\iota}(\xi, \eta,\sigma)}\big( \tilde{d}_{+, \kappa, \iota}(\xi, \eta, \sigma)-   e(\xi) \big) \]
\[\times   \psi_{k}(\xi) \widehat{ \Gamma^1   \Gamma^2   g_{k_1 }^{ }}(t, \xi-\eta) \widehat{g^{\kappa}_{k_2 }}(t, \eta-\sigma) \widehat{g^{\iota}_{k_3 }}(t,  \sigma) d \eta d\sigma d \xi d t,
\]
\[
T_{k,k_1, k_2, k_3  }^{+, \kappa,\iota,1;3}=   \int_{t_1}^{t_2} \int_{\R^2} \int_{\R^2} \overline{  \widehat{\Gamma^1  \Gamma^2  g}(t, \xi) } e^{i t\Phi^{+, \kappa,\iota}(\xi, \eta,\sigma)} \tilde{d}_{+, \kappa, \iota}(\xi-\eta, \eta-\sigma, \sigma) \widehat{g^{\kappa}_{k_2 }}(t, \eta-\sigma) \]
\[\times \psi_{k}(\xi)   \widehat{g^{\iota}_{k_3 }}(t,  \sigma)   \big(\Gamma^1_\xi \Gamma^2_\xi \widehat{g_{k_1 }^{ }}(t, \xi-\eta) -\widehat{ \Gamma^1   \Gamma^2   g_{k_1 }^{ }}(t, \xi-\eta)\big)d \eta d\sigma d \xi d t,
\]
where $e(\xi)$ is defined in (\ref{e88988}).  After switching the role of $\xi$ and $\xi-\eta$ inside $T_{k_1,k_2,k_3}^{+, \kappa,\iota,1;1}$, we have
\[
\textup{Re}[T_{k_1, k_2 ,k_3 }^{+, \kappa,\iota,1;1} ] = \textup{Re}[\tilde{T}_{k_1, k_2 ,k_3 }^{+, \kappa,\iota} ], \quad  \tilde{T}_{k_1,k_2, k_3 }^{+, \kappa,\iota}:= \frac{1}{2}\int_{t_1}^{t_2} \int_{\R^2} \int_{\R^2} \overline{  \widehat{\Gamma^1  \Gamma^2  g}(t, \xi) } e^{i t\Phi^{+, \kappa,\iota}(\xi, \eta,\sigma)} \]
\[\times \tilde{d}_{k, k_1}(\xi, \eta, \sigma) \widehat{ \Gamma^1   \Gamma^2   g_{ }^{ }}(t, \xi-\eta)  \widehat{g^{\kappa}_{k_2 }}(t, \eta-\sigma) \widehat{g^{\iota}_{k_3 }}(t,  \sigma) d \eta d\sigma d \xi d t,
\]
where
\[
 \tilde{d}_{ k, k_1}(\xi, \eta, \sigma):=  {e(\xi)}  \psi_{k_1}(\xi-\eta)\psi_k(\xi) +  \overline{ e(\xi-\eta)}     \psi_{k_1}(\xi)\psi_{k}(\xi-\eta).
\]
From Lemma \ref{Snorm} and (\ref{e88988}), we have
\be\label{eqn1688}
\|   \tilde{d}_{k, k_1}(\xi, \eta, \sigma) \psi_{k_2}(\eta-\sigma) \psi_{k_3}(\sigma)\|_{\mathcal{S}^\infty } \lesssim 2^{\max\{k_2,k_3\} + k_1 + 4k_{1,+}}.
\ee
From  (\ref{eqn1688}),  (\ref{e88987}) and the $L^2-L^\infty -L^\infty$ type trilinear estimate (\ref{bilinearesetimate}) in Lemma \ref{multilinearestimate}, we have
\[
 \sum_{i=1,2,3} |\textup{Re}\big[T_{k,k_1 ,k_2 ,k_3 }^{+, \kappa,\iota,1; i}\big]|  \lesssim 2^{m+k_1 +4k_{1,+} + \max\{k_2,k_3\}} \| \Gamma^1 \Gamma^2 g_{k_1}\|_{L^2}\big(2^{2k_1}\|\nabla_\xi^2\widehat{g}_{k_1}(t,\xi)\|_{L^2}   \]
\[  + 2^{ k_1}\|\nabla_\xi  \widehat{g}_{k_1}(t,\xi)\|_{L^2}  + \| g_{k_1}(t)\|_{L^2}\big) \| e^{-it \Lambda } g_{k_2}\|_{L^\infty} \| e^{-it\Lambda} g_{k_3}\|_{L^\infty} \lesssim 2^{-m/2+50\beta m}\epsilon_0^2.
\]

Now, we proceed to estimate $ T_{k_1,k_2,k_3}^{+, \kappa,\iota, i}$ and $ T_{k_1,k_2,k_3}^{+, \kappa,\iota, i}(t, \xi)$, $i\in\{ 3,4\}$. Recall (\ref{eqn1605}), (\ref{eqn1609}), and (\ref{eqq4}).  By doing integration by parts in ``$\eta$'' many times, we can rule out the case when $\max\{j_1,j_2\}\leq m +k_{1,-} -\beta m $.  From the $L^2-L^\infty-L^\infty$ type trilinear estimate (\ref{trilinearesetimate}) in Lemma \ref{multilinearestimate}, the following estimate holds when $\max\{j_1,j_2\}\geq m +k_{1,-} -\beta m $, 
\[
 \sum_{i=3,4}\big\|\sum_{\max\{j_1,j_2\}\geq m +k_{1,-} -\beta m  }T_{k,k_1, j_1,k_2,j_2, k_3,j_3 }^{+, \kappa,\iota,i}(t, \xi) \big\|_{L^2} \lesssim 2^{ m+3k_1+k_2 + 4k_{1,+}} 	 \| e^{-it \Lambda} g_{k_3}\|_{L^\infty}\]
\[\times  \Big[\sum_{j_1\geq \max\{j_2,  m +k_{1,-} -\beta m\}}  (2^{k_1+j_1}+ (1+2^{m+k_1+k_2}))   \| g_{k_1,j_1}\|_{L^2} \| e^{-it \Lambda} g_{k_2,j_2}\|_{L^\infty}+ \sum_{j_2\geq \max\{j_1,  m +k_{1,-} -\beta m\}}  \]
\[ \big( 2^{k_1}\|e^{-it \Lambda} \mathcal{F}^{-1}[\nabla_\xi \widehat{g_{k_1,j_1}}(t, \xi)]\|_{L^\infty}+ (1+ 2^{m +k_1+k_2})\|e^{-it \Lambda} g_{k_1,j_1}\|_{L^\infty}\big)    \|  g_{k_2,j_2}\|_{L^2}  \Big]
\lesssim 2^{-3m/2+50\beta m }\epsilon_0 .
\]
Note that above estimate is very sufficient for the estimate of $ T_{k_1,k_2,k_3}^{+, \kappa,\iota, i}$, $i\in\{ 3,4\}$.

\end{proof}
\begin{lemma}\label{cubicZ2estimatepart2}
For fixed $k_1,k_2,k_3$ in the range \textup{ (\ref{restrictedrangeforcubicterms}) },  our desired estimate \textup{(\ref{desiredcubicZ2estimate}) }holds if either $|k_1-k_2|\leq 10$ and $k_3\leq k_2-10$ or $|k_1-k_2|\leq 10$, $|k_3-k_2|\leq 10, k\leq k_1-10$.
\end{lemma}

\begin{proof} 
The estimate of $T_{k,k_1,k_2,k_3}^{\tau, \kappa,\iota, 1}(t,\xi)$ is straightforward. As $|k_1-k_2|\leq 10$, the size of symbol compensates the decay rate of $e^{-it \Lambda} g_{k_2}(t)$.  From the $L^2-L^\infty-L^\infty$ type trilinear estimate (\ref{trilinearesetimate}) in Lemma \ref{multilinearestimate}, the following estimate holds
\[
\| T_{k,k_1,k_2,k_3}^{\tau, \kappa,\iota, 1}(t, \xi)\|_{L^2}\lesssim   2^{  2k_1 + 4k_{1,+}} \big( 2^{2k_1}  \| \nabla_\xi^2 \widehat{g}_{k_1}(t, \xi)\|_{L^2}  + 2^{ k_1} \| \nabla_\xi \widehat{g}_{k_1}(t, \xi)\|_{L^2}  + \| g_{k_1}(t)\|_{L^2}\big)
\]
\be\label{eqn1698}
\times \| e^{-it \Lambda} g_{k_2}\|_{L^\infty}\| e^{-it \Lambda} g_{k_3}\|_{L^\infty} \lesssim 2^{-3m/2 +50\beta m }\epsilon_0.
\ee

Now, we proceed to estimate $T_{k,k_1,k_2,k_3}^{\tau, \kappa,\iota, 3}(t, \xi)$ and $T_{k,k_1,k_2,k_3}^{\tau, \kappa,\iota, 4}(t, \xi)$. Recall (\ref{eqn1605}) and (\ref{eqn1609}).   Note that, if either $|k_1-k_2|\leq 10$ and $k_3\leq k_2-10$ or $|k_1-k_2|\leq 10$, $|k_3-k_2|\leq 10, k\leq k_1-10$, we know that $\nabla_\eta \Phi^{\tau, \kappa, \iota}(\xi, \eta, \kappa)$ has a lower bound, which is $2^{k-4\beta m }$. To take advantage of this fact, we  do integration by parts in ``$\eta$'' many times to rule out the case when $\max\{j_1,j_2\} \leq m + k_{-}-5\beta m $. From the $L^2-L^\infty-L^\infty$ type trilinear estimate (\ref{trilinearesetimate}) in Lemma \ref{multilinearestimate}, the following estimate holds when $\max\{j_1,j_2\}\geq m + k_{-}- 5\beta m $, 
 \[
\sum_{i=3,4} \big\| \sum_{ \max\{j_1,j_2\}\geq  m + k_{-}- 5\beta m} T_{k,k_1, j_1,k_2,j_2, k_3,j_3 }^{+, \kappa,\iota,i}  (t, \xi) \big\|_{L^2} 
\lesssim   2^{ m+ 2k+  k_1 + 4k_{1,+}}\| e^{-it \Lambda} g_{k_3}\|_{L^\infty}    \]
\[\times \Big( \sum_{j_2\geq \max\{j_1 ,  m +k_{-}-5 \beta m\}}  \big( (1+2^{ m+ 2k_1  } ) \| e^{-it \Lambda} g_{k_1,j_1}\|_{L^\infty}+ 2^{    k_1  }  \| e^{-it \Lambda}\mathcal{F}^{-1}[\nabla_\xi \widehat{g_{k_1,j_1}}(t, \xi)]	\|_{L^\infty} \big)       \| g_{k_2,j_2}\|_{L^2} \]
 \be\label{e1435}
  +     \sum_{j_1\geq \max\{j_2,    m + k_{-}- 5\beta m\}}  \big( 2^{ m+ 2 k_1}+ 2^{    j_1 +   k_1} \big)  \| e^{-it \Lambda} g_{k_2,j_2}\|_{L^\infty} \| g_{k_1,j_1}\|_{L^2}  \Big)  \lesssim 2^{-3m/2 + 50\beta m }\epsilon_0 .
\ee
From (\ref{eqn1698}) and (\ref{e1435}), it is easy to see our desired estimate for $T_{k,k_1,k_2,k_i}^{\tau, \kappa,\iota, 3}$ , $i\in \{1,3,4\}, $  in (\ref{desiredcubicZ2estimate}) 
holds. 
\end{proof}
 
\begin{lemma}
For fixed $k_1,k_2,k_3$ in the range \textup{ (\ref{restrictedrangeforcubicterms}) },  our desired estimate \textup{(\ref{desiredcubicZ2estimate}) }holds if  $|k_1-k_2|\leq 10$,  $|k_3-k_2|\leq 10$, and $|k-k_1|\leq 10$.
\end{lemma}
\begin{proof}
 Since we still have $|k_1-k_2|\leq 10$, the estimate of  $T_{k,k_1,k_2,k_3}^{\tau, \kappa,\iota, 1}(t, \xi)$ is same as what we did in (\ref{eqn1698}), which is derived from the $L^2-L^\infty-L^\infty$ type trilinear estimate. We omit details here.

The estimate of   $T_{k,k_1,k_2,k_3}^{\tau, \kappa,\iota, 3}(t, \xi)$ and $T_{k,k_1,k_2,k_3}^{\tau, \kappa,\iota, 4}(t, \xi)$  is more delicate. We will handle it in the similar way as we did in the $Z_1$-norm estimate of cubic terms in subsubsection \ref{allcomparable}. We can still either do integration by parts in  ``$\eta$'' or in ``$\sigma$'' as long as either $\xi-\eta$ is not close to $\eta-\sigma$ or $\eta-\sigma$ is not close to $\sigma-\eta$. 

Note that, we already canceled out the case when $ (\tau, \kappa, \iota) \in \mathcal{S}_4 $ (see (\ref{eqn720}))  and  $(\xi-\eta, \eta-\sigma, \sigma)$ is very close to $(\xi/3,\xi/3,\xi/3)$ in the normal form transformation. Therefore, for the case   when $(\tau, \kappa, \iota)\in \mathcal{S}_4$,  we only have to consider the case when     $(\xi-\eta, \eta-\sigma, \sigma)$ is not  close to $(\xi/3,\xi/3,\xi/3)$, in which case either $\nabla_\eta \Phi^{\tau, \kappa, \iota}(\xi, \eta, \kappa)$ or $\nabla_\sigma \Phi^{\tau, \kappa, \iota}(\xi, \eta, \kappa)$
has a good lower bound, which allows us to do integration by parts either in $\eta$ or in $\sigma$. The estimate of this case is
 similar to and also easier than the estimate of (\ref{e1435}) in the proof of Lemma \ref{cubicZ2estimatepart2}. We omit details here.

Now, we   focus on the case when $(\tau, \kappa, \iota) \in \mathcal{S}_i$, $i\in \{1,2,3\}.$ By the symmetries between   inputs, it would be sufficient to consider the case when $(\tau, \kappa, \iota)\in \mathcal{S}_1$, i.e., $(\tau, \kappa, \iota)\in \{ (+,-,-), (-,+,+)  \}$. After changing the variables as follows $(\xi, \eta, \sigma)\longrightarrow ( \xi, 2\xi+\eta+\sigma, \xi+\sigma) $, we have the following
decomposition for
$i \in \{3,4\},$
\[
T_{k,k_1,  k_2, k_3}^{\tau, \kappa,\iota, i}(t, \xi):= \sum_{l_1, l_2\geq \bar{l}_{\tau}} H^{l_1,l_2,\tau,i-2}_{}(t, \xi), \quad H^{l_1,l_2,\tau,i-2 }(t,\xi) = \sum_{j_1\geq -k_{1,-}, j_2\geq -k_{2,-}}   H_{j_1,j_2}^{l_1,l_2,\tau,i-2}(t,\xi) , \]
\[  H_{j_1,j_2}^{l_1,l_2,\tau,1 }(t, \xi):=   \sum_{\{l,n\}=\{1,2\}}   \int_{\R^2} \int_{\R^2}  e^{i t\widetilde{\Phi}^{\tau, \kappa,\iota}(\xi, \eta,\sigma)}  it \big(\Gamma^l_\xi \Phi^{\tau, \kappa,\iota}(\xi, 2\xi+\eta+\sigma,\xi+\sigma)\big)  \]
\[  
\times \Gamma^n_\xi \big(\tilde{d}_{\tau, \kappa, \iota}(-\xi-\eta-\sigma, \xi+\eta , \xi+\sigma)   \widehat{g_{k_1,j_1}^{\tau}}(t, -\xi-\eta-\sigma) \big) \widehat{g^{\kappa}_{k_2,j_2}}(t, \xi+\eta ) \]
\[\times \widehat{g^{\iota}_{k_3,j_3}}(t,  \xi+\sigma){\varphi}_{l_1;\bar{l}_{\tau}}(\eta) {\varphi}_{l_2;\bar{l}_{\tau}}(\sigma) \psi_k(\xi) d \eta d\sigma  ,
\]
\[  H_{j_1,j_2}^{l_1,l_2,\tau,2 }:=  -   \int_{\R^2} \int_{\R^2}    e^{i t \widetilde{\Phi}^{\tau, \kappa,\iota}(\xi, \eta, \sigma)} t^2 \big(  \Gamma^1_\xi  \Phi^{\tau, \kappa,\iota}(\xi, 2\xi+\eta+\sigma,\xi+\sigma) \]
\[
\times \Gamma^2_\xi  \Phi^{\tau, \kappa,\iota}(\xi, 2\xi+\eta+\sigma,\xi+\sigma) \big) \tilde{d}_{\tau, \kappa, \iota}(-\xi-\eta-\sigma, \xi+\eta , \xi+\sigma)     \widehat{g_{k_1, j_1}^{\tau}}(t, -\xi-\eta-\sigma)  \]
\[ \times   \widehat{g^{\kappa}_{k_2,j_2 }}(t, \xi+ \eta )  \widehat{g^{\iota}_{k_3 }}(t,  \xi+ \sigma) \psi_k(\xi)   {\varphi}_{l_1;\bar{l}_{\tau}}(\eta) {\varphi}_{l_2;\bar{l}_{\tau}}(\sigma)  d \eta d\sigma,
\]
where $\widetilde{\Phi}^{\tau, \kappa, \iota}(\xi, \eta, \sigma)$ is defined in (\ref{eqn724}), the cutoff function $\varphi_{l;\bar{l}}(\cdot)$ is defined in (\ref{thresholdcutoff}) and the thresholds are chosen as follows, $\bar{l}_{+}:=k_1-10$ and $\bar{l}_{-}:= -m/2+ 10\delta  m+ k_{1,+}/2$. 

  $\oplus$  If $\tau=+$, i.e., $(\tau, \kappa, \iota)=(+,-,-)$. \quad  Recall the normal form transformation that we did in (\ref{goodvariable}), see (\ref{normalformatransfor}) and (\ref{eqn200}). For the case we are considering, $(\tau, \kappa, \iota)\in \widetilde{S}$, we have already canceled out the case  when    $\max\{l_1, l_2\} = \bar{l}_{+}$. Hence it would be sufficient   to  consider the case when $\max\{l_1, l_2\} > \bar{l}_{+}$.   Due to the symmetry between inputs, we assume that $l_2=\max\{l_1,l_2\}$. As $l_2 > \bar{l}_{+}$, we can take the advantage of the fact that ``$\nabla_\eta \widetilde{\Phi}^{\tau, \kappa, \iota}(\xi, \eta, \sigma)$'' is big by doing integration by parts in $\eta$. From (\ref{eqn723}), we can rule out the case when $\max\{j_1,j_2\}\leq m + k_{-}- \beta m $ by doing integration by parts in ``$\eta$'' many times. 

  From the $L^2-L^\infty-L^\infty$ type trilinear estimate (\ref{trilinearesetimate}) in Lemma \ref{multilinearestimate}, the following estimates holds when $\max\{j_1,j_2\} \geq m +k_{-}- \beta m$,
\[
\sum_{\max\{j_1,j_2\} \geq  m +k_{-}- \beta m } \sum_{i=1,2}\|H_{j_1,j_2}^{l_1,l_2,\tau,i}(t, \xi)\|_{L^2}  \lesssim    2^{ m+ 4k_{1}+  4k_{1,+}}  \Big(  \sum_{j_1\geq \max\{j_2, m +k_{-}-\beta m\}} \big(2^{m + 2k_1}+ 2^{ k_1+  j_1  }\big) 
\]
\[
\times\| g_{k_1,j_1}(t)\|_{L^2} \| e^{-it\Lambda} g_{k_2,j_2}(t)\|_{L^\infty} + \sum_{j_2\geq \max\{j_1,m+k_{-}-\beta m \}}  \big( (2^{m+2k_1} + 1) \| e^{-it\Lambda} g_{k_1,j_1}(t)\|_{L^\infty}
\]
\be\label{eqq10}
 + 2^{k_1}\| e^{-it\Lambda}\mathcal{F}^{-1}[\nabla_\xi \widehat{ g_{k_1,j_1}}(t, \xi)]\|_{L^\infty} \big)  \| g_{k_2,j_2}(t)\|_{L^2} \Big)\| e^{-it \Lambda}  g_{k_3}(t)\|_{L^\infty}  
   \lesssim 2^{-2m+50\beta m } \epsilon_0.
\ee

  $\oplus$  If $\tau=-$, i.e., $(\tau,\kappa, \iota)=(-,+,+)$. \quad Note that, estimate (\ref{eqn739}) holds for the case we are considering.
 We first consider the case when $\max\{l_1,l_2\} > \bar{l}_{-}$.  Same as before, due to the symmetry between inputs, we assume that $l_2=\max\{l_1,l_2\}$. Recall (\ref{eqn723}), by doing integration by parts in ``$\eta$'' many times, we can rule out the case when $\max\{j_1,j_2\}\leq m+l_2 -4\beta m $.  From the $L^2-L^\infty-L^\infty$ type trilinear estimate, the following estimate holds when $\max\{j_1,j_2\} \geq m+l_2 -4\beta m $.
\[
\sum_{\max\{j_1,j_2\} \geq m+l_2 -4\beta m} \sum_{i=1,2} \| H_{j_1,j_2}^{l_1,l_2,-, i}(t, \xi)\|_{L^2} \lesssim    2^{  m+3k_1+l_2+  4k_{1,+}}   \Big( \sum_{ j_1\geq \max\{j_2, m+l_2 -4\beta m\}}   \big(2^{m+ k_1+l_2} \]
\[   + 2^{j_1 + k_1}\big) \| g_{k_1,j_1}(t)\|_{L^2} \| e^{-it \Lambda}g_{k_2,j_2}(t)\|_{L^\infty} + \sum_{ j_2\geq \max\{j_1, m+l_2 -4\beta m\}} \big((1+2^{ m+k_1+l_2}) \| e^{-it \Lambda}g_{k_1,j_1}(t)\|_{L^\infty}    \]
\be\label{eqq11} 
 + 2^{  k_1}   \| e^{-it\Lambda}\mathcal{F}^{-1}[\nabla_\xi \widehat{ g_{k_1,j_1}}(t, \xi)]\|_{L^\infty} \big)     \| g_{k_2,j_2}(t)\|_{L^2}   \Big)\| e^{-it \Lambda}  g_{k_3}(t)\|_{L^\infty}
   \lesssim 2^{-2m +50\beta m }\epsilon_0^2.
\ee

Lastly, we consider the case when $\max\{l_1,l_2\}= \bar{l}_{-}=-m/2+ 10\delta  m+ k_{1,+}/2 $. Recall (\ref{eqn739}).  For this case, we use the volume of support in $\eta$ and $\sigma$. As a result, the following estimate holds, 
\[
\sum_{i=1,2}\| H_{ }^{\bar{l}_{-},\bar{l}_{-}, i}(t, \xi) \|_{L^2} \lesssim   2^{   4k_{1,+}} \big(2^{2m + 6\bar{l}+ 4k_1} +2^{ m + 5\bar{l}+ 4k_1}\big)\big(2^{-k_1}\|   g_{k_1}(t)\|_{L^2}  +\|  \nabla_\xi \widehat{ g_{k_1}}(t, \xi)(t)\|_{L^2}\big) \]
\be\label{eqq14}
  \times \|   g_{k_2}(t)\|_{L^1}\| g_{k_3}(t)\|_{L^1}\lesssim 2^{-m + 100\delta  m}\epsilon_0^2.
\ee
 From (\ref{eqq10}), (\ref{eqq11}), and (\ref{eqq14}), it is sufficient to derive our desired estimate (\ref{desiredcubicZ2estimate}). 
\end{proof}
 \subsection{The $Z_2$ norm estimate of the quartic terms}

 Recall (\ref{eqn441}). For any $\Gamma^1_\xi, \Gamma^2_\xi \in\{ \hat{L}_\xi, \hat{\Omega}_\xi\},$ we have 
 \[
\Gamma^1_\xi \Gamma^2_\xi   \Lambda_{4}[\p_t \widehat{g}(t, \xi) ] \psi_{k}(\xi) = \sum_{ \mu_1, \mu_2,\nu_1, \nu_2\in\{+,-\} }\sum_{k_4\leq k_3\leq k_2\leq k_1}\sum_{i=1,2,3,4}    K_{k,k_1,k_2,k_3,k_4}^{ \mu_1, \mu_2,\nu_1, \nu_2,i}(t, \xi),
 \]
 \[
 K_{k,k_1,k_2,k_3,k_4}^{ \mu_1, \mu_2,\nu_1, \nu_2,i}(t, \xi)= \sum_{j_1\geq -k_{1,-}, j_2\geq -k_{2,-} } K_{k,k_1,j_1,k_2,j_2,k_3,k_4}^{ \mu_1, \mu_2,\nu_1, \nu_2,i}(t, \xi), \quad i\in\{ 3,4 \},
\]	
where
 \[ K_{k,k_1,k_2,k_3,k_4}^{\mu_1, \mu_2,\nu_1, \nu_2, 1}(t, \xi):=    \int_{\R^2}  \int_{\R^2} \int_{\R^2}    e^{i t \Phi^{\mu_1, \mu_2,\nu_1, \nu_2}(\xi, \eta, \sigma,\kappa)}   \tilde{e}_{\mu_1, \mu_2,\nu_1, \nu_2}(\xi-\eta, \eta-\sigma, \sigma-\kappa, \kappa)  \]
\be\label{eqn1676}
\times\Gamma^1_\xi \Gamma^2_\xi \widehat{g_{k_1 }^{\mu_1}}(t, \xi-\eta) \widehat{g^{\mu_2}_{k_2 }}(t, \eta-\sigma) \widehat{g^{\nu_1}_{k_3 }}(t,  \sigma-\kappa) \widehat{g^{\nu_2}_{k_4 }}(t,   \kappa) \psi_{k}(\xi)  d \kappa d\sigma d \eta  ,
\ee
 
\[  K_{k,k_1,k_2,k_3,k_4}^{ \mu_1, \mu_2,\nu_1, \nu_2,2}(t, \xi):=   \int_{\R^2} \int_{\R^2}  \int_{\R^2}     e^{i t \Phi^{\mu_1, \mu_2,\nu_1, \nu_2}(\xi, \eta, \sigma,\kappa)} \psi_k(\xi) \big[\Gamma^1_\xi \Gamma^2_\xi \big( \tilde{e}_{\mu_1, \mu_2,\nu_1, \nu_2}(\xi-\eta, \eta-\sigma, \sigma-\kappa, \kappa)\big)   \]
\[
 \times   \widehat{g_{k_1  }^{\mu_1}}(t, \xi-\eta) +\sum_{\{l,n\}=\{1,2\}}\Gamma_\xi^l  \tilde{e}_{\mu_1, \mu_2,\nu_1, \nu_2}(\xi-\eta, \eta-\sigma, \sigma-\kappa, \kappa)    \Gamma^n_\xi  \widehat{g_{k_1  }^{\mu_1}}(t, \xi-\eta)  \big]
\]
\be\label{eqn1684} 
\times  \widehat{g^{\mu_2}_{k_2  }}(t, \eta-\sigma)  \widehat{g^{\nu_1}_{k_3 }}(t,  \sigma-\kappa) \widehat{g^{\nu_2}_{k_4 }}(t,   \kappa) d \kappa d\sigma d \eta  ,
\ee

\[
K_{k,k_1,j_1,k_2,j_2,k_3,k_4}^{ \mu_1, \mu_2,\nu_1, \nu_2,3}(t, \xi):=   \sum_{\{l,n\}=\{1,2\}}   \int_{\R^2} \int_{\R^2}\int_{\R^2}  \psi_k(\xi)   e^{i t \Phi^{\mu_1, \mu_2,\nu_1, \nu_2}(\xi, \eta, \sigma,\kappa)}   it \big(\Gamma^l_\xi  \Phi^{\mu_1, \mu_2,\nu_1, \nu_2}(\xi, \eta, \sigma,\kappa)\big)  \]
\be 
 \times  \Gamma^n_\xi \big(   \tilde{e}_{\mu_1, \mu_2,\nu_1, \nu_2}(\xi-\eta, \eta-\sigma, \sigma-\kappa, \kappa)  \widehat{g_{k_1,j_1}^{\mu_1}}(t, \xi-\eta) \big)\widehat{g^{\mu_2}_{k_2,j_2}}(t, \eta-\sigma)  \widehat{g^{\nu_1}_{k_3 }}(t,  \sigma-\kappa) \widehat{g^{\nu_2}_{k_4 }}(t,  \kappa) d \kappa d\sigma d \eta ,
\ee
\[
K_{k,k_1,j_1,k_2,j_2,k_3,k_4}^{ \mu_1, \mu_2,\nu_1, \nu_2,4}(t, \xi):= -  \int_{\R^2} \int_{\R^2}\int_{\R^2} \psi_k(\xi)    e^{i t \Phi^{\mu_1, \mu_2,\nu_1, \nu_2}(\xi, \eta, \sigma,\kappa)}  t^2   \Gamma^1_\xi  \Phi^{\mu_1, \mu_2,\nu_1, \nu_2}(\xi, \eta, \sigma,\kappa)  \]
\[  \times \Gamma^2_\xi  \Phi^{\mu_1, \mu_2,\nu_1, \nu_2}(\xi, \eta, \sigma,\kappa)    \tilde{e}_{\mu_1, \mu_2,\nu_1, \nu_2}(\xi-\eta, \eta-\sigma, \sigma-\kappa, \kappa)  \widehat{g_{k_1,j_1}^{\mu_1}}(t, \xi-\eta)  \widehat{g^{\mu_2}_{k_2,j_2}}(t, \eta-\sigma)   \]
\be 
 \times  \widehat{g^{\nu_1}_{k_3 }}(t,  \sigma-\kappa) \widehat{g^{\nu_2}_{k_4 }}(t,   \kappa)  d \kappa d \eta d\sigma .
\ee

The main goal of this subsection is to prove the following proposition
\begin{proposition}\label{quarticZ2norm3}
 Under the bootstrap assumption (\textup{\ref{smallness}}), the following estimates hold,
 \be\label{eqq20}
 \sup_{t_1, t_2\in[2^{m-1}, 2^m]}  \big| \sum_{k } \int_{t_1}^{t_2}\int_{\R^2} \overline{ \Gamma^1_\xi \Gamma^2_\xi \widehat{g}(t, \xi) } \Gamma^1_\xi \Gamma^2_\xi   \Lambda_{4}[\p_t \widehat{g}(t, \xi) ] \psi_{k}(\xi)d \xi d t \big|\lesssim 2^{2\tilde{\delta}m }\epsilon_0^2.
 \ee
  \be\label{eqq21}
  \sup_{t\in[2^{m-1}, 2^m ]} \|  \Gamma^1_\xi \Gamma^2_\xi   \Lambda_{4}[\p_t \widehat{g}(t, \xi) ]  \|_{L^2} \lesssim 2^{-m +\tilde{\delta}m }\epsilon_0^2.
 \ee
\end{proposition}

As usual,  we first rule out the very high   frequency case and the very low frequency case.   Same as before, we move the derivative $\nabla_\xi = -\nabla_\eta$ in front of $\widehat{g}_{k_1}(t, \xi-\eta)$ around by doing integration by parts in $\eta$ such that there is no derivative in front of $\widehat{g}_{k_1}(t, \xi-\eta)$. As a result, the following estimate holds, 
\[
\sum_{i=1,2,3,4} \|K_{k,k_1,k_2,k_3,k_4}^{\mu_1, \mu_2,\nu_1, \nu_2, i}(t,\xi) \|_{L^2} \lesssim (1+2^{2m + 2k }) 2^{ 6k_{1,+}  }  \| g_{k_1}(t)\|_{L^2} \big( \| \nabla_\xi^2  \widehat{g_{k_2}}(t, \xi)\|_{L^2} +2^{-k_2}\| \nabla_\xi  \widehat{g_{k_2}}(t, \xi)\|_{L^2} 
\]
\be\label{eqn1740}
+ 2^{-2k_2}\| g_{k_2}(t)\|_{L^2} \big)  2^{k_3+k_4} \|g_{k_3}(t)\|_{L^2}  \|g_{k_4}(t)\|_{L^2}  \lesssim 2^{ 2m+\beta m -(N_0-10) k_{1,+}  }\epsilon_0^2.
\ee
Hence, we can rule out the case when $k_1\geq 4\beta m $. It remains to consider the case when $k_1\leq 4\beta m $.

If either  $k_4\leq -3m -30\beta m $ or $k \leq -2m$, then the following estimate holds,  
 \[
\sum_{i=1,2,3,4}  \|K_{k,k_1,k_2,k_3,k_4}^{\mu_1, \mu_2,\nu_1, \nu_2, 1}(t, \xi)\|_{L^2} \lesssim    (1+2^{2m+2k}) 2^{ k+k_4 + 4k_{1,+}}  \big( 2^{2k_1} \| \nabla_\xi^2\widehat{g}_{k_1}(t, \xi)\|_{L^2}  +  2^{ k_1} \| \nabla_\xi \widehat{g}_{k_1}(t, \xi)\|_{L^2}
 \]
\[
 +   \|   {g}_{k_1}(t )\|_{L^2} \big) \| e^{-it \Lambda} g_{k_2}(t )\|_{L^\infty} \|  g_{k_3}(t )\|_{L^2} \|   g_{k_4}(t )\|_{L^2} \lesssim 2^{-m-\beta m}\epsilon_0^2.
\] 

Now it would be sufficient to consider   fixed
 $k,k_1,k_2,k_3$, and $k_4$   in the following range, 
\be\label{Z2normestimatequarticrange}
-3m-30\beta m  \leq k_4\leq k_3\leq k_2\leq k_1\leq 4\beta m, \quad -2m \leq k \leq 3\beta m .
\ee
From the    $L^2-L^\infty-L^\infty-L^\infty$ type estimate, the following estimate holds,
\[
\sum_{i=1,2}  \|K_{k,k_1,k_2,k_3,k_4}^{\mu_1, \mu_2,\nu_1, \nu_2, i}(t, \xi)\|_{L^2} \lesssim 2^{ 2k_1+4k_{1,+}} \big( 2^{2k_1}\| \nabla_\xi^2 \widehat{g}_{k_1}(t, \xi)\|_{L^2} + 2^{k_1}\| \nabla_\xi \widehat{g}_{k_1}(t, \xi)\|_{L^2} + \|   {g}_{k_1}(t, \xi)\|_{L^2}  \big)
\]
\[
\times   \| e^{-i t\Lambda} g_{k_2}(t)\|_{L^\infty}  \| e^{-i t\Lambda} g_{k_3}(t)\|_{L^\infty}\| e^{-i t\Lambda} g_{k_4}(t)\|_{L^\infty}\lesssim 2^{-3m /2 + 50\beta m }\epsilon_0^2.
\]
\begin{lemma}
Under the bootstrap assumption \textup{(\ref{smallness})}, the following estimate holds for fixed  $k,k_1,k_2,k_3$, and $k_4$  in the range \textup{(\ref{Z2normestimatequarticrange})},
\be\label{Z2normestimatequartic}
\sum_{i= 3,4} \|K_{k,k_1,k_2,k_3,k_4}^{\mu_1, \mu_2,\nu_1, \nu_2, i}(t, \xi)\|_{L^2} \lesssim 2^{-m+\tilde{\delta} m /2}\epsilon_0^2.
\ee
\end{lemma}
\begin{proof}
  We first consider the case when $k_1-10 \leq k_3$.  For this case, the following estimate holds from  the    $L^2-L^\infty-L^\infty-L^\infty$ type estimate, the following estimate holds, 
\[
\sum_{i= 3,4} \|K_{k,k_1,k_2,k_3,k_4}^{\mu_1, \mu_2,\nu_1, \nu_2, i}(t, \xi)\|_{L^2} \lesssim 2^{m+4k_1+4k_{1,+}} \big[  \big( \| e^{-it \Lambda} g_{k_1}(t)\|_{L^\infty} + 2^{k_1} \| e^{-it \Lambda} \mathcal{F}^{-1}[\nabla_\xi \widehat{g}_{k_1}(t, \xi)]\|_{L^\infty}  \big) 
\]
\be\label{eqq30}
 + 2^{m+2k_{1}} \| e^{-it \Lambda} g_{k_1}(t)\|_{L^\infty} \big]  \| e^{-i t\Lambda} g_{k_2}(t)\|_{L^\infty} \| e^{-i t\Lambda} g_{k_3}(t)\|_{L^\infty} \| g_{k_4}(t)\|_{L^2}\lesssim 2^{-m+ \tilde{\delta} m /2  }\epsilon_0^2.
 \ee

Now, we proceed to consider the case when $k_3\leq k_1-10$. Recall (\ref{quarticsymbolnormalform}) and (\ref{quarticsymbolnormalform2}). Because of  the construction of normal form transformation we did in subsection \ref{goodvariable}, we know that the case when $\eta$ is very close to $\xi/2$ and $|\sigma|, |\kappa|\ll |\xi|$ is canceled out. As a result, we know that 
``$\nabla_\eta  \Phi^{\mu_1, \mu_2,\nu_1, \nu_2}(\xi, \eta, \sigma,\kappa)$'' has a lower bound, which is $2^{k-k_{1,+} }$.   To take advantage of this fact, we do integration by parts in ``$\eta$'' many times to rule out the case when $\max\{j_1,j_2\}\leq m+ k_{-}-5\beta m$. 
From    the    $L^2-L^\infty-L^\infty-L^\infty$ type estimate, the following estimate holds when  when $\max\{j_1,j_2\}\geq m+ k_{-}-5\beta m$,
\[
\sum_{i=3,4}\sum_{\max\{j_1,j_2\}\geq m+ k_{-}-5\beta m} \|K_{k,k_1,j_1,k_2,j_2,k_3,k_4}^{ \mu_1, \mu_2,\nu_1, \nu_2,i}(t, \xi)\|_{L^2} \lesssim \sum_{ j_1 \geq \max\{j_2,m+ k_{-}-5\beta m \} } 2^{ m + k+k_2+2k_1  + 4k_{1,+}}   \]
\[  \times\big( 2^{m+k+k_1}  +2^{k_1+ j_1} \big) \| g_{k_1,j_1}\|_{L^2} \| e^{-it \Lambda} g_{k_2,j_2}\|_{L^\infty} \| e^{-it \Lambda} g_{k_3}\|_{L^\infty} \| e^{-it \Lambda} g_{k_4}\|_{L^\infty} \]
\[+ \sum_{ j_2 \geq \max\{j_1,m+ k_{-}-5\beta m  \} } \ 2^{ m + k+k_2+2k_1  + 4k_{1,+}}  \big( 2^{m+k+k_1}  \|e^{-it \Lambda}  g_{k_1,j_1}\|_{L^\infty}     \]
\be\label{eqq31}
+ 2^{k_1} \|e^{-it \Lambda}  \mathcal{F}^{-1}[\nabla_\xi \widehat{g_{k_1,j_1}}(t, \xi)]\|_{L^\infty}  \big) 2^{k_2  }\|  g_{k_2,j_2}\|_{L^2} \| e^{-it \Lambda} g_{k_3}\|_{L^\infty}  \|   g_{k_4}\|_{L^2} \lesssim 2^{-3m/2+50\beta m}\epsilon_0^2.
\ee
From (\ref{eqq30}) and (\ref{eqq31}), it is easy to see our desired estimate (\ref{Z2normestimatequartic}) holds. 
\end{proof}

\begin{lemma}\label{Z2normcubicandhigher}
Under the bootstrap assumption \textup{(\ref{smallness})}, the following estimates hold for any  $t \in [2^{m-1}, 2^{m}]$ and any $\Gamma^1_\xi, \Gamma^2_\xi \in \{\hat{L}_\xi, \hat{\Omega}_\xi\},$
\be\label{eqq425}
 \| \Gamma^1 \Gamma^2 \Lambda_{\geq 3}[\p_t \widehat{g_k}(t, \xi)] \|_{L^2}\lesssim 2^{-m + \tilde{\delta}m } \big(1+ 2^{2\tilde{\delta}m + k+5k_{+}} \big)\epsilon_0,
\ee
\end{lemma}
\begin{proof}
The desired estimate (\ref{eqq425}) follows easily from estimate (\ref{eqq2}) in Proposition (\ref{cubicproposition1}), estimate (\ref{eqq21}) in Proposition (\ref{quarticZ2norm3}), and estimate (\ref{eqnj878}) in Lemma \ref{remaindertermweightednorm}.
\end{proof}
\section{Fixed time weighted norm  estimates and the estimate of   remainder terms}\label{reminderestimatefixed}
\begin{lemma}\label{derivativeL2estimate1}
Under the bootstrap assumption \textup{(\ref{smallness})}, the following estimates hold,
\be\label{eqn751}
\sup_{t\in[2^{m-1}, 2^m]} \| \p_t \widehat{g}_k(t, \xi)-  \sum_{\mu, \nu\in\{+,-\}} \sum_{(k_1,k_1)\in \chi_k^1}B_{k,k_1,k_2}^{\mu, \nu}(t, \xi)  \|_{L^2} \lesssim 2^{-21m/20}\epsilon_0,
\ee
\be\label{eqn52}
\sup_{t\in[2^{m-1}, 2^m]} \| \p_t \widehat{g}_k(t, \xi)\|_{L^2}   \lesssim \min\{ 2^{-2m-k + 2\tilde{\delta} m }, 2^{-m+\delta m }\}\epsilon_0 +2^{-21 m/20}\epsilon_0,
\ee
\be\label{L2cubicandhigher}
\sup_{t\in[2^{m-1}, 2^m]} \| \Lambda_{\geq 3}[\p_t \widehat{g}_k(t, \xi)]\|_{L^2}   \lesssim 2^{-3m/2+\beta m }\epsilon_0,
\ee
where $B_{k,k_1,k_2}^{\mu, \nu}(t, \xi)$ is defined in \textup{(\ref{eqn650})}.
\end{lemma}
\begin{proof}
For the cubic and higher order terms, after putting the input with the smallest frequency in $L^2$ and all other inputs in $L^\infty$, the decay rate of $L^2$ norm is at least $2^{-3m/2+\beta m } $, which gives us our desired estimate (\ref{L2cubicandhigher}). Hence to prove (\ref{eqn751}) and (\ref{eqn52}), 
  we only have to consider the quadratic terms ``$B^{\mu, \nu}_{k,k_1,k_2}(t, \xi)$''. Recall (\ref{eqn650}), after doing spatial  localizations for two inputs, we have
\[
   B_{k,k_1,k_2}^{\mu, \nu}(t, \xi)= \sum_{j_1 \geq -k_{1,-}, j_2\geq -k_{2,-}}   B^{\mu, \nu,j_1,j_2	}_{k,k_1, k_2 }(t, \xi), 
\] 
\[
B^{\mu, \nu,j_1,j_2	}_{k,k_1, k_2 }(t, \xi) = \int_{\R^2} e^{i t\Phi^{\mu, \nu}(\xi ,\eta)} \tilde{q}_{\mu, \nu}(\xi, \eta) \widehat{g_{k_1,j_1}^{\mu}}(t,  \xi-\eta)\widehat{g_{k_2,j_2}^{\nu}}(t, \eta) \psi_{k}(\xi) d \eta.
\]

 We first consider the case when $|k_1-k_2|\leq 10 $. From the $L^2-L^\infty$ type bilinear estimate (\ref{bilinearesetimate}) in Lemma \ref{multilinearestimate}, we have
 \be\label{eqn700}
\sum_{|k_1-k_2|\leq 10}\|  B^{\mu, \nu}_{k,k_1,k_2}(t, \xi)\|_{L^2}\lesssim  \sum_{|k_1-k_2|\leq 10} 2^{2k_1} \| g_{k_1}\|_{L^2} \| e^{-it \Lambda} g_{k_2}(t)\|_{L^\infty} \lesssim 2^{-m +\delta m}\epsilon_0. 
 \ee
Meanwhile, after doing integration by parts in ``$\eta$'' once, the following estimate also holds,
\[
\sum_{|k_1-k_2|\leq 10} \|  B^{\mu, \nu}_{k,k_1,k_2}(t, \xi) \|_{L^2}\lesssim \sum_{|k_1-k_2|\leq 10} 2^{2k_1} 2^{-m-k+k_{1,+}}\big( \| e^{-it \Lambda}  g_{k_1}\|_{L^\infty}  + \| e^{-it \Lambda}  g_{k_2}\|_{L^\infty} \big) \]
\be\label{eqn701}
\times \big( \| \nabla_\xi \widehat{g}_{k_1}(t, \xi)\|_{L^2} + \| \nabla_\xi \widehat{g}_{k_2}(t, \xi)\|_{L^2} + 2^{-k_1}\| g_{k_1}(t)\|_{L^2} \big) \lesssim 2^{-2m-k + 2\tilde{\delta} m } \epsilon_0.
\ee

Now, we consider the case when $ k_2\leq k_1-10$ and $k_{1,-} +k_2\leq -18 m /19 $. From estimate  (\ref{eqn400}) in Lemma \ref{Linftyxi}, we have
\[
\sum_{k_{1,-} +k_2\leq -18m/19  } \| B^{\mu, \nu}_{k,k_1,k_2}(t, \xi) \| \lesssim \sum_{k_{1,-} +k_2\leq -18m/19    }  \| g_{k_1}(t)\|_{L^2} \min\{ 2^{2k_1+k_2}  \| g_{k_2}(t)\|_{L^2},
\]
\[
 2^{ k_1+3k_2} \| \widehat{g}_{k_2}(t, \xi)\|_{L^\infty_\xi} + 2^{ 2k_1+2k_2} \| \widehat{\textup{Re}[v]}_{}(t, \xi)\psi_{k_2}(\xi) \|_{L^\infty_\xi}  \}\lesssim \sum_{k_{1,-} +k_2\leq -18m/19  }2^{3\tilde{\delta} m }\min\{2^{2 k_{1,-} + k_2   },  \]
 \[   2^{ 2k_2  }\big(2^{ k_{1,-}+ k_2+ m}+   2^{2k_{1,-}+2k_2+2m}\big)  \} \lesssim 2^{-21m/20}\epsilon_0.
\]

Lastly, we consider the case when $k_2\leq k_1-10$ and $k_{1,-}+k_2\geq  -18 m /19$. After doing integration by parts in ``$\eta$'' many times, we can rule out the case when $\max\{j_1,j_2\} \leq m +k_{1,-}-\beta m .$  From the $L^2-L^\infty$ type bilinear estimate (\ref{bilinearesetimate}) in Lemma \ref{multilinearestimate}, the following estimate holds when $\max\{j_1,j_2\} \geq m +k_{1,-}- \beta m$,
\[
\sum_{\max\{j_1,j_2\} \geq m +k_{1,-}- \beta m}\|  B^{\mu, \nu,j_1,j_2	}_{k,k_1, k_2 }(t, \xi)\|_{L^2}\lesssim \sum_{j_1 \geq \max\{j_2,m +k_{1,-}- \beta m \}} 2^{2k_1} \| e^{-it \Lambda} g_{k_2,j_2}\|_{L^\infty} \| g_{k_1,j_1}\|_{L^2} 
\]
\be\label{eqn705}
+ \sum_{j_2\geq \max\{j_1, m +k_{1,-}- \beta m\}} 2^{2k_1} \| e^{-it \Lambda} g_{k_1,j_1}\|_{L^\infty} \| g_{k_2,j_2}\|_{L^2}  \lesssim 2^{-3m-2k_2-k_{1,-} + 3\beta m }\epsilon_0 \lesssim 2^{-21 m/20}\epsilon_0.
\ee
Combining estimates (\ref{eqn700}), (\ref{eqn701}), and (\ref{eqn705}), it is easy to see our desired estimate (\ref{eqn52}) holds. 
\end{proof}

\begin{lemma}\label{derivativeL2estimate2}
Under the bootstrap assumption \textup{(\ref{smallness})}, the following estimate holds for any $t\in [2^{m-1}, 2^m]$,
\be\label{eqn730}
\|   \p_t \widehat{\Gamma_1 \Gamma_2 g_k}(t, \xi) -\sum_{\nu\in\{+,-\}}\sum_{(k_1,k_2)\in \chi_k^2}\widetilde{B}_{k,k_1 ,k_2   }^{+, \nu }(t, \xi)  \|_{L^2 } \lesssim  2^{-m + \tilde{\delta}m +\delta m} \big(1+ 2^{2\tilde{\delta}m +k+5k_{+}} \big)\epsilon_0,  
\ee
where $\Gamma_1, \Gamma_2\in \{L, \Omega\}$ and $\widetilde{B}_{k,k_1 ,k_2   }^{+, \nu }(t, \xi) $ is defined as follows,
\be\label{eqn892}
\widetilde{B}_{k,k_1 ,k_2   }^{+, \nu }(t, \xi):= \int_{\R^2} e^{it \Phi^{+, \nu}(\xi, \eta)} \tilde{q}_{+,\nu}(\xi-\eta, \eta) \widehat{\Gamma_1 \Gamma_2 g_{k_1}}(t, \xi-\eta)\widehat{g^{\nu}_{k_2}}(t, \eta)\psi_k(\xi) d \eta.
\ee
\end{lemma}
\begin{proof}
From (\ref{eqq2}) in Proposition \ref{cubicproposition1}, (\ref{eqq21}) in Proposition \ref{quarticZ2norm3}, and  (\ref{eqnj878}) in Lemma \ref{remaindertermweightednorm}, we know  that all terms except quadratic terms inside $ \p_t \widehat{\Gamma_1 \Gamma_2 g_k}(t, \xi) $   already satisfy the desired estimate \ref{eqn730}.  Hence, we only need to estimate the quadratic terms. 

We first consider the case when  $(k_1,k_2)\in \chi_k^1$ ,  we have
\[
\Gamma^1_\xi \Gamma^2_\xi B^{\mu, \nu}_{k_1,k_2}(t, \xi)= \sum_{i=1,2,3} K_{k_1,k_2}^{\mu, \nu,1;i}, \]
\[K_{k_1,k_2}^{\mu, \nu,1;1}:=   \int_{\R^2} e^{i t\Phi^{\mu, \nu}(\xi, \eta)}\Gamma^1_\xi \Gamma^2_\xi \big(\tilde{q}_{\mu, \nu}(\xi-\eta, \eta) \widehat{g_{k_1}^{\mu}}(t, \xi-\eta) \big)  \widehat{g^{\nu}_{k_2}}(t, \eta) d \eta,
\]
\[
K_{k_1,k_2}^{\mu, \nu,1;2}:=   \sum_{{l,m}=\{1,2\}} \int_{\R^2} e^{i t\Phi^{\mu, \nu}(\xi, \eta)}it \big(\Gamma^l_\xi \Phi^{\mu, \nu}(\xi, \eta)\big) \Gamma^n_\xi \big(\tilde{q}_{\mu, \nu}(\xi-\eta, \eta) \widehat{g_{k_1}^{\mu}}(t, \xi-\eta) \big)  \widehat{g^{\nu}_{k_2}}(t, \eta) d \eta
\]
\[
K_{k_1,k_2}^{\mu, \nu,1;3}:= -\int_{\R^2} e^{i t\Phi^{\mu, \nu}(\xi, \eta)} t^2 \big(  \Gamma^1_\xi \Phi^{\mu, \nu}(\xi, \eta) \Gamma^2_\xi \Phi^{\mu, \nu}(\xi, \eta) \big) \tilde{q}_{\mu, \nu}(\xi-\eta, \eta) \widehat{g_{k_1}^{\mu}}(t, \xi-\eta) \big)  \widehat{g^{\nu}_{k_2}}(t, \eta) d \eta.
\]
From the $L^2-L^\infty$ type bilinear estimate (\ref{bilinearesetimate}) in Lemma \ref{multilinearestimate}, we have
\[
\sum_{|k_1-k_2|\leq 10}  \| K_{k_1,k_2}^{\mu, \nu,1;1}\|_{L^2}\lesssim 2^{2k_1} \big( 2^{2 k}  \| \nabla_\xi^2 \widehat{g}_{k_1}(t, \xi)\|_{L^2} + 2^{ k}  \| \nabla_\xi   \widehat{g}_{k_1}(t, \xi)\|_{L^2}+    \|   \widehat{g}_{k_1}(t, \xi)\|_{L^2} \big) \]
\[\times \| e^{-it \Lambda} g_{k_2}(t)\|_{L^\infty} \lesssim 2^{-m +\tilde{\delta} m }\epsilon_0.
\]

For $K_{k_1,k_2}^{\mu, \nu,1;2}$, we do integration by parts in ``$\eta $ '' once. Meanwhile, for $K_{k_1,k_2}^{\mu, \nu,1;3}$, we do integration by parts in ``$\eta$ '' twice.  As a result, we have
\[
\sum_{|k_1-k_2|\leq 10} \sum_{i=2,3} \|  K_{k_1,k_2}^{\mu, \nu,1;i}\|_{L^2}   \lesssim \sum_{|k_1-k_2|\leq 10} 2^{2k_1} \big( \sum_{i=0,1,2} 2^{ i k_1}  \| \nabla_\xi^i \widehat{g}_{k_1}(t, \xi)\|_{L^2} + 2^{ i k_1}  \| \nabla_\xi^i \widehat{g}_{k_2}(t, \xi)\|_{L^2} \big)  \]
\[\times \big( \| e^{-it \Lambda} g_{k_1}(t)\|_{L^\infty}+ \| e^{-it \Lambda} g_{k_2}(t)\|_{L^\infty} \big)  + \sum_{|k_1-k_2|\leq 10}\sum_{j_1\geq j_2} 2^{4k_1}   \| e^{-it \Lambda} \mathcal{F}^{-1}[ \nabla_\xi \widehat{g}_{k_2,j_2}(t,\xi)]\|_{L^\infty}\]
\[ \times   \| \nabla_\xi \widehat{g}_{k_1, j_1}(t, \xi)\|_{L^2} +  \sum_{|k_1-k_2|\leq 10}\sum_{j_2\geq j_1} 2^{4k_1}   \| e^{-it \Lambda} \mathcal{F}^{-1}[ \nabla_\xi \widehat{g}_{k_1,j_1}(t,\xi)]\|_{L^\infty}  \| \nabla_\xi \widehat{g}_{k_2, j_2}(t, \xi)\|_{L^2}
\]
\[
\lesssim 2^{-m +\tilde{\delta} m }\epsilon_0 +\sum_{j_1} 2^{-m +4k_1 +2j_1} \| g_{k_1,j_1}(t)\|_{L^2} \sum_{j_2\geq j_1} 2^{ j_2 } \| g_{k_2,j_2}(t)\|_{L^2}\]
\[
+ \sum_{j_2} 2^{-m +4k_1 +2j_1} \| g_{k_2,j_2}(t)\|_{L^2} \sum_{j_1\geq j_2} 2^{ j_1 } \| g_{k_1,j_1}(t)\|_{L^2} \lesssim 2^{-m +\tilde{\delta}m}\epsilon_0.
 \]

 Now, we proceed to consider the case when $(k_1,k_2)\in \chi_k^2$. We split it into two cases based on the size of $k_1+k_2$. If $ k_1+k_2\leq -18 m /19$, the following estimate holds from estimates (\ref{eqn400})   in Lemma \ref{Linftyxi}, 
 \[
 \sum_{i=1,2,3}\|\sum_{\nu\in\{+,-\}}  K_{k_1,k_2}^{\mu, \nu,1;i}\|_{L^2}+ \|\sum_{\nu\in\{+,-\}}\widetilde{B}_{k,k_1 ,k_2   }^{+, \nu }(t, \xi)\|_{L^2}  \]
 \[\lesssim \Big( \big(\sum_{i=0,1,2}2^{ i k_1}  \| \nabla_\xi^i \widehat{g}_{k_1}(t, \xi)\|_{L^2}  \big) + 2^{m+k_1+k_2}\big(\sum_{i=0,1 }2^{ i k_1}  \| \nabla_\xi^i \widehat{g}_{k_1}(t, \xi)\|_{L^2}  \big)+ 2^{2m+2k_1+2k_2} \| g_{k_1}(t)\|_{L^2}  \Big) \]
 \[ \times  \min\{  2^{k_1+ 3k_2}\| \widehat{g}_{k_2}(t)\|_{L^\infty_\xi}  +2^{2k_1+ 2k_2}\| \widehat{\textup{Re}[v]}(t, \xi)\psi_{k_2}(\xi)\|_{L^\infty_\xi},  2^{2k_1+k_2} \| g_{k_2}\|_{L^2}\}
 \]
 \[
\lesssim ( 2^{k_1 + \tilde{\delta} m}+  2^{2m + 3k_1+2k_2 + 2\tilde{\delta} m })\min\{2^{k_1+k_2},   2^{ 3k_2+m}+ 2^{k_1+4k_2 +2m}  \}\epsilon_0\lesssim 2^{-m-\beta m }\epsilon_0.
 \]

Now, we will rule out the case when $k_1$ is relatively large. Same as before, we move the derivative $\nabla_\xi = -\nabla_\eta$ in front of $\widehat{g}_{k_1}(t, \xi-\eta)$ around by doing integration by parts in $\eta$ such that there is no derivative in front of $\widehat{g}_{k_1}(t, \xi-\eta)$. As a result, the following estimate holds when $k_1+k_2\geq -18m/19$ and $k_1 \geq 5\beta m $,
\[
\sum_{k_1+k_2 \geq-18m/19, k_1 \geq 5\beta m}\sum_{i=1,2,3}\|  K_{k_1,k_2}^{\mu, \nu,1;i}\|_{L^2}+ \|\widetilde{B}_{k,k_1 ,k_2   }^{+, \nu }(t, \xi)\|_{L^2}   \]
\[ \lesssim 2^{2m+2k_1+k_2 + 4k_{1,+}} \| g_{k_1}(t)\|_{L^2}\big( \| \nabla_\xi^2\widehat{g}_{k_2}(t, \xi)\|_{L^2} + 2^{-k_2} \| \nabla_\xi  \widehat{g}_{k_2}(t, \xi)\|_{L^2} + 2^{-2k_2}\| g_{k_2}(t)\|_{L^2} \big) \]
\[ \lesssim \sum_{k_1+k_2 \geq-18m/19,k_1 \geq 5\beta m} 2^{2m+\beta m +2k_1-k_2 -(N_0-10) k_{1,+}} \epsilon_1^2\lesssim 2^{-m -\beta m }\epsilon_0.
\]
Lastly, we consider the case when $k_1+k_2\geq -18 m /19$ and $k_1\leq 5\beta m $. Note that
\[
  \Gamma^1_\xi \Gamma^2_\xi  B^{\mu, \nu}_{k_1,k_2}(t, \xi)-  \int_{\R^2 } e^{i t\Phi^{+, \nu}(\xi, \eta)}     \tilde{q}_{+, \nu}(\xi-\eta, \eta) \widehat{  \Gamma^1   \Gamma^2 g_{k_1}}(t, \xi-\eta)\widehat{g_{k_2}^\nu}(t, \eta)  d\eta = \sum_{i=1}^4 K_{k_1,k_2}^{+, \nu,2;i},\]
  where
\[  K_{k_1,k_2}^{+, \nu,2;1}=  \int_{\R^2 } e^{i t\Phi^{+, \nu}(\xi, \eta)} \tilde{q}_{+, \nu}(\xi-\eta, \eta) \widehat{  g_{k_1}}(t, \xi-\eta) \widehat{\Gamma^1   \Gamma^2 g_{k_2}^\nu}(t, \eta ) d \eta,\]
\[ K_{k_1,k_2}^{+, \nu,2;2}= \sum_{j_1 \geq k_{1,-}, j_2\geq -k_{2,-}} K_{k_1,j_1,k_2,j_2}^{+, \nu,2;2}, \quad K_{k_1,j_1,k_2,j_2}^{+, \nu,2;2}:=  \sum_{(l,n)\in\{(1,2),(2,1)\}}   \int_{\R^2 } e^{i t\Phi^{+, \nu}(\xi, \eta)} \Big[\tilde{q}_{+, \nu}(\xi-\eta, \eta)\]
\[   \times \widehat{ \Gamma^l g_{k_1,j_1}}(t, \xi-\eta) \widehat{\Gamma^n g_{k_2,j_2}^\nu}(t, \eta)   + (\Gamma_\xi^l +\Gamma_\eta^l + d_{\Gamma^l} )\tilde{q}_{+, \nu}(\xi-\eta, \eta) \big(\widehat{ \Gamma^n g_{k_1,j_1}}(t, \xi-\eta) \widehat{  g_{k_2,j_2}^\nu}(t, \eta)  \]
\[
+ \widehat{ g_{k_1,j_1}}(t, \xi-\eta) \widehat{ \Gamma^n  g_{k_2,j_2}^\nu}(t, \eta) \big) 
 +it   (\Gamma_\xi^l +\Gamma_\eta^l )\Phi^{+, \nu}(\xi, \eta)(\Gamma_\xi^n +\Gamma_\eta^n  + d_{\Gamma^n})\tilde{q}_{+, \nu}(\xi-\eta, \eta) \widehat{   g_{k_1,j_1}}(t, \xi-\eta) 
  \]
  \[ 
  \times \widehat{g_{k_2,j_2}^\nu}(t, \eta)   +(\Gamma_\xi^1 +\Gamma_\eta^1  + d_{\Gamma^1}) (\Gamma_\xi^2 +\Gamma_\eta^2 +  d_{\Gamma^2})\tilde{q}_{+, \nu}(\xi-\eta, \eta) \widehat{ g_{k_1,j_1}}(t, \xi-\eta)   \widehat{   g_{k_2,j_2}^\nu}(t, \eta)  d \eta.  \] 
\[
  K_{k_1,k_2}^{+ \nu,2;3}=  \sum_{(l,n)\in\{(1,2),(2,1)\}}  \int_{\R^2 } e^{i t\Phi^{+, \nu}(\xi, \eta)}  it   (\Gamma_\xi^l +\Gamma_\eta^l )\Phi^{+, \nu}(\xi, \eta) \tilde{q}_{+, \nu}(\xi-\eta, \eta) \]
  \[ \times\big( \widehat{g_{k_2 }^\nu}(t, \eta) \widehat{\Gamma^n  g_{k_1 }}(t, \xi-\eta)+   \widehat{g_{k_1 }}(t, \xi-\eta) 
   \widehat{\Gamma^n g_{k_2 }^\nu}(t, \eta)\big) d \eta
\]
\[
  K_{k_1,k_2}^{+, \nu,2;4}=  -\int_{\R^2 } e^{i t\Phi^{+, \nu}(\xi, \eta)}  t^2 (\Gamma_\xi^1 +\Gamma_\eta^1 ) \Phi^{\mu, \nu}(\xi, \eta) (\Gamma_\xi^2 +\Gamma_\eta^2 ) \Phi^{+, \nu}(\xi, \eta)  \tilde{q}_{+, \nu}(\xi-\eta, \eta)   \widehat{g_{k_1  }}(t, \xi-\eta) \widehat{  g_{k_2  }^\nu}(t, \eta) d \eta.
\]

From the $L^2-L^\infty$ type estimate (\ref{bilinearesetimate}) in Lemma \ref{multilinearestimate}, we have
\[
\big\|  K_{k_1,k_2}^{+, \nu,2;1} \big\|_{L^2}\lesssim2^{2k_1} \| \Gamma^1 \Gamma^2 g_{k_2}(t)\|_{L^2} \| e^{-it \Lambda} g_{k_1}(t)\|_{L^\infty} \lesssim 2^{-m +\tilde{\delta} m }\epsilon_0. 
 \]

  Now, we proceed to estimate  $K_{k_1,k_2}^{+, \nu,2;2}$. By doing integration by parts in $\eta$ many times, we can rule out the case when $\max\{j_1,j_2\} \leq m +k_{1,-}-\beta m $. From the $L^2-L^\infty$ type estimate (\ref{bilinearesetimate}) in Lemma \ref{multilinearestimate}, the following estimate holds when  $\max\{j_1,j_2\} \geq  m +k_{1,-}-\beta m $,
 \[
\sum_{ \max\{j_1,j_2\} \geq  m +k_{1,-}-\beta m } \|K_{k_1,j_1,k_2,j_2}^{+, \nu,2;2}\|_{L^2} \lesssim \sum_{  j_1 \geq  \max\{ m +k_{1,-}-\beta m,j_2\}  } 2^{2k_1}\big(2^{j_1+k_1+k_2+j_2} + 2^{m+k_1+k_2}\big)\]
\[ \| g_{k_1,j_1}(t)\|_{L^2} 2^{-m}\|   g_{k_2,j_2}(t)\|_{L^1}  + \sum_{  j_2 \geq  \max\{ m +k_{1,-}-\beta m,j_1\}  } 2^{2k_1}\big(2^{j_1+k_1+k_2+j_2} + 2^{m+k_1+k_2}\big) \]
\[\times \| g_{k_2,j_2}(t)\|_{L^2} 2^{-m} \| g_{k_1,j_1}(t)\|_{L^1} \lesssim 2^{-2m-k_2 +20\beta m }\epsilon_0 \lesssim 2^{-m-\beta m }\epsilon_0.
 \]
For $ K_{k_1,k_2}^{\mu, \nu,2;3}$, we do integration by parts in ``$\eta$'' once. Meanwhile, for $ K_{k_1,k_2}^{\mu, \nu,2;4}$, we do integration by parts in ``$\eta$'' twice. As a result, we have
\[
 \|K_{k_1,k_2}^{+, \nu,2;3}\|_{L^2} + \| K_{k_1,k_2}^{+, \nu,2;4}\|_{L^2} \lesssim  \big(\sum_{i=0,1,2} 2^{i k_2} \| \nabla_\xi^i \widehat{g}_{k_2}(t, \xi)\|_{L^2}+ 2^{i k_1} \| \nabla_\xi^i \widehat{g}_{k_1}(t, \xi)\|_{L^2}   \big)\]
\[\times \big( 2^{2k_1} \| e^{-it\Lambda} g_{k_1}\|_{L^\infty}  +  2^{k_1+k_2} \| e^{-it\Lambda} g_{k_2}\|_{L^\infty} \big) + \sum_{j_1 \geq j_2} 2^{-m+3k_1+ k_2+j_1+2j_2}  \| g_{k_1,j_1}\|_{L^2}\| g_{k_2,j_2}\|_{L^2}
\]
\[
  +\sum_{j_2 \geq j_1 } 2^{-m+3k_1+ k_2+j_2+2j_1}  \| g_{k_1,j_1}\|_{L^2}\| g_{k_2,j_2}\|_{L^2} \lesssim 2^{-m+2\tilde{\delta} m +\delta m /2 +k} \epsilon_0.
\]
 Now, it is easy to see our desired estimate (\ref{eqn730}) holds. Hence finishing the proof.
\end{proof}

The rest of this section is devoted to prove the weighted norm estimates for the remainder term $\mathcal{R}_1$ in (\ref{realduhamel}), which will be done by using the fixed point type formulation (\ref{fixedpoint}). Before that, we first prove the weighted norm estimates for a very general multilinear form.

For $g_i\in H^{N_0-10}\cap Z_2\cap Z_1$, $i \in \{1,\cdots,5\}$, we define 
  a multilinear form  as follows, 
\[
Q^{\tau, \kappa, \iota}_{k,\mu, \nu}(g_1(t), g_2(t), g_3(t), g_4(t), g_5(t))(\xi):= \int_{\R^2} \int_{\R^2} \int_{\R^2} \int_{\R^2}e^{i t \Phi^{\tau, \kappa, \iota}_{\mu, \nu}(\xi, \eta, \sigma,\eta', \sigma')} q^{\tau, \kappa, \iota}_{\mu, \nu}(\xi, \eta, \sigma,\eta', \sigma') 
\]
\[
\times  \widehat{g^{\tau}_{ 1}}(t, \xi-\eta) \widehat{g^{\kappa}_{ 2}}(t, \eta-\sigma) \widehat{g^{\iota}_{ 3}}(t,  \sigma-\eta') \widehat{g^{\mu}_{ 4}}(t, \eta'-\sigma')\widehat{g^{\nu}_{ 5}}(t, \sigma') \psi_k(\xi) d \sigma' d \eta' d \eta d \sigma,
\]
where the symbol $ q^{\tau, \kappa, \iota}_{\mu, \nu}(\xi, \eta, \sigma,\eta', \sigma') $ satisfies the following estimate,
\[
\| q^{\tau, \kappa, \iota}_{\mu, \nu}(\xi, \eta, \sigma,\eta', \sigma')  \psi_k(\xi) \psi_{k_1}(\xi-\eta) \psi_{k_2}( \eta-\sigma) \psi_{k_3}(\sigma-\eta') \psi_{k_4}(\eta'-\sigma')\]
\[\times \psi_{k_5}(\eta'-\sigma') \|_{\mathcal{S}^\infty}\lesssim 2^{2k_1 + 6\max\{k_1,\cdots,k_5\}_{+}}.
\]
 We define auxiliary  function spaces as follows, 
\be\label{eqn19876}
\| f\|_{\widetilde{Z}_i}:=\sup_{k\in \mathbb{Z}} \sup_{j \geq -k_{-}} \| f\|_{\widetilde{B}_{k,j}^i}, \quad \| f\|_{\widetilde{B}_{k,j}^i}:= 2^{ (1-\delta)k +k_{+} + (20-5i)k_{+}+ i j +\delta j} \| \varphi_j^k(x)  P_k f \|_{L^2}, \quad i\in \{0,1,2\}.
\ee
From above definition, it is easy to verify that the following estimates hold, 
\[
 \sum_{k\in \mathbb{Z}} 2^{k+(20-5i)k_{+}}  \| \nabla_\xi^i \widehat{f}_k(t, \xi)\|_{L^2}    \lesssim \| f\|_{ \widetilde{Z}_i}, \quad \| f\|_{Z_l} \lesssim  \| f\|_{ \widetilde{Z}_l}, \quad i\in  \{0,1,2\}, l\in \{1,2\}.
\]

\begin{lemma}\label{quarticestimatefixed}
Let $g_i(t)\in H^{N_0-10}\cap Z_2\cap Z_1$, $i\in\{1,\cdots, 5\}$. Assume that the following estimate holds for any $t\in [2^{m-1}, 2^m]$,
\[
2^{-\delta  m } \| g_i(t)\|_{H^{N_0-10}} + 2^{-\tilde{\delta}  m } \| g_i(t)\|_{Z_2}  +\| g_i(t)\|_{Z_1} \lesssim \epsilon_1:=\epsilon_0^{5/6}, \quad i \in \{1,2,3\},
\]
then the following estimates hold for  any $t\in[2^{m-1}, 2^m]$ and  any $\mu, \nu, \kappa, \iota, \tau\in \{+,-\},$
\be\label{quinticestimateZ2}
 \sum_{i=0,1,2}     2^{( 3-i) m }    \| \mathcal{F}^{-1}\big[  Q^{\tau, \kappa, \iota}_{k,\mu, \nu}(g_{1 }(t), g_{2 }(t),g_{3 }(t),g_{4 }(t), g_{5 }(t))(\xi)\big] \|_{\widetilde{Z}_i}\lesssim 2^{-m/2+190\beta m }\epsilon_0^2. 
\ee
\end{lemma}
\begin{proof}

As usual, we  rule out the very high frequency case and the very low frequency case first. Without loss of generality, we assume that $k_5\leq k_4\leq k_3\leq k_2\leq k_1$. From the $L^2-L^\infty-L^\infty-L^\infty-L^\infty$ type multilinear estimate and the $L^\infty\rightarrow L^2$ type Sobolev estimate, we have
\[
\sum_{i=0,1,2} 2^{( 3-i) m }  \| \mathcal{F}^{-1}[Q^{\tau, \kappa, \iota}_{k,\mu, \nu}(g_{1,k_1}(t), g_{2,k_2}(t),g_{3,k_3}(t),g_{4,k_4}(t), g_{5,k_5}(t))(\xi)]\|_{ \widetilde{B}_{k,j}^i}\lesssim 2^{3m+ (2+\delta) j } \]
\be\label{eqn1889}
  \times 2^{ 30 k_{1,+} + (1-\delta)k+k_5}\| g_{k_1}\|_{L^2} \| e^{-it \Lambda}g_{k_2}\|_{L^\infty}\| e^{-it \Lambda}g_{k_3}\|_{L^\infty} \| e^{-it \Lambda}g_{k_4}\|_{L^\infty}\|  g_{k_5}\|_{L^2}.
\ee

From   estimate (\ref{eqn1889}), we can rule out the case when $k_{1,+} \geq (3m+2j)/(N_0-45)$ or $ k_5\leq -3m -2(1+2\delta)j$, or $k\leq -3m -2(1+2\delta)j$. Hence it would be sufficient to consider fixed $k, k_1,k_2,k_3,k_4$, and $k_5$ in the following range, 
\be\label{restrictedrangequartichigh}
   -3m  -2(1+2\delta)j\leq k_5, k \leq k_1 +2\leq  (3m+2j)/(N_0-45). 
\ee

	From now on, $k, k_i$, $i\in\{1,\cdots, 5\}$, are restricted inside the range  (\ref{restrictedrangequartichigh}). 
We first consider the case when $j\geq (1+\delta)\big( m +k_{1,+}\big)+ \beta m$. For this case, we do spatial localization for inputs ``$g_{k_1}$'' and ``$g_{k_2}$''.  Note that the following estimate holds for the case we are considering, 
\[
\big|\nabla_\xi \big[x\cdot \xi + t  \Phi^{\tau, \kappa, \iota}_{\mu, \nu}(\xi, \eta, \sigma,\eta', \sigma')\big]\big| \varphi_j^k(x) \sim 2^{j}.
\]
Therefore, by doing integration by parts in ``$\xi$'' many times, we can rule out the case when $\min\{j_1, j_2\}\leq j-\delta j - \delta m$, where $j_1$ and $j_2$ are the spatial concentrations of $g_{k_1}$ and $g_{k_2}$ respectively. For the case when $ \min\{j_1, j_2\} \geq j-\delta j -\delta m $, the following estimate holds from the $L^2-L^\infty-L^\infty-L^\infty-L^\infty$ type multilinear estimate, 
\[
\sum_{\min\{j_1, j_2\} \geq j-\delta j -\delta m  } \sum_{i=0,1,2}    2^{( 3-i) m }  \|\mathcal{F}^{-1}\big[ Q^{\tau, \kappa, \iota}_{k,\mu, \nu}(g_{1,k_1,j_1}(t), g_{2,k_2,j_2}(t),g_{3,k_3}(t),g_{4,k_4}(t), g_{5,k_5}(t))(\xi)\big]\|_{\widetilde{B}_{k,j}^i}
\]
\[\lesssim \sum_{i=0,1,2} \sum_{ \min\{j_1, j_2\} \geq j-\delta j -\delta m  } 2^{ (3-i)m+ i j +\delta j + 3\beta m +  (3-\delta)k_1+30 k_{1,+}} \| g_{1,k_1,j_1}\|_{L^2} 2^{k_2}\|   g_{2,k_2,j_2 }\|_{L^2} 
\]
\[
\times \| e^{-it \Lambda} g_{3,k_3}\|_{L^\infty} \| e^{-it \Lambda} g_{4,k_4}\|_{L^\infty} \| e^{-it \Lambda} g_{5,k_5}\|_{L^\infty} \lesssim 2^{-m/2 +50\beta m } \epsilon_0^2.
\]

It remains to consider the case when $j \leq (1+\delta)\big(m  + k_{1,+} \big) + \beta m $. Recall   (\ref{restrictedrangequartichigh}).   Note that $j$ now is bounded, we have $-6m \leq k_5\leq k_{1}\leq 5\beta m $.

 We split   into three cases based on  sizes of the difference between $k_1$ and $k_2$ and the difference between $k_2$ and $k_3$ as follows.

$\oplus$\quad 
 If $k_2\leq k_1-10$. \quad For this case, we have a good lower bound for $\nabla_\eta \Phi^{\tau, \kappa, \iota}_{\mu, \nu}(\xi, \eta, \sigma,\eta', \sigma')$. Hence, we can do integration by parts in $\eta$ many times to rule out the  case when $\max\{j_1, j_2\} \leq m + k_{1,-} - \beta m $. From the $L^2-L^\infty-L^\infty-L^\infty-L^\infty$ type multilinear estimate, the following estimate holds, 
\[
\sum_{ \max\{j_1, j_2\} \geq m + k_{1,-} - \beta m} \sum_{i=0,1,2}    2^{( 3-i) m }   \|\mathcal{F}^{-1}\big[ Q^{\tau, \kappa, \iota}_{k,\mu, \nu}(g_{1,k_1,j_1}(t), g_{2,k_2,j_2}(t),g_{3,k_3}(t),g_{4,k_4}(t), g_{5,k_5}(t))(\xi)\big]\|_{\widetilde{B}_{k,j}^i} \]
\[ \lesssim \sum_{j_1 \geq \max\{j_2, m + k_{1,-} - \beta m\}} 2^{3m+4\beta m + 3k_1 +30k_{1,+}}\| g_{1,k_1,j_1}\|_{L^2} \| e^{-it\Lambda} g_{2,k_2,j_2}\|_{L^\infty} \| e^{-it \Lambda} g_{3,k_3}\|_{L^\infty}
\]
\[
\times \| e^{-it \Lambda} g_{4,k_4}\|_{L^\infty} \| e^{-it \Lambda} g_{5,k_5}\|_{L^\infty} +  \sum_{j_2 \geq \max\{j_1, m + k_{1,-} - \beta m\}}2^{3m+4\beta m + 3k_1 + 30k_{1,+} + k_4+k_5}\| g_{2,k_2,j_2}\|_{L^2} 
\]
\[
\times\| e^{-it\Lambda} g_{1,k_1,j_1}\|_{L^\infty} \| e^{-it \Lambda} g_{3,k_3}\|_{L^\infty} \|  g_{4,k_4}\|_{L^2} \|   g_{5,k_5}\|_{L^2} \lesssim 2^{-m/2+180\beta m }\epsilon_0^2.
\]

$\oplus$\quad If $|k_1-k_2|\leq 10$ and $k_3\leq k_1-20$. \quad Note that, $\nabla_\sigma \Phi^{\tau, \kappa, \iota}_{\mu, \nu}(\xi, \eta, \sigma,\eta', \sigma')$ has a good lower bound for the case we are considering. Hence, by doing integration by parts in $\sigma$, we can rule out the case when $\max\{j_2,j_3\}\leq m +k_{2,-}-\beta m $, where $j_2$ and $j_3$ are the spatial concentrations of inputs $g_{k_2}$ and $g_{k_3}$ respectively.  From the $L^2-L^\infty-L^\infty-L^\infty-L^\infty$ type multilinear estimate, the following estimate holds, 
\[
\sum_{ \max\{j_2, j_3\} \geq m + k_{2,-} - \beta m}  \sum_{i=0,1,2}    2^{( 3-i) m }  \|\mathcal{F}^{-1}\big[ Q^{\tau, \kappa, \iota}_{k,\mu, \nu}(g_{1,k_1 }(t), g_{2,k_2,j_2}(t),g_{3,k_3,j_3}(t),g_{4,k_4}(t), g_{5,k_5}(t))(\xi)\big]\|_{\widetilde{B}_{k,j}^i} \]
\[
  \lesssim \sum_{j_2 \geq \max\{j_3, m + k_{1,-} - \beta m\}} 2^{3m+4\beta m + 3k_1+ 30k_{1,+}}\| g_{3,k_3,j_3}\|_{L^2} \| e^{-it\Lambda} g_{2,k_2,j_2}\|_{L^\infty} \| e^{-it \Lambda} g_{1,k_1}\|_{L^\infty} 
\]
\[
\times \| e^{-it \Lambda} g_{4,k_4}\|_{L^\infty}  \| e^{-it \Lambda} g_{k_5}\|_{L^\infty} +  \sum_{j_3 \geq \max\{j_2, m + k_{1,-} - \beta m\}}2^{3m+4\beta m +3k_1 + 30k_{1,+} + k_4+k_5}\| g_{3,k_3,j_3}\|_{L^2} 
\]
\[
\times\| e^{-it\Lambda} g_{2,k_2,j_2}\|_{L^\infty} \| e^{-it \Lambda} g_{1,k_1}\|_{L^\infty} \|  g_{4,k_4}\|_{L^2} \|   g_{5,k_5}\|_{L^2} \lesssim 2^{-m/2+180\beta m }\epsilon_0^2.
\]

$\oplus$\quad If $|k_1-k_2|\leq 10$ and $|k_2-k_3|\leq 10$. This case is straightforward. By the $L^2-L^\infty-L^\infty-L^\infty-L^\infty$ type multilinear estimate, the following estimate holds, 
\[
 \sum_{i=0,1,2}    2^{( 3-i) m }  \|\mathcal{F}^{-1}\big[ Q^{\tau, \kappa, \iota}_{k,\mu, \nu}(g_{1,k_1 }(t), g_{2,k_2 }(t),g_{3,k_3 }(t),g_{4,k_4}(t), g_{5,k_5}(t))(\xi)\big]\|_{\widetilde{B}_{k,j}^i}  \]
 \[\lesssim 2^{3m+4\beta m +3k_{1}+30k_{1,+}}  \| g_{5,k_5}(t)\|_{L^2}  \| e^{-it\Lambda} g_{1,k_1}(t)\|_{L^\infty} \| e^{-it\Lambda} g_{2,k_2}(t)\|_{L^\infty} 	\| e^{-it\Lambda} g_{3,k_3}(t)\|_{L^\infty} \]
 \[
 \times\| e^{-it\Lambda} g_{4,k_4}(t)\|_{L^\infty} \lesssim 2^{-m/2+180\beta m }\epsilon_0^2.
\]
To sum up, our desired estimate (\ref{quinticestimateZ2}) holds, hence finishing the proof.
\end{proof}

In the following, we will use a fixed point type argument to estimate the Remainder terms. Recall (\ref{normalformatransfor}) and $ u = \tilde{\Lambda} h + i \tilde{\psi}$. To estimate the weighted norms of the reminder term $\mathcal{R}_1$, from estimate (\ref{quinticestimateZ2}) in Lemma \ref{quarticestimatefixed}, we know that it would be sufficient to estimate the weighted norms, i.e., $ {Z}_1$ norm and $ {Z}_2$ norm, of $  e^{it \Lambda  } \Lambda_{\geq 5}[B(h)\psi]=  e^{it \Lambda  }\Lambda_{\geq 5}[\nabla_{x,z}\varphi](t  )\big|_{z=0}$.

\begin{lemma}\label{L2estimateremaider}
	Under the bootstrap assumption \textup{ (\ref{smallness})}, 
the following estimate holds for any $t \in[2^{m-1}, 2^m]$,
\be\label{eqn19878}
  \|   \Lambda_{\geq i}[\nabla_{x,z}\varphi]\|_{L^\infty_z H^{20}}   \lesssim 2^{-im/2+ \beta m }\epsilon_0^2, \quad   i\in \{5,6\}. 
\ee
\end{lemma}
\begin{proof}
  Recall the fixed point type formulation for $\nabla_{x,z}\varphi$ in (\ref{fixedpoint}). Very similar to the proof of (\ref{eqn2202}) in Lemma \ref{Sobolevestimate}, we can derive the following estimate for $i\in \{5,6\},$
\[
\| \Lambda_{\geq i}[\nabla_{x,z}\varphi]\|_{L^\infty_z H^{20}}\lesssim  \|(h, \psi)\|_{W^{30,1}}\|(h, \psi)\|_{W^{30,0}}^{i-2} \| (h, \psi)\|_{H^{30}} + \|(h, \psi)\|_{W^{30}} \| \Lambda_{\geq i}[\nabla_{x,z}\varphi]\|_{L^\infty_z H^{20}},
\]
which further implies the following estimate, 
\[
  \| \Lambda_{\geq i}[\nabla_{x,z}\varphi]\|_{L^\infty_z H^{20}}\lesssim  \|(h, \psi)\|_{W^{30,1}}\|(h, \psi)\|_{W^{30,0}}^{i-2} \| (h, \psi)\|_{H^{30}} \lesssim 2^{-im/2+\beta m }\epsilon_0^2,\quad i\in\{5,6\}.
\]
Hence, it is easy to see our desired estimate holds.

\end{proof}
Now we will use above estimate of Sobolev norm to estimate the weighted norms of $ e^{it \Lambda  } \Lambda_{\geq 5}[\nabla_{x,z}\varphi](t  )$.  More precisely, we have
\begin{lemma}\label{remaindertermweightednorm}
Under the bootstrap assumption \textup{(\ref{smallness})}, the following estimate holds for the remainder term $\mathcal{R}_1$,
 \be\label{eqnj878}
\sup_{t\in[2^{m-1}, 2^m]}  \sum_{i= 1,2}      \| e^{it \Lambda} \mathcal{R}_1  \|_{Z_i} \lesssim  	2^{-3m/2+200\beta m }\epsilon_0. 
\ee
\end{lemma}
\begin{proof}
Recall the fixed point type formulation for $\nabla_{x,z}\varphi$ in (\ref{fixedpoint}). We decompose $\Lambda_{\geq 5}[g_i(z)]$ into two parts: one of them doesn't depend on $\Lambda_{\geq 5}[\nabla_{x,z}\varphi]$ while the other part does depend (linearly depend) on  $\Lambda_{\geq 5}[\nabla_{x,z}\varphi]$. For the first part, estimate  (\ref{quinticestimateZ2}) in Lemma \ref{quarticestimatefixed} is very sufficient. Hence, it remains to estimate the second part. As usual, by doing integration by parts in $\xi$ many times, we can rule out the case  when $j\geq (1+\delta)\big(\max\{m +k_{1,+}, -k_{-}\}\big)	  +\beta m $. It remains to consider the case when  $j\leq (1+\delta)\big( \max\{m +k_{1,+}, -k_{-}\}	\big) +\beta m  $. For this case, the following estimate holds from estimate (\ref{quinticestimateZ2}) in Lemma \ref{quarticestimatefixed}, estimate  (\ref{eqn19878}) in Lemma \ref{L2estimateremaider} and $L^2-L^\infty$ type bilinear estimate,
\[
\sum_{i=1,2}  \|  e^{it \Lambda  } \Lambda_{\geq 5}[\nabla_{x,z}\varphi](t  )  \|_{L^\infty_z  {Z}_i} \lesssim 2^{-3m/2+200\beta m }\epsilon_0 + 2^{2m+3\beta m} \big(   \| e^{-it \Lambda} g \|_{W^{20,0}} \]
\[\times \|\Lambda_{\geq 6}[\nabla_{x,z}\varphi](t, \xi) \|_{L^\infty_z H^{15}} + \| g\|_{H^{20}} \| \d \Lambda_{5}[\nabla_{x,z}\varphi]\|_{L^\infty_z H^{15}}\big)+ 2^{3\beta m } \| g\|_{H^{20}}
\]
\be
 \times  \|\Lambda_{\geq 5}[\nabla_{x,z}\varphi](t, \xi) \|_{L^\infty_z H^{20}} \lesssim 2^{-3m/2 +200\beta m }\epsilon_0.
\ee
 From above estimate and estimate \textup{(\ref{eqn19878}) in Lemma \ref{L2estimateremaider}}, now it is easy to derive our desired estimate (\ref{eqnj878}) for the remainder term $\mathcal{R}_1$.
\end{proof}

\end{document}